\documentclass{amsart}

\usepackage{amssymb,amsmath,amscd,bbm,a4wide,graphicx,tikz,hyperref,amsbsy,enumerate,mathtools,stmaryrd}
\usepackage[all]{xy}
\usepackage{color}

\makeatletter

\def\chaptermark#1{}

\def\chapter{%
  \if@openright\cleardoublepage\else\clearpage\fi
  \thispagestyle{plain}\global\@topnum\z@
  \@afterindenttrue \secdef\@chapter\@schapter}

\def\@chapter[#1]#2{\refstepcounter{chapter}%
  \ifnum\c@secnumdepth<\z@ \let\@secnumber\@empty
  \else \let\@secnumber\thechapter \fi
  \typeout{\chaptername\space\@secnumber}%
  \def\@toclevel{0}%
  \ifx\chaptername\appendixname \@tocwriteb\tocappendix{chapter}{#2}%
  \else \@tocwriteb\tocchapter{chapter}{#2}\fi
  \chaptermark{#1}%
  \addtocontents{lof}{\protect\addvspace{10\p@}}%
  \addtocontents{lot}{\protect\addvspace{10\p@}}%
  \@makechapterhead{#2}\@afterheading}
\def\@schapter#1{\typeout{#1}%
  \let\@secnumber\@empty
  \def\@toclevel{0}%
  \ifx\chaptername\appendixname \@tocwriteb\tocappendix{chapter}{#1}%
  \else \@tocwriteb\tocchapter{chapter}{#1}\fi
  \chaptermark{#1}%
  \addtocontents{lof}{\protect\addvspace{10\p@}}%
  \addtocontents{lot}{\protect\addvspace{10\p@}}%
  \@makeschapterhead{#1}\@afterheading}
\newcommand\chaptername{Chapter}

\def\@makechapterhead#1{\global\topskip 7.5pc\relax
  \begingroup
  \fontsize{\@xivpt}{18}\bfseries\centering
    \ifnum\c@secnumdepth>\m@ne
      \leavevmode \hskip-\leftskip
      \rlap{\vbox to\z@{\vss
          \centerline{\normalsize\mdseries
              \uppercase\@xp{\chaptername}\enspace\thechapter}
          \vskip 3pc}}\hskip\leftskip\fi
     #1\par \endgroup
  \skip@34\p@ \advance\skip@-\normalbaselineskip
  \vskip\skip@ }
\def\@makeschapterhead#1{\global\topskip 7.5pc\relax
  \begingroup
  \fontsize{\@xivpt}{18}\bfseries\centering
  #1\par \endgroup
  \skip@34\p@ \advance\skip@-\normalbaselineskip
  \vskip\skip@ }
\def\appendix{\par
  \c@chapter\z@ \c@section\z@
  \let\chaptername\appendixname
  \def\thechapter{\@Alph\c@chapter}}

\newcounter{chapter}

\newif\if@openright

\makeatother

\newtheorem{theorem}{Theorem}[section]
\newtheorem{proposition}[theorem]{Proposition}
\newtheorem{lemma}[theorem]{Lemma}
\newtheorem{corollary}[theorem]{Corollary}

\theoremstyle{remark}
\newtheorem{remark}[theorem]{Remark}
\newtheorem{example}[theorem]{Example}
\newtheorem{definition}[theorem]{Definition}

\numberwithin{equation}{chapter}
\numberwithin{figure}{chapter}
 \renewcommand{\thesection}{\thechapter.\arabic{section}}

\newcommand{\N}{\mathbb{N}}
\newcommand{\Z}{\mathbb{Z}}
\newcommand{\Q}{\mathbb{Q}}
\newcommand{\R}{\mathbb{R}}

\newcommand{\PP}{\mathbb{P}}

\newcommand{\Lg}{\mathcal{L}}
\newcommand{\A}{\mathcal{A}}
\newcommand{\Ra}{\mathcal{R}}
\newcommand{\Co}{\mathcal{C}}

\newcommand{\bu}{\mathbf{u}}
\newcommand{\bv}{\mathbf{v}}
\newcommand{\bz}{\mathbf{z}}
\newcommand{\bt}{\mathbf{t}}

\newcommand{\bw}{\mathbf{w}}
\newcommand{\bx}{\mathbf{x}}
\newcommand{\be}{\mathbf{e}}
\newcommand{\by}{\mathbf{y}}
\newcommand{\bone}{\mathbf{1}}
\newcommand{\bM}{\mathbf{M}}
\newcommand{\bsigma}{{\boldsymbol{\sigma}}}

\begin{document} 
\title[${S}$-adic sequences]{Chapter 3 \\[0.7cm]  $\boldsymbol{S}$-adic sequences.\\ A bridge between dynamics, arithmetic, and geometry}
\author{J\"org M.~Thuswaldner}
\address{J\"org M.~Thuswaldner, Department Mathematics and Information Technology, University of Leoben, Franz-Josef-Strasse~18, A-8700 Leoben}
\email{joerg.thuswaldner@unileoben.ac.at}

\maketitle

\begin{abstract}
A Sturmian sequence is an infinite nonperiodic string over two letters with minimal subword complexity. In two papers, the first written by Morse and Hedlund in 1940 and the second by Coven and Hedlund in 1973, a surprising correspondence was established between Sturmian sequences on one side and rotations by an irrational number on the unit circle on the other. In 1991 Arnoux and Rauzy observed that an induction process (invented by Rauzy in the late 1970s), related with the classical continued fraction algorithm, can be used to give a very elegant proof of this correspondence. This process, known as the Rauzy induction, extends naturally to interval exchange transformations (this is the setting in which it was first formalized).
It has been conjectured since the early 1990s that these correspondences carry over to rotations on higher dimensional tori, generalized continued fraction algorithms, and so-called $S$-adic sequences generated by substitutions. The idea of working towards such a generalization is known as Rauzy's program. Recently Berth\'e, Steiner, and Thuswaldner made some progress on Rauzy's program and were indeed able to set up the conjectured generalization of the above correspondences. Using a generalization of Rauzy's induction process in which generalized continued fraction algorithms show up, they proved that under certain natural conditions an $S$-adic sequence gives rise to a dynamical system which is measurably conjugate to a rotation on a higher dimensional torus. Moreover, they established a metric theory which shows that counterexamples like the one constructed in 2000 by Cassaigne, Ferenczi, and Zamboni are rare.
It is the aim of the present chapter to survey all these ideas and results.
\end{abstract}

\thispagestyle{empty}

\makeatletter
\def\@makechapterhead#1{%
  \vspace*{10\p@}%
  {\parindent \z@ \raggedright \normalfont
    \interlinepenalty\@M
    \Large \bfseries #1\par\nobreak
    \vskip 0\p@
  }}
\def\@makeschapterhead#1{%
  \vspace*{10\p@}%
  {\parindent \z@ \raggedright
    \normalfont
    \interlinepenalty\@M
    \Large \bfseries  #1\par\nobreak
    \vskip 0\p@
  }}
\makeatother

\addtocontents{toc}{\protect\setcounter{tocdepth}{-1}}
\chapter[$S$-adic Sequences]{}
\addtocontents{toc}{\protect\setcounter{tocdepth}{3}}

\thispagestyle{empty}

\setcounter{tocdepth}{3}
\tableofcontents

\newpage

\section{Introduction}
A {\em Sturmian sequence} is an infinite string over two letters with low subword complexity. In particular, it has exactly $n+1$ different subwords of a given length $n\in\N$. Sturmian sequences have been studied extensively in the literature from various points of view and we refer to Lothaire~\cite[Chapter~2]{Lothaire:02} or Pytheas Fogg~\cite[Chapter~6]{Fog02} for detailed accounts. The history of the research surveyed in the present chapter starts with two papers written by Morse and Hedlund~\cite{Morse&Hedlund:1940} as well as Coven and Hedlund~\cite{Coven&Hedlund:1973} in 1940 and 1973, respectively. In these papers the authors established a surprising correspondence between Sturmian sequences and rotations by an irrational number $\alpha$ on the torus $\mathbb{T}=\mathbb{R}/\mathbb{Z}$.  In their proof ``balance properties'' of Sturmian sequences play a prominent role. Several decades later, Arnoux and Rauzy~\cite{Arnoux-Rauzy:91} observed that an induction process in which the classical continued fraction algorithm appears can be used to give another very elegant proof of this correspondence (see also Rauzy's earlier papers~\cite{Rauzy:77,Rauzy:79} on this induction process). Their proof also shows how arithmetic and Diophantine properties of an irrational number $\alpha$ are encoded in the corresponding Sturmian sequence. 

It has been conjectured since the early 1990s that these correspondences between rotations on $\mathbb{T}$, continued fractions, and Sturmian sequences carry over to rotations on higher dimensional tori, generalized continued fraction algorithms, and so-called $S$-adic sequences generated by substitutions. The idea of working towards such a generalization is known as \emph{Rauzy's program} and starting with Rauzy~\cite{Rauzy:82} a number of examples which hint at such a generalization was devised. A natural class of $S$-adic sequences to study in this context are so-called {\em Arnoux-Rauzy sequences} which go back to Arnoux and Rauzy~\cite{Arnoux-Rauzy:91}. These are sequences over three letters that behave analogously to Sturmian sequences in many regards. However, in 2000 Cassaigne, Ferenczi, and Zamboni~\cite{Cassaigne-Ferenczi-Zamboni:00} could construct Arnoux-Rauzy sequences with strong ``imbalance'', a property which cannot occur for a Sturmian sequence. Cassaigne, Ferenczi, and Messaoudi~\cite{Cassaigne-Ferenczi-Messaoudi:08} even constructed Arnoux-Rauzy sequences that give rise to weakly-mixing dynamical systems which are far from rotations in their dynamical behavior. All this shows the limitations of Rauzy's program and indicates that the situation in the general setting is more complicated than it is in the classical case.

Nevertheless, recently Berth\'e, Steiner, and Thuswaldner~\cite{Berthe-Steiner-Thuswaldner} made some progress on Rauzy's program and were indeed able to set up the conjectured generalization of the above correspondences. Using a generalization of Rauzy's induction process in which generalized continued fraction algorithms show up, they proved that under certain natural conditions an $S$-adic sequence gives rise to a dynamical system which is measurably conjugate to a rotation on a higher dimensional torus. Moreover, they established a metric theory which shows that exceptional cases like the ones constructed in \cite{Cassaigne-Ferenczi-Messaoudi:08} and \cite{Cassaigne-Ferenczi-Zamboni:00} are rare. A prominent role in this generalization is played by tilings induced by generalizations of the classical \emph{Rauzy fractal} introduced by Rauzy~\cite{Rauzy:82}.

Another idea which can be linked to the above results goes back to Artin~\cite{Artin:24}, who observed that the classical continued fraction algorithm and its natural extension can be viewed as a Poincar\'e section of the geodesic flow on the space of two-dimensional lattices $\mathrm{SL}_{2}(\mathbb{Z})\setminus \mathrm{SL}_{2}(\mathbb{R})$. Arnoux and Fisher~\cite{AF:01} revisited Artin's idea and showed that the correspondence between continued fractions, rotations, and Sturmian sequences can be interpreted in a very nice way in terms of an extension of this geodesic flow to pointed lattices which is called the \emph{scenery flow}. Currently, Arnoux {\it et al.}~\cite{ABMST:18} are setting up a generalization of this connection between continued fraction algorithms and geodesic flows. In particular, they code the {\em Weyl Chamber Flow}, a diagonal $\R^{d-1}$-action on the space of $d$-dimensional lattices $\mathrm{SL}_{d}(\mathbb{Z})\setminus \mathrm{SL}_{d}(\mathbb{R})$, arithmetically and geometrically by generalized continued fraction algorithms. In this coding, which provides a new view of the relation between $S$-adic sequences and rotations on higher dimensional tori, non-stationary Markov partitions defined in terms of generalized Rauzy fractals are of great importance.

\medskip

It is the aim of the present chapter to survey all these ideas and results. In Section~\ref{sec:sturm} we deal with the case of Sturmian sequences and Section~\ref{sec:problems3} discusses the problems with the extension of the theory to the more general situation. From Section~\ref{sec:genset} onwards we set up the general theory of $S$-adic sequences and their relation to generalized continued fraction algorithms and rotations on higher dimensional tori.

\section{The classical case}\label{sec:sturm}

We start our journey by giving some elements of the interaction between Sturmian sequences, the classical continued fraction algorithm, and irrational rotations on the circle. After that we discuss natural extensions of continued fractions and show how all these objects turn up in the study of the geodesic flow acting on the space $\mathrm{SL}_2(\mathbb{Z})\backslash \mathrm{SL}_2(\mathbb{R})$ of lattices and its extension to pointed lattices. We will prove most of the results that we state and although our exposition is self-contained we recommend the reader to have a look at the survey \cite[Chapter~6]{Fog02} in order to find more background information on the subject of this section. 

\subsection{Sturmian sequences and their basic properties}\label{sec:sturmianprop}
For a finite set $\{1,2,\ldots,d\}$ denote by $\{1,2,\ldots,d\}^*$ the set of all finite \emph{words} $v_{0}\ldots v_{n-1}$ whose \emph{letters} $v_{i}$, $0\le i <n$, are contained in $\{1,2,\ldots,d\}$. Moreover, let $\{1,2,\ldots, d\}^\mathbb{N}$ be the space of \emph{(right-infinite) sequences} $w=w_{0}w_{1}\ldots$ whose letters $w_{i}$, $i\in\N$, are elements of $\{1,2,\ldots,d\}$. The \emph{shift} $\Sigma:\{1,2,\ldots, d\}^\mathbb{N}\to\{1,2,\ldots, d\}^\mathbb{N}$ on this space of sequences is defined by $\Sigma(w_0w_1\ldots)=w_1w_2\ldots$ 
Let $w=w_{0}w_{1}\ldots \in \{1,2,\ldots, d\}^\mathbb{N}$ be a sequence. A {\em factor} (or \emph{subword}) of $w$ is a word $v_{0}\ldots v_{n-1} \in \{1,2,\ldots,d\}^*$ for which there is $k \ge 0$ such that $w_{k}\ldots w_{k+n-1}= v_{0}\ldots v_{n-1}$. In this case we say that $v$ occurs in $w$ at position $k$. The \emph{complexity function} $p_w:\mathbb{N}\to\mathbb{N}$ of $w$ assigns to each integer $n$ the number of words $v_{0}\ldots v_{n-1}\in\{1,2,\ldots,d\}^*$ that are factors of $w$. If $w$ is \emph{ultimately periodic} in the sense that there exist $k>0$ and $N\ge 0$ with $w_n=w_{n+k}$ for each $n\ge N$ then $p_w$ is a bounded function. On the other hand, a result by Coven and Hedlund~\cite{Coven&Hedlund:1973} which is not hard to prove states that a sequence $w\in\{1,2,\ldots, d\}^\N$ that admits the inequality $p_w(n)\le n$ for a single choice of $n$ is ultimately periodic (see also~\cite[Proposition~1.1.1]{Fog02}). It is the class of not ultimately periodic sequences with smallest complexity function that we are interested in.

\begin{definition}[Sturmian sequence]\label{def:sturm}\index{Sturmian sequence}\index{sequence!Sturmian}
A sequence $w\in\{1,2\}^\N$ is called a \emph{Sturmian sequence} if its complexity function satisfies $p_w(n)=n+1$ for all $n\in\N$.
\end{definition}

It is {\it a priori} not clear that Sturmian sequences exist at all. However, we will see in Theorem~\ref{th:sturmrot} below that they can be characterized as so-called \emph{natural codings} of irrational rotations which are easy to construct (and will be defined in Section~\ref{sec:sturmrot}).

A detailed account on the early history of Sturmian sequences, which goes back to Johann Bernoulli~\cite{Bernoulli:1772}, is given in \cite[Notes to Chapter~2]{Lothaire:02}. The name ``Sturmian sequence'' was coined in 1940 by Morse and Hedlund~\cite{Morse&Hedlund:1940}. Sturmian sequences have been studied extensively. For an overview on fundamental properties of Sturmian sequences we refer in particular to Lothaire~\cite[Chapter~2]{Lothaire:02}, Pytheas Fogg~\cite[Chapter~6]{Fog02}, or Allouche and Shallit~\cite[Chapters~9 and~10]{Allouche-Shallit:03}. Belov~{\it et al.}~\cite{Belov-Kondakov-Mitrofanov:11} discuss some aspects of Sturmian sequences which are related to the present survey.

We start with the discussion of basic properties of Sturmian sequences. The fact that $p_w(n)=n+1$ holds for a Sturmian sequence entails that for each $n$ there is only one factor $v_{0}\ldots v_{n-1}$ of $w$ with the property that both words $v_{0}\ldots v_{n-1}1$ and $v_{0}\ldots v_{n-1}2$ are factors of $w$. Such a word $v_{0}\ldots v_{n-1}$ is called \emph{right special factor} of $w$. Left special factors are defined analogously.

Our first lemma deals with \emph{recurrence} of Sturmian sequences.  Recall that a sequence $w\in \{1,2\}^\N$ is called \emph{recurrent} if each factor of $w$ occurs infinitely often, {\it i.e.}, at infinitely many positions, in~$w$.

\begin{lemma}[{{\em cf.~e.g.}~\cite[Proposition~6.1.2]{Fog02}}]\label{lem:sturmianrec}
A Sturmian sequence is recurrent.
\end{lemma}

\begin{proof}
Suppose that this is wrong and let $w$ be a nonrecurrent Sturmian sequence. Then there exists a factor $v$ of length $n$, say, that occurs only finitely many times in $w$. Then there exists $k\in \N$ such that $w'=\Sigma^kw$ does not contain $v$ as a factor. However, as $p_w(n)=n+1$ this implies that $p_{w'}(n)\le n$ and, hence, $w'$ is ultimately periodic. However, then also $w$ is ultimately periodic, a contradiction.
\end{proof}

Next we discuss \emph{balance}. To give a formal definition we introduce some notation. For a word $v\in\{1,2\}^*$ we denote by $|v|$ its \emph{length}, {\it i.e.}, the number of letters of $v$. Moreover, for $i\in\{1,2\}$, we write $|v|_i$ for the number of occurrences of the letter $i$ in $v$. 

\begin{definition}[Balanced sequence]\label{def:sturmbalance}\index{balance}
A sequence $w\in\{1,2\}^\N$ is called \emph{balanced} if each pair of factors $(v,v')$ of $w$ with $|v|=|v'|$  satisfies $\big| |v|_1-|v'|_1\big| \le 1$.  
\end{definition}

As was observed already in \cite{Morse&Hedlund:1940}, there is a tight relation between Sturmian sequences and balance. 

\begin{proposition}\label{prop:sturmbalance}
Let $w\in\{1,2\}^\N$ be given. Then $w$ is a Sturmian sequence if and only if $w$ is not ultimately periodic and balanced.
\end{proposition}

The proof of this result is combinatorial. It is based on the observation that for a sequence $w$ which is not balanced there is a word $v\in \{1,2\}^*$ such that $1v1$ and $2v2$ are factors of $w$. Since the details are a bit tricky we do not give them here and refer the reader to \cite{Morse&Hedlund:1940} or \cite[Chapter~6, p.~147{\it ff}\,]{Fog02}.

The fact that Sturmian sequences are balanced will now be exploited in order to prove that they can be \emph{coded} using the \emph{Sturmian substitutions}
\begin{equation}\label{eq:sturmsubs}
\sigma_1: 
\begin{cases} 1 \mapsto 1,\\ 2 \mapsto 21,\end{cases} \qquad
\sigma_2: 
\begin{cases} 1 \mapsto 12,\\ 2 \mapsto 2.\end{cases} 
\end{equation}
The domain of these substitutions can naturally be extended from $\{1,2\}$ to $\{1,2\}^*$ and $\{1,2\}^\mathbb{N}$ by concatenation. The next statement essentially says that balance is maintained by ``desubstitution''.

\begin{lemma}[{see {\em e.g.}~\cite[Lemma~4.2]{AF:01}}]\label{lem:balbal}
If a sequence $w\in\{1,2\}^\mathbb{N}$ is  not balanced, then for each $a\in\{1,2\}$ the sequence $\sigma_1(aw)$ is not balanced.
\end{lemma}

\begin{proof}
If $w$ is not balanced it is easy to see that there are words $u$ and $v$ with $|u|=|v|$ and $|u|_1=|v|_1$ such that $1u1$ and $2v2$ are factors of $w$. Since $1u1$ occurs in $w$ there is $b\in\{1,2\}$ such that $b1u1$ occurs in $aw$ (we need $a$ in case $1u1$ is the initial word of $w$). As $\sigma_1(b)$ always ends with $1$ and $\sigma_1(2)$ begins with $2$, the words $11\sigma_1(u)1$ and $21\sigma_1(v)2$ have the same length and both occur in $\sigma_1(aw)$. As the number of 1s in these two words clearly differs by 2 the lemma follows.
\end{proof}

Let $w=w_0w_1\ldots\in\{1,2\}^\mathbb{N}$ be given. If $w$ is a Sturmian sequence, it contains exactly three of the four factors $11,12,21,22$. Since it clearly contains $12$ and $21$ as factors, it either doesn't contain $22$, in which case we say that $w$ is of \emph{type~1}, or it doesn't contain $11$, in which case we say it is of \emph{type~2}.
Using recurrence
one can easily see that for each Sturmian sequence $w\in\{1,2\}^\mathbb{N}$ at least one of the sequences $1w$ and $2w$ is Sturmian as well.
A Sturmian sequence $w\in\{1,2\}^\mathbb{N}$ is called \emph{special} if $1w$ as well as $2w$ are both Sturmian sequences.  With these notions we get the following ``desubstitution'' of Sturmian sequences (see also \cite[Section~1]{Arnoux-Rauzy:91} where an analog of this was proved along somewhat different lines). 

\begin{lemma}[{see {\em e.g.}~\cite[Proposition~4.3]{AF:01}}]\label{prop:sturmDesubs}
Let $u$ be a Sturmian sequence of type~1. 
\begin{enumerate}
\item[(i)] If $u$ is not special then either $u=\sigma_1(v)$ with $v$ Sturmian, or $u=\Sigma\sigma_1(v)$ with $v$ Sturmian starting with $2$ (but not both).
\item[(ii)] If $u$ is special then $u=\sigma_1(v_1)=\Sigma\sigma_1(v_2)$ where $\Sigma v_1=\Sigma v_2$ is a special Sturmian sequence.
\end{enumerate}
If $u$ is of type~2 the same statement with the symbols $1$ and $2$ interchanged holds.
\end{lemma}

\begin{proof}
Since $u$ is of type~1 it is immediate that it can be written as $u=\sigma_1(v)$ for some $v\in\{1,2\}^\N$. 

To prove (i) suppose that $u$ is not special. Then either $1u$ or $2u$ is Sturmian, but not both. 

If $1u$ is Sturmian, $1u=\sigma_1(v')$ with $v'$ starting with $1$ and, hence, by Lemma~\ref{lem:balbal} and Proposition~\ref{prop:sturmbalance}, $v=\Sigma v'$ is Sturmian, and $u=\sigma_1(v)$. If $u$ starts with $2$ then $u\not=\Sigma\sigma_1(v')$ for $v'$ starting with $2$. If $u$ starts with $1$ then also $v$ starts with $1$. If we replace the first letter of $v$ by $2$ this yields a sequence $w$ satisfying $u=\Sigma\sigma_1(w)$. However, if $w$ is also Sturmian $\Sigma v=\Sigma w$ is special and, hence, one easily checks that $u$ is special, a contradiction and we are done. 

If $2u$ is Sturmian then $12u$ has to be Sturmian (since $22$ is forbidden) and thus $12u=\sigma_1(1v)$ with $v$ Sturmian and beginning with $2$. Thus $u=\Sigma\sigma_1(v)$. 
As before, we can write $u=\sigma_1(w)$ where $w$ is the word obtained from $v$ by replacing the first letter by $1$.
This leads again to the contradiction of $u$ being special.

To show (ii) assume $u$ is special. Then, as $u$ has to start with $1$ the sequences $12u=\sigma_1(12v)$ and $21u=\sigma_1(21v)$ are Sturmian ($11u$ cannot be Sturmian for imbalance reasons, see \cite[Proposition~6.1.23]{Fog02}). By Lemma~\ref{lem:balbal} and Proposition~\ref{prop:sturmbalance} the sequences $1v$ and $2v$ are Sturmian, so $v$ is special and $u=\sigma_1(1v)=\Sigma\sigma_1(2v)$. 

The proof of the type~2 case is analogous.
\end{proof}

From the proof of Lemma~\ref{prop:sturmDesubs} we see that for a special sequence $u$ of type~1 there exists a special sequence $v$ such that $21u=\sigma_1(21v)$ and $12u=\sigma_1(12v)$ are Sturmian sequences. If $u$ is special of type~2 we get the existence of a special sequence $v$ with $21u=\sigma_2(21v)$ and $12u=\sigma_2(12v)$ Sturmian by analogous reasoning.  
If $u$ is a special Sturmian sequence then the two Sturmian sequences $12u$ and $21u$ are called \emph{limit sequences} or \emph{fixed sequences}. By the above arguments they can be desubstituted to sequences that are limit sequences as well. This process can be iterated: let $w$ be a limit sequence. Then there is a sequence $(w^{(n)})_{n\ge 0}$ of limit sequences with 
\[
w=w^{(0)} \quad\hbox{and}\quad w^{(n)}=\sigma_{i_n}(w^{(n+1)}) \hbox{ for }n\ge 0.
\]
This can be rewritten as 
\begin{equation}\label{eq:2lettercoding}
w=\sigma_{i_0}\circ\cdots\circ\sigma_{i_n}(w^{(n+1)}). 
\end{equation}
As $w$ is Sturmian, the sequence $(i_n)\in\{1,2\}^\N$ has to change its value infinitely often because otherwise $w$ would be ultimately constant. Now observe that a sequence $w^{(n)}$ starting with a letter $a$ results in a sequence $w^{(0)}$ also starting with $a$. Moreover, since the sequence $(i_n)$ changes its value infinitely often we see that the first letter of $w^{(n)}$ determines a prefix of $w$ whose length tends to infinity with $n$. Thus, equipping $\{1,2\}^\N$ with the product topology of the discrete topology yields
\begin{equation}\label{eq:sadic2letters}
w=\lim_{n\to\infty}\sigma_{i_0}\circ\cdots\circ\sigma_{i_n}(a),
\end{equation}
where $a$ is the first letter of $w$ (note that we slightly abuse notation here: to be exact the argument of $\sigma_{i_n}$ should be $aa\ldots\in\mathcal{A}^\N$ since the limit is not defined for finite words). We could also group the blocks of the sequence $(i_n)$. So if it starts with a block of $a_0$ times the symbol 1 followed by a block of $a_1$ times the symbol 2 and so on we can rewrite \eqref{eq:sadic2letters} as
\begin{equation}\label{eq:sadic2lettersMult}
w=\lim_{k\to\infty}\sigma_{1}^{a_0}\circ\sigma_{2}^{a_1}\circ\sigma_{1}^{a_2}\circ\cdots\circ\sigma_{1}^{a_{2k}}(a).
\end{equation}
A sequence $w$ that can be represented by iteratively composing substitutions as in \eqref{eq:2lettercoding} is called an \emph{$S$-adic sequence}.

Note that for arbitrary Sturmian sequences a similar coding as in \eqref{eq:2lettercoding} is possible, however, in the general case shifts have to be inserted between the composed substitutions on the appropriate places according to Lemma~\ref{prop:sturmDesubs}. Inserting these shifts does not change the collection of factors (called \emph{language}) of the sequence. Thus each Sturmian sequence $w$ is associated with a sequence $(\sigma_{i_m})$ which determines its language. We call this sequence the \emph{coding sequence} of $w$. Summing up we proved the following proposition.

\begin{proposition}[see {\cite[Section~1]{Arnoux-Rauzy:91}}]\label{prop:STsubs}
Let $\sigma_1,\sigma_2$ be the Sturmian substitutions. Then for each Sturmian sequence $w$ there exists a coding sequence $\bsigma=(\sigma_{i_n})$, where $(i_n)$ takes each symbol in $\{1,2\}$ an infinite number of times, such that $w$ has the same language as
\begin{equation*}
u=\lim_{n\to\infty}\sigma_{i_0}\circ\sigma_{i_1}\circ\cdots\circ\sigma_{i_n}(a).
\end{equation*}
Here $a\in\{1,2\}$ can be chosen arbitrarily.
\end{proposition}

Since it will turn out that \eqref{eq:sadic2letters} and \eqref{eq:sadic2lettersMult} are nonabelian versions of the classical continued fraction algorithm we will now review the basics of this well-known concept.
  
\subsection{The classical continued fraction algorithm}\label{sec:CF}

The ``$S$-adic'' representations of a Sturmian sequence given in \eqref{eq:sadic2letters} and \eqref{eq:sadic2lettersMult} are related to continued fraction expansions of irrational numbers. For this reason we provide a brief discussion of the classical continued fraction algorithm (see {\it e.g.}~\cite[Chapter~3]{EW:11} for an introduction to continued fractions of a dynamical flavor or \cite{Berthe:11} for a discussion of continued fractions in a context related to the present paper). 

We start with the well-known additive Euclidean algorithm. Given a pair of two nonnegative real numbers $(a,b)\not=(0,0)$ we define the mapping $F:\R_{\ge 0}^2\setminus\{\mathbf{0}\}\to \R^2_{\ge 0} \setminus\{\mathbf{0}\}$ by
\[
F(a,b) = \begin{cases}
(a-b,b), & \hbox{if }a > b, \\
(a,b-a), & \hbox{if }a \le b.
\end{cases}
\] 
If we iterate this mapping starting with $(a,b)\in\R^2_{>0}$ we see that we reach a pair of the form $(0,c)$ or $(c,0)$ with $c> 0$ if and only if the ratio $a/b$ is rational. If $a/b\not\in\Q$ the iterations of $F$ on $(a,b)$ produce an infinite sequence of pairs of strictly positive numbers. Setting 
\begin{equation}\label{eq:matrAE}
M_1=\begin{pmatrix}
1&1\\
0&1
\end{pmatrix}
\qquad
\hbox{and}
\qquad
M_2=\begin{pmatrix}
1&0\\
1&1
\end{pmatrix}
\end{equation}
we see that $F(a,b)^t=M_1^{-1}(a,b)^t$ if $a > b$ and $F(a,b)^t=M_2^{-1}(a,b)^t$ if $a\le b$. Thus iterating $F$ on a pair $(a,b)$ with $a/b\not\in\Q$ produces an infinite sequence $(M_{i_n})_{n\in \N}\in\{M_1,M_2\}^\N$ defined by
\begin{equation}\label{eq:addseq}
(a,b)^t=M_{i_0}F(a,b)^t=M_{i_0}M_{i_1}F^2(a,b)^t=M_{i_0}M_{i_1}M_{i_2}F^3(a,b)^t=\cdots.
\end{equation}
This sequence $(M_{i_n})$ is called the \emph{additive continued fraction expansion} of $(a,b)$. In \eqref{eq:concvone2lett} we will see that, up to a scalar factor, $(a,b)$ is determined by the sequence $(i_n)$. 

Since the sequence $(M_{i_n})$ is invariant under the multiplication of $(a,b)$ by a scalar, we may use projective coordinates. This motivates the following definition.
Let $\PP$ be the projective line and $X=\{[a:b] \in \PP\,:\, a\ge 0,\, b\ge 0\}$. Define $M:X\to \{M_1,M_2\}$ by $M([a:b]) = M_1$ if $a > b$ and $M([a:b]) = M_2$ if $a \le b$. Then the mapping 
\begin{equation}\label{eq:linearAE}
F:X\to X; \quad \bx \mapsto M(\bx)^{-1}\bx
\end{equation}
is called the \emph{linear additive continued fraction mapping}.

Since $(a,b)\not=(0,0)$ we can define a \emph{projective} version of \eqref{eq:linearAE}. Indeed, we can write $[a:b]=[1,b/a]$ if $a>b$ and $[a:b]=[a/b,1]$ if $a\le b$ and the mapping $F$ can be written as ($c\in[0,1]$)
\begin{equation}\label{eq:ACFc}
F[1:c]= 
\begin{cases}
[1-c:c] = [\frac{1-c}{c}:1], & \hbox{if }c > \frac12,\\
[1-c:c] = [1:\frac{c}{1-c}], & \hbox{if }c \le \frac12,
\end{cases}
\qquad
F [c:1]= 
\begin{cases}
[1:\frac{1-c}{c}], & \hbox{if }c > \frac12,\\
[\frac{c}{1-c}:1], & \hbox{if }c \le \frac12.
\end{cases}
\end{equation}
Since the coordinate $1$ contains no information in \eqref{eq:ACFc} and $c\in[0,1]$, this defines a mapping $f:[0,1]\to[0,1]$ by
\[
f(x) = \begin{cases}
\frac{1-x}{x}, &\hbox{if }x > \frac12,\\
\frac{x}{1-x} ,&\hbox{if }x \le \frac12.
\end{cases}
\]
The mapping $f$ is called \emph{projective additive continued fraction mapping} or \emph{Farey map}. It is visualized in Figure~\ref{fig:additiveCF}.

\begin{figure}[h]
\includegraphics[width=0.3\textwidth]{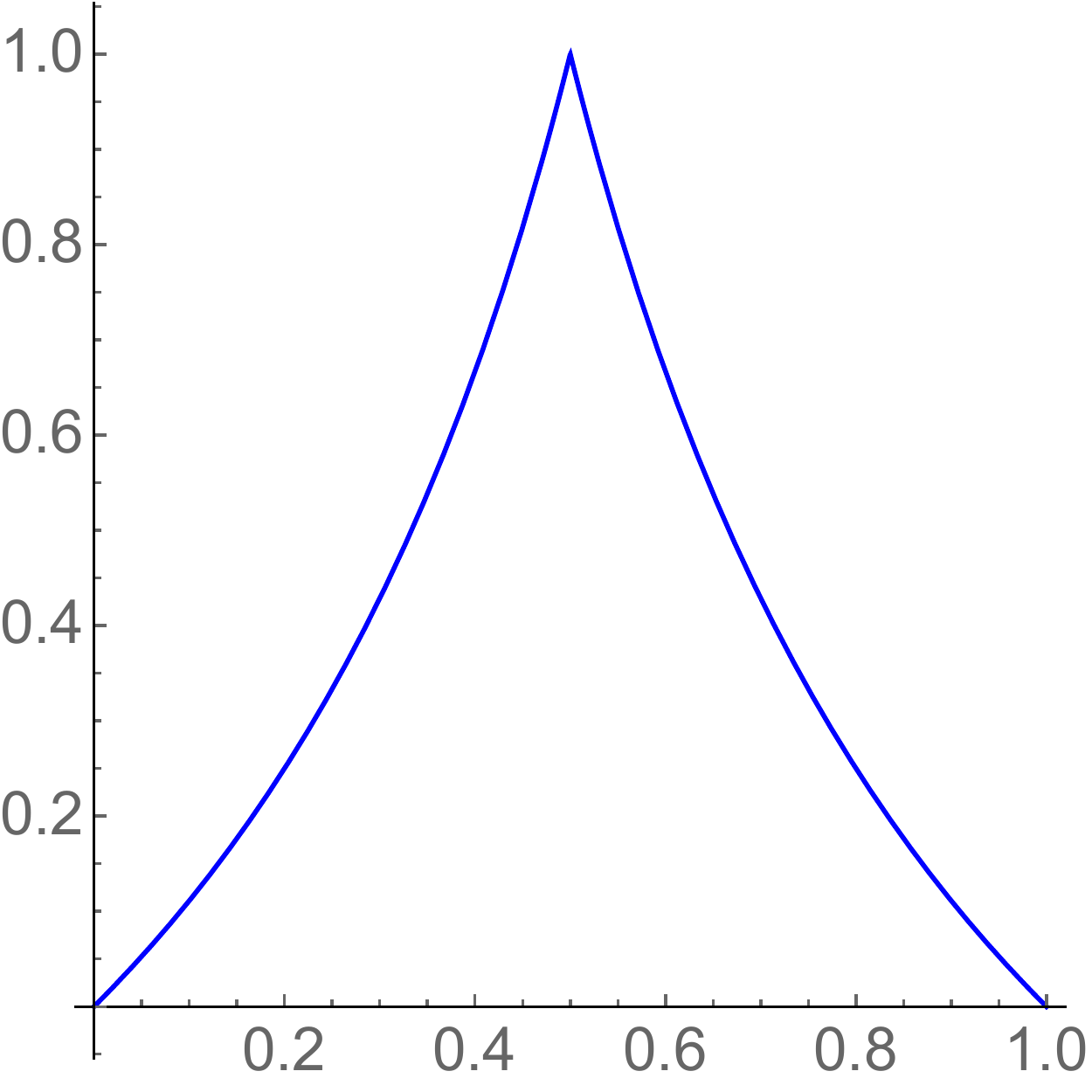}
\caption{The Farey map. \label{fig:additiveCF}}
\end{figure}

The additive continued fraction algorithm can be ``accelerated'' in the following way. Assume that $a,b>0$ are given. If $a>b$ we do not just subtract $b$ from $a$. We subtract it $m$ times where $m$ is chosen in a way that $0\le a-mb < b$. If $a \le b$ we proceed analogously. This results in the \emph{multiplicative Euclidean algorithm} $G:\R_{>0}^2\to \R^2_{\ge 0} \setminus\{\mathbf{0}\}$ with 
\[
G(a,b) = \begin{cases}
(a-\lfloor \frac ab \rfloor b,b), & \hbox{if }a > b, \\
(a,b-\lfloor \frac ba \rfloor a), & \hbox{if }a \le b.
\end{cases}
\] 
As in \eqref{eq:addseq}, iterating $G$ on a pair $(a,b)\in\R_{>0}^2$ yields a sequence of matrices $M_{1}^{a_0},M_{2}^{a_1},M_{1}^{a_2},\ldots$ with positive integers $a_0,a_1,\ldots$ satisfying (we assume $a>b$ here; otherwise the sequence would start with a power of $M_2$)
\begin{equation}\label{eq:mulitCF}
(a,b)^t=M_{1}^{a_0}G(a,b)^t=M_1^{a_0}M_{2}^{a_1}G^2(a,b)^t=M_{1}^{a_0}M_{2}^{a_1}M_{1}^{a_2}G^3(a,b)^t=\cdots.
\end{equation}
However, contrary to \eqref{eq:addseq} this sequence stops if the iteration runs into a vector one of whose coordinates is $0$ because $G$ is not defined for such vectors. Indeed, as can easily be verified, we run into such a vector if and only if $a/b \in \mathbb{Q}$.

Again we move to the projective line and set $X=\{[a:b] \in \PP\,:\, a> 0,\, b> 0\}$. Define $M:X\to \{M_1^m, M_2^m \,:\, m \ge 1\}$ by $M([a:b]) = M_1^m$ if $a > b$ and $0\le a-mb < b$ and $M([a:b]) = M_2^m$ if $a \le b$ and $0\le b-ma < b$. Then the mapping 
\begin{equation}\label{eq:linearME}
G:X\to X; \quad \bx \mapsto M(\bx)^{-1}\bx
\end{equation}
is called the \emph{linear multiplicative continued fraction mapping}.

Similar to the additive case assume that $a,b>0$ and choose the representatives $[a:b]=[1,b/a]$ if $a>b$ and $[a:b]=[a/b,1]$ if $a\le b$. The mapping $G$ can then be written as ($c\in(0,1]$)
\begin{equation}\label{eq:MCFc}
G[1:c]= [1-\lfloor  \scalebox{1}{$\frac1c$} \rfloor c:c] =[\{  \scalebox{1}{$\frac1c$} \} c:c] =[\{  \scalebox{1}{$\frac1c$} \}:1], \qquad
G [c:1]= [1:\{  \scalebox{1}{$\frac1c$} \}].
\end{equation}
As the coordinate $1$ contains no information in \eqref{eq:MCFc} this defines a mapping $g:(0,1]\to[0,1)$ by
\begin{equation}\label{eq:gaussmap}
g(x) = \Big\{\frac1x \Big\}.
\end{equation}
The mapping $g$ is called \emph{projective multiplicative continued fraction mapping} or \emph{Gauss map}\index{Gauss map}. It is visualized in Figure~\ref{fig:multiplicativeCF}.

\begin{figure}[hh]
\includegraphics[width=0.3\textwidth]{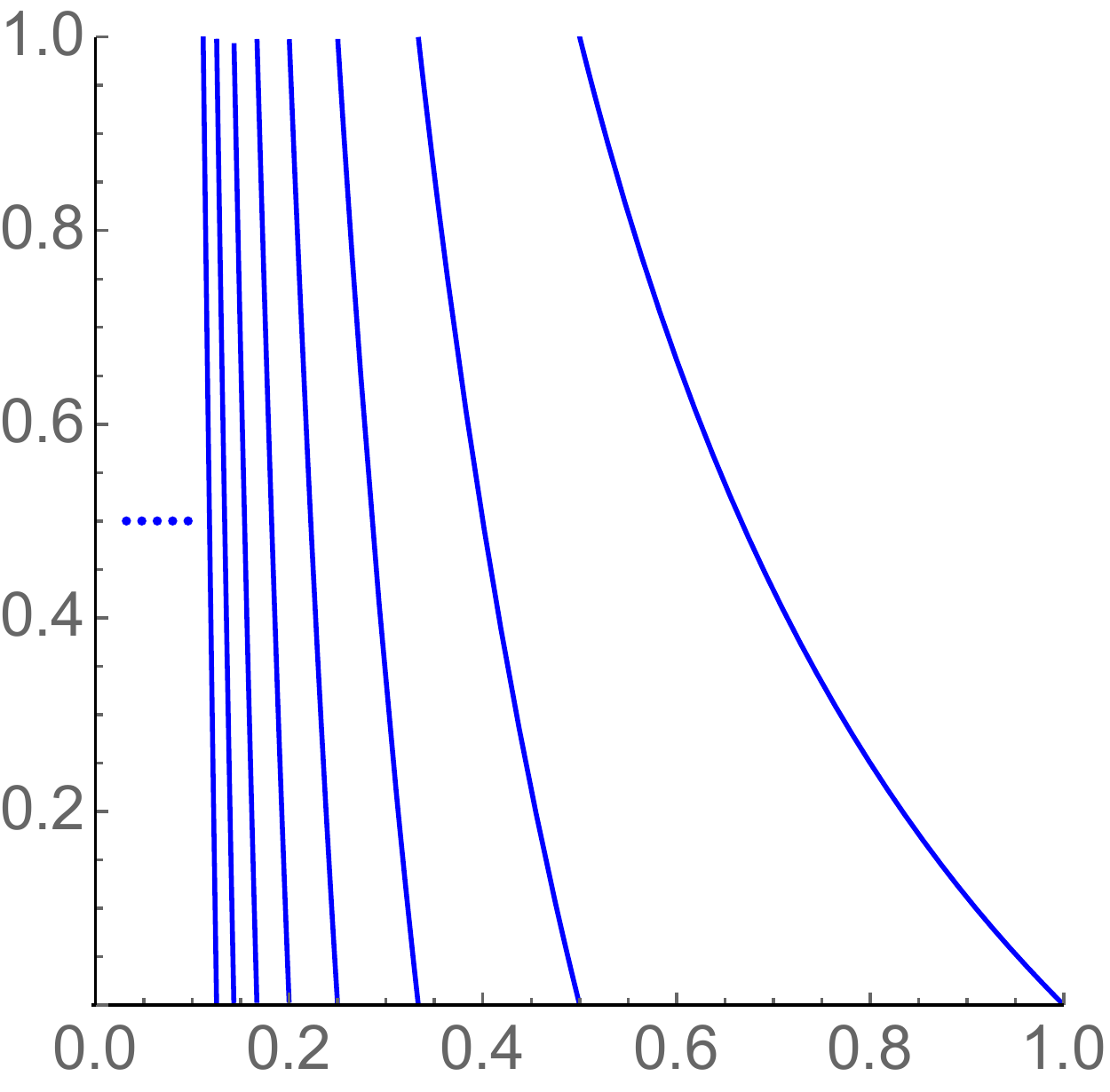}
\caption{The Gauss map $x\mapsto\{\frac1x\}$. \label{fig:multiplicativeCF}}
\end{figure}

By direct calculation (see {\it e.g.}\ \cite[Chapter~3]{EW:11}) it follows from the definition that for each irrational $x\in(0,1)$ the Gauss map $g$ can be iterated infinitely often. This iteration process determines a sequence $(a_n)$ of positive integers defined by $a_n=\big\lfloor \frac{1}{g^n(x)} \big\rfloor$ which admits to develop $x$ in its  \emph{(multiplicative) continued fraction expansion}
\[
x= \frac{\displaystyle 1}{\rule[2ex]{0pt}{1.7ex} \displaystyle a_0+\frac{\displaystyle 1}{\rule[2ex]{0pt}{1.7ex}\displaystyle a_1+\frac{\displaystyle 1}{\rule[2ex]{0pt}{1.7ex}\displaystyle a_2+ \frac{\displaystyle 1}{\rule[2ex]{0pt}{1.7ex}\displaystyle a_3+\ddots}}}}
\]
(which will be denoted by $x=[a_0,a_1,\ldots]$).
By definition this is the same sequence $(a_n)$ as the one we obtain in the exponents of the matrices in \eqref{eq:mulitCF} when setting $(a,b)=(1,x)$. One can show that this sequence is ultimately periodic if and only if $x$ is a quadratic irrational. If $x$ is rational one can associate a finite sequence with $x$ in this way. 

Continued fractions play an eminent role in Diophantine approximation. It is therefore of special interest that they will appear in our theory of Sturmian sequences naturally without being presupposed.
  
\subsection{Dynamical properties of Sturmian sequences}

We want to have a look at the ``abelianized'' version of \eqref{eq:sadic2letters} and \eqref{eq:sadic2lettersMult} in order to get a link between Sturmian sequences and the classical continued fraction algorithm. For a word $v\in\{1,2\}^*$ define the \emph{abelianization} $\mathbf{l}(v)=(|v|_1,|v|_2)^t$, and for $i\in\{1,2\}$ associate to the Sturmian substitution $\sigma_i$ from \eqref{eq:sturmsubs} the \emph{incidence matrix} $M_i=(|\sigma_{i}(k)|_j)_{1\le j,k\le 2}$. Then $M_1$ and $M_2$ are the matrices defined in \eqref{eq:matrAE} which were used to define the linear version of the classical additive continued fraction algorithm in \eqref{eq:linearAE}. Indeed, since $\mathbf{l}\sigma_i(v)=M_i\mathbf{l}(v)$ we see that the vectors (here $\be_1,\be_2$ are the standard basis vectors)
\begin{equation}\label{eq:2letterabel}
M_{i_0}\cdots M_{i_n}\mathbf{e}_a
\end{equation}
form an abelianized version of the expression in the limit of \eqref{eq:sadic2letters}. Since $(i_n)$ changes its value infinitely often,  $M_{i_n}M_{i_{n+1}}$ is a positive matrix for infinitely many $n$ (in particular, $M_{i_n}M_{i_{n+1}}=M_1M_2$ for infinitely many $n$; we therefore call the whole sequence $(M_{i_n})$ a \emph{primitive} sequence of matrices). This property entails that the positive cone $\mathbb{R}^2_{\ge 0}$ is shrunk to a line by these matrices, more precisely, there exists a vector $\bu\in\mathbb{R}^2_{>0}$ such that
\begin{equation}\label{eq:concvone2lett}
\bigcap_{n\ge 0}M_{i_0}\cdots M_{i_n}\mathbb{R}^2_{\ge 0} = \R_+\bu
\end{equation}
(see~\cite[pp.~91-95]{Furstenberg:60}, \cite[Chapter~26]{Viana:06}, or Proposition~\ref{prop:furstmatrix} below). This says that the additive continued fraction algorithm defined by \eqref{eq:addseq} is \emph{weakly convergent} (as is well known, this algorithm is even strongly convergent which is related to the balance property of Sturmian sequences). We call $\bu$, which is uniquely defined up to scalar factors by the sequence $(M_{i_n})$, a \emph{generalized right eigenvector} of $(M_{i_n})$. We also see from  \eqref{eq:concvone2lett} that the vector $(a,b)^t$ in \eqref{eq:addseq} is defined by the sequence $(M_{i_n})$ up to a scalar factor. 

We go back to the (nonabelian) $S$-adic setting. Assume that a Sturmian sequence $w=w_0w_1\ldots$ has a coding sequence $(\sigma_{i_n})$ whose associated sequence of incidence matrices $(M_{i_n})$ satisfies \eqref{eq:concvone2lett}. We will now prove that in this case $w$ has \emph{uniform letter frequencies}, {\it i.e.}, the limit
\[
f_i(w) = \lim_{\ell \to \infty}\frac{|w_k\ldots w_{k+\ell-1}|_i}{\ell}
\] 
exists uniformly in $k$ for each $i\in\{1,2\}$. We get even more, namely, the following lemma holds. In its proof and in all the remaining part of this section we use the abbreviations 
\[
\sigma_{i_{[m,n)}}=\sigma_{i_m}\circ\cdots\circ\sigma_{i_{n-1}}\quad\hbox{and}\quad M_{i_{[m,n)}}=M_{i_m}\cdots M_{i_{n-1}}.
\]

\begin{lemma}\label{lem:sturmletterfreq}
Let $w=w_0w_1\ldots$ be a Sturmian sequence with coding sequence $(\sigma_{i_n})$ whose associated sequence of incidence matrices $(M_{i_n})$ has a generalized right eigenvector $\bu$. Then $w$ has uniform letter frequencies and $(f_1(w),f_2(w))^t=\frac{\bu}{\Vert\bu\Vert_1}$.
\end{lemma}

\begin{proof}
Let $\bu/{\Vert\bu\Vert_1}=(u_1,u_2)^t$. By Proposition~\ref{prop:STsubs}  for all $k,\ell,n\in \N$ we can write
\begin{equation*}
w_k\ldots w_{k+\ell-1} = p \sigma_{i_{[0,n)}}(v) s
\end{equation*}
for some $p,v,s\in\{1,2\}^*$, where the lengths of $p,s$ are bounded by $\max\{|\sigma_{i_{[0,n)}}(1)|,|\sigma_{i_{[0,n)}}(2)|\}$. Now for each $a\in \{1,2\}$
\begin{equation}\label{eq:strumDTeinfach}
\left| \frac{|w_k\ldots w_{k+\ell-1}|_a}{\ell} -  u_a\right| \le \frac{\big| |p|_a - |p|u_a\big|}{\ell} + \frac{\big| |\sigma_{i_{[0,n)}}(v)|_a-|\sigma_{i_{[0,n)}}(v)|u_a\big|}{\ell}  + \frac{\big| |s|_a- |s|u_a\big|}{\ell}.
\end{equation}
By the convergence of the positive cone to $\bu$ in \eqref{eq:concvone2lett} we know that $|\sigma_{i_{[0,n)}}(b)|_a/|\sigma_{i_{[0,n)}}(b)|$ is close to $u_a$ for all $a,b\in\{1,2\}$ if $n$ is large. Thus for each $\varepsilon > 0$ there is $N\in\N$ such that whenever $\ell \ge N$ we can choose $n$ in a way that $|p|,|s| \le \varepsilon \ell$ and $\big||\sigma_{i_{[0,n)}}(b)|_a-|\sigma_{i_{[0,n)}}(b)|u_a \big| < \varepsilon |\sigma_{i_{[0,n)}}(b)|$ for all letters $a$ and $b$. This proves that the right hand side of \eqref{eq:strumDTeinfach} is bounded by $3\varepsilon$ and thus $\lim_{\ell \to \infty}{|w_k\ldots w_{k+\ell-1}|_a}/{\ell}=u_a$ uniformly in $k$.
\end{proof}

For a proof of Lemma~\ref{lem:sturmletterfreq} along similar lines in a more general setting we refer to Lemma~\ref{lem:SadicFreq} (see also Berth\'e and Delecroix~\cite[Theorem~5.7]{Berthe-Delecroix}; a proof using balance, which also gives irrationality of the frequencies, is contained in \cite[Proposition~6.1.10]{Fog02}). 

In the same way as for letters, we can define uniform frequencies for factors of an infinite sequence $w\in\{1,2\}^\N$. Let $w$ be a Sturmian sequence with coding sequence $(\sigma_{i_n})$. The sequence is the shifted image of another Sturmian sequence under an arbitrary large block $\sigma_{i_{[0,n)}}$ of substitutions. This enables one to show that for the words $\sigma_{i_{[0,n)}}(a)$ there exist uniform frequencies in $w$. Since $(i_n)$ changes its value infinitely often, the length of the words $\sigma_{i_{[0,n)}}(a)$ tends to infinity for each letter $a$ if $n\to\infty$.  Using this fact one can prove the following result along similar lines as Lemma~\ref{lem:sturmletterfreq} (for details we refer to the proof of Lemma~\ref{lem:SadicFreq} below; see also \cite[Theorem~5.7]{Berthe-Delecroix}).

\begin{lemma}\label{lem:sturmpatternfreq}
Let $w=w_0w_1\ldots\in\{1,2\}^\N$ be a Sturmian sequence with coding sequence $(\sigma_{i_n})$ whose associated sequence of incidence matrices $(M_{i_n})$ has a generalized right eigenvector $\bu$. Let $v \in \{1,2\}^*$ be given, and let $|w_kw_{k+1}\dots w_{k+\ell-1}|_v$ be the number of occurrences of $v$ in the factor $w_kw_{k+1}\dots w_{k+\ell-1}$ of $w$. Then $|w_kw_{k+1}\dots w_{k+\ell-1}|_v/\ell$ tends to a limit $f_v(w)$ for $\ell\to\infty$ uniformly in $k$. 
\end{lemma}
 
We can associate a dynamical system with a Sturmian sequence $w$ in a very natural way. Let $X_w = \overline{\{\Sigma^k w \;:\; k\in \N\}}$ be the closure of the shift orbit of $w$. Alternatively, $X_w$ can be viewed as the set of all sequences $u$ whose \emph{language} $L(u)$ ({\it i.e.}, its set of factors) satisfies $L(u)\subseteq L(w)$.
Thus if $\bsigma=(\sigma_{i_n})$ is the coding sequence of $w$, Proposition~\ref{prop:STsubs} implies that $X_{w}$ contains all Sturmian sequences with coding sequence $\bsigma$.  Since $X_w$ is shift invariant the shift $\Sigma$ acts on $X_{w}$ and the dynamical system $(X_w,\Sigma)$ is well defined. We call $(X_w,\Sigma)$ a \emph{Sturmian system}. From what we know about Sturmian sequences we can derive a number of properties for these dynamical systems. The notions of \emph{minimality} and \emph{unique ergodicity} of a dynamical system used in the following lemma are defined precisely in Definitions~\ref{def:min} and~\ref{def:ue}, respectively.
 
\begin{proposition}\label{lem:sturmsys}
A Sturmian system $(X_w, \Sigma)$ has the following properties.
\begin{enumerate}
\item[(i)] The system $(X_w, \Sigma)$ is minimal.
\item[(ii)]  The set $X_w$ is the set of all Sturmian sequences having the same language.
\item[(iii)]  The set $X_w$ is the set of all Sturmian sequences having the same coding sequence $\bsigma$.
\item[(iv)] The system $(X_w,\Sigma)$ is uniquely ergodic.
\item[(v)]  We have $X_w=X_{w'}$ for any $w'\in X_w$.
\end{enumerate}
\end{proposition} 

\begin{proof}
Let $(\sigma_{i_n})$ be the coding sequence of $w$ with $(M_{i_n})$ being the associated sequence of matrices.

We start with (i). By Proposition~\ref{prop:STsubs} we may assume w.l.o.g.\ that $w = \lim_{n\to \infty} \sigma_{i_{[0,n)}}(1)$. Let $v \in X_w$ be given. To prove minimality it suffices to show that $L(v) = L(w)$. Since $L(v)\subseteq L(w)$ is true by definition we need to prove the reverse inclusion. Let $u\in L(w)$.  
By the definition of $w$ and the primitivity of the sequence $(M_{i_n})$ there is $m\in\N$ such that $u$ occurs in $\sigma_{i_{[0,m)}}(1)$. However, there is a Sturmian word $w^{(m)}$ satisfying $w=\sigma_{i_{[0,m)}}(w^{(m)})$. Since $w^{(m)}$ is balanced by Proposition~\ref{prop:sturmbalance}, the letter $1$ occurs in $w^{(m)}$ with bounded gaps. This implies that $\sigma_{i_{[0,m)}}(1)$ and, hence, $u$ occurs in $w$ with bounded gaps. Thus $u$ occurs in each element of the orbit closure $X_w$ of $w$, hence, also in $v$. Thus $L(v) = L(w)$ is established.

Since $L(v)=L(w)$ holds for each $v\in X_w$ according to the previous paragraph we have $p_v(n)=p_w(n)=n+1$ for all $n\in\N$, hence, $v$ is Sturmian with the same language as $w$. This proves (ii).

To prove (iii) we follow the proof of \cite[Lemma~6.3.12]{Fog02}. Assume w.l.o.g.\ that the elements of $X_w$ are of type $1$ and let $u,u'\in X_w$. Then according to Lemma~\ref{prop:sturmDesubs} there exist Sturmian words $v,v'$ such that $u=\sigma_1(v)$ or $u=\Sigma\sigma_1(v)$ as well as $u'=\sigma_1(v')$ or $u'=\Sigma\sigma_1(v')$. We first prove that $v,v'$ belong to the same Sturmian system. By (ii) we have to show that $L(v)=L(v')$. Suppose that $x \in L(v)$. Since $x$ occurs infinitely often in $v$ by recurrence, there is $y\in L(v)$ starting with the letter $2$ such that $x$ is a subword of $y$. The word $\sigma_1(y)$ occurs in $u$ and by (ii) it occurs also in $u'$ and because $\sigma_1(y)$ begins with $2$ and ends with $1$ it can be desubstituted in only one way by $\sigma_1$, namely to $y$. This proves that $y$ and, hence, also $x$ occurs in $v'$. Thus $L(v)\subseteq L(v')$. The other inclusion follows by interchanging the roles of $v$ and $v'$. Iterating this argument yields that $u$ and $u'$ have the same coding sequence. Thus all elements of $X_w$ have the same coding sequence. As Sturmian sequences with the same coding sequence have the same language by Proposition~\ref{prop:STsubs}, $X_w$ contains all Sturmian sequences having the same coding sequence as $w$.

Item (iv) follows immediately by combining Lemma~\ref{lem:sturmpatternfreq} with \cite[Proposition~5.1.21]{Fog02} (see also Proposition~\ref{prop:freq_unique_ergod} below) which states that the existence of uniform word frequencies implies unique ergodicity. Alternatively, one can use Boshernitzan~\cite{Boshernitzan:84}.

Finally, (v) follows from (ii).
\end{proof}

We emphasize on the fact that for minimality and unique ergodicity of $(X_{w},\Sigma)$ the recurrence of $w$ as well as the primitivity of the sequence $(M_{i_n})$ is of importance. This will be the same in the general case (see Section~\ref{sec:pr}  below). In view of assertion (iii) of the previous lemma we will write $X_\bsigma$ instead of $X_w$, where $\bsigma$ is the coding sequence of $w$. 

\subsection{Sturmian sequences code rotations}\label{sec:sturmrot}
It was observed already by Morse, Coven, and Hedlund~\cite{Coven&Hedlund:1973,Morse&Hedlund:1940} that each Sturmian sequence is a \emph{natural coding} of a rotation by some irrational number $\alpha$. We now sketch a proof of this fact which goes back to Rauzy and in which the multiplicative continued fraction expansion of $\alpha$ pops up when we represent such a coding in an $S$-adic fashion. For proofs of this kind we refer to \cite{AFH:99,AF:01,BFZ:05,BHZ:06}; a different, combinatorial proof along the lines of the original proof by Morse, Coven, and Hedlund is presented in~\cite[Theorem~2.1.13]{Lothaire:02} and \cite[Section~10.5]{Allouche-Shallit:03}.

Before we give the main result of this section we provide some definitions. Let $\mathbb{T}$ be the $1$-torus, {\it i.e.}, the unit interval $[0,1]$ with its end points glued together. A \emph{rotation}\index{torus rotation} or \emph{translation}\index{torus translation} on $\mathbb{T}$ by a real number $\alpha$ is a mapping $R_\alpha: \mathbb{T}\to\mathbb{T}$ with $x \mapsto x + \alpha \pmod{1}$. If $\alpha\not\in\Q$ this gives a minimal dynamical system.  
\begin{figure}[hh]
\includegraphics[width=0.4\textwidth]{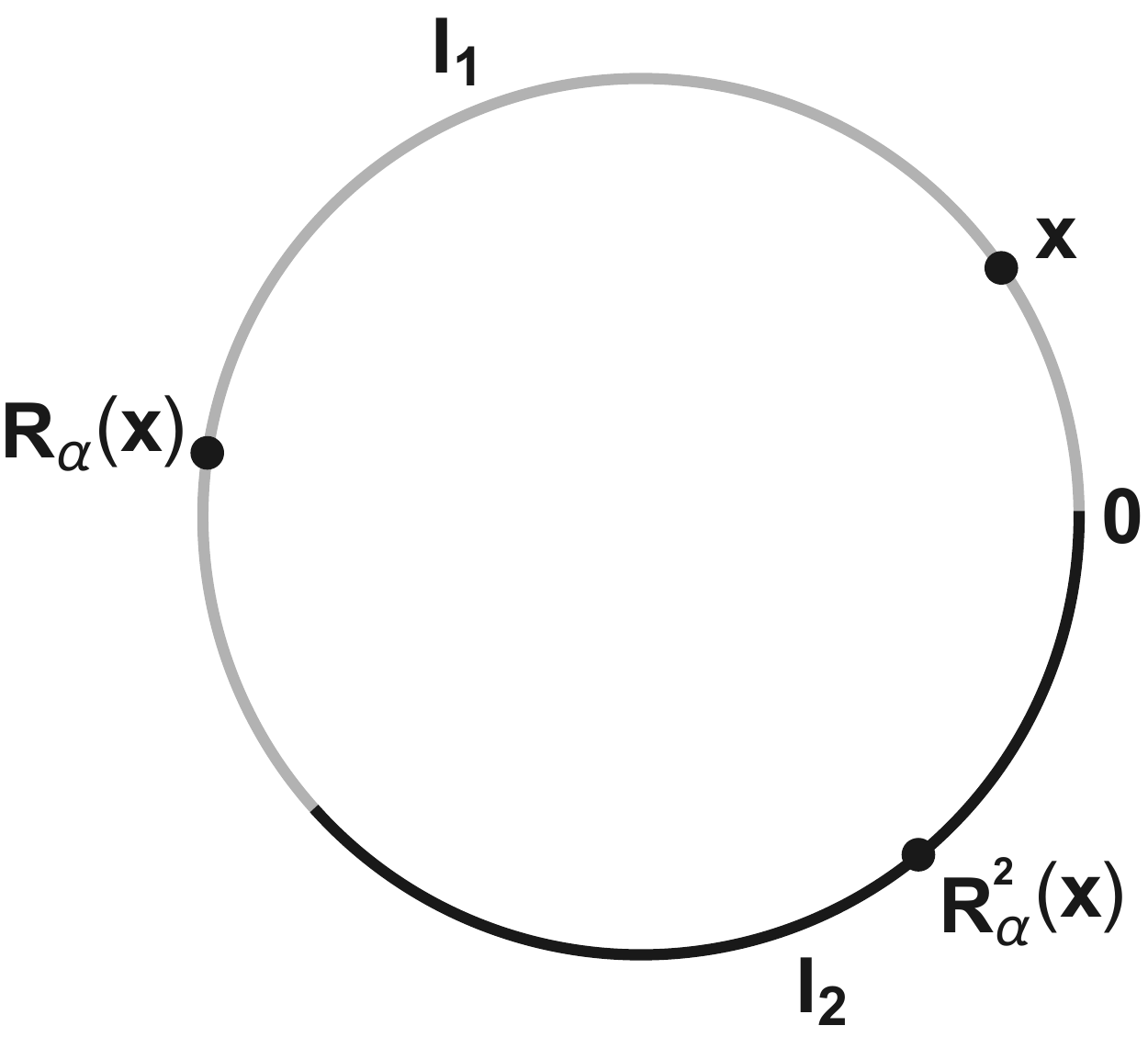}
\caption{Two iterations of the irrational rotation $R_\alpha$ on $\mathbb{T}$ which is subdivided into the two intervals $I_1$ and $I_2$. \label{fig:irrot}}
\end{figure}
Moreover, observe that $R_\alpha$ can be regarded as a \emph{two interval exchange} of the intervals $I_1=[0,1-\alpha)$ and $I_2=[1-\alpha,1)$ or of the intervals $I'_1=(0,1-\alpha]$ and $I'_2=(1-\alpha,1]$, see Figure~\ref{fig:irrot}. We say that a sequence $w=w_0w_1\ldots\in\{1,2\}^\mathbb{N}$ is a \emph{natural coding}\index{natural coding} of $R_\alpha$ if there is $x\in\mathbb{T}$ such that $R_\alpha^k(x) \in I_{w_k}$ for each $k\in\N$ or $R_\alpha^k(x) \in I'_{w_k}$ for each $k\in\N$. 

\begin{theorem}\label{th:sturmrot}
A sequence $w\in\{1,2\}^\N$ is Sturmian if and only if there exists $\alpha\in\R\setminus\Q$ such that $w$ is a natural coding of the rotation $R_\alpha$.
\end{theorem}

The sufficiency part of the theorem is easy. Indeed, it just follows from the observation that 
\begin{equation}\label{eq:strumianrotationiff}
v_0\ldots v_{n-1} \hbox{ is a factor of a natural coding of } R_\alpha \quad\Longleftrightarrow\quad \bigcap_{k=0}^{n-1} R_\alpha^{-k}I_{v_k}\not=\emptyset,
\end{equation}
whose proof is an easy exercise (see \cite[Lemma~2.7]{BFZ:05}).

The proof of the necessity part of Theorem~\ref{th:sturmrot} needs more work and we will see that the classical continued fraction algorithm pops up along the way without being presupposed. We need the following key lemma.

\begin{lemma}\label{lem:rotCF}
For $\alpha \in(0,1)$ irrational let $u$ be the coding of the point $1-\alpha/(\alpha+1)$ under the irrational rotation $R_{\alpha/(\alpha+1)}$. Then there is a sequence $(\sigma_{i_n})$ of substitutions such that
\[
u=\lim_{n\to\infty}\sigma_{i_0}\circ\cdots\circ\sigma_{i_n}(2).
\]
The sequence $(i_n)\in\{1,2\}^\N$ is of the form $1^{a_0}2^{a_1}1^{a_2}2^{a_3}\ldots$ where $[a_0,a_1,a_2,a_3,\ldots]$ is the continued fraction expansion of $\alpha$. For $\alpha >1$ a similar result with switched symbols holds.
\end{lemma}

\begin{proof}
We assume $\alpha <1$ ($\alpha > 1$ can be treated in a similar way). For computational reasons consider the rotation $R$ by $\alpha$ on the interval $J=[-1,\alpha)$ with the partition $P_1=[-1,0)$ and $P_2=[0,\alpha)$. The natural coding $u$ of $1-\alpha/(\alpha+1)$ by $R_{\alpha/(\alpha+1)}$ is the natural coding of $0$ by $R$. Let $R'$ be the first return map of $R$ to the interval $J'=\left[\alpha  \left\lfloor \frac1\alpha\right\rfloor-1  , \alpha\right)$. Let $v$ be a coding of the orbit of $0$ for $R'$. As can be seen from Figure~\ref{fig:induction}, after each occurrence of $2$ in $u$ we leave the interval $J'$ and there follows a block of $1$s of 
\begin{figure}[hh]
\includegraphics[width=\textwidth]{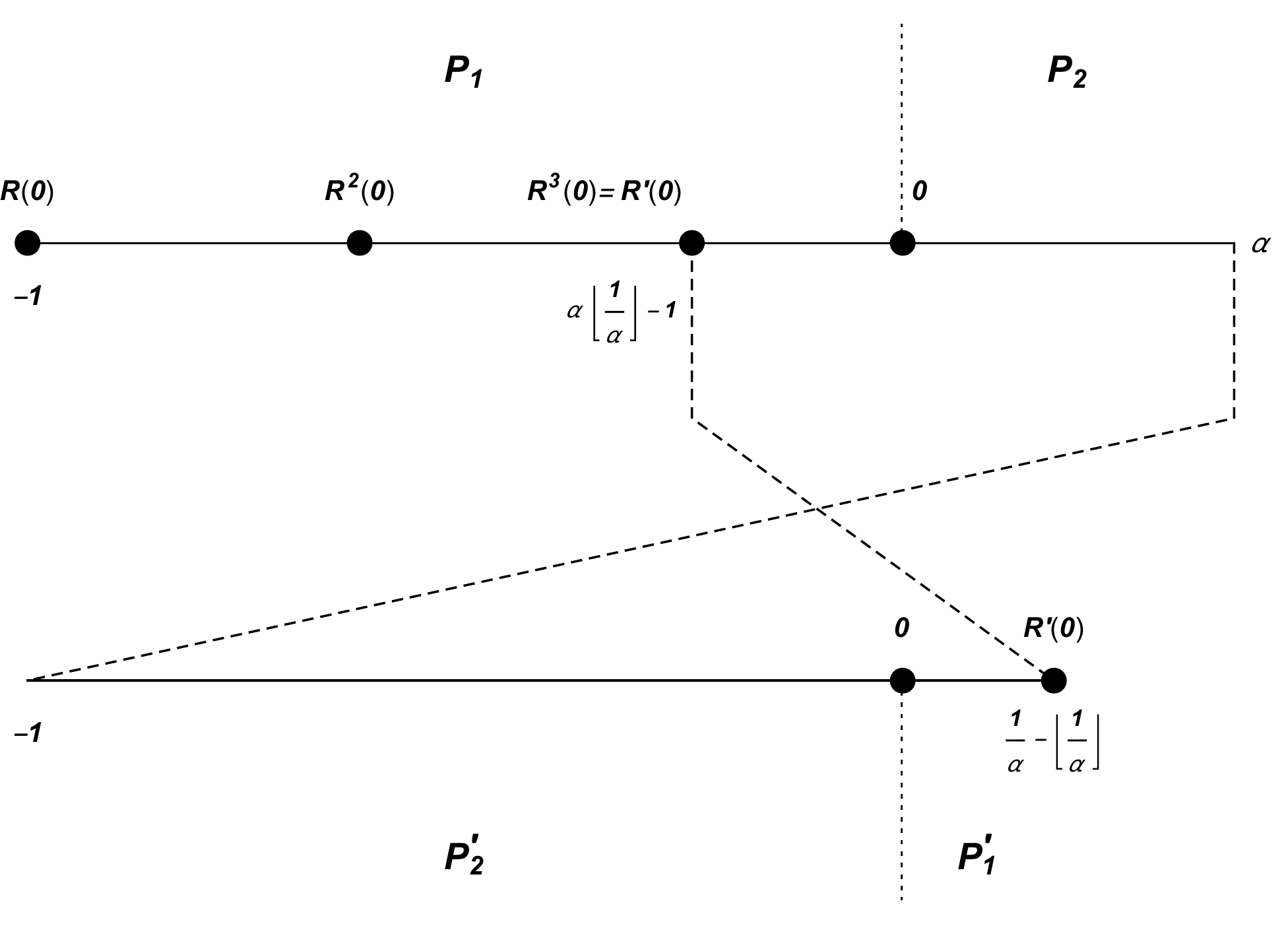}
\caption{The rotation $R'$ induced by $R$. \label{fig:induction}}
\end{figure}
length $\left \lfloor \frac1\alpha\right\rfloor$ before we enter the interval $J'$ again. Thus $v$ emerges from $u$ by removing such a block of $1$s after each letter $2$ occurring in $u$. By the definition of $\sigma_1$ this just means that $u=\sigma_1^{\lfloor 1/\alpha \rfloor}(v)$. We can now renormalize the interval $J'$ by dividing it by $-\alpha$ and, as illustrated in Figure~\ref{fig:induction}, then $R'$ is conjugate to a rotation (called $R'$ again) by $\left\{\frac1\alpha\right\}$ on the interval $\left(-1,\left\{\frac1\alpha\right\}\right]$, where $v$ is the natural coding of the partition $P'_2=(-1,0]$ and $P'_1=\left(0,\left\{\frac1\alpha\right\}\right]$. Note that the \emph{Gauss map} $\alpha\mapsto\left\{\frac1\alpha\right\}$ from \eqref{eq:gaussmap} comes up here without being presupposed. 
Since we are in the same setting as before (just with the letters $1$ and $2$ interchanged), we can iterate this process and thereby obtain a sequence $(u^{(n)})_{n\ge 0}$ of natural codings such that 
\[
u=u^{(0)} \quad\hbox{and}\quad u^{(n)}=\sigma_{i_n}(u^{(n+1)}) \hbox{ for }n\ge 0
\]
for some sequence $(\sigma_{i_n})$ with $(i_n)\in\{1,2\}^\N$ having infinitely many changes between the letters $1$ and $2$.
Arguing in the same way as in Section~\ref{sec:sturmianprop} we gain that 
\[
u=\lim_{n\to\infty}\sigma_{i_0}\circ\dots\circ\sigma_{i_n}(a)
\]
where $a=2$ is the first letter of $u$. The assertion on the continued fraction expansion follows from the above proof as well. Just note that the interval we use has length $\alpha+1$ so that the rotation by $\alpha$ on this interval is conjugate to $R_{\alpha/(\alpha+1)}$. 
\end{proof}

\begin{proof}[Conclusion of the proof of Theorem~\ref{th:sturmrot}] The sufficiency assertion has been treated in \eqref{eq:strumianrotationiff}. The necessity part of the theorem can now be obtained as follows. Let $w$ be a Sturmian sequence. Consider its coding sequence $(\sigma_{i_n})$ and write $(i_n)\in\{1,2\}^\N$ as $1^{a_0}2^{a_1}1^{a_2}2^{a_3}\ldots$ Then $u=\lim_{n\to\infty}\sigma_{i_0}\circ\dots\circ\sigma_{i_n}(2)$ is a natural coding of $R_{\alpha/(1+\alpha)}$ where $\alpha=[a_0,a_1,a_2,\ldots]$. By Proposition~\ref{prop:STsubs} the sequence $w$ has the same language as $u$ and \eqref{eq:strumianrotationiff} together with an approximation argument implies that $w$ is a natural coding of $R_{\alpha/(1+\alpha)}$ (it is easy to verify that there are limit cases where we really need the intervals $I_1'$, $I_2'$ to define the natural coding for $w$).
\end{proof}

The fact that Sturmian sequences have irrational uniform letter frequencies is an immediate consequence of Theorem~\ref{th:sturmrot}. Moreover, we have the following corollary of Theorem~\ref{th:sturmrot} for Sturmian systems.

\begin{corollary}
A  Sturmian system $(X_\bsigma, \Sigma,\mu)$ is measurably conjugate to an irrational rotation $(\mathbb{T},R_\alpha,\lambda)$. Here $\mu$ is the unique $\Sigma$-invariant measure on $X_\bsigma$ and $\lambda$ is the Haar measure on~$\mathbb{T}$.
\end{corollary}

\begin{proof}
Let $\varphi: X_\bsigma \to \mathbb{T}$ be defined by $\varphi(w_0w_1\ldots)=x$ if $R_\alpha^k(x) \in I_{w_k}$ for each $k\in\N$ or $R_\alpha^k(x) \in I'_{w_k}$ for each $k\in\N$. Using Theorem~\ref{th:sturmrot} and the minimality of $R_\alpha$ it is easy to check that this is well defined. Surjectivity of $\varphi$ follows immediately from Theorem~\ref{th:sturmrot}. To investigate injectivity let $u=u_0u_1\ldots$ and $v=v_0v_1\ldots$ be distinct elements of $X_\bsigma$ with $\varphi(u)=\varphi(v)$.  By the minimality of $R_\alpha$ this is only possible if the orbit of $\varphi(u)$ passes through $0$ and $u$ is naturally coded by $I_1$, $I_2$ while $v$ is naturally coded by $I_1'$, $I_2'$ (or vice versa).\footnote{This implies that $u$ and $v$ have $\Sigma x$ and $\Sigma y$ in their orbit where $x$ and $y$ are the two limit sequences of $\bsigma$. This interesting fact, which is not needed in this proof, should be proved by the reader.} Since the set of such elements $u$ and $v$ is countable, $\varphi$ is bijective everywhere save for a countable set. Moreover, $\varphi$ is easily seen to be continuous and $\varphi\circ \Sigma = R_\alpha \circ \varphi$ holds by the definition of $\varphi$. This implies the result. 
\end{proof}

We illustrate the concepts of this section by a classical example.

\begin{example}[A variant of the Fibonacci sequence]
Let $\sigma$ be given by
\[
\sigma=\sigma_1\circ\sigma_2: 
\begin{cases}
1 \mapsto 121,\\
2 \mapsto 21.\end{cases} 
\]
This is a reordering of the square of the well-known \emph{Fibonacci substitution} (which is defined by $1\mapsto12$, $2\mapsto1$; see for instance in \cite[Section~1.2.1]{Fog02}).
Consider the coding sequence $\bsigma=(\sigma)$. In this case the associated limit sequences are ``purely substitutive''. One of the two limit sequences is 
\[
w=\lim_{n\to\infty}\sigma^n(2)=21121121211211212112121121121\ldots
\]
Since only one substitution plays a role here, the associated ``$S$-adic'' system $(X_\bsigma,\Sigma)$ is called a \emph{substitutive system}. Let $\varphi=\frac{1+\sqrt{5}}{2}$. By the Perron-Frobenius Theorem the generalized right eigenvector $\bu$ of the sequence of incidence matrices $\bM$ of $\bsigma$ is the eigenvector $(\varphi,1)^t$ corresponding to the dominant eigenvalue $\varphi^2$ of the incidence matrix of $\sigma$. Let $L$ be the eigenline defined by this eigenvector. Being a Sturmian sequence, $w$ is balanced by Proposition~\ref{prop:sturmbalance} and has uniform letter frequencies $(f_1(w),f_2(w))^t=\frac{1}{1+\varphi}(\varphi,1)^t$ by Lemma~\ref{lem:sturmletterfreq}. This is reflected by the fact that the ``broken line'' 
\begin{equation}\label{eq:brokensturmianline}
B= \{\mathbf{l}(p) \;:\; p \hbox{ is a prefix of } w \}
\end{equation}
associated with the sequence $w$ stays at bounded distance from the eigenline $L$ (see Figure~\ref{fig:brokenLineRauzy}). 

Because $w=\lim_{n\to\infty}\sigma^n(2)= \lim_{n\to\infty}(\sigma_1\circ\sigma_2)^n(2)$, it has coding sequence $\sigma_1,\sigma_2,\sigma_1,\sigma_2,\ldots$ Since the ``run lengths'' of $\sigma_{i}$ in this sequence are always equal to $1$ we set $\alpha=[1,1,1,\ldots] =\varphi^{-1}$ and, hence, $\alpha/(\alpha+1)=\varphi^{-2}$.
Thus from Theorem~\ref{th:sturmrot} and its proof we see that $w$ is a natural coding of the rotation by $\varphi^{-2}$ of the point $1-\alpha/(\alpha+1)=\varphi^{-1}\in[0,1)$ with respect to the partition $I_1=[0,\varphi^{-1})$, $I_2=[\varphi^{-1},1)$ (or the according partition $I_1',I_2'$) of $[0,1)$. This gives us an easy way to construct $w$ (and the broken line $B$). Indeed, start at the origin, write out $2$ and go up to the lattice point $(0,1)^t$. After that, inductively proceed as follows: whenever the current lattice point is above $L$, write out $1$ and go right to the next lattice point by adding the vector $(1,0)^t$ and whenever the current lattice point is below $L$, write out $2$ and go up to the next lattice point by adding the vector $(0,1)^t$.\footnote{We could also have started with writing out 1 and going to the right from the origin. This would have produced the second limit sequence of $(\sigma)$ which coincides with $w$ save for the first two letters.}

\begin{figure}[hh]
\includegraphics[width=0.7\textwidth]{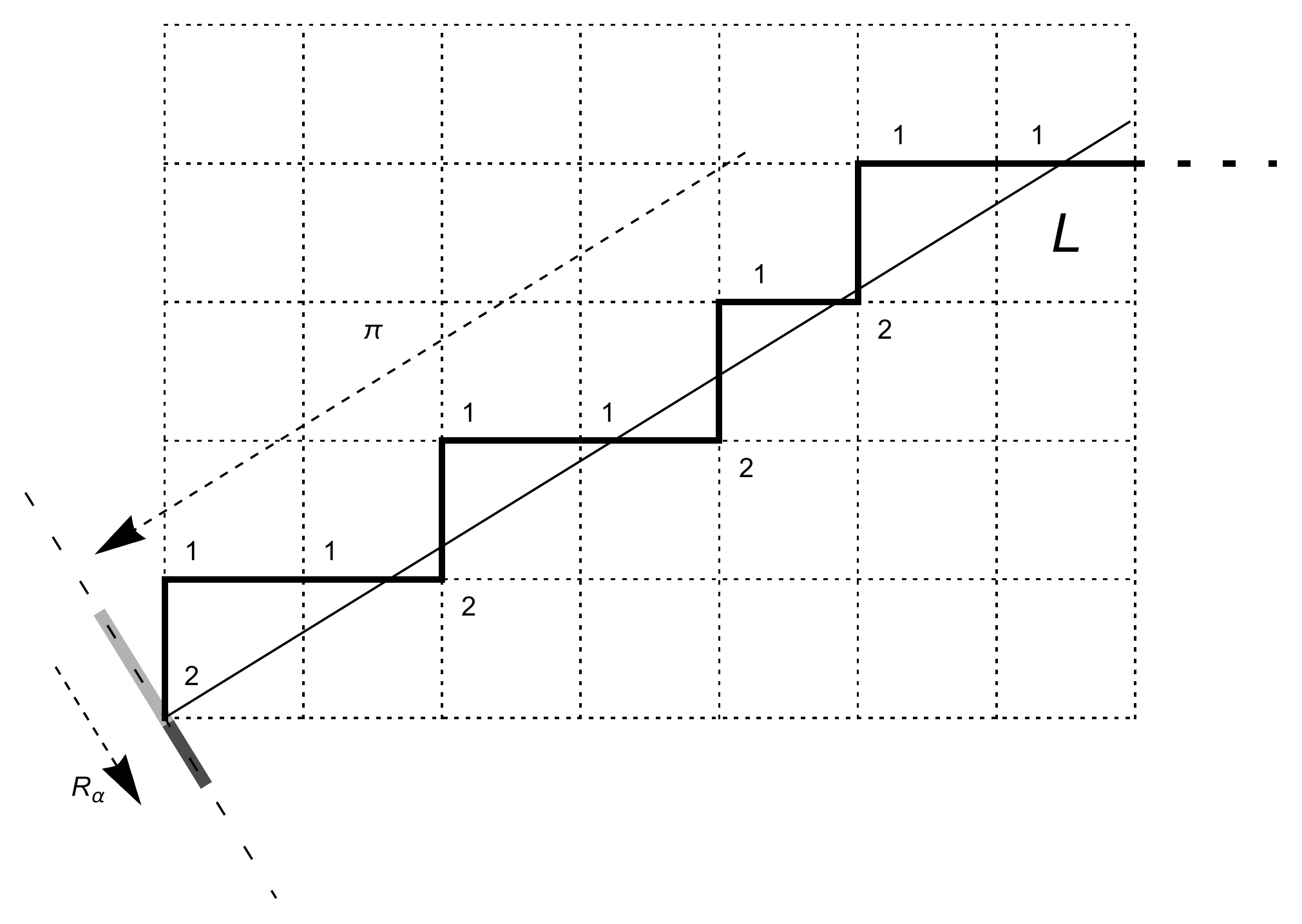}
\caption{The broken line and its projection to the Rauzy fractal. \label{fig:brokenLineRauzy}}
\end{figure}

Let $\pi$ be the projection along $L$ to the line $L^\bot$ orthogonal to $L$. If we project all points on the broken line and take the closure of the image, due to the irrationality of $\bu$ we obtain the interval 
\[
\mathcal{R}_\bu = \overline{\{\pi\mathbf{l}(p)\;:\; p \hbox{ is a prefix of }w  \}}
\]
on $L^\bot$ (the subscript $\bu$ indicates that $\Ra_\bu$ lives in the space $L^\bot=\bu^\bot$ orthogonal to $\bu$ which is an arbitrary choice; other choices will play a roll in subsequent sections). We color the part of the interval for which we write out $1$ at the associated lattice point light grey, the other part dark grey. This subdivides the interval $\mathcal{R}_\bu$ into two subintervals $\mathcal{R}_\bu(1)$ and $\mathcal{R}_\bu(2)$, where
\[
\mathcal{R}_\bu(i) = \overline{\{\pi\mathbf{l}(p)\;:\; pi \hbox{ is a prefix of }w  \}}\qquad(i=1,2).
\] 
Moreover, we see that moving a step along the broken line amounts to exchanging these two intervals in the projection: points in $\mathcal{R}_\bu(1)$ are moved downwards by a fixed vector, while points in $\mathcal{R}_\bu(2)$ are moved upwards by a fixed vector. 

Thus passing along the broken line each step amounts to exchanging the intervals $\mathcal{R}_\bu(1)$ and $\mathcal{R}_\bu(2)$ in the projection. If we identify the end points of $\mathcal{R}_\bu$ this interval exchange becomes a rotation. This is the rotation which is coded by the Sturmian sequence $w$. The union $\mathcal{R}_\bu=\mathcal{R}_\bu(1)\cup\mathcal{R}_\bu(2)$ is called the \emph{Rauzy fractal} associated with the substitution $\sigma$ (or with the sequence $\bsigma=(\sigma)$). The reason why we speak about \emph{fractals} here will be come apparent in Section~\ref{sec:Srauzy} when we define the analogs of $\Ra_{\bu}$ in a more general setting.
\end{example}

Suppose we would be given an arbitrary sequence $w\in\{1,2\}^\N$ with letter frequency vector $\bu$ whose broken line stays within bounded distance of the line $L=\mathbb{R}_+\bu$. Then we could draw a similar picture as in Figure~\ref{fig:brokenLineRauzy}. However, although the projection $\pi$ would project the vertices of the associated broken line to a bounded set, there is no reason for its closure $\mathcal{R}_\bu$ to be an interval. Also, if we use two colors as in the example above, it may well happen that the two sets $\mathcal{R}_\bu(1)$ and $\mathcal{R}_\bu(2)$ have considerable overlap. This bad behavior prevents us from seeing a rotation in the projections.

Making sure that the closure of the projection of the broken line behaves topologically well and allows a partition whose atoms are essentially different will be our main concern when we establish a theory of $S$-adic sequences that are codings of rotations on higher dimensional tori in the subsequent sections and, hence, give rise to dynamical systems that are measurably conjugate to torus rotations.

\subsection{Natural extensions and the geodesic flow on $\mathrm{SL}_2(\Z)\setminus \mathrm{SL}_2(\R)$}\label{sec:NatlGauss}

In this section we talk about natural extensions of the Gauss map and of the coding map of Sturmian sequences by substitutions. Moreover, we show how to relate these natural extensions to the geodesic flow on the space $\mathrm{SL}_2(\Z)\setminus \mathrm{SL}_2(\R)$ of unimodular two-dimensional lattices. 

So far we could relate Sturmian sequences to rotations on the circle by using the classical continued fraction algorithm. In our discussion we coded a Sturmian sequence $w$ by a sequence of substitutions $(\sigma_{i_n})$ as
\[
w=\lim_{k\to\infty}\sigma_{1}^{a_0}\circ\sigma_{2}^{a_1}\circ\sigma_{1}^{a_2}\circ\cdots\circ\sigma_{1}^{a_{2k}}(a)
\]
(see \eqref{eq:sadic2lettersMult}). In the induction process used in the proof of Theorem~\ref{th:sturmrot} we recoded $w$ by a ``desubstitution'' process. If we look at the first step of this process we produce the sequence
\[
u=\lim_{k\to\infty}\sigma_{2}^{a_1}\circ\sigma_{1}^{a_2}\circ\cdots\circ\sigma_{1}^{a_{2k}}(a).
\]
However, the mapping $w\mapsto u$ cannot be inverted since it is not possible to reconstruct $a_{0}$ from $u$. Similarly, the Gauss map $g$ cannot be inverted since $g([a_{0},a_{1},\ldots])=[a_{1},a_{2},\ldots]$, and $a_{0}$ cannot be reconstructed from the image  $[a_{1},a_{2},\ldots]$.

In this section we want to make both of these mappings bijective by constructing a geometric model for their \emph{natural extensions} (in the sense of Rohlin~\cite{Rohlin:64}). To this matter we look again at the induction used in Lemma~\ref{lem:rotCF} which is visualized once more in Figure~\ref{fig:inductionRestack}~(a). In this figure we see why this induction process cannot be reversed: the intervals $[R(0),R^2(0))$ and $[R^2(0),R^3(0))$ get lost during the induction process and cannot be reconstructed.

\begin{figure}[hh]
\includegraphics[width=0.94\textwidth]{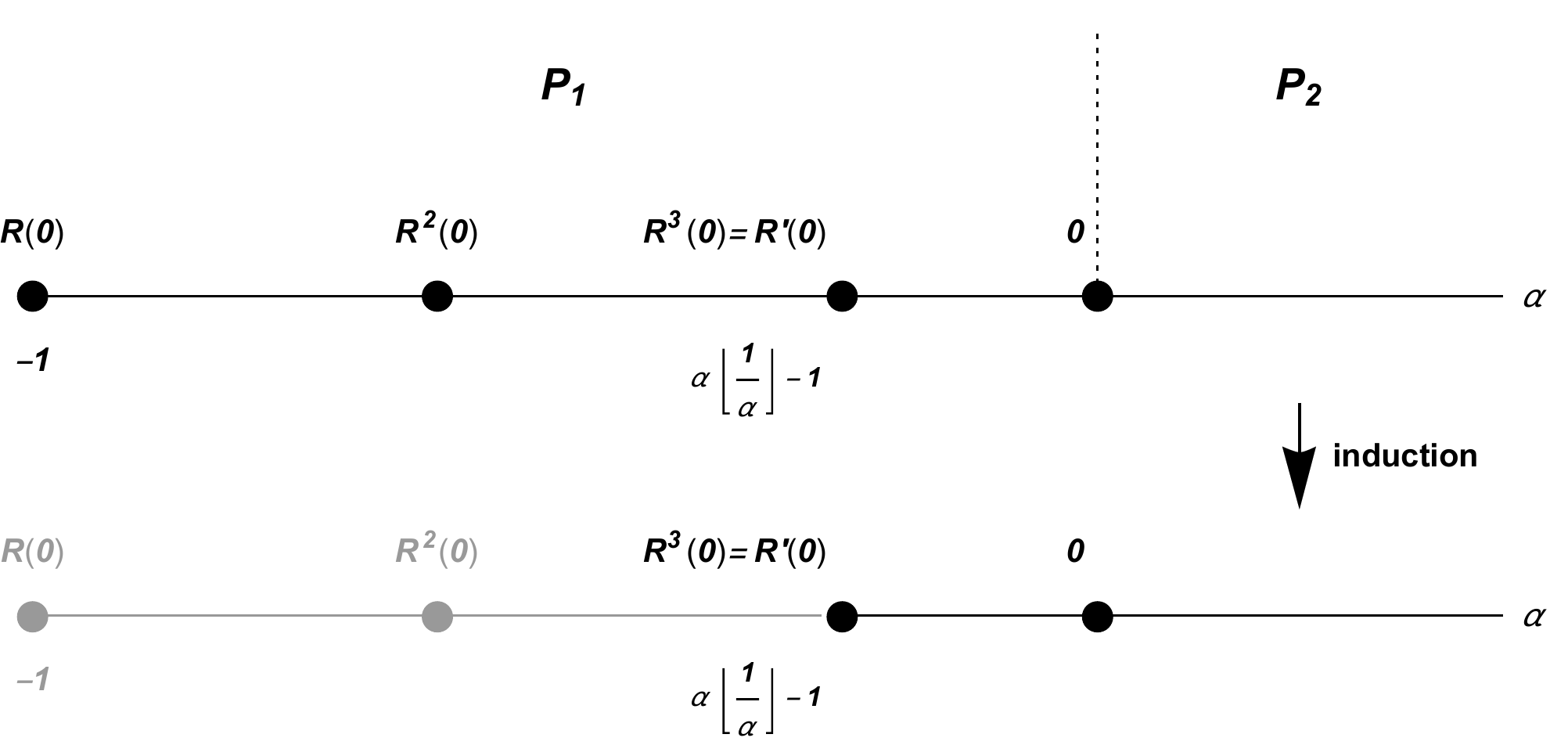}
\vskip 0.3cm
(a) Induction without restacking loses some part of the information. The intervals $[R(0),R^2(0))$ and $[R^2(0),R^3(0))$ depicted in light gray are no longer present in the induced rotation.
\vskip 1cm
\includegraphics[width=\textwidth]{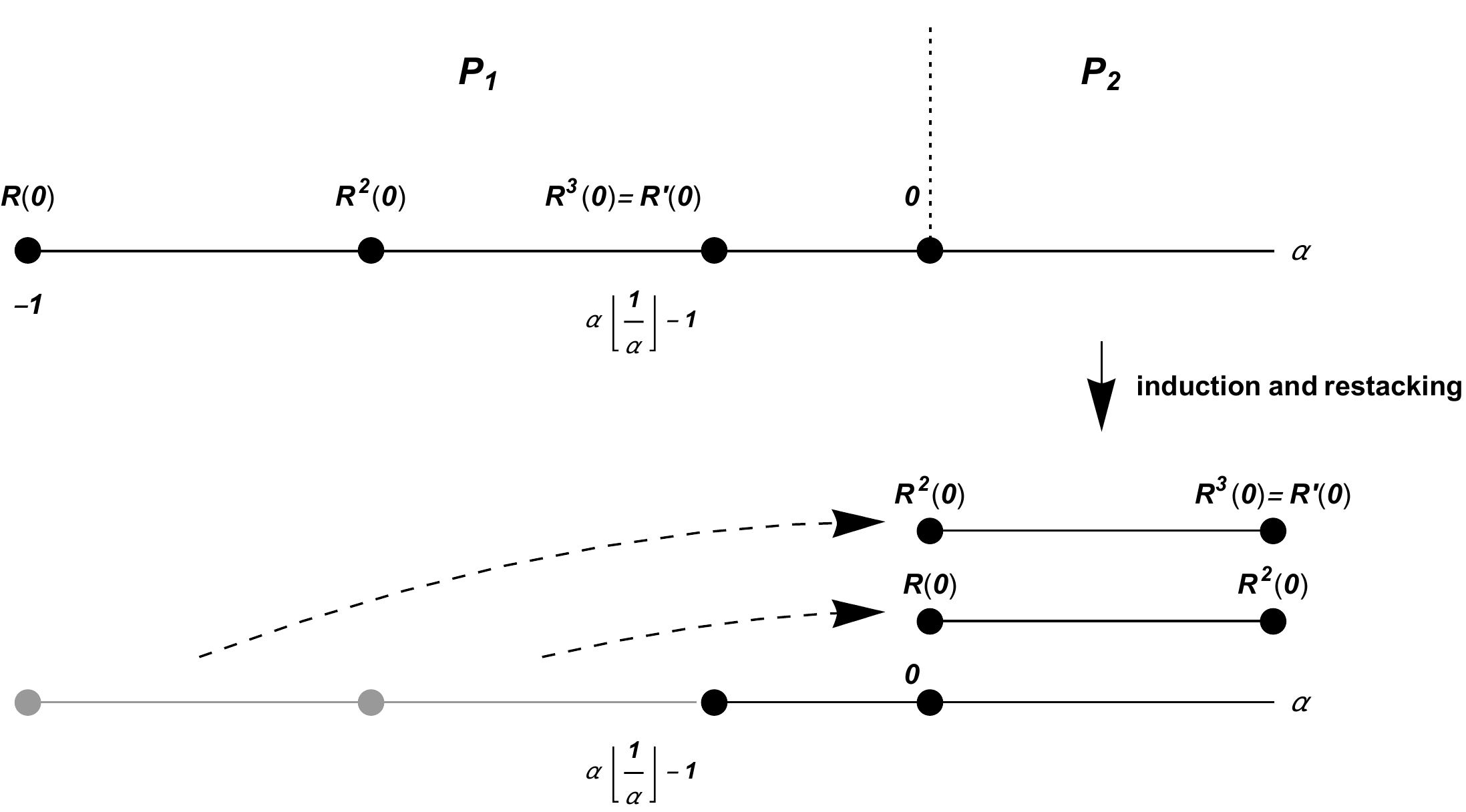}
(b) Induction together with restacking the intervals keeps all the information. The light gray intervals $[R(0),R^2(0))$ and $[R^2(0),R^3(0))$ are stacked on the longer interval of the induced rotation.
\caption{Induction without (a)  and with (b) restacking.\label{fig:inductionRestack}}
\end{figure}

A first idea on how to mend this is indicated in Figure~\ref{fig:inductionRestack}~(b): one could ``stack'' the lost intervals on the larger interval of the induced rotation. This would keep the information of the last induction step. However, acting in this way we can go back at most to the setting from which we started but not farther to the ``past''. 

\begin{figure}[hh]
\includegraphics[width=\textwidth]{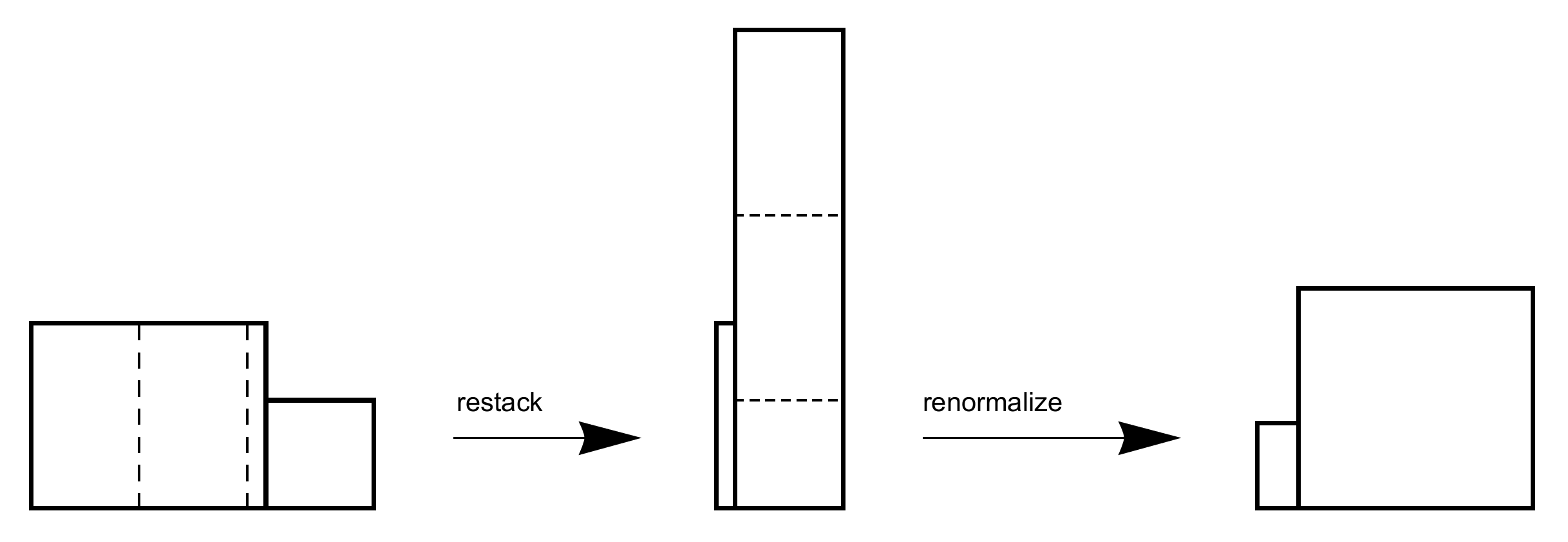}
\caption{Step 1: Restack the boxes. Step 2: Renormalize in a way that the larger box has length $1$ again.\label{fig:RestackSturmian}}
\end{figure}

To make the induction process bijective, it is more convenient to build rectangular boxes above the intervals as indicated in Figure~\ref{fig:RestackSturmian} (this approach is extensively exploited in Arnoux and Fisher~\cite{AF:01}; we follow here \cite[Section~6.6]{Fog02}). The lengths of the boxes are given by the intervals on which the induction process starts: one box is of length $1$, the other one has length $\alpha$ for some $\alpha\in(0,1)\setminus\Q$. The heights are chosen in a way that the longer rectangle is also the higher one and that the total area of the two rectangles is equal to one. The induction process can now be performed on the rectangles as indicated in Figure~\ref{fig:RestackSturmian}: let $a\times d$ be the size of the left rectangle and $b\times c$ the size of the right one. Slice the larger rectangle by vertical cuts into pieces of lengths equal to $b$ until a slice of length less than $b$ remains. Then stack all slices of length $b$ on the smaller rectangle. The result can be seen in the middle of Figure~\ref{fig:RestackSturmian}. After that renormalize the resulting pair of rectangles (as we did in the induction process on the intervals) by making it ``thinner'' and ``longer'' in a way that the length of the larger rectangle is equal to $1$ again and the area of the whole region remains $1$.

Call the resulting mapping on the rectangles $\Psi$. {\it A priori}, the mapping $\Psi$ is a mapping from a subset of $\R^4$ to a subset of $\R^4$. However, since $ad+bc=1$ and $\max\{a,b\}=1$, we can eliminate two coordinates and we are left with a mapping in two variables.

We make this precise in the following definition.

\begin{definition}[Natural extension of the gauss map, see~\cite{AF:01}]
Let $\Delta_{\mathrm{m}}$ be the set of pairs $(a\times d, b\times c)$ of rectangles of total area $1$ such that the widest one is the highest one ({\em i.e.}, $a > b \;\Leftrightarrow \; d > c$) and such that the width of the widest one is equal to $1$ ({\em i.e.}, $\max\{a,b\}=1$). Let $\Delta_{\mathrm{m},0}$ be the subset of $\Delta_{\mathrm{m}}$ with $a=1$, and $\Delta_{\mathrm{m},1}$ the subset of $\Delta_{\mathrm{m}}$ with $b=1$.

The mapping $\Psi$ is defined on $\Delta_{\mathrm{m},1}$ as
\[
(a,d) \mapsto \Big(\Big\{\frac1a \Big\}, a-da^2\Big),
\]
and similarly on $\Delta_{\mathrm{m},0}$. It is called the \emph{natural extension} of the Gauss map (which is seen in the first coordinate).
\end{definition}

\begin{remark}
The subscript ``$\mathrm{m}$'' stands for \emph{multiplicative} since we work here with the multiplicative version of the classical continued fraction algorithm defined by the Gauss map. An analogous theory exists for the additive algorithm as well, see~\cite{AF:01}.
\end{remark}

The mapping $\Psi$ is bijective as becomes clear from its geometric interpretation. Moreover, it is easy to show that $\Psi$ preserves the Lebesgue measure. By integrating away the second coordinate one can show that the invariant measure of the Gauss map is $\frac{dx}{\ln 2(1+x)}$ (see~{\it e.g.}~\cite[Chapter~3]{EW:11}). We mention that another natural extension of the Gauss map defined on the unit square is provided in~\cite{NIT:77}.
 
\begin{figure}[hh]
\includegraphics[width=0.6\textwidth]{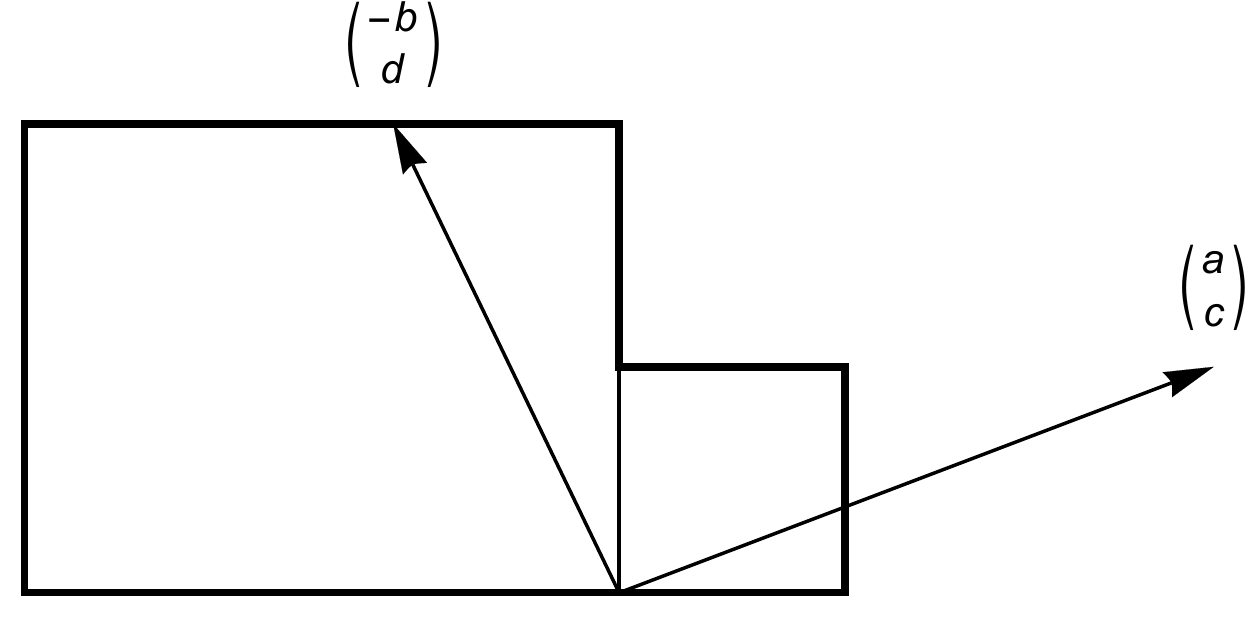}
\caption{A pair of boxes is a fundamental domain of a lattice\label{fig:latticeBasis}}
\end{figure} 
 
We can also see Sturmian sequences in the rectangular boxes. To this end note first that a pair of boxes $a\times d$ and $b\times c$ is a fundamental domain of the lattice spanned by the vectors $(a,c)^t$ and $(-b,d)^t$. This is illustrated in Figure~\ref{fig:latticeBasis} and has the consequence that the ``L-shaped'' region formed by this pair of boxes can be used to tile the plane with respect to this lattice as indicated in Figure~\ref{fig:SturmianTiling}. 
 
\begin{figure}[hh]
\includegraphics[width=0.6\textwidth]{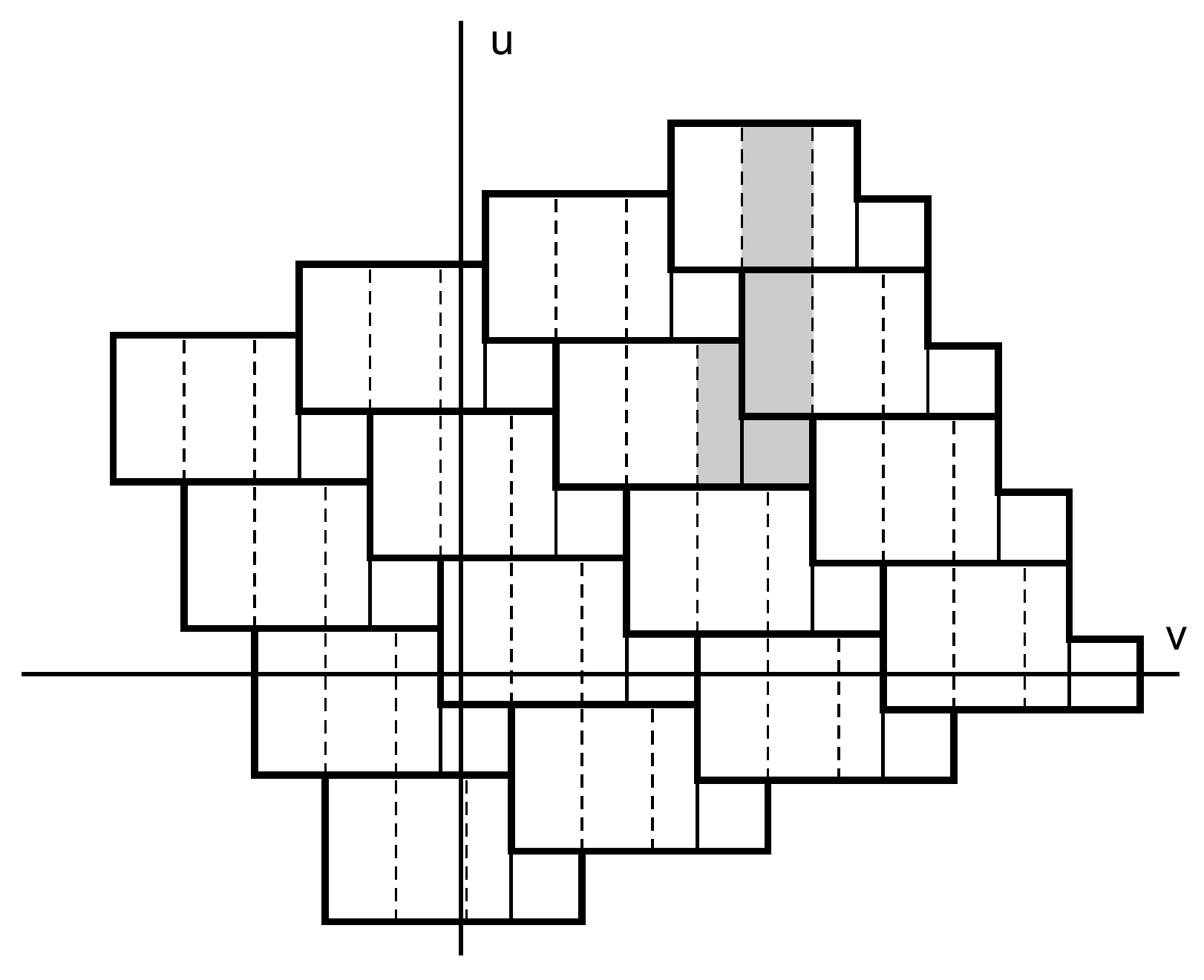}
\caption{The vertical line is coded by a  Sturmian sequence $u$, the horizontal line by a Sturmian sequence $v$. The restacking procedure desubstitutes $u$ and substitutes $v$. The shaded region is a restacked fundamental domain.\label{fig:SturmianTiling}}
\end{figure}

Let us mark a point in this tiling. If we start from this point and move upwards and write out $1$ whenever we pass through a large rectangle, and $2$, whenever we pass through a small one, we get the coding $u$ of a rotation by $\alpha$ on the interval $(-1,\alpha)$ which, by Theorem~\ref{th:sturmrot}, is a Sturmian sequence. This is indicated in Figure~\ref{fig:SturmianTiling}. In the same way we can produce a Sturmian sequence $v$ by moving horizontally.

If we restack each of the fundamental domains, according to the procedure described above, we get a new fundamental domain (indicated by the shaded region in Figure~\ref{fig:SturmianTiling}). We now code the same vertical line using this restacked region. Doing this we obtain another Sturmian sequence $u^{(1)}$ which, by the definition of the restacking process, satisfies $u= \sigma(u^{(1)})$, where $\sigma$ is the substitution defining the induction process as in the proof of Lemma~\ref{lem:rotCF}. On the other hand, looking at the horizontal line we get $v^{(-1)}=\sigma(v)$ as the new coding. Thus the restacking process corresponds to the mapping
\[
(u,v) \mapsto (u^{(1)},v^{(-1)}).
\]
As mentioned at the beginning of this section, we cannot reconstruct $u$ from $u^{(1)}$, however we \emph{can} reconstruct $(u,v)$ from $(u^{(1)},v^{(-1)})$ since the type of the Sturmian sequence $v^{(-1)}$ tells us (which power of) which of the two substitutions $\sigma_{1},\,\sigma_{2}$ from \eqref{eq:sturmsubs} we have to use to get back. This makes the coding process bijective as well. We could mark the pair of rectangles discussed above by a point $(x,y)$ and look at the itinerary of this point under the restacking process. This would give an extension $\tilde \Psi$ of the mapping $\Psi$ that is defined on the $\mathbb{T}^2$-fibers over $\Delta_{\mathrm{m}}$ (see~\cite{AF:01}).

The following remark is of particular importance. 

\begin{remark}
Regardless of the point in the ``L-shaped'' region in which we start, the ``vertical'' Sturmian sequence will always be contained in the same Sturmian system. Thus we can say that the ``L-shaped'' pairs of rectangles parametrize the Sturmian systems (which are characterized by their coding sequence according to Proposition~\ref{lem:sturmsys}~(iii)), while the ($x$-coordinates of the) points in a given region parametrize the sequences contained in this system. The same is true for the ``vertical'' Sturmian sequence w.r.t.\ the $y$-coordinates. 
\end{remark}

We also mention that the vertical line producing the coding $u$ can also be extended downwards. This yields a sequence $\tilde u \in \A^\Z$ as a coding. Such a sequence is an example of a \emph{bi-infinite Sturmian sequence} (the same can be done in the horizontal direction). Bi-infinite Sturmian sequences are studied for instance in \cite[Section~6.2]{Fog02}. It turns out that some of their properties are nicer than in our one-sided case since one no longer has troubles coming from ``the beginning'' of the sequences.

\medskip

Artin~\cite{Artin:24} observed that the continued fraction algorithm can be viewed as a Poincar\'e section of the geodesic flow on the unit tangent bundle $\mathrm{SL}_2(\Z)\setminus \mathrm{SL}_2(\R)$ of the modular surface $\mathrm{SL}_2(\Z)\setminus \mathbb{H}$. In the meantime this correspondence between the continued fraction algorithm and the geodesic flow was studied by many authors (see {\em e.g.} Series~\cite{Series:85}) and discussed in connection with our setting by Arnoux~\cite{Arnoux:94} and later by Arnoux and Fisher~\cite{AF:01}. The necessary details on the modular surface and its unit tangent bundle including an explanation why the flow ${\rm diag}(e^{t},e^{-t})$ which will come up below is a \emph{geodesic flow} on the homogeneous space $\mathrm{SL}_2(\Z)\setminus \mathrm{SL}_2(\R)$ can be found for instance in \cite{Arnoux:94} or \cite[Chapter~9]{EW:11}.

We now explain briefly how the geodesic flow on $\mathrm{SL}_2(\Z)\setminus \mathrm{SL}_2(\R)$ enters our model. We have to restack the rectangles as above and then  renormalize the lattice again. This can be done also in the following way. First multiply the basis of the lattice from the right by ${\rm diag}(e^{t},e^{-t})$ for $t$ varying from $0$ to the threshold value for which the width of the \emph{smallest} rectangle equals $1$. Then restack as above to end up at a pair of rectangles whose \emph{larger} rectangle has width $1$. Altogether, starting from a pair of rectangles drawn on the left hand side of Figure~\ref{fig:RestackSturmian} we ended up with a pair drawn on its right side. We just did the renormalization smoothly and we did it before the restacking instead of after it. 

What we do can be explained more precisely as follows:

\begin{itemize}
\item Define the set 
\[
\begin{split}
\Omega_{\mathrm{m}} =  \Omega_{\mathrm{m},0} \cup \Omega_{\mathrm{m},1}  = 
&\left\{
M=\begin{pmatrix}a & c \\ -b &d \end{pmatrix}\; :\; 
0 < a  <1 \le c, \; 0 < d < b,\; ad+bc=1
\right\}  \cup \\
& 
\left\{
M=\begin{pmatrix}a & c \\ -b &d \end{pmatrix}\; :\; 
0 < c  < 1 \le a, \; 0 < b < d, \; ad+bc=1
\right\}.
\end{split}
\]
One can show that a.e.\ lattice has exactly one basis made of row vectors of a matrix in $\Omega_{\mathrm{m}}$ (see~\cite{Arnoux:94}). Thus $\Omega_{\mathrm{m}}$ is a (measure theoretic) \emph{fundamental domain} for the action of $\mathrm{SL}_2(\Z)$ on $\mathrm{SL}_2(\R)$.

\item Start with a lattice, associate with it a basis taken from $\Omega_{\mathrm{m}}$.

\item Hit this lattice (together with the chosen basis) with the geodesic flow ${\rm diag}(e^{t},e^{-t})$, $t\ge 0$.

\item For increasing $t$ this will eventually deform the basis in a way that the width of the \emph{smaller} rectangle gets equal to $1$ (and we would leave $\Omega_{\mathrm{m}}$ when deforming this basis further). If we restack at this point we end up with a pair of rectangles contained in the \emph{Poincar\'e section} $\Delta_{\mathrm{m}}$: indeed, after restacking the \emph{larger} rectangle will have width $1$.

\item Change the basis of the lattice to the basis corresponding to the new pair of rectangles according to Figure~\ref{fig:latticeBasis}. Note that restacking \emph{does not change the lattice}, so the geodesic flow, which acts on $\mathrm{SL}_2(\Z)\setminus \mathrm{SL}_2(\R)$, is not affected by this base change. However, this restacking has the effect that it creates a new basis of the lattice that remains inside $\Omega_{\mathrm{m}}$ when it gets further deformed by the action of the flow. Thus we can repeat the procedure.

\item Repeating this procedure, the geodesic flow yields a sequence of restackings: any time the width of the smaller rectangle gets equal to $1$ by restacking, the according basis gets inside the \emph{Poincar\'e section} $\Delta_{\mathrm{m}}$. This restacking performs one step of the natural extension of the Gauss map.

\item Thus the geodesic flow on $\mathrm{SL}_2(\Z)\setminus \mathrm{SL}_2(\R)$ can be regarded as a so-called \emph{suspension flow} of the natural extension of the Gauss map.
\end{itemize}

This viewpoint has many advantages and one can prove results on continued fractions using the well-developed theory of the geodesic flow on $\mathrm{SL}_2(\Z)\setminus \mathrm{SL}_2(\R)$.

The same procedure can also be performed for \emph{pointed} pairs of rectangles (which we needed to study Sturmian sequences, see Figure~\ref{fig:SturmianTiling}). This has the effect that the geodesic flow on $\mathrm{SL}_2(\Z)\setminus \mathrm{SL}_2(\R)$ has to be replaced by the so-called \emph{scenery flow} which also takes care of the distinguished point in the ``L-shaped'' region. All this is described in detail in~\cite{AF:01}. 

We mention that similar results have been obtained for variants of the classical continued fraction algorithm. For instance, Arnoux and Schmidt~\cite{Arnoux-Schmidt:13,Arnoux-Schmidt:14} proved that the $\alpha$-continued fraction algorithm, Rosen's continued fraction algorithm as well as Veech's continued fraction algorithm can be viewed as Poincar\'e sections of a geodesic flow. The material presented in this section also forms an easy case of the wide and appealing field of interval exchange transformations and their dynamics (see {\em e.g.} Viana~\cite{Viana:06} for a survey).


\section{Problems with the generalization to higher dimensions}\label{sec:problems3}

According to Cassaigne, Ferenczi, and Zamboni~\cite{Cassaigne-Ferenczi-Zamboni:00} it was conjectured since the beginning of the 1990s that the beautiful correspondence between Sturmian sequences, continued fractions, and irrational rotations on the circle described in Section~\ref{sec:sturm} can be extended to higher dimensions. 
The same paper gives strong indications towards the wrongness of this conjecture. Indeed, in~\cite{Cassaigne-Ferenczi-Zamboni:00} Arnoux-Rauzy sequences over a three letter alphabet that are not balanced and that cannot be viewed as natural codings of rotations on the two dimensional torus with finite fundamental domain are constructed. It is the objective of the present section to explain their work and to give an account on further results by Cassaigne, Ferenczi, and Messaoudi~\cite{Cassaigne-Ferenczi-Messaoudi:08} concerning weakly mixing Arnoux-Rauzy systems as well as Arnoux-Rauzy systems with nontrivial eigenvalues. 

\subsection{Arnoux-Rauzy sequences} In an attempt to pave the way for a generalization to higher dimensions of the correspondence between combinatorics, arithmetics, and dynamical systems outlined in Section~\ref{sec:sturm}, Arnoux and Rauzy~\cite{Arnoux-Rauzy:91} defined sequences over the alphabet $\{1,2,3\}$ whose properties are inspired by Sturmian sequences.   

In the following definition a \emph{right special factor}\index{right special factor}\index{special factor!right} of a sequence $w\in\{1,2,3\}^\N$ is a factor $v$ of $w$ for which there are distinct letters $a,b\in\{1,2,3\}$ such that $va$ and $vb$ both occur in $w$. A \emph{left special factor}\index{left special factor}\index{special factor!left} is defined analogously. The definition of several other objects and notations from Section~\ref{sec:sturm} carry over from two to three letter alphabets without any change and we will use them without defining them again (we will give exact definitions for the general setting from Section~\ref{sec:genset} onwards).

\begin{definition}[Arnoux-Rauzy sequence, see~\cite{Arnoux-Rauzy:91}]\label{def:ar}\index{Arnoux-Rauzy sequence}\index{sequence!Arnoux-Rauzy}
A sequence $w\in\{1,2,3\}^\N$ is called \emph{Arnoux-Rauzy sequence} if $p_w(n)=2n+1$ and if $w$ has only one right special factor and only one left special factor for each given length $n$.
\end{definition}

Let $w$ be an Arnoux-Rauzy sequence. Let $(\Gamma_n)$ be a sequence of directed graphs defined in the following way. For each $n\in\N$ the vertices of $\Gamma_{n}$ are the factors of length $n$ of $w$. There is a directed edge from $u$ to $v$ if and only if there are letters $a,b\in\{1,2,3\}$ and a word $x\in\{1,2,3\}^*$ such that $u=ax$ and $v=xb$. Inspecting these graphs we see that two cases can occur. If the left special factor $v$ of length $n$ is also the right special factor then $\Gamma_n$ is a bouquet of three circles whose common vertex is $v$, otherwise it is a union of three circles that share the line between the vertices corresponding to the right and left special factor. An investigation of these graphs (as done in~\cite[Section~2]{Arnoux-Rauzy:91}) shows that Arnoux-Rauzy sequences are ``$S$-adic'' and we get the following analog of Proposition~\ref{prop:STsubs}.

\begin{proposition}[{see~\cite[Section~2]{Arnoux-Rauzy:91}}]
Let the \emph{Arnoux-Rauzy substitutions} $\sigma_1,\sigma_2,\sigma_3$ be defined by 
\begin{equation}\label{eq:ARsubs}
\sigma_1: 
\begin{cases} 1 \mapsto 1,\\ 2 \mapsto 12,\\ 3 \mapsto 13,\end{cases} \qquad
\sigma_2: 
\begin{cases} 1 \mapsto 21,\\ 2 \mapsto 2,\\ 3 \mapsto 23,\end{cases} \qquad
\sigma_3: 
\begin{cases} 1 \mapsto 31,\\ 2 \mapsto 32,\\ 3 \mapsto 3.\end{cases}
\end{equation}
Then for each Arnoux-Rauzy sequence $w$ there exists a sequence $\bsigma=(\sigma_{i_n})$, where $(i_n)$ takes each symbol in $\{1,2,3\}$ an infinite number of times, such that $w$ has the same language as
\begin{equation}\label{eq:ARSadic}
u=\lim_{n\to\infty}\sigma_{i_0}\circ\sigma_{i_1}\circ\cdots\circ\sigma_{i_n}(1).
\end{equation}
\end{proposition}

By this proposition each Arnoux-Rauzy sequence $w$ has a \emph{coding sequence} $\bsigma$ of Arnoux-Rauzy substitutions and we may define the dynamical system $(X_w,\Sigma)=(X_\bsigma,\Sigma)$ as the dynamical system associated with $w$, where $X_w=X_\bsigma$ is the set of sequences whose language equals the language of $w$ and which just depends on $\bsigma$. These dynamical systems are called \emph{Arnoux-Rauzy systems}.

Let $w$ be an Arnoux-Rauzy sequence with coding sequence $\bsigma=(\sigma_{i_n})$ and let $(M_{i_n})$ be the associated sequence of incidence matrices.
Since each symbol in $\{1,2,3\}$ occurs infinitely often in $(i_n)$ the associated sequence of incidence matrices $(M_{i_n})$ is easily seen to be primitive in the sense that for each $m\in\N$ there is $n>m$ such that $M_{i_{[m,n)}}$ is a positive matrix. Indeed, a block $M_{i_{[m,n)}}$ is primitive if and only if it contains each of the three matrices $M_1,M_2,M_3$ at least once. 

\begin{lemma}\label{lem:ARmin}
Let $w$ be an Arnoux-Rauzy sequence with coding sequence $\bsigma$. Then the dynamical system $(X_\bsigma,\Sigma)$ is minimal and uniquely ergodic.
\end{lemma}

\begin{proof}
Minimality follows if we can show that $L(v)=L(w)$ for each $v\in X_\bsigma$. This in turn holds if each factor of $w$ occurs infinitely often in $w$ with bounded gaps, which we will now prove. Let $x$ be a factor of $w$. As $w$ has the same language as the sequence $u$ in \eqref{eq:ARSadic}, by primitivity of $(M_{i_n})$ there is $m\in\N$ such that $x$ occurs in $\sigma_{i_{[0,m)}}(1)$.  Using primitivity again we see that there exists $n>m$ such that $M_{i_{[m,n)}}$ is a positive matrix. This entails that the word $\sigma_{i_{[m,n)}}(b)$ contains $1$ for any $b\in\{1,2,3\}$ and, hence, $\sigma_{i_{[0,n)}}(b)$ contains $\sigma_{i_{[0,m)}}(1)$ and, {\it a fortiori}, also $x$ for each $b\in\{1,2,3\}$. Thus $x$ occurs in $w$ infinitely often with gaps bounded by $2\max\{|\sigma_{i_{[0,n)}}(b)| \,:\, b\in\{1,2,3\}\}$.  

Unique ergodicity of $(X_\bsigma, \Sigma)$ can be derived from a general result of Boshernitzan~\cite{Boshernitzan:84} due to the fact that $(X_\bsigma, \Sigma)$ is minimal and its elements have linear complexity with slope less than $3$. 

\end{proof}

This proof implies that each Arnoux-Rauzy sequence is {\em uniformly recurrent}.

Generalizing an idea of Arnoux~\cite{Arnoux:88}, in \cite{Arnoux-Rauzy:91} it was shown that each Arnoux-Rauzy sequence $w$ can be viewed as a coding of a $6$-interval exchange transformation (by using sequences over an alphabet with only $3$ letters!) and that each Arnoux-Rauzy system can be represented by such a $6$-interval exchange. In view of a result by Katok~\cite{Katok:80} this implies that Arnoux-Rauzy systems cannot be mixing. The incidence matrices of Arnoux-Rauzy substitutions can be used to define a generalized continued fraction algorithm in the sense of Section~\ref{sec:genCF} below. However, this algorithm only works for vectors taken from a set of measure zero, the so-called \emph{Rauzy gasket}. For more on this interesting set we refer to \cite{AS:13,AHS:15b,AHS:15,DD09,Levitt:93}.

Another interesting class of sequences of complexity $2n+1$ over the alphabet $\{1,2,3\}$ has been defined recently in~\cite{CLL} and is currently subject to intensive investigation. Compared to Arnoux-Rauzy sequences it has the advantage that it is defined in terms of only two substitutions and gives rise to a continued fraction algorithm that works on a set of full measure.

\subsection{Imbalanced Arnoux-Rauzy sequences}\label{sec:ARImbalance}
To get the perfect analogy with the Sturmian case it would be desirable to represent a given Arnoux-Rauzy sequence $w$ as a natural coding of a rotation on the two-dimensional torus $\mathbb{T}^2$. In the seminal paper of Rauzy~\cite{Rauzy:82}, this was achieved for the sequence $w=\lim\sigma^n(1)$, where $\sigma$ is the famous \emph{Tribonacci substitution} defined by 
\begin{equation}\label{eq:substribo}
\sigma: 
\begin{cases} 1 \mapsto 12,\\ 2 \mapsto 13,\\ 3 \mapsto 1.\end{cases}
\end{equation}
Since $\sigma^3=\sigma_1\circ \sigma_2 \circ \sigma_3$ the sequence $w$ is an example of an Arnoux-Rauzy sequence (with periodic coding sequence). Several years ago Barge, \v{S}timac, and Williams~\cite{Barge-Stimac-Williams:13} as well as Berth\'e, Jolivet, and Siegel~\cite{Berthe-Jolivet-Siegel:12} could generalize this result and proved that each Arnoux-Rauzy sequence $w$ with periodic coding sequence is a natural coding of a rotation on $\mathbb{T}^2$ (a weaker result in this direction is already contained in \cite{Arnoux-Ito:01}). A general theory for nonperiodic sequences was established only recently, see Berth\'e, Steiner, and Thuswaldner~\cite{Berthe-Steiner-Thuswaldner}, and we will come back to this in later sections.

We recall that a sequence $w=w_0w_1\ldots\in\{1,2,3\}^\N$ is a \emph{natural coding} of a rotation $R$ on $\mathbb{T}^2$  if there exists a fundamental domain $\Omega$ of $\mathbb{T}^2$ in $\mathbb{R}^2$ together with a partition $\Omega=\Omega_1\cup\Omega_2\cup\Omega_3$ such that on each $\Omega_i$ the map $R'$ induced on $\Omega$ by the rotation $R$ acts as a translation by a vector $\mathbf{a}_i\in\R^2$ and for some point $x\in\Omega$ we have $R'^k(x)\in \Omega_{w_k}$ for each $k\in\N$
(see also Definition~\ref{def:NC}).

An Arnoux-Rauzy sequence is not always a coding of a rotation on $\mathbb{T}^2$ with \emph{bounded} fundamental domain. The reason for this is the lack of balance for some particular instances of such sequences.  Following Cassaigne, Ferenczi, and Zamboni~\cite{Cassaigne-Ferenczi-Zamboni:00} we now sketch the construction of an Arnoux-Rauzy sequence that is not balanced.

Let $C\ge 1$ be an integer. Generalizing the notion of balance from Section~\ref{sec:sturmianprop} we say that a sequence $w\in\{1,2,3\}^\N$ is \emph{$C$-balanced} if each pair of factors $(u,v)$ of $w$ having the same length satisfies $\big| |u|_a - |v|_a \big| \le C$ for each $a\in\{1,2,3\}$. The following result implies that there is no uniform $C$ that gives $C$-balance for each Arnoux-Rauzy sequence. Here a substitution is called primitive if its incidence matrix is primitive.

\begin{lemma}[{see \cite[Proposition~2.2]{Cassaigne-Ferenczi-Zamboni:00}}]\label{lem:CFZ22}
For each integer $C\ge 1$ there is a finite sequence of Arnoux-Rauzy substitutions $\sigma_{i_1},\ldots,\sigma_{i_k}$ such that $\sigma= \sigma_{i_1}\circ\cdots\circ \sigma_{i_k}$ is primitive and for each Arnoux-Rauzy sequence $w$ the Arnoux-Rauzy sequence $\sigma(w)$ is not $C$-balanced.
\end{lemma}

\begin{proof}
We prove by induction that for each $n\ge 2$ there exist $a_n,b_n,c_n\in\N$ and a primitive composition of Arnoux-Rauzy matrices $\sigma^{(n)}$ such that for each Arnoux-Rauzy sequence $w$ the sequence $\sigma^{(n)}(w)$ contains two factors $u^{(n)}$ and $v^{(n)}$ of equal length with 
\[
\begin{pmatrix}
|u^{(n)}|_i\\
|u^{(n)}|_j\\
|u^{(n)}|_k
\end{pmatrix}=
\begin{pmatrix}
a_n\\
b_n+n\\
c_n
\end{pmatrix} \quad\hbox{and}\quad
\begin{pmatrix}
|v^{(n)}|_i\\
|v^{(n)}|_j\\
|v^{(n)}|_k
\end{pmatrix}=
\begin{pmatrix}
a_n+1\\
b_n\\
c_n+n-1
\end{pmatrix}
\]
for some choice $i,j,k$ with $\{i,j,k\}=\{1,2,3\}$. This will prove the result because $\big| |u^{(n)}|_j - |v^{(n)}|_j\big| = n$ shows that $\sigma^{(n)}(w)$ is not $(n-1)$-balanced.

For the induction start take $n=2$ and $\sigma^{(2)}=\sigma_1\sigma_2$ with $u^{(2)}=212$ and $v^{(2)}=131$. 

To perform the induction step assume that the result is true for some $n$ and let $u^{(n)}$, $v^{(n)}$, $a_n$, $b_n$, $c_n$, $i$, $j$, $k$, and $\sigma^{(n)}$ be as above. Set $\sigma^{(n+1)}=\sigma_k^n\circ\sigma_i^n\circ\sigma^{(n)}$. We now construct $u^{(n+1)}$ and $v^{(n+1)}$. Let $u$ be a nonempty factor of some Arnoux-Rauzy sequence $w$. Then for each $a\in\{1,2,3\}$ the word $\sigma_a(u)a$ is a factor of $\sigma_a(w)$ which begins with $a$. If we define $\sigma_{(a,+)}(u)=\sigma_a(u)a$ and $\sigma_{(a,-)}(u)$ as the suffix of $\sigma_a(u)$ of length $|\sigma_a(u)|-1$ ({\it i.e.}, the first letter of $\sigma_a(u)$ is canceled) we see that $u^{(n+1)}=\sigma_{(k,-)}^n\sigma_{(i,+)}^n(v_n)$ and $v^{(n+1)}=\sigma_{(k,+)}^n\sigma_{(i,-)}^n(v_n)$ are factors of $\sigma^{(n+1)}(w)$. 
Using the definition of $\sigma_{(a,+)}$ and $\sigma_{(a,-)}$ one can now check directly that
\[
\begin{pmatrix}
|u^{(n+1)}|_k\\
|u^{(n+1)}|_i\\
|u^{(n+1)}|_j
\end{pmatrix}=
\begin{pmatrix}
a_{n+1}\\
b_{n+1}+n+1\\
c_{n+1} 
\end{pmatrix} \quad\hbox{and}\quad
\begin{pmatrix}
|v^{(n+1)}|_k\\
|v^{(n+1)}|_i\\
|v^{(n+1)}|_j
\end{pmatrix}=
\begin{pmatrix}
a_{n+1}+1\\
b_{n+1}\\
c_{n+1}+n
\end{pmatrix},
\]
where
\[
\begin{pmatrix}
a_{n+1}\\
b_{n+1}\\
c_{n+1}
\end{pmatrix}
=
\begin{pmatrix}
c_{n} + n - 1 + n(a_n +n(b_n+c_n+n)+b_n)\\
a_{n} + n(b_n+c_n+n-1)\\
b_{n}
\end{pmatrix}.
\]
\end{proof}

This lemma can even be sharpened in the following way.

\begin{lemma}[{see \cite[Proposition~2.3]{Cassaigne-Ferenczi-Zamboni:00}}]\label{lem:CFZ23}
For each integer $C\ge 1$ and each composition of Arnoux-Rauzy substitutions $\sigma$ there exists a primitive composition of Arnoux-Rauzy sequences $\sigma'$ such that for each Arnoux-Rauzy sequence $w$ the Arnoux-Rauzy sequence $\sigma\circ\sigma'(w)$ is not $C$-balanced.
\end{lemma}

The proof is technical and we do not provide it here. The idea is to use Lemma~\ref{lem:CFZ22} in order to choose $\sigma'$ in a way that $\sigma'(w)$ is not $K$-balanced for each Arnoux-Rauzy sequence $w$, where $K$, which depends on the incidence matrix of $\sigma$, is so large that even after the application of $\sigma$ we cannot reach $C$-balance. 

We are now able to establish the following result.

\begin{theorem}[{see \cite[Theorem~2.4]{Cassaigne-Ferenczi-Zamboni:00}}]\label{prop:CFZ24}
There exists an Arnoux-Rauzy sequence which is not $C$-balanced for any $C\ge1$.
\end{theorem} 

\begin{proof}
By Lemma~\ref{lem:CFZ23} one can construct primitive compositions of Arnoux-Rauzy substitutions $\sigma^{(1)}, \ldots, \sigma^{(C)}$ such that $\sigma^{(1)}\circ\cdots\circ\sigma^{(C)}(w)$ is not $C$-balanced for any Arnoux-Rauzy sequence $w$. Thus
$
u=\lim_{C\to\infty}\sigma^{(1)}\circ\cdots\circ\sigma^{(C)}(w)
$
is the desired sequence.
\end{proof}

Using this proposition we are able to establish the following result of \cite{Cassaigne-Ferenczi-Zamboni:00} which strongly indicates that an unconditional generalization of the theory presented in Section~\ref{sec:sturm} is not possible. 

\begin{corollary}[{{\em cf.}~\cite[Corollary~2.6]{Cassaigne-Ferenczi-Zamboni:00}}]\label{thm:cfz00}
There exists an Arnoux-Rauzy sequence which is not a natural coding of a minimal rotation on the $2$-torus with \emph{bounded} fundamental domain.
\end{corollary}

\begin{proof}
By Theorem~\ref{prop:CFZ24} there is an Arnoux-Rauzy sequence $w$ which is not $C$-balanced for any $C>0$. Assume that $w$ is a natural coding of a minimal rotation on $\mathbb{T}^2$ with bounded fundamental domain $\Omega$. Each letter $j\in\{1,2,3\}$ corresponds to a translation $\mathbf{a}_j$ on $\Omega$ and, hence, to each word  $u=u_0\ldots u_{n-1} \in \{1,2,3\}^*$ there corresponds the translation $\mathbf{a}_{u}=\sum_{k=0}^{n-1}\mathbf{a}_{u_k}$ on $\Omega$. Since $\Omega$ is bounded and the rotation is minimal one easily checks that the vectors $\mathbf{a}_1,\mathbf{a}_2,\mathbf{a}_3$ 
satisfy $\R_+\mathbf{a}_1 +\R_+\mathbf{a}_2 + \R_+\mathbf{a}_3 = \R^2$.
This implies that there exists a constant $\gamma > 0$ such that two words $u,v \in \{1,2,3\}^*$ with $\big| |u|_i - |v|_i \big| \ge C$ for some $i\in\{1,2,3\}$ satisfy $\Vert\mathbf{a}_{u}-\mathbf{a}_{v} \Vert_1 > \gamma C$. 

Since $w$ is not balanced there is a letter $i\in\{1,2,3\}$ such that for each $C>0$ there exist two factors $u,v\in \{1,2,3\}^*$ of $w$ with $\big| |u|_i - |v|_i \big| \ge C$.  Thus $\Vert \mathbf{a}_{u}-\mathbf{a}_{v} \Vert_1 > \gamma C$. Since $C$ can be arbitrarily large, this difference can be made arbitrarily large. Thus one of the two vectors $\mathbf{a}_{u},\mathbf{a}_{v}$ can be made arbitrarily large. Assume w.l.o.g.\ that this is $\mathbf{a}_{u}$. Since there is an element $\mathbf{x} \in \Omega$ with $\mathbf{x}+\mathbf{a}_{u} \in \Omega$, the diameter of $\Omega$ is bounded from below by the length of $\mathbf{a}_{u}$. This contradicts the boundedness of the fundamental domain $\Omega$.
\end{proof}

\begin{remark}
We mention that in \cite[Corollary~2.6]{Cassaigne-Ferenczi-Zamboni:00} it is claimed that Corollary~\ref{thm:cfz00} is true without assuming that the fundamental domain is bounded. However, we were not able to verify this proof.
\end{remark}

\subsection{Weak mixing and the existence of eigenvalues}\label{sec:WM}
In Cassaigne, Ferenczi, and Messaoudi~\cite{Cassaigne-Ferenczi-Messaoudi:08} the authors 
give a criterion for \emph{weak mixing} for some class of Arnoux-Rauzy systems. On the other hand they provide a class of Arnoux-Rauzy systems that admit nontrivial \emph{eigenvalues}. Before we give the details, we recall the required terminology from ergodic theory (good references here are for instance Einsiedler and Ward~\cite{EW:11} or Walters~\cite{Walters:82}; we also mention Halmos~\cite{Hal60} where some concepts are illustrated in an intuitive way). 

Let $(X,T,\mu)$ be a dynamical system with invariant measure $\mu$. We say that a complex number $\lambda$ is a \emph{measurable eigenvalue}\index{measurable eigenvalue}\index{eigenvalue!measurable} of $T$  if there exists $f\in L^1(\mu)$, $f\not=0$, such that $f(Tx)=\lambda f(x)$ for $\mu$-almost every $x$. Such an $f$ is called an \emph{eigenfunction}\index{eigenfunction} for $\lambda$. For topological dynamical systems the notion of \emph{topological eigenvalue}\index{topological eigenvalue}\index{eigenvalue!topological} is defined analogously by using continuous eigenfunctions instead of functions from $L^1(\mu)$. 

The transformation $T$ is called \emph{weakly mixing}\index{weakly mixing} if for each $A,B\subset X$ of positive measure we have 
\[
\lim_{n\to\infty}\frac1n\sum_{0\le k< n}|\mu(T^{-k}(A)\cap B) - \mu(A)\mu(B)|=0. 
\]
Weak mixing is equivalent to the fact that $1$ is the only measurable eigenvalue of $T$ and the only eigenfunctions are constants (in this case the dynamical system is said to have \emph{continuous spectrum}\index{continuous spectrum}\index{spectrum!continuous}). We note that rotations are never weakly mixing. They  have \emph{pure discrete spectrum} (with will be defined in Definition~\ref{def:PDS}), meaning that they have ``a lot of eigenfunctions'' and therefore they have a completely different dynamical behavior. Indeed, from the definition of weak mixing we see that iterated preimages of each set tend to ``smear'' (or \emph{mix}) over the whole space, this is of course not the case for the iterated preimages of a rotation.
 
We now come back to the aim of this section and discuss mixing properties of Arnoux-Rauzy systems. Let
\[
u=\lim_{n\to\infty}\sigma_{i_1}^{k_1}\circ\sigma_{i_2}^{k_2}\circ\cdots\circ\sigma_{i_n}^{k_n}(1)
\]
with $i_n\not=i_{n+1}$ be an Arnoux-Rauzy sequence. We define $(n_\ell)$ to be the sequence of indices $n$ for which $i_n\not=i_{n+2}$. The sequence $u$ is uniquely defined by the sequences $(k_n)$ and $(n_\ell)$ (up to permutation of letters). The following result shows a result on weak mixing Arnoux-Rauzy systems for large partial quotients $(k_n)$. 

\begin{theorem}[{see \cite[Theorem~2]{Cassaigne-Ferenczi-Messaoudi:08}}]\label{thm:cfm08}
For an Arnoux-Rauzy sequence $w$ with coding sequence $\bsigma$ and associated sequences $(k_n)$ and $(n_\ell)$  the system $(X_\bsigma,\Sigma,\mu)$  (with $\mu$ being the unique invariant measure) is weakly mixing if 
\[
\hbox{the sequence } (k_{n_\ell+2})_{\ell\in\N} \hbox{ is unbounded},\quad \sum_{\ell\ge 1}\frac1{k_{n_\ell+1}} < \infty,\quad\hbox{and}\quad \sum_{\ell\ge 1}\frac1{k_{n_\ell}} < \infty.
\]
This implies that $(X_\bsigma,\Sigma,\mu)$ is not measurably conjugate to a rotation on $\mathbb{T}^2$.
\end{theorem}

The proof of this result is quite involved. In fact, to get weak mixing, by definition one has to show that there exists no measurable eigenvalue apart from $1$ for the system $(X_\bsigma,\Sigma,\mu)$. This is achieved by verifying the following criterion (see \cite[Proposition~10]{Cassaigne-Ferenczi-Messaoudi:08}): if $\vartheta$ is a measurable eigenvalue of $(X_\bsigma,\Sigma,\mu)$, then $k_{n+1}\{h_n\vartheta\}\to 0$ for $n\to\infty$. Here $h_n$ is the length of $\sigma_{i_1}^{k_1}\circ\cdots\circ\sigma_{i_{n}}^{k_n}(1)$. This criterion is proved using a sequence of nested Rohlin towers which are naturally built using the coding sequence $\bsigma$. 
As mentioned above, because an Arnoux-Rauzy system can be represented by a $6$-interval exchange, it cannot be mixing in view of Katok~\cite{Katok:80}.

To give this section a good end we mention that \cite{Cassaigne-Ferenczi-Messaoudi:08} also contains results that support the hope that at least something along the lines of Section~\ref{sec:sturm} can be done in higher dimensions. Indeed, the authors are able to exhibit criteria for the existence of nontrivial continuous eigenvalues (not equal to $1$) for Arnoux-Rauzy systems which implies that these systems have a rotation as a continuous factor. The novelty here is the fact that these systems still have unbounded partial quotients $(k_n)$. For bounded partial quotients criteria for the existence of continuous and measurable eigenvalues are provided in the more general setting of linear recurrent minimal Cantor systems in Cortez {\it et al.}~\cite{Cortez-Durand-Host-Maass:03}.

It will be our concern in the subsequent sections to exhibit $S$-adic sequences that are even measurable conjugates of rotations on tori of dimension greater than or equal to two.

\section{The general setting}\label{sec:genset}

So far we have seen some elements of the correspondence between Sturmian sequences, the classical continued fraction algorithm, and rotations on the circle. We have also reviewed some results that highlight the problems and limitations of a generalization of this nice interplay between several branches of mathematics to higher dimensions. Nevertheless, we are able to set up a quite general extension of the results contained in Section~\ref{sec:sturm}. Indeed, in the subsequent sections of this chapter we will relate sequences generated by substitutions on alphabets over $d$ letters to generalized continued fraction algorithms and to rotations on the $(d-1)$-dimensional torus. From this point on we will give exact definitions of all objects we use. This may seem redundant as some objects have already been introduced before but as the subject is quite difficult and a variety of concepts and notations is needed along the way we found it better for the reader to do it that way. 

\subsection{$S$-adic sequences}

We now define so-called \emph{$S$-adic sequences} which form analogs of sequences of the form \eqref{eq:2lettercoding} and \eqref{eq:ARSadic} for arbitrary ``coding sequences'' of substitutions over a fixed finite alphabet. To this end we need some notation.

Let $\A=\{1,2,\ldots,d\}$ be a finite \emph{alphabet} whose elements will be called \emph{letters} or \emph{symbols}. Define $\A^*$ to be the free monoid generated by $\A$ equipped with the operation of concatenation. The elements of $\A^*$, which are of the form $v=v_{0}v_{1}\ldots v_{n-1}$ with $n\in\N$ and $v_{i}\in\A$ for $i\in\{0,1,\ldots,n-1\}$, will be referred to as \emph{words}. The integer $n$, which is equal to the number of letters in the word $v$, is called the \emph{length} of $v$ and will be denoted by $|v|$. The unique word of length $0$ is called the \emph{empty word}. Let $\A^\N$ be the space of \emph{right infinite sequences} $w=w_0w_1\ldots$ with $w_i\in\A$ for each $i\in\N$. We equip $\A^\N$ with the product topology of the discrete topology on $\A$. To a sequence $w=w_0w_1\ldots \in\A^\N$ we associate a function $p_{w}:\N\to\N$ which is defined by $$n\mapsto |\{v \in\A^*\,:\, v = w_kw_{k+1}\ldots w_{k+n-1} \hbox{ for some }k\in\N  \}|.$$ The function $p_{w}$ is called the \emph{complexity function} of the sequence $w$. For more on this function we refer for instance to Cassaigne and Nicolas~\cite{CN:10}.

A \emph{substitution} $\sigma$ over the alphabet $\A$ is an endomorphism on $\A^*$ that in our setting will always assumed to be {\em nonerasing} in the sense that the image of each letter is a nonempty word taken from $\A^*$. Being a morphism, a substitution is completely defined by giving its image for each letter. Thus our previous examples of Sturmian substitutions in \eqref{eq:sturmsubs} and of Arnoux-Rauzy substitutions in \eqref{eq:ARsubs} are indeed substitutions. We can extend the domain of a substitution $\sigma$ to $\A^\N$ in a natural way by defining it symbol-wise, {\it i.e.}, by setting $\sigma(w_0w_1\ldots)= \sigma(w_0)\sigma(w_1)\ldots$ The mapping $\sigma$ defined in this way is continuous on $\A^\N$.

With each substitution $\sigma$ over the alphabet $\A$ we associate the $|\mathcal{A}|\times|\mathcal{A}|$ \emph{incidence matrix} $M_\sigma$ whose columns are the abelianized images of $\sigma(i)$ for $i\in\A$. More precisely, letting $|v|_i$ be the number of occurrences of a given letter $i\in\A$ in a word $v\in\A^*$ this matrix is given by $M_{\sigma}=(m_{ij})=(|\sigma(j)|_i)$. The incidence matrix can be seen as the \emph{abelianized} version of $\sigma$. If we define the \emph{abelianization mapping} $\mathbf{l}:\A^* \to \N^d$ by $\mathbf{l}(w)=(|w|_1,\ldots,|w|_d)^t$ (here $\mathbf{x}^t$ is the transpose of a vector $\mathbf{x}\in\mathbb{R}^d$) we have the commutative diagram
\begin{equation}\label{eq:MsigmaDiagram}
\begin{CD}
\A^* @> \sigma >> \A^* \\
@VV\mathbf{l} V @VV\mathbf{l} V\\
\N^d @> M_\sigma >> \N^d
\end{CD}
\smallskip
\end{equation}
which says that $\mathbf{l}\sigma(w)=M_\sigma\mathbf{l}(w)$ holds for each $w\in\A^*$.

We will be interested in special classes of substitutions. Let $\sigma$ be a substitution. Then  $\sigma$ is called \emph{unimodular} if $|\det M_\sigma|=1$, it is called \emph{primitive} if $M_\sigma$ is a primitive matrix ({\it i.e.}, $M_\sigma$ has a power each of whose entries is greater than zero), it is called \emph{irreducible} if $M_\sigma$ has irreducible characteristic polynomial, and it is called \emph{Pisot} if the characteristic polynomial of $M_\sigma$ is the minimal polynomial of a \emph{Pisot number}. We recall that a Pisot number is an algebraic integer $\beta>1$ whose Galois conjugates (apart from $\beta$ itself) are all smaller than $1$ in modulus.

In full generality substitutions are studied for instance in~\cite{Allouche-Shallit:03,Berstel:79,HU:79} and, in a context related to the present chapter, in~\cite{Fog02}.

We will now define the analogs of the ``coding sequences'' used in Sections~\ref{sec:sturm} and~\ref{sec:problems3} for a more general setting. We go in the reverse direction: in the mentioned earlier sections the sequence (of letters) was there first and we constructed a sequence of substitutions that generates this sequence. Now we start with a sequence of substitutions in order to define a sequence of letters.

Let $\bsigma=(\sigma_n)_{n\in \N}$ be a sequence of substitutions over a given finite alphabet $\A$.
For convenience, we will set $M_n=M_{\sigma_n}$ for the incidence matrix of $\sigma_n$ and write $\bM=(M_n)$ for the sequence of these incidence matrices. Moreover, as we will often need blocks of substitutions as well as blocks of matrices we set
\[
\sigma_{[m,n)}=\sigma_m\circ\sigma_{m+1}\circ\cdots\circ\sigma_{n-1}
\quad\hbox{and}\quad M_{[m,n)}=M_mM_{m+1}\cdots M_{n-1}
\]
for positive integers $m\le n$ (here we set $\sigma_{[n,n)}(a)=a$ for all $a\in\A$ and define $M_{[n,n)}$ to be the $|\A|\times |\A|$ identity matrix).

We associate with $\bsigma$ a sequence of languages
\[
\Lg^{(m)}_\bsigma=\{v\in\A^* \;:\; v \hbox{ is a factor of }\sigma_{[m,n)}(a) \hbox{ for some } a\in\A,\, m\le n \} \qquad(m\in\N)
\]
and call $\Lg_\bsigma=\Lg^{(0)}_\bsigma$ the \emph{language of $\bsigma$}.  Here $u\in\A^*$ is a \emph{factor} of $v\in\A^*$ if $v\in\A^*u\A^*$, or, more informally, if the word $u$ occurs somewhere as \emph{subword} in the word $v$. We will use this notation also for (right infinite) sequences later. Then $u\in\A^*$ is a \emph{factor} of $v\in\A^\N$ if $v\in\A^*u\A^\N$. The set of all factors of a sequence $v$ is called the \emph{language} of $v$. It is denoted by $L(v)$. We also introduce the notion of prefix and suffix that will be used later. A \emph{prefix} of a word $v\in\A^*$ is a word $u\in\A^*$ with $v\in u\A^*$ and a \emph{suffix} of $v\in\A^*$ is a word $u\in\A^*$ with $v\in \A^*u$. A prefix of a sequence $v\in\A^\N$ is a word $u\in\A^*$ with $v\in u\A^\N$.

After these preparations we can define \emph{$S$-adic sequences} for a given sequence of substitutions $\bsigma$. The terminology ``$S$-adic'' goes back to Ferenczi~\cite{Ferenczi:96}. In our definition we follow Arnoux, Mizutani, and Sellami~\cite{AMS:14} (see also~\cite[Section~2.2]{Berthe-Steiner-Thuswaldner}).  

\begin{definition}[$S$-adic sequence]\label{def:sadicsequence}\index{$S$-adic sequence}\index{sequence!$S$-adic}
Let $\A$ be a given finite alphabet, let $\bsigma=(\sigma_n)_{n\ge 0}$ be a sequence of substitutions over $\A$, and set $S:=\{\sigma_n\,:\, n\in \N\}$. We call a sequence $w\in\A^\N$ an \emph{$S$-adic sequence} (or a \emph{limit sequence}\index{limit sequence}\index{sequence!limit}) for $\bsigma$ if there exists a sequence $(w^{(n)})_{n\ge 0}$ of sequences $w^{(n)}\in\A^\N$ with
\begin{equation}\label{eq:limitworddesubs}
w^{(0)}=w, \quad w^{(n)}=\sigma_n(w^{(n+1)}) \quad (\hbox{for all } n\in\N).
\end{equation}
In this case we call $\bsigma$ the \emph{coding sequence} or the \emph{directive sequence} for $w$. (Note that \eqref{eq:limitworddesubs} says that $w$ can be ``desubstituted'' infinitely often).
\end{definition}

Let $S$ be a finite set of substitutions over a given alphabet $\A$. For this case $S$-adic sequences have been thoroughly studied in the literature. With Sturmian sequences and Arnoux-Rauzy sequences we already discussed two prominent  classes of $S$-adic sequences. Durand~\cite{Durand:00a,Durand:00b} proved that linearly recurrent\footnote{A sequence is called \emph{linearly recurrent} if there is a constant $K$ such that each of its factors $u$ occurs infinitely often in the sequence with gaps bounded by $K|u|$.} sequences are $S$-adic with finite $S$. Ferenczi~\cite{Ferenczi:96} and Leroy~\cite{Leroy:12} showed that a uniformly recurrent\footnote{A sequence is called \emph{uniformly recurrent} if each of its  factors occurs infinitely often in the sequence with bounded gaps.} sequence $w$ with an at most linear complexity function $p_w$ is $S$-adic with finite $S$; see also~\cite{Leroy:14}. The so-called \emph{$S$-adic conjecture} (see {\it e.g.} \cite[Section~12.1.2]{Fog02} or \cite{Durand-Leroy-Richomme:13,Leroy:12}) is also formulated for a finite set of substitutions $S$. It asks to what extent a converse of this assertion can be true, {\it i.e.}, which criteria are needed for an $S$-adic sequence $w$ to have linear complexity function $p_{w}$. Berth\'e and Labb\'e~\cite{Berthe-Labbe:15} show linearity of the complexity of $S$-adic sequences associated with the Arnoux-Rauzy-Poincar\'e multidimensional continued fraction algorithm (their bound $p_w(n)\le \frac52n + 1$ is even strong enough to conclude from Boshernitzan~\cite{Boshernitzan:84} that, like Arnoux-Rauzy sequences, these sequences pertain to uniquely ergodic dynamical systems). Arnoux, Mizutani, and Sellami~\cite{AMS:14} study $S$-adic sequences in the same context as we will do it. However, they restrict their attention to sets of substitutions $S$ whose elements have a common incidence matrix. If $S$ is a singleton, an $S$-adic sequence is called \emph{substitutive}. Substitutive sequences are very well studied (see for instance \cite{Fog02}; moreover in the paragraphs following Definition~\ref{def:SadicSystem} we review the literature on substitutive sequences related to our subject). They are strongly related to automatic sequences by \emph{Cobham's Theorem}, see {\it e.g.}~\cite[Theorem~6.3.2]{Allouche-Shallit:03}.

Generalizing Sturmian systems we introduce dynamical systems for $S$-adic sequences. To this end, for a finite alphabet $\A$ define the \emph{shift} on $\A^\N$ as $\Sigma:\A^\N\to\A^\N$ by $\Sigma(w_0w_1\ldots)=w_1w_2\ldots$

\begin{definition}[$S$-adic system]\label{def:SadicSystem}\index{$S$-adic system}
For an $S$-adic sequence $w$ over a finite alphabet $\A$ we denote by $X_w=\overline{\{\Sigma^kw\,:\, k\in\N\}}$ the orbit closure of $w$ under the action of the shift $\Sigma$. If we denote the restriction of $\Sigma$ to $X_w$ by $\Sigma$ again we call the pair $(X_w,\Sigma)$ the \emph{$S$-adic system} (or \emph{$S$-adic shift}) generated by $w$. 
\end{definition}

Alternatively, the set $X_w$ can be defined using languages by setting $X_w=\{v \in \A^\N \,:\, L(v)\subseteq L(w)\}$. The proof of the fact that both definitions of $X_w$ agree  is an easy exercise. Also the set $X_{\bsigma}=\bigcup X_{w}$, where the union is extended over all $S$-adic sequences with directive sequence $\bsigma$, and the associated dynamical system $(X_{\bsigma},\Sigma)$ are of interest.\footnote{If we impose the additional property of \emph{primitivity} on the coding sequence of a sequence $w\in\A^\N$ it turns out that $X_w$ depends only on the directive sequence $\bsigma$ defining the $S$-adic sequence $w$ and we have $X_{\bsigma}=X_{w}$. This will be worked out precisely in Section~\ref{sec:pr}.}  A recent survey on $S$-adic systems is provided in~\cite{Berthe-Delecroix}. 

In all what follows we will assume that all our substitutions and matrices are unimodular.

The case of $\bsigma=(\sigma)$, the constant sequence formed by a given unimodular substitution $\sigma$ over some alphabet $\mathcal{A}$, has been studied extensively. In this case we call $(X_{(\sigma)},\Sigma)$ a \emph{substitutive system} (see Queffelec~\cite{Queffelec:10} for a profound study of dynamical properties of these systems). The theory of Section~\ref{sec:sturm} can be generalized quite well to substitutive systems if $\sigma$ is a unimodular Pisot substitution. 
The seed for such a generalization was planted by Rauzy~\cite{Rauzy:82}. Constructing the prototype of what is now called \emph{Rauzy fractal}, he proved that the dynamical system $(X_\bsigma,\Sigma)$ is measurably conjugate to a rotation on $\mathbb{T}^2$ if $\bsigma=(\sigma)$ with $\sigma$ being the Tribonacci substitution introduced in \eqref{eq:substribo}. It was conjectured since then that each unimodular Pisot substitution $\sigma$ gives rise to a substitutive system $(X_{(\sigma)},\Sigma)$ which is measurably conjugate to a rotation on the torus. This conjecture is still open and known as \emph{Pisot (substitution) conjecture}.

In the meantime, the Pisot conjecture was studied by many people and interesting partial results have been achieved. We mention Arnoux and Ito~\cite{Arnoux-Ito:01} as well as Ito and Rao~\cite{Ito-Rao:06} who could prove the Pisot conjecture subject to some combinatorial \emph{coincidence conditions}. Conditions of this type will also play an important role in the general theory we will develop here, see Section~\ref{sec:coinc}. Recently, Barge~\cite{Barge:16b,Barge:16} made considerable progress on this subject using refinements of the notion of \emph{proximality} (see~\cite{Auslander:1988,Barge-Kellendonk:13}). For survey papers on the subject we refer {\it e.g.}\ to \cite{ABBLS,CANTBST}. For extensions of this theory to the nonunimodular case see~\cite{MinervinoThuswaldner14,Siegel:03}.

\subsection{Generalized continued fraction algorithms}\label{sec:genCF}

We now generalize the concept of continued fraction algorithm defined in Section~\ref{sec:CF} and introduce \emph{generalized continued fraction algorithms}. Standard references for these objects are Brentjes~\cite{BRENTJES} and Schweiger~\cite{Schweiger:00}. Also Labb\'e's \emph{Cheat Sheets}~\cite{Lab15} for $3$-dimensional continued fraction algorithms are highly recommended. For discussions of generalized continued fraction algorithms in a context similar to ours we refer {\it e.g.} to~\cite{Arnoux:16,AL:15,Arnoux-Nogueira,Berthe:11}.

\begin{definition}[Generalized continued fraction algorithm]\label{def:MCF}
\index{generalized continued fraction algorithm}\index{continued fraction algorithm!generalized}
For $d\ge 2$ let $X$ be a closed subset of the projective space $\mathbb{P}^{d-1}$ and let 
$\{X_i\}_{i\in I}$ be a partition of $X$ (up to a set of measure~$0$) indexed by a countable set $I$. Let $\mathcal{M}=\{M_i\,:\, i\in I\}$ be a set of unimodular $d\times d$ integer matrices (that act on $\mathbb{P}^{d-1}$ by homogeneity) satisfying $M_i^{-1}X_i \subset X$ and let $M:X\to \mathcal{M}$ given by $M(\bx)=M_i$ whenever $\bx\in X_i$. 
The \emph{generalized continued fraction algorithm} associated with this data is given by the mapping
\[
F:X\to X; \quad \bx \mapsto M(\bx)^{-1}\bx.
\]
If $I$ is a finite set, the algorithm given by $F$ is called \emph{additive}, otherwise it is called \emph{multiplicative}.
\end{definition}

Note that $F$ is defined only almost everywhere since $\{X_i\}_{i\in I}$ in general is only a partition up to measure zero. We confine ourselves to unimodular matrices. Thus the algorithms in Definition~\ref{def:MCF} are sometimes called \emph{unimodular algorithms}.
Interesting examples of nonunimodular continued fraction algorithms are provided by the \emph{$N$-continued fraction algorithm} introduced by Burger {\it et al.}~\cite{BGKWY:08} and by the \emph{Reverse algorithm}, a certain ``completion'' of the Arnoux-Rauzy algorithm studied in~\cite[Section~4]{AL:15}.

We illustrate the definition of generalized continued fraction algorithms by a classical example: \emph{Brun's continued fraction algorithm}.

\begin{example}[Brun's algorithm]\label{ex:brun}\index{Brun continued fraction algorithm}\index{continued fraction algorithm!Brun}
The linear version of Brun's algorithm is defined on the subset 
\[
X=\{[w_{1}:w_{2}:w_{3}]\;:\; 0\le w_{1}\le w_{2}\le  w_{3} \} \subset \mathbb{P}^2.
\]
It maps a vector $[w_{1}:w_{2}:w_{3}]$ to ${\rm sort}[w_{1}:w_{2}:w_{3}-w_{2}]$, {\it i.e.}, it subtracts the second largest entry from the largest one and sorts the resulting entries in ascending order. By a straightforward calculation we see that $\mathcal{M}=\{M_{1},M_{2},M_{3}\}$ with 
\begin{equation}\label{eq:brunmatrices}
M_{1} = \begin{pmatrix}0&1&0\\0&0&1\\1&0&1 \end{pmatrix},\quad
M_{2} = \begin{pmatrix}1&0&0\\0&0&1\\0&1&1 \end{pmatrix},\quad
M_{3} = \begin{pmatrix}1&0&0\\0&1&0\\0&1&1 \end{pmatrix},
\end{equation}
and that the partition $X=X_{1}\cup X_{2}\cup X_{3}$ is given by Figure~\ref{fig:brunsimplex}. 

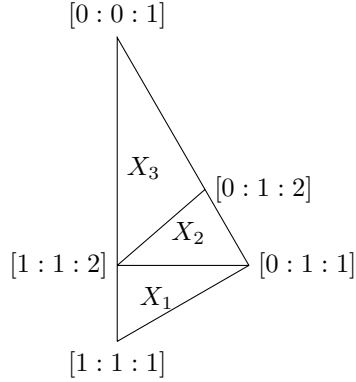
\begin{figure}
\begin{tikzpicture}[scale=3.5]
\coordinate [label={below :$[1:1:1]$}] (D) at (0,1.732/3);
\coordinate [label={above:$[0:0:1]$}] (E) at (0,1.732);
\coordinate [label={right:$[0:1:1]$}] (F) at (.5,1.732/2);
\coordinate [label={left:$[1:1:2]$}] (G) at (0,1.732/2);
\coordinate [label={right:$[0:1:2]$}] (H) at (1/3,2/3*1.732);

\draw (D) -- node[above] {} (E) -- node[right] {} (F) -- node[below] {} (D) (G)--(F) (G)--(H);
\draw (.15,1.732*3/7)node[]{$X_1$} (.1,1.732*5/7)node[]{$X_3$} 
(.27,1.732*4/7)node[]{$X_2$};
\end{tikzpicture}
\caption{The partition of $X$ induced by Brun's continued fraction algorithm.\label{fig:brunsimplex}}
\end{figure}
With this data the linear Brun continued fraction mapping can be defined according to Definition~\ref{def:MCF} by
\[
F_{B}:X\to X; \quad \bx \mapsto M_{i}^{-1}\bx \quad \hbox{ for }\bx \in X_{i}.
\]
Since $M_{1}$, $M_{2}$, and $M_{3}$ are unimodular, Brun's algorithm is a unimodular continued fraction algorithm.
As we did for the classical continued fraction algorithm in Section~\ref{sec:CF}, we can define a projective version also in the case of Brun's algorithm. This projective version is the original version of this algorithm and goes back to Brun~\cite{BRUN}. It is defined on the set 
\begin{equation}\label{eq:deltaset}
\Delta=\{(x_1,x_2)\in\R^2\,:\, 0\le x_1\le x_2 \le 1\}
\end{equation}
by 
\begin{equation}\label{eq:brunmap}
f_{\rm B}: (x_1,x_2) \mapsto 
\begin{cases}
\left(\frac{x_1}{1-x_2},\frac{x_2}{1-x_2}\right), & \hbox{for } x_2 \le \frac12, \\
\left(\frac{x_1}{x_2},\frac{1-x_2}{x_2}\right), & \hbox{for } \frac12 \le x_2 \le 1-x_1,\\
\left(\frac{1-x_2}{x_2},\frac{x_1}{x_2}\right), & \hbox{for }1-x_1 \le x_2.
\end{cases}
\end{equation}
To see that $f_{B}$ is the projective version of $F_{B}$ we use the same reasoning as in the classical case in Section~\ref{sec:CF}. 

We refer to Example~\ref{ex:brunprim} where we provide $S$-adic sequences associated with Brun's algorithm.
\end{example}

Other well-known generalized continued fraction algorithms include the Jacobi-Perron algorithm~\cite{Perron:07} and the Selmer algorithm~\cite{Selmer:61}.

\section{The importance of primitivity and recurrence}\label{sec:pr}

As indicated in Section~\ref{sec:problems3} it is not possible to generalize the results of Section~\ref{sec:sturm} to higher dimensions (or, equivalently, to alphabets of cardinality greater than two) without additional conditions on the sequence of substitutions $\bsigma$. 
In this section we will discuss two natural conditions that we will have to impose on our sequences of substitutions. The first one is \emph{primitivity}, the second one is \emph{recurrence}. Both of them will have important consequences for the underlying $S$-adic system: primitivity will imply minimality, and if we assume recurrence on top of primitivity, the system will be uniquely ergodic.

\subsection{Primitivity and minimality}
In the following definition a matrix is called \emph{nonnegative} if each of its entries is greater than or equal to zero. In a \emph{positive matrix} each entry is greater than zero.

\begin{definition}[Primitivity]\label{def:prim}\index{primitivity}
A sequence $\bM=(M_n)_{n\ge 0}$ of nonnegative integer matrices is \emph{primitive} if for each $m\in\N$ there is $n>m$ such that $M_{[m,n)}$ is a positive matrix. A sequence $\bsigma$ of substitutions is \emph{primitive} if its associated sequence of incidence matrices is primitive. 
\end{definition}

Note that primitivity of $(M_n)_{n\ge 0}$ implies primitivity of the ``shifted'' sequence $(M_{n+k})_{n\ge 0}$ for each $k\in\N$. The same applies for primitive sequences of substitutions.

Our definition of primitivity is taken from \cite[Section~2.2]{Berthe-Steiner-Thuswaldner}. It coincides with the notion of weak primitivity introduced in \cite[Definition~5.1]{Berthe-Delecroix} and with the notion of nonstationary primitivity defined in \cite[p.~339]{Fisher:09}. 
The more restrictive property of strong primitivity which is also introduced in \cite[Definition~5.1]{Berthe-Delecroix} requires that the integer $n$ in Definition~\ref{def:prim} can be chosen in a way that the difference $n-m$ is uniformly bounded in $m$. In other papers, this stronger property is called primitivity (see {\it e.g.}~\cite{Durand:00a,Durand:00b,Durand-Leroy-Richomme:13}). 

As we will see in the first result of this section, the assumption of primitivity entails \emph{minimality} of the associated $S$-adic systems. We recall the definition of this basic concept.

\begin{definition}[Minimality]\label{def:min}\index{minimality}
Let $(X,T)$ be a topological dynamical system. $(X,T)$ is called \emph{minimal} if the orbit of each point is dense in $X$, {\em i.e.,} if $\overline{\{T^nx\,:\;n\in\N\}}=X$ holds for each $x\in X$.
\end{definition}

The following lemma summarizes the consequences of primitivity for an $S$-adic system. It is proved for instance in~\cite[Proposition~2.1 and~2.2]{AMS:14}; the minimality assertion can already be found in~\cite[Lemma~7]{Durand:00a}.

\begin{proposition}\label{prop:sadicminimal}
If $\bsigma$ is a primitive sequence of substitutions, the following properties hold.
\begin{enumerate}
\item[(i)] There exists at least one and at most $|\A|$  limit sequences for $\bsigma$.
\item[(ii)] Let $w,w'$ be two $S$-adic sequences with directive sequence $\bsigma$. Then $(X_w,\Sigma)=(X_{w'},\Sigma)$. 
\item[(iii)] For a limit sequence $w$ of $\bsigma$ the $S$-adic system $(X_w,\Sigma)$ is minimal.
 \end{enumerate}
\end{proposition}

\begin{proof}
To show (i) let $\bsigma=(\sigma_n)$ and for each $n\in\N$ let $\A_n$ be the set of all first letters occurring in the family $\sigma_{[0,n)}(\A)$ of words. Then $(\A_n)$ is a decreasing sequence of nonempty subsets of $\A$. Hence, there is $a\in\bigcap_{n\ge0}\A_n$. By construction there is a sequence $(a_n)$ with $a_0=a$ such that  $a_{n}$ is the first letter of $\sigma_{n}(a_{n+1})$. Moreover, 
$\sigma_{[0,n)}(a_n)$ is a prefix of $\sigma_{[0,n+1)}(a_{n+1})$. By primitivity, the lengths of these words tend to infinity which implies that
$
w=\lim_{n\to\infty}\sigma_{[0,n)}(a_na_n\ldots)
$
converges.\footnote{Only the first letter $a_{n}$ in the argument of $\sigma_{[0,n)}$ is relevant for the limit. However, since we use the topology on $\A^\N$ and $\sigma_{[0,n)}(a_{n})\not\in\A^\N$ we have to write $\sigma_{[0,n)}(a_na_n\ldots)$.} By the same reasoning (here we use that primitivity also holds for ``shifted'' sequences), we see that $w^{(m)}=\lim_{n\to\infty}\sigma_{[m,n)}(a_na_n\ldots)$ converges as well and the sequence $(w^{(m)})$ satisfies the conditions of Definition~\ref{def:sadicsequence}. Thus $w=w^{(0)}$ is an $S$-adic sequence with directive sequence $\bsigma$.

If $w$ is an $S$-adic sequence with directive sequence $\bsigma$, by Definition~\ref{def:sadicsequence} we can associate a sequence $(w^{(n)})$ with it. For $n\in\N$ let $a_{n}$ be the first letter of $w^{(n)}$. Primitivity implies that $|\sigma_{[0,n)}(a_{n})|\to\infty$ for $n\to\infty$ and, hence, the sequence $w$ is determined by the sequence $(a_{n})$. In particular, we can write  $w= \lim_{n\to\infty}\sigma_{[0,n)}(a_na_n\ldots)$. Since $a_n$ uniquely determines $a_p$ for each $p<n$, there are at most $|\A|$ possible different choices for such a sequence. 

To prove (ii) let $w$ and $w'$ be two $S$-adic sequences with directive sequence $\bsigma$. Associate the sequences $(a_n)$ and $(a_n')$, respectively, with them as above. If $u$ is a factor of $w$ then $u$ is a factor of $\sigma_{[0,m)}(a_m)$ for some $m$. By primitivity, there exists $n > m$ such that $a_m$ occurs in $\sigma_{[m,n)}(a_n')$. Thus $\sigma_{[0,m)}(a_m)$ and {\it a fortiori} also $u$ is a factor of $w'$ and, hence, $L(w)\subseteq L(w')$. Exchanging the roles of $w$ and $w'$ we can therefore conclude that $L(w)=L(w')$ which implies that $X_w=X_{w'}$.

It remains to prove (iii). This follows if we can show that $L(v)=L(w)$ for each $v\in X_w$.
This in turn holds if each factor of $w$ occurs infinitely often in $w$ with bounded gaps, which we will now prove. Let $u$ be a factor of $w$ and let $(a_n)$ be the sequence of letters associated to $w$ as above. Then $u$ is a factor of $\sigma_{[0,m)}(a_m)$ for some $m$. By primitivity, there exists $n > m$ such that $u$ is a factor of $\sigma_{[0,n)}(a)$ for each $a\in\A$. Since $w$ is an $S$-adic sequence, $w=\sigma_{[0,n)}(w^{(n)})$ holds for some $w^{(n)}\in \A^\N$. Thus $u$ occurs in $w$ infinitely often with gaps bounded by $2\max\{|\sigma_{[0,n)}(a)| \,:\, a\in\A\}$.  
\end{proof}

If $\bsigma$ is a primitive sequence of substitutions, assertion (ii) of this proposition implies that $X_{\bsigma}=X_{w}$ and, hence, $(X_\bsigma,\Sigma)=(X_w,\Sigma)$ for $w$ being an arbitrary $S$-adic sequence with directive sequence $\bsigma$.  Since we will assume primitivity throughout the remaining part of the paper we will always work with $X_\bsigma$.

\subsection{Recurrence, weak convergence, and unique ergodicity}
The next concept we introduce is \emph{recurrence}. Let $S$ be a finite set of substitutions. If we take a random sequence of substitutions $\bsigma\in S^\N$ whose elements are taken from a finite set $S$ we will almost always (w.r.t.\ any natural measure on the space $S^\N$) get a sequence $\bsigma$ each of whose patterns occurs infinitely often. This infinite repetition of patterns is made precise in the following definition.

\begin{definition}[Recurrence]\label{def:recurrence}\index{recurrence}
A sequence $\bM=(M_{n})$ of integer matrices is called \emph{recurrent} if for each $m\in \N$ there is $n\ge1$ such that $(M_0,\ldots,M_{m-1}) = (M_n,\ldots,M_{n+m-1})$.  
A sequence $\bsigma=(\sigma_{n})$ of substitutions is called \emph{recurrent} if for each $m\in \N$ there is $n\ge1$ such that $(\sigma_0,\ldots,\sigma_{m-1}) = (\sigma_n,\ldots,\sigma_{n+m-1})$.  
\end{definition}

Note that recurrence of a sequence of substitutions $\bsigma$ implies that each block of substitutions that occurs once in $\bsigma$ must occur infinitely often (the same is true for sequences of matrices). Thus recurrence of $(\sigma_n)_{n\in\N}$ implies recurrence of $(\sigma_{m+n})_{n\in\N}$ for each $m\in \N$ and an analogous statement holds for sequences of matrices. We also emphasize that a nonrecurrent sequence of substitutions may well have a recurrent sequence of incidence matrices. This is due to the fact that two different substitutions can have the same incidence matrix. 

We now study consequences of primitivity and recurrence. We start with the following result which follows from contraction properties of the \emph{Hilbert metric}, a metric on projective space that goes back to Birkhoff~\cite{Bir:57} and Furstenberg~\cite[pp.~91--95]{Furstenberg:60} (we mention \cite[Appendix~A]{Fisher:09} and \cite[Chapter~26]{Viana:06} as more recent references). A special case of this result is stated in Section~\ref{sec:sturm}, see \eqref{eq:concvone2lett}.

\begin{proposition}\label{prop:furstmatrix}
Let $\bM=(M_n)$ be a primitive and recurrent sequence of nonnegative integer matrices. Then there is a vector $\bu\in\mathbb{R}^d_{> 0}$ satisfying
\begin{equation}\label{eq:coneshrink}
\bigcap_{n\ge 0}M_{[0,n)}\mathbb{R}^d_{\ge 0} = \R_+\bu.
\end{equation}
\end{proposition}

\begin{proof}
To prove this result we define a metric on the space $\mathcal{W}=\{\R_{+}\bw \,:\, \bw\in\R^d_{\ge 0}\setminus\{\mathbf{0}\}\}$ of nonnegative rays through the origin by (see \cite[Appendix~A]{Fisher:09}) 
\[
d_\mathcal{W}(\R_+\bv,\R_+\bw)=\max_{1\le i,j\le d} \log\frac{v_iw_j}{v_jw_i},
\]
where $\bv=(v_1,\ldots,v_d)$ and $\bw=(w_1,\ldots,w_d)$. It can be checked by direct calculation that this is a metric on $\mathcal{W}$ which is the so-called \emph{Hilbert Metric} ({\it cf. e.g.}~\cite[Lemma~A.5]{Fisher:09} or \cite[Chapter~26]{Viana:06}).
Let $\mathrm{diam}_\mathcal{W}(A)$ be the diameter of a set $A\subset \mathcal{W}$ w.r.t.\ this metric. Then $\mathrm{diam}_\mathcal{W}(\mathcal{W})=\infty$ and $\mathrm{diam}_\mathcal{W}(M\mathcal{W})<\infty$ for every positive matrix $M$. It follows from the definitions that a nonnegative matrix $M$ is nonexpanding in the sense that $d_\mathcal{W}(M\R_+\bv,M\R_+\bw)\le d_\mathcal{W}(\R_+\bv,\R_+\bw)$ for all $\R_+\bv,\R_+\bw\in \mathcal{W}$. Moreover, one can show that each positive matrix $M$ is a contraction, {\it i.e.}, there is $\kappa<1$ (depending on $M$) such that $d_\mathcal{W}(M\R_+\bv,M\R_+\bw)\le \kappa \, d_\mathcal{W}(\R_+\bv,\R_+\bw)$ for all $\R_+\bv,\R_+\bw\in \mathcal{W}$ (see for instance~\cite{Bir:57} or \cite[Proposition~26.3]{Viana:06} for a proof of this).

We now apply these contraction properties to our setting. Since $\bM$ is primitive and recurrent, there exists a positive matrix $B$ and an integer $h>0$ such that $B=M_{[m_i,m_i+h)}$ for a sequence of positive integers $(m_i)_{i\ge 0}$ satisfying $m_i + h \le m_{i+1}$. By the preceding paragraph we get that $\mathrm{diam}_\mathcal{W}(B\mathcal{W})= \gamma$  for some $\gamma > 0$ and that $B$ is a contraction with some contraction factor $\kappa<1$. Thus for each $m\in \{m_i + h, m_{i+1}+h-1\}$ we have
\[
\mathrm{diam}_\mathcal{W} \bigg( \bigcap_{0\le n\le m}M_{[0,n)}\R^d_{\ge 0} \bigg) \le \gamma \kappa^i.
\]
Since $\kappa<1$ and $i\to\infty$ for $m\to\infty$ this yields the result. Positivity of the entries of $\bu$ follows from the primitivity of $\bM$.
\end{proof}

This result motivates the following definition.

\begin{definition}[Weak convergence and generalized right eigenvector]\label{def:weakconv}
If a sequence of nonnegative integer matrices satisfies \eqref{eq:coneshrink} for some $\bu\in\mathbb{R}^d_{\ge 0}\setminus\{\mathbf{0}\}$ we say that $\bM$ is \emph{weakly convergent} to $\bu$. In this case we call $\bu$ a \emph{generalized right eigenvector} of $\bM$. If a sequence $\bsigma$ of substitutions has a sequence of incidence matrices $\bM$ which is weakly convergent to $\bu$, we say that $\bsigma$ is \emph{weakly convergent}\index{weak convergence}\index{convergence!weak} to $\bu$ and call $\bu$ a \emph{generalized right eigenvector}\index{generalized right eigenvector} of $\bsigma$.
\end{definition}
 
Our next goal is to establish \emph{unique ergodicity} of $S$-adic systems with primitive and recurrent directive sequences. We start with a fundamental definition (and refer to \cite[\S6.5]{Walters:82} for background material on this).

\begin{definition}[Unique ergodicity]\label{def:ue}\index{unique ergodicity}
A topological dynamical system $(X,T)$ on a compact space $X$ is said to be \emph{uniquely ergodic} if there is a unique $T$-invariant Borel probability measure on $X$. 
\end{definition}

By a theorem of Krylov and Bogoliubov (see {\em e.g.}\ \cite[Corollary~6.9.1]{Walters:82}) there always exists an invariant probability measure on $(X,T)$ if $X$ is compact. 

A uniquely ergodic dynamical system is ergodic (thus the name) since otherwise there would be a $T$-invariant set $E$ with $\mu(E)\in(0,1)$ which could be used to define a second $T$-invariant Borel probability measure $\nu(B)=\frac{\mu(B\cap E)}{\mu(E)}$ on $X$. Unique ergodicity is equivalent to the fact that each point is generic in the sense that Birkhoff's ergodic theorem holds everywhere ({\it cf.} \cite[Theorem~6.19]{Walters:82}). Roughly speaking, this is true since nongeneric points (as for instance periodic points) could be used to construct a second invariant measure.

We note that unique ergodicity is close to minimality in the sense that there are many dynamical systems that either enjoy both or none of the two properties. If $(X,T)$ is uniquely ergodic with $T$-invariant measure $\mu$ having full support then minimality follows. However, there are examples of systems that have only one of these two properties. For a discussion of such examples in a context similar to ours see \cite{Ferenzi-Fisher-Talet:09} and the references given there. What happens for these examples is that although we have a primitive sequence of matrices (leading to minimality) this primitivity is so weak that it does not make the positive  cone converge to a single line as in \eqref{eq:coneshrink}. This entails that no letter frequencies exist which permits to construct many invariant measures (see also \cite{BKMS:2010,BKMS12,Fisher:09}).

It has been mentioned already in Section~\ref{sec:sturm} that the existence of uniform frequencies of letters and words in a shift $(X_w, \Sigma)$ entail unique ergodicity. We want to give the elegant proof of this result here before we use it in order to establish unique ergodicity of primitive and recurrent $S$-adic systems. To this matter we need the following definition (see Lemma~\ref{lem:sturmpatternfreq} for the special case of Sturmian sequences). 

\begin{definition}[Uniform word and letter frequencies]\label{def:wordfreq}
\index{uniform word frequencies}
\index{uniform letter frequencies}
Let $w=w_0w_1\ldots \in\A^\N$ be given and for each $k,\ell\in\N$ and each $v\in\A^*$ let $|w_k\ldots w_{k+\ell-1}|_v$ be the number of occurrences of $v$ in $w_k\ldots w_{k+\ell-1}$. We say that $w$ has \emph{uniform word frequencies} if for each $v\in\A^*$ the ratio $|w_k\ldots w_{k+\ell-1}|_v/\ell$ tends to a limit $f_v(w)$ (which does not depend on $k$) for $\ell\to\infty$ uniformly in~$k$. It has \emph{uniform letter frequencies} if this is true for each $v\in\A$.
\end{definition}

\begin{proposition}[{see~\cite[Proposition~5.1.21]{Fog02}}]\label{prop:freq_unique_ergod}
Let $w \in\A^\N$ be a sequence with uniform word frequencies and let $X_w=\overline{\{\Sigma^kw\,:\, k\in\N\}}$ be the shift orbit closure of $w$. Then $(X_w,\Sigma)$ is uniquely ergodic.
\end{proposition}  

\begin{proof}
For every factor $v$ of $w=w_0w_1\ldots$ let $[v]$ be the \emph{cylinder} of all sequences in $X_w$ that have $v$ as a prefix. Define a function $\mu$ on these cylinders by $\mu([v])=\mu(\Sigma^{-n}[v])=f_v(w)$. Since cylinders generate the topology on $X_w$ this defines a Borel measure $\mu$ on $X_w$. Our goal is to show that every element of $X_w$ is generic in the sense of Birkhoff's ergodic theorem. To this end note first that (here $\mathbbm{1}_{Y}$ denotes the characteristic function of a set $Y\subset X_w$)
\[
\frac1N \sum_{n<N} \mathbbm{1}_{[v]}(\Sigma^{n+j}w) \to \mu([v]) = \int\mathbbm{1}_{[v]}d\mu
\]
holds uniformly in $j \in \N$ for every $v\in\A^*$ by the existence of uniform word frequencies for $w$. Since continuous functions are monotone limits of simple functions this extends to 
\begin{equation}\label{eq:ferencUnique}
\frac1N \sum_{n<N} g(\Sigma^{n+j}w) \to  \int g d\mu 
\end{equation}
uniformly in $j\in\N$ for each $g\in C(X_w)$.  By this uniform convergence, in \eqref{eq:ferencUnique} we may choose $j=n_k$ with any sequence $(n_k)$ and \eqref{eq:ferencUnique} holds uniformly in $k$. Since each $u\in X_w$ is the limit of $(\Sigma^{n_k}w)$ for some sequence $(n_k)$ this implies that
\[
\frac1N \sum_{n<N} g(\Sigma^{n}u) \to  \int g d\mu 
\]
holds for each $g\in C(X_w)$ and each $u\in X_w$. Thus each point is generic in the sense of Birkhoff's ergodic theorem which is equivalent to unique ergodicity (by~\cite[Theorem~6.19]{Walters:82} which was already mentioned above).
\end{proof}

We now show that the conditions we introduced so far imply unique ergodicity of $S$-adic systems. In view of Proposition~\ref{prop:freq_unique_ergod} we will establish the following lemma (see also \cite[Theorem~5.7]{Berthe-Delecroix}).

\begin{lemma}\label{lem:SadicFreq}
Let $\bsigma$ be a sequence of substitutions with associated sequence of incidence matrices $\bM$. If $\bM$ is primitive and recurrent then each sequence $w\in X_\bsigma$ has uniform word frequencies.
\end{lemma}

\begin{proof}
Let $w=w_0w_1\ldots \in X_\bsigma$ be given. We follow the proof of \cite[Theorem~5.7]{Berthe-Delecroix} to establish that $w$ has uniform word frequencies.

{\it Part 1: Uniform letter frequencies.}
Since $\bM$ satisfies the conditions of Proposition~\ref{prop:furstmatrix}, it admits a generalized right eigenvector $\bu$. Let $\bu/{\Vert\bu\Vert_1}=(u_1,u_2,\ldots, u_d)^t$.  Since $w\in X_\bsigma$, for all $k,\ell,n\in \N$ we can write
\begin{equation*}
w_k\dots w_{k+\ell-1} = p \sigma_{[0,n)}(v) s
\end{equation*}
for some $p,v,s\in\A^*$, where the lengths of $p,s$ are bounded by $\max\{|\sigma_{[0,n)}(a)|\,:\,a\in \A\}$. Now for each $a\in\A$
\begin{equation}\label{eq:sadicDTeinfach}
\left| \frac{|w_k\dots w_{k+\ell-1}|_a}{\ell} -  u_a\right| \le \frac{\big| |p|_a - |p|u_a\big|}{\ell} + \frac{\big| |\sigma_{[0,n)}(v)|_a-|\sigma_{[0,n)}(v)|u_a\big|}{\ell}  + \frac{\big| |s|_a- |s|u_a\big|}{\ell}.
\end{equation}
By the convergence of the positive cone to $\bu$ in Proposition~\ref{prop:furstmatrix} we know that $|\sigma_{[0,n)}(b)|_a/|\sigma_{[0,n)}(b)|$ is close to $u_a$ for all $a,b\in\A$ if $n$ is large. Thus for each $\varepsilon > 0$ there is $N\in\N$ such that whenever $\ell\ge N$ we can choose $n$ in a way that $|p|,|s| \le \varepsilon \ell$ and $\big| |\sigma_{[0,n)}(b)|_a-|\sigma_{[0,n)}(b)|u_a \big| < \varepsilon |\sigma_{[0,n)}(b)|$ for all letters $a$ and $b$. This proves that the right hand side of \eqref{eq:sadicDTeinfach} is bounded by $3\varepsilon$ and, hence, $\lim_{\ell \to \infty}{|w_k\ldots w_{k+\ell-1}|_a}/{\ell}=u_a$ uniformly in $k$. Thus $w$ has uniform letter frequencies.

{\it Part 2: Uniform word frequencies.} For $m\in\N$ let $\bu^{(m)}$ be a right eigenvector of the shifted sequence $\bsigma^{(m)}=(\sigma_{m+n})_{n\in\N}$ and set $\bu^{(m)}/\Vert\bu^{(m)}\Vert_1=(u_1^{(m)},\ldots,u_d^{(m)})$. Such an eigenvector exists by Proposition~\ref{prop:furstmatrix} since the shifted sequence $\bsigma^{(m)}$ has a primitive and recurrent sequence of incidence matrices as well.

Fix $v\in \mathcal{L}_\bsigma$. We claim that for each $m\in\N$ and each $w^{(m)}=w_0^{(m)}w_1^{(m)}\ldots \in X_{\bsigma^{(m)}}$  we have
\begin{equation}\label{eq:freqclaim}
\lim_{j\to \infty} \frac{\sum_{i=q}^{q+j-1}|\sigma_{[0,m)}(w^{(m)}_{i})|_v}{|\sigma_{[0,m)}(w^{(m)}_q\ldots w^{(m)}_{q+j-1})|} = \frac{\sum_{a\in\A}u_a^{(m)}|\sigma_{[0,m)}(a)|_v}{\sum_{a\in\A}u_a^{(m)}|\sigma_{[0,m)}(a)|} =: g(v,m)
\end{equation}
uniformly in $q\in\N$.
This claim follows because, since $w^{(m)}$ has uniform letter frequencies $(u_1^{(m)},\ldots,u_d^{(m)})$ by Part 1, we get that 
\begin{align*}
\lim_{j\to \infty}\frac{|\sigma_{[0,m)}(w_q^{(m)}\ldots w_{q+j-1}^{(m)})|}{j}&=\sum_{a\in\A}u_a^{(m)}|\sigma_{[0,m)}(a)| \quad\hbox{and}  \\
\lim_{j\to \infty}\frac{\sum_{i=q}^{q+j-1}|\sigma_{[0,m)}(w^{(m)}_{i})|_v}{j}&=\sum_{a\in\A}u_a^{(m)}|\sigma_{[0,m)}(a)|_v
\end{align*}
uniformly in $q\in\N$.

Now we proceed similarly to Part 1. First define 
\[
m_{n}^+=\max\{|\sigma_{[0,n)}(a)|\,:\,a\in \A\}
\quad\hbox{and}\quad
m_{n}^- = \min\{|\sigma_{[0,n)}(a)|\,:\,a\in \A\},
\]
and observe that primitivity of $\bsigma$ implies that both of these quantities tend to $\infty$ for $n\to \infty$. For each $n\in\N$ choose a fixed $w^{(n)}=w^{(n)}_0w^{(n)}_1\ldots \in X_{\bsigma^{(n)}}$. As $w\in X_\bsigma$, for all $k,\ell\in \N$ we can write
\begin{equation*}
w_k\dots w_{k+\ell-1} = p \sigma_{[0,n)}(w^{(n)}_q\ldots w^{(n)}_{q+r-1}) s
\end{equation*}
for some $q,r\in\N$, where the lengths of $p,s\in\A^*$ are bounded by $m_{n}^+$. There are three possibilities for an occurrence of $v$ in $w_k\dots w_{k+\ell-1}$. Firstly, $v$ can overlap with $p$ or $s$. This can happen at most $2m_{n}^+$ times. Secondly, $v$ can have nonempty overlap with the images $\sigma_{[0,n)}(w_i^{(n)})$ and $\sigma_{[0,n)}(w_{i+1}^{(n)})$ of two consecutive letters $w^{(n)}_{i}$ and $w^{(n)}_{i+1}$ of $w^{(n)}_q\ldots w^{(n)}_{q+r-1}$. This can happen at most $|v|(r-1)\le |v|\frac{\ell}{m_{n}^-}$ times. Thirdly, $v$ can occur as a factor of  $\sigma_{[0,n)}(w^{(n)}_{i})$ for some $i\in\{q,\ldots,q+r-1\}$ which happens exactly $\sum_{i=q}^{q+r-1}|\sigma_{[0,n)}(w^{(n)}_{i})|_v$ times. 
Each of these three possibilities contributes one of the summands of the right hand side of the estimate
\begin{equation}\label{eq:3trmsabc}
\Big| \frac{|w_k\dots w_{k+\ell-1}|_v}{\ell} - g(v,n)\Big| 
\le 
\frac{2m_{n}^+}{\ell} + \frac{|v|}{m_{n}^-} + 
\bigg|
\frac{ \sum_{i=q}^{q+r-1}|\sigma_{[0,n)}(w^{(n)}_{i})|_v } {\ell} -g(v,n)
\bigg|.
\end{equation}
Letting $\ell\to\infty$ and using \eqref{eq:freqclaim} for the third term on the right this yields that
\begin{equation}\label{eq:letfreqconclusion}
\limsup_{\ell\to\infty} \Big| \frac{|w_k\dots w_{k+\ell-1}|_v}{\ell} - g(v,n)\Big| \le \frac{|v|}{m_{n}^-}.
\end{equation}
Since for $n\to\infty$ the quantity $\frac{|w_k\dots w_{k+\ell-1}|_v}{\ell}$ does not change while $\frac{|v|}{m_{n}^-} \to 0$ we conclude from \eqref{eq:letfreqconclusion} that $(g(v,n))_{n\in \N}$ is a Cauchy sequence 
converging to the frequency $f_v(w)$ of $v$ in $w$. Since $\frac{|v|}{m_{n}^-}$ does not depend on $k$ and the convergence in \eqref{eq:freqclaim} is uniform in $q$, the estimate \eqref{eq:3trmsabc} implies that $\frac{|w_k\dots w_{k+\ell-1}|_v}{\ell}\to f_v(w)$ for $\ell\to\infty$ uniformly in $k$ and the proof is finished.
\end{proof}

The following main result of this section is an immediate consequence of Proposition~\ref{prop:sadicminimal}, Proposition~\ref{prop:freq_unique_ergod}, and Lemma~\ref{lem:SadicFreq}.

\begin{theorem}\label{prop:SadicUE}
Let $\bsigma$ be a sequence of substitutions with associated sequence of incidence matrices $\bM$. If $\bM$ is primitive and recurrent then $(X_\bsigma,\Sigma)$ is minimal and uniquely ergodic.
\end{theorem}

A proof of a similar result as Theorem~\ref{prop:SadicUE} is sketched in Berth\'e and Delecroix~\cite{Berthe-Delecroix}. Moreover, we refer to Fisher~\cite{Fisher:09} and Bezuglyi~{\it et al.}~\cite{BKMS:2010,BKMS12}, where theorems of this flavor are proved in the context of Bratteli-Vershik systems. 

\begin{example}\label{ex:brunprim}
We associate substitutions with the matrices $M_{1}$, $M_{2}$, and $M_{3}$ that came up in  \eqref{eq:brunmatrices} during the definition of Brun's continued fraction algorithm. Indeed, the substitutions
\begin{equation}\label{eq:brun}
\sigma_1 : \begin{cases} 1 \mapsto 3,  \\ 2 \mapsto 1, \\ 3 \mapsto 23, \end{cases} \quad
\sigma_2 : \begin{cases} 1 \mapsto 1,  \\ 2 \mapsto 3, \\ 3 \mapsto 23, \end{cases} \quad
\sigma_3 : \begin{cases} 1 \mapsto 1,  \\ 2 \mapsto 23, \\ 3 \mapsto 3. \end{cases} 
\end{equation}
are called \emph{Brun substitutions}\index{Brun substitutions} (see \cite[Sections~3.3 and~9.2]{Berthe-Steiner-Thuswaldner} where also the relation between these substitutions and a slightly different set of ``Brun substitutions'' studied in \cite{BBJS16} is discussed).

It is immediate that $M_{1}M_{2}M_{1}M_{2}$ is a strictly positive matrix. Thus we get the following result.

\begin{proposition}\label{prop:brunMUE}
Let $S=\{\sigma_{1},\sigma_{2},\sigma_{3}\}$ be the set of Brun substitutions and $\bsigma\in S^\N$. If $\bsigma$ is recurrent and contains the block $(\sigma_{1},\sigma_{2},\sigma_{1},\sigma_{2})$ then the associated $S$-adic system $(X_\bsigma,\Sigma)$ is minimal and uniquely ergodic.
\end{proposition}

\begin{proof}
Since $\bsigma$  is recurrent it contains the block $(\sigma_{1},\sigma_{2},\sigma_{1},\sigma_{2})$ infinitely often. Thus $\bsigma$ is primitive and the result follows from Theorem~\ref{prop:SadicUE}.
\end{proof}
\end{example}

\section{The importance of balance and algebraic irreducibility}

Let $\bsigma$ be a sequence of unimodular substitutions over an alphabet $\A=\{1,2,\ldots,d\}$ and let $(X_\bsigma,\Sigma)$ be the $S$-adic system defined by it. At the end of Section~\ref{sec:sturmrot} we gave some rough idea on how we want to prove that $(X_\bsigma,\Sigma)$ is measurably conjugate to a rotation on $\mathbb{T}^{d-1}$. Indeed, we wish to project the broken line (see \eqref{eq:brokensturmianline} for an example) associated with a limit sequence $w\in X_\bsigma$ to a hyperplane in $\R^d$ not containing the frequency vector $\bu$ of the sequences in $X_\bsigma$. On a natural subdivision $\mathcal{R}(1),\ldots,\mathcal{R}(d)$ of the closure $\mathcal{R}$ of this projection we want to define a domain exchange and a rotation. This is possible only if the sets $\mathcal{R}(i)$, $i\in\A$, have suitable topological properties and the mentioned subdivision has no essential overlaps. 

In the present section we will define these sets $\mathcal{R}$ and $\mathcal{R}(i)$, $i\in\A$, and discuss basic properties of them. Besides primitivity and recurrence, the crucial conditions we will have to impose on $\bsigma$ in order to get suitable properties of the sets $\Ra$ and $\Ra(i)$ will be \emph{algebraic irreducibility} of the sequence of incidence matrices of $\bsigma$ and \emph{balance} of the language $\Lg_\bsigma$. Both of these conditions will be defined and first consequences of them will be discussed. This paves the way to obtain deeper topological and measure theoretic properties of $\Ra$ and $\Ra(i)$ in Section~\ref{sec:rauzyprop}. The theory we will outline in the present as well as in the forthcoming sections is mainly due to Berth\'e, Steiner, and Thuswaldner~\cite{Berthe-Steiner-Thuswaldner} and we refer to this paper for rigorous proofs of the statements we give.

\subsection{$S$-adic Rauzy fractals}\label{sec:Srauzy}

Following Berth\'e, Steiner, and Thuswaldner~\cite[Section~2.9]{Berthe-Steiner-Thuswaldner} we will now define $S$-adic Rauzy fractals. As mentioned before, on these objects we will be able to ``see'' the rotations to which we want to (measurably) conjugate our $S$-adic systems. In the definition we will use the following notations. For a vector $\bw\in\R^d\setminus\{ {\bf 0} \}$ we write $\bw^\bot$ for the hyperplane orthogonal to $\bw$, {\it i.e.}, $\bw^{\bot}=\{\bx\in\R^d\,:\, \langle\bx,\bw\rangle=0 \}$ with $\langle \cdot,\cdot \rangle$ being the dot product on $\mathbb{R}^d$, and we equip the space $\bw^\bot$ with the $(d-1)$-dimensional Lebesgue measure $\lambda_{\bw}$. Since its orthogonal hyperplane will be of special interest later we introduce the vector $\bone=(1,\ldots, 1)^t$. 

For vectors $\bu,\bw \in \mathbb{R}^d\setminus\{{\bf 0}\}$ satisfying $\bu\not\in\bw^\bot$ we denote the projection along $\bu$ to $\bw^\bot$ by $\pi_{\bu,\bw}$.

\begin{definition}[$\boldsymbol{S}$-adic Rauzy fractal and subtiles]\index{$S$-adic Rauzy fractal}\index{Rauzy fractal!$S$-adic}
Let $\bsigma$ be a sequence of unimodular substitutions over the alphabet $\mathcal{A}$ and assume that $\bsigma$ is weakly convergent to a generalized right eigenvector $\bu\in\R^d_{>0}$. The \emph{$S$-adic Rauzy fractal} in the representation space $\bw^\bot$, $\bw\in\R_{\ge0}^{d}\setminus\{\mathbf{0}\}$, associated with $\bsigma$ is the set
\[
\mathcal{R}_\bw := \overline{ \{\pi_{\bu,\bw}\mathbf{l}(p) \;:\; p \hbox{ is a prefix of a limit sequence of } \bsigma \} }.
\]
The set $\mathcal{R}_\bw$ can be covered by the \emph{subtiles}
\begin{equation}\label{eq:subtileRi}
\mathcal{R}_\bw(i) := \overline{ \{\pi_{\bu,\bw}\mathbf{l}(p) \;:\; pi \hbox{ is a prefix of a limit sequence of } \bsigma \} } \qquad(i\in\A).
\end{equation}
For convenience we will use the notation $\Ra(i)=\Ra_\bone(i)$ and $\Ra=\Ra_\bone$.
\end{definition}

The prototype of a Rauzy fractal goes back to Rauzy~\cite{Rauzy:82} and was used there in order to show that a certain substitutive dynamical system is measurably conjugate to a rotation on the torus, see Example~\ref{ex:rauzy} below. In the meantime there exists a vast literature on Rauzy fractals. For constant sequences $\bsigma=(\sigma)$ with $\sigma$ being a Pisot substitution fundamental properties of Rauzy fractals were studied for instance by Ito and Kimura~\cite{ItoKimura91}, Holton and Zamboni~\cite{HZ:98}, Arnoux and Ito~\cite{Arnoux-Ito:01}, Canterini and Siegel~\cite{CanteriniSiegel01b}, Sirvent and Wang~\cite{SirventWang02}, Hubert and Messaoudi~\cite{HubertMessaoudi06}, and Ito and Rao~\cite{Ito-Rao:06}. Akiyama~\cite{Akiyama98,Akiyama02} and Messaoudi~\cite{Messaoudi00,Messaoudi06} consider versions of Rauzy fractals for $\beta$-numeration, in Siegel~\cite{Siegel:03}, Minervino and Thuswaldner~\cite{MinervinoThuswaldner14}, and Minervino and Steiner~\cite{Minervino-Steiner:14} Rauzy fractals with $\mathfrak{p}$-adic factors are related to nonunimodular substitutions. For versions of Rauzy fractals corresponding to substitutions with reducible incidence matrices (whose most prominent representative is the so-called ``Hokkaido Fractal'' studied by Akiyama and Sadahiro~\cite{AS:98}) we refer to \cite{Akiyama02,EiItoRao06,LM:17}. A case of a non-Pisot substitution is treated in~\cite{AFHI:11}. Surveys containing information on Rauzy fractals are provided in \cite{CANTBST,SiegelThuswaldner10} (see also \cite{ABBLS} for their relation to the Pisot substitution conjecture). An easily accessible treatment of Rauzy fractals intended for a broad  audience is given in~\cite{AH:14}.

Recently, Boyland and Severa~\cite{BS:18} considered a particular family of $S$-adic sequences associated with the so-called \emph{infimax $S$-adic family} over three letters. These sequences do not fit into our framework. 
Indeed, they have two ``expanding directions'' which entails that the authors have to project on a $1$-dimensional subspace of $\R^3$ in order to obtain compact Rauzy fractals. Their Rauzy fractals turn out to be Cantor sets which can be subdivided naturally into three subtiles whose convex hulls are intervals that intersect on their boundary points. This fact is used to show that the infimax $S$-adic systems can be geometrically represented as $3$-interval exchange transformations.

In what follows, instead of ``$S$-adic Rauzy fractal'' we will often just say ``Rauzy fractal''. This will cause no confusion. To give the reader a feeling for a Rauzy fractal and its importance in the remaining part of this chapter we provide an example. 

\begin{example}[Tribonacci substitution]\label{ex:rauzy}\index{classical Rauzy fractal}\index{Rauzy fractal!classical}
We explain the definition of a Rauzy fractal for the constant sequence $\bsigma=(\sigma)_{n\in\N}$ with $\sigma$ being the Tribonacci substitution introduced in \eqref{eq:substribo}. The sequence $\bsigma$ is easily checked to be primitive and obviously it is recurrent. Thus it admits a generalized right eigenvector $\bu$ which is just the Perron-Frobenius eigenvector of $M_\sigma$. Since each of the words $\sigma(1)$, $\sigma(2)$, and $\sigma(3)$ begins with $1$ the only limit sequence of $(\sigma)$ is given by
\[
w=\lim_{n\to\infty}\sigma^n(1) = 1213121121312121312112131213121\dots
\] 
and, hence, $\mathcal{R}_\bw := \overline{ \{\pi_{\bu,\bw}\mathbf{l}(p) \;:\; p \hbox{ is a prefix of } w\}}$ for $\bw\in\R_{\ge0}^{d}\setminus\{\mathbf{0}\}$. In Figure~\ref{fig:RauzyTribo} we illustrate the definition of $\Ra_\bu$ and its subtiles (we choose $\bw=\bu$ in this case so the occurring projection $\pi_{\bu,\bu}$ is an orthogonal projection). As mentioned before, this famous prototype of a Rauzy fractal first appears in Rauzy~\cite{Rauzy:82}.
\begin{figure}[hh]
\includegraphics[trim=0 120 0 150,clip,width=0.7\textwidth]{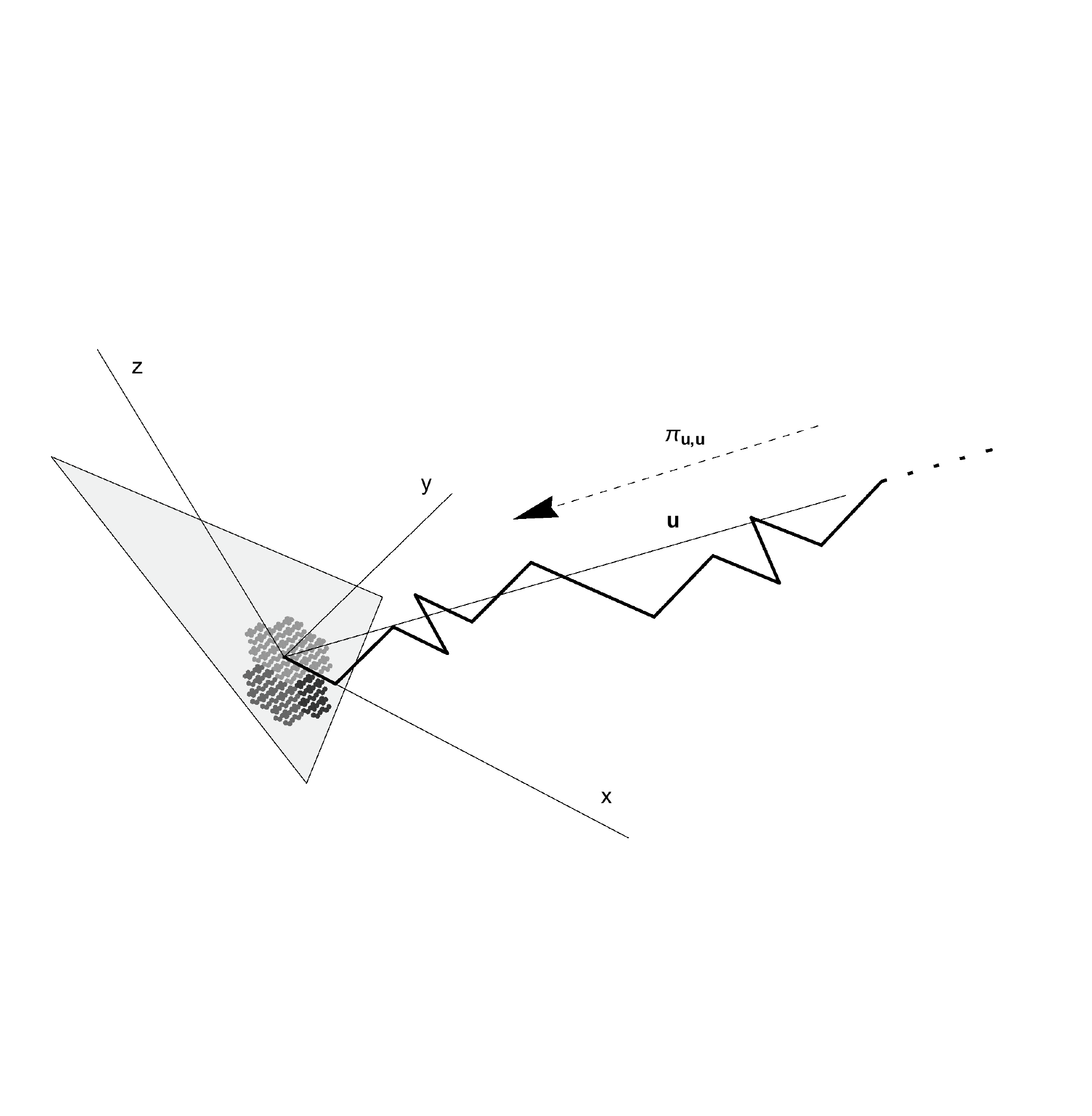}
\caption{The broken line and its projection to $\bu^\bot$ defining the Rauzy fractal $\Ra_\bu$ for the case of the Tribonacci substitution (note that only the vertices of the broken line are projected; not the whole edges).  Each of the three subtiles $\Ra_\bu(i)$ is shaded differently. The shaded triangle represents a part of the plane $\bu^\bot$ in which $\Ra_\bu$ is situated. \label{fig:RauzyTribo}}
\end{figure}

For this example it is known since Rauzy~\cite{Rauzy:82} that one can define a rotation on the Rauzy fractal using the broken line. This can be used to prove  that the substitutive system $(X_{(\sigma)},\Sigma)$ is measurably conjugate to a rotation on $\mathbb{T}^2$. We want to give an idea on how this works without going into the details. To this end it is convenient to work with $\Ra=\Ra_\bone$ and its subtiles. It was shown in \cite{Rauzy:82} that each of the three subtiles $\Ra(i)$, $i\in\{1,2,3\}$, is a compact subset of the space $\bone^\bot$ which is equal to the closure of its interior and has a boundary of $\lambda_\bone$-measure $0$. Moreover, it is proved that  these subtiles are pairwise disjoint apart from overlaps on their boundaries. Thus we can almost everywhere define a ``domain exchange'' $E$ in the following way. If we set 
\[
\tilde \Ra(i) := \overline{ \{\pi_{\bu,\bone}\mathbf{l}(pi) \;:\; pi \hbox{ is a prefix of } w \} } \qquad(i\in\{1,2,3\})
\]
we see from the definition of $\Ra(i)$ that $\tilde \Ra(i)= \Ra(i) +\pi_{\bu,\bone}\mathbf{l}(i)$ (recall that $w$ is the only limit sequence of $\bsigma$). As the Lebesgue measure $\lambda_\bone$ doesn't change under translation and we still have that $\Ra=\tilde\Ra(1)\cup\tilde\Ra(2)\cup\tilde\Ra(3)$ also the translated pieces only overlap on a set of measure $0$. The domain exchange
\[
E: \Ra \to \Ra; \quad \bx \mapsto \bx  +\pi_{\bu,\bone}\mathbf{l}(i) \quad \hbox{for }\bx\in \Ra(i)
\]
is thus well defined almost everywhere and it moves $\Ra(i)$ to $\tilde \Ra(i)$ for each $i\in\{1,2,3\}$. By what was said above, $E$ is an almost everywhere bijective symmetry. The effect of $E$ on the points of $\Ra$ is illustrated in Figure~\ref{fig:tribodomain}. As in the Sturmian case discussed in Section~\ref{sec:sturmrot}, each step on the broken line performs the domain exchange on $\Ra$.
\begin{figure}[hh]
\includegraphics[width=0.3\textwidth]{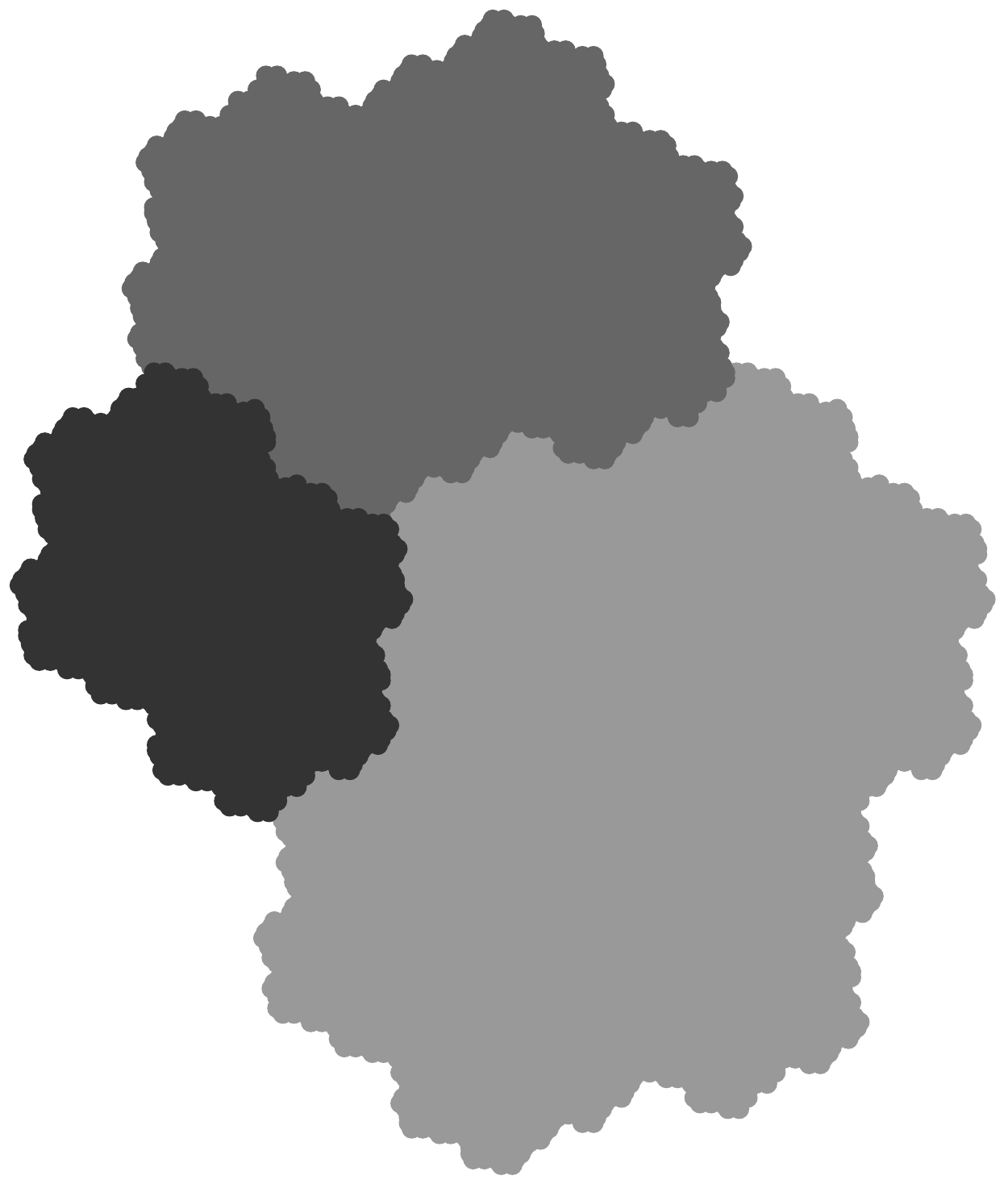}
\put(10,70){$\xrightarrow{\hskip 0.9cm E \hskip 0.9cm}$}
\hskip3cm
\includegraphics[width=0.3\textwidth]{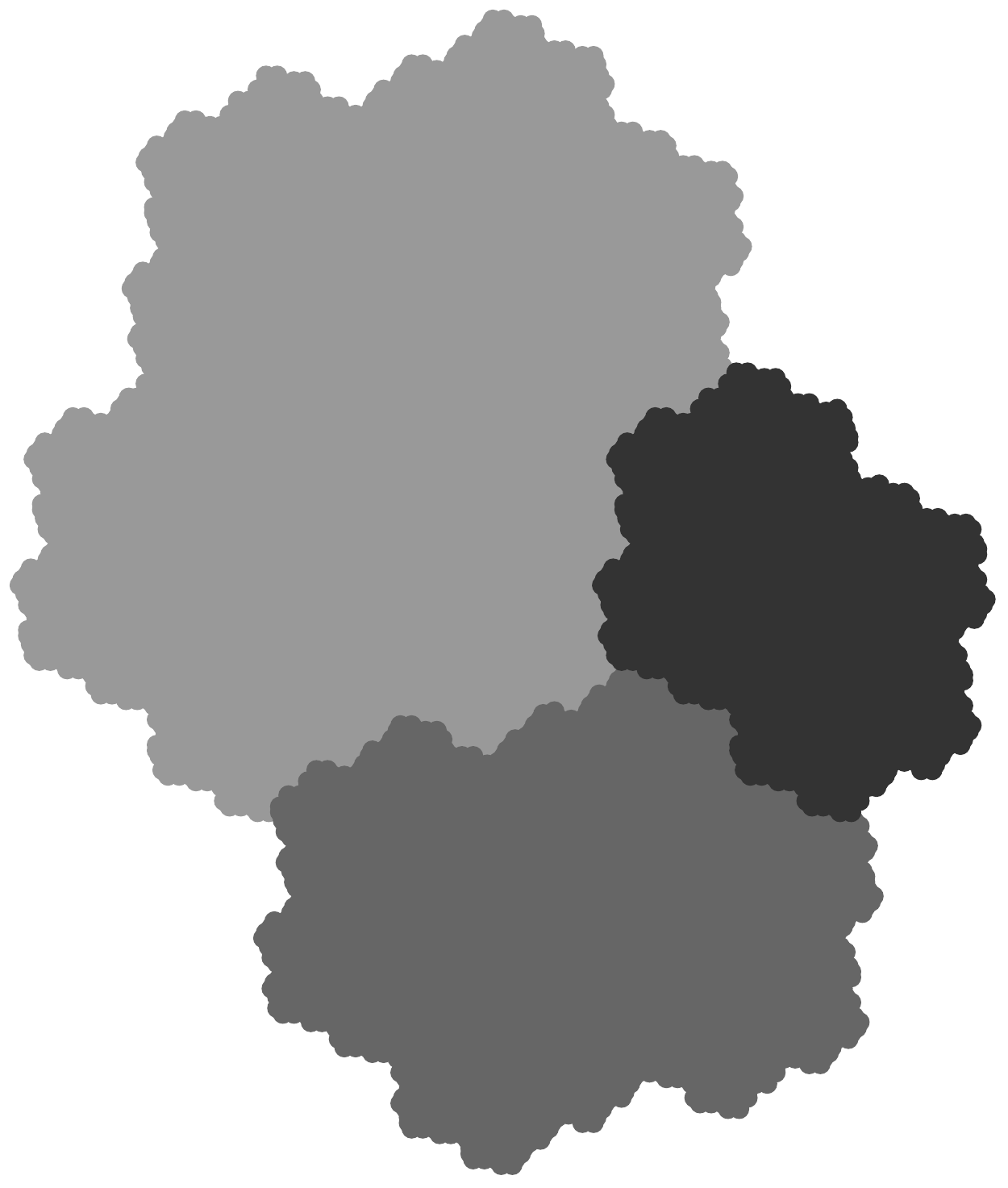}
\caption{The domain exchange on the classical Rauzy fractal associated with the Tribonacci substitution: the bright domain $\Ra(1)$ is translated by $\pi_{\bu,\bone}\mathbf{l}(1)$, the darker domain $\Ra(2)$ is translated by $\pi_{\bu,\bone}\mathbf{l}(2)$, and finally the darkest domain $\Ra(3)$ is translated by $\pi_{\bu,\bone}\mathbf{l}(3)$. The union of the translated domains gives $\Ra$ again. \label{fig:tribodomain}}
\end{figure}

The problem that remains is the transition from the domain exchange to the rotation. In the Sturmian case this was achieved by identifying the endpoints of an interval. Here things become more complicated as intervals are replaced by fractals and we have to make identifications on $\partial \Ra$. 

To settle this, Rauzy~\cite{Rauzy:82} proved that $\Ra$ forms a fundamental domain of the lattice 
$$
\Lambda=(\pi_{\bu,\bone}\mathbf{l}(1) - \pi_{\bu,\bone}\mathbf{l}(2))\Z \oplus (\pi_{\bu,\bone}\mathbf{l}(1)-\pi_{\bu,\bone}\mathbf{l}(3))\Z,
$$ 
{\it i.e.}, it forms a tiling of $\bone^\bot$ when translated by elements of $\Lambda$. Thus $\Ra$ can be seen as a subset of the $2$-torus $\bone^{\bot}/\Lambda$ and since it is a fundamental domain of $\Lambda$ it covers the torus without overlaps (apart from the boundary). This gives the desired identifications on $\partial \Ra$. If we look at the domain exchange on this torus we see that $ \pi_{\bu,\bone}\mathbf{l}(i) \equiv  \pi_{\bu,\bone}\mathbf{l}(j) \pmod {\Lambda}$ holds for $i,j\in\{1,2,3\}$. Thus on this torus all the translations performed by the domain exchange $E$ become the same and, hence, on the torus the mapping $E$ induces a rotation by $\pi_{\bu,\bone}\mathbf{l}(1)$. One can show (by defining a suitable ``representation map'' for the elements of $X_{(\sigma)}$ on the torus $\bone^\bot/\Lambda$) that $(X_{(\sigma)},\Sigma)$ is measurably conjugate to $(\bone^\bot/\Lambda, +\, \pi_{\bu,\bone}\mathbf{l}(1))$, which is a rotation on the $2$-torus. We also refer to \cite[Section~8]{Berthe-Steiner-Thuswaldner} where rigorous arguments are given in a general context (a sketch of these arguments is provided in Section~\ref{sec:rotproof} below). 
\end{example}

In the preceding example various properties of the Rauzy fractal were needed in order to get the measurable conjugacy between the substitutive system and the rotation. Our aim is to establish these conditions for $S$-adic Rauzy fractals under a set of natural conditions. Since tiling properties of $S$-adic Rauzy fractals will play an important role we will now define some collections of Rauzy fractals that will later be shown to provide \emph{tilings} in the following sense.

\begin{definition}[Multiple tiling and tiling]\index{tiling}\index{multiple tiling}
A collection $\mathcal{K}$ of subsets of a Euclidean space $\mathcal{E}$ is called a \emph{multiple tiling} of $\mathcal{E}$ if each element of $\mathcal{K}$ is a compact set which is equal to the closure of its interior, and if there is $m\in\N$ such that almost every point (w.r.t.\ Lebesgue measure) of $\mathcal{E}$ is contained in exactly $m$ elements of $\mathcal{K}$. If $m=1$ then a multiple tiling is called a \emph{tiling}. 
\end{definition}

 The collections of tiles we need in our setting are defined in terms of so-called \emph{discrete hyperplanes}\index{discrete hyperplane}. These objects were first defined and studied in the context of theoretical computer science (see \cite{Reveilles:91} and later~\cite{AAS:97,JT:06}) and have interesting connections to generalized continued fraction algorithms ({\it cf. e.g.}~\cite{BDJP:14,Fernique:06,Fer:08,ItoOtsuki94,JLP:16}). The formal definition reads as follows. Pick $\bw\in \R^d_{\ge 0}\setminus\{\mathbf{0}\}$, then (setting $\be_i=\mathbf{l}(i)$ for $i\in\A$)
 \[
 \Gamma(\bw) = \{[\bx,i] \in \Z^d\times \A \;:\; 0 \le \langle \bx,\bw \rangle < \langle \be_i,\bw \rangle\}.
 \]
This has a geometrical meaning: if we interpret the symbol $[\bx,i]\in \Z^d\times\A$ as the hypercube or ``face''
 \begin{equation}\label{eq:hypercube}
 [\bx,i] =\bigg\{
 \bx + \sum_{j\in\A\setminus\{i\}}\lambda_j\be_j \;:\; \lambda_j\in[0,1]
 \bigg\},
 \end{equation}
 the set $\Gamma(\bw)$ turns into a ``stepped hyperplane'' that approximates $\bw^\bot$ by hypercubes. In Figure~\ref{fig:steppedsurface} this is illustrated for two cases: for a rational vector $\bw$, which leads to a periodic pattern and for an irrational vector $\bw$ which yields an aperiodic one. A finite subset of a discrete hyperplane will often be called a \emph{patch}.
 
 \begin{figure}[hh]
\includegraphics[trim=0 90 0 50,clip,width=0.45\textwidth]{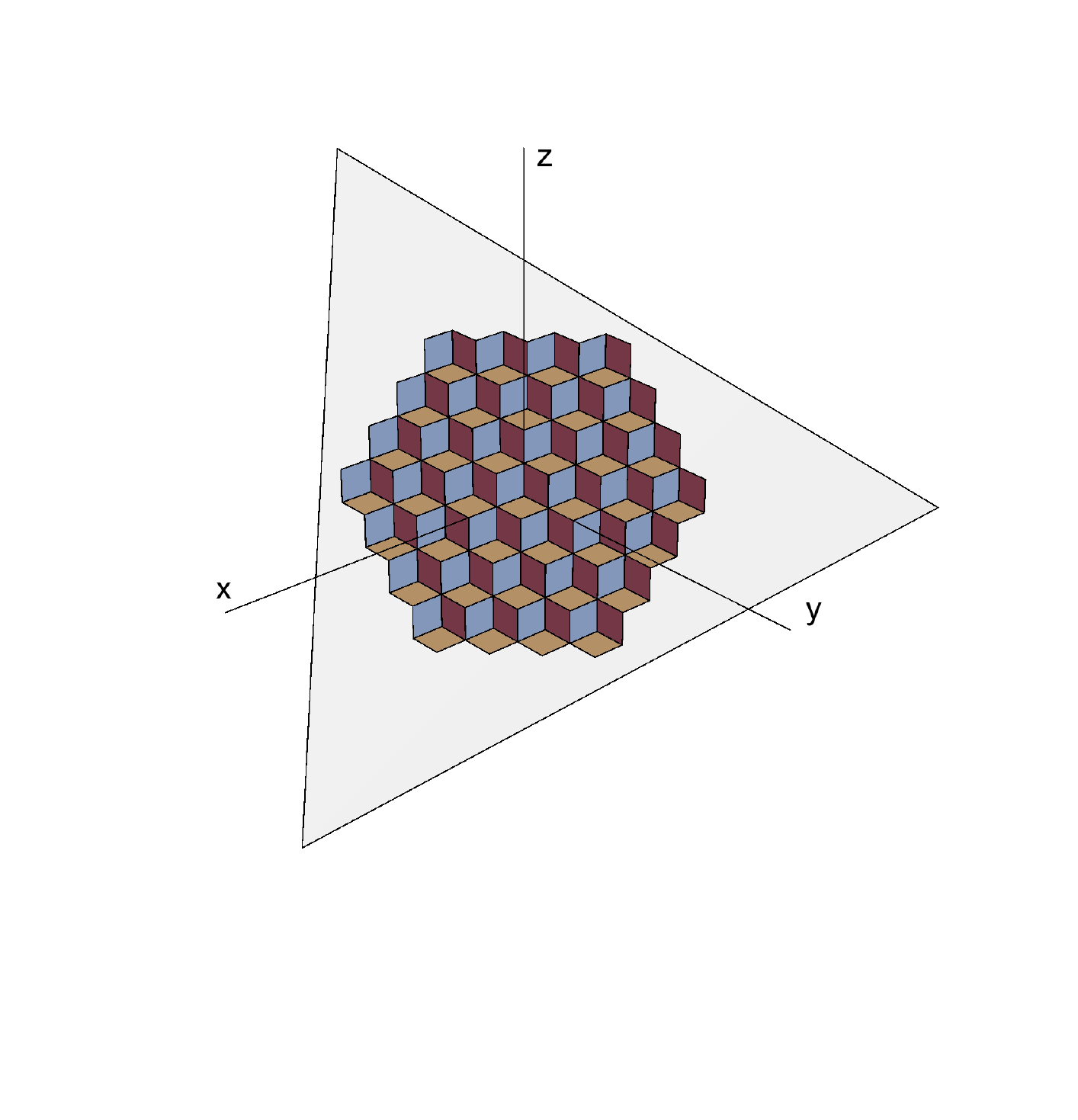}
\includegraphics[trim=0 90 0 50,clip,width=0.45\textwidth]{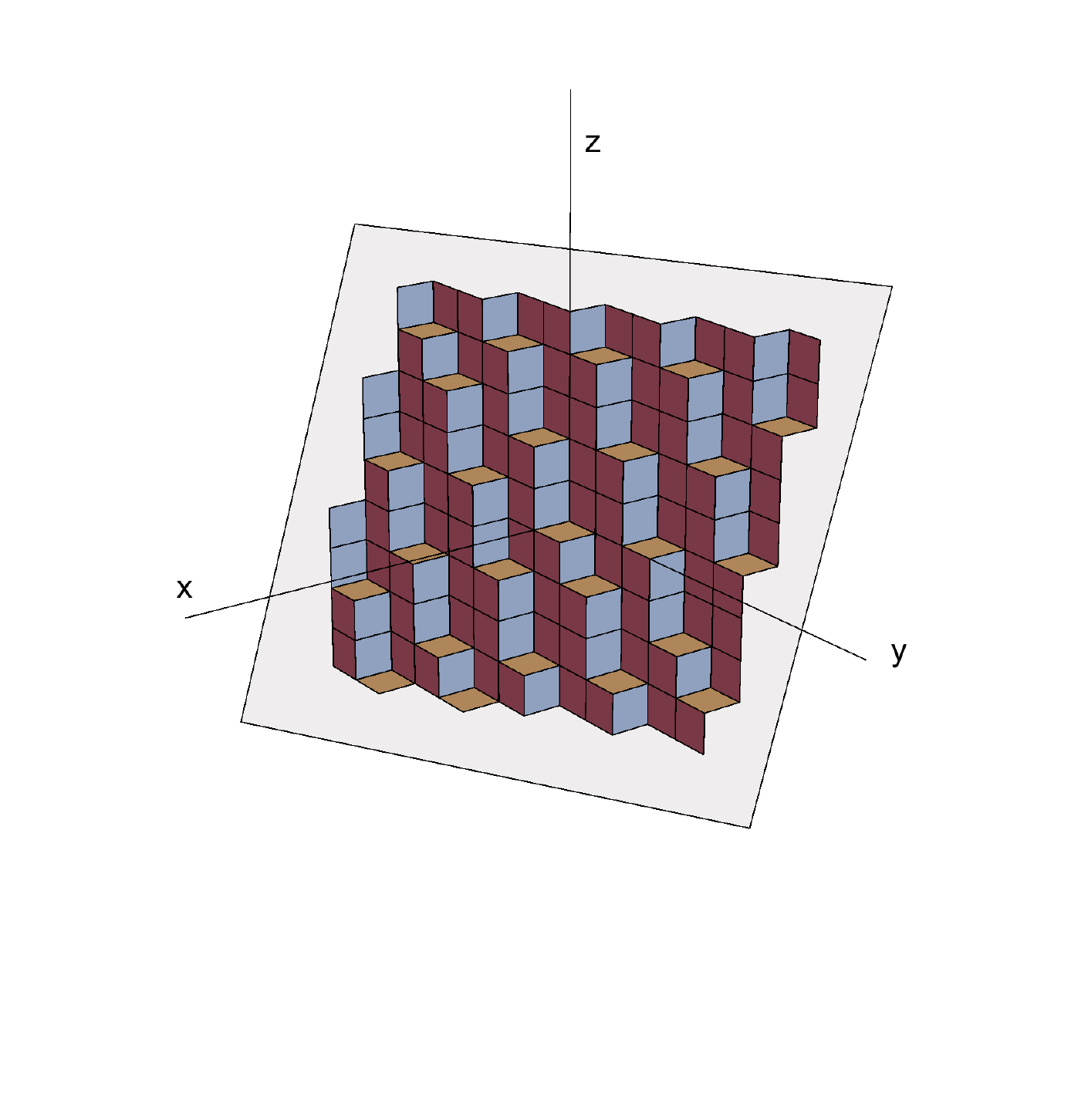}
\caption{Examples of stepped planes. On the left hand side the stepped plane $\Gamma(\bone)$, on the right hand side $\Gamma(\bu)$ with $\bu$ as in Example~\ref{ex:rauzy}. Since $\bone$ is rational the stepped plane $\Gamma(\bone)$ is periodic, while the irrationality of $\bu$ leads to an aperiodic structure in $\Gamma(\bu)$. \label{fig:steppedsurface}}
\end{figure}

Using the concept of discrete hyperplane we define the following collections of Rauzy fractals. Let $\bsigma$ be a sequence of substitutions with generalized right eigenvector $\bu\in\R_{> 0}^d$ and choose $\bw\in\R_{\ge 0}^d\setminus\{\mathbf{0}\}$. Then, following \cite[Section~2.10]{Berthe-Steiner-Thuswaldner}, we set
\begin{equation}\label{eq:cVcoll}
\Co_\bw=\{\pi_{\bu,\bw}\bx + \Ra_\bw(i)\;:\; [\bx,i]\in \Gamma(\bw)  \}.
\end{equation}
As mentioned above, we will see that each of these collections forms a tiling of the space $\bw^\bot$ under natural conditions. A special role will be played by the collection $\Co_\bone$ which will give rise to a periodic tiling of $\bone^\bot$ by lattice translates of the Rauzy fractal $\Ra$.

\subsection{Balance, algebraic irreducibility, and strong convergence}

Let $(X_\bsigma,\Sigma)$ be an $S$-adic system. As we mentioned already, the associated  Rauzy fractals can be used to prove that $(X_\bsigma,\Sigma)$ is measurably conjugate to a rotation on a torus provided that they have suitable properties. In the present section we will discuss two conditions that have to be imposed on $\bsigma$ in order to guarantee that each of the associated Rauzy fractals $\Ra_\bw$, $\bw\in\R_{\ge 0}\setminus\{\mathbf{0}\}$, as well as each of their subtiles $\Ra_\bw(i)$, $i\in\A$, is a compact set that is the closure of its interior and has a boundary of zero measure $\lambda_\bw$. 

The first property is \emph{balance} and as we will see immediately it entails compactness of $\Ra_\bw$ and its subtiles (see {\it e.g.}\ \cite{Adamczewski:03,Berthe-Delecroix} or \cite[Section~2.4]{Berthe-Steiner-Thuswaldner} for similar definitions).

\begin{definition}[Balance]\label{def:balance}\index{balance}
Let $\A$ be an alphabet and consider a pair of words $(u,v)\in\A^*\times\A^*$ of the same length. If there is $C>0$ such that $\big| |v|_i - |u|_i\big| \le C$ holds for each letter $i\in \A$, the pair $(u,v)$ is called \emph{$C$-balanced}. A language $\Lg\subset \A^*$ is called \emph{$C$-balanced} if each pair $(u,v)\in\Lg\times\Lg$ with $|u|=|v|$ is $C$-balanced. It is called \emph{finitely balanced} if it is $C$-balanced for some $C>0$.
\end{definition}

In Definition~\ref{def:sturmbalance} and in Section~\ref{sec:ARImbalance} we defined balance of an infinite sequence and applied this notion to Sturmian sequences as well as to Arnoux-Rauzy sequences. For a general $S$-adic system $(X_\bsigma,\Sigma)$ it is more convenient to look at balance of the associated language $\Lg_\bsigma$ since there might be more than one limit sequence associated with the given directive sequence $\bsigma$. Of course, by Proposition~\ref{prop:sadicminimal}~(ii) primitivity of $\bsigma$ implies that each of these limit sequences has the language $\Lg_\bsigma$ of factors.

The following result goes back essentially to \cite[Proposition~7]{Adamczewski:03} and, in the form we present it here, is contained in \cite[Lemma~4.1]{Berthe-Steiner-Thuswaldner} (in fact, the conditions that are imposed on $\bsigma$ in that paper are slightly weaker than ours). 
 
\begin{proposition}\label{prop:Rcompact}
Let $\bsigma$ be a primitive and recurrent sequence of unimodular substitutions. Then $\Ra_{\bw}$ and each of its subtiles is compact for each $\bw\in\R_{\ge0}^d\setminus\{\mathbf{0}\}$ if and only if $\Lg_\bsigma$ is finitely balanced.
\end{proposition}
 
\begin{proof}
Since $\Ra_{\bw}$ as well as each of its subtiles is closed by definition it suffices to prove that 
\begin{equation}\label{eq:boundbal}
\Ra_{\bw} \hbox{ is bounded for each }\bw\in\R_{\ge0}^d\setminus\{\mathbf{0}\} \quad\Longleftrightarrow\quad \Lg_\bsigma \hbox{ is finitely balanced. }
\end{equation}

We start with proving \eqref{eq:boundbal} for the case $\bw=\bone$ and follow~\cite{Berthe-Steiner-Thuswaldner}. Let $\bu$ be a generalized right eigenvector for $\bsigma$ which exists by Proposition~\ref{prop:furstmatrix}. 

If $\Ra$ is bounded then there is $C>0$ such that $\Vert\pi_{\bu,\bone}\mathbf{l}(p)\Vert_\infty \le C$ for each prefix of a limit sequence of $\bsigma$. Let $u,v\in\Lg_\bsigma$ be of equal length. Then, by primitivity these words are factors of a limit sequence which entails that $\Vert\pi_{\bu,\bone}\mathbf{l}(u)\Vert_\infty,\Vert\pi_{\bu,\bone}\mathbf{l}(v)\Vert_\infty\le 2C$. As $\mathbf{l}(u)-\mathbf{l}(v)\in\bone^\bot$ this yields $\Vert\mathbf{l}(u)-\mathbf{l}(v)\Vert_\infty=\Vert\pi_{\bu,\bone}\mathbf{l}(u)-\pi_{\bu,\bone}\mathbf{l}(v)\Vert_\infty \le 4C$ and, hence, $\Lg_\bsigma$ is $4C$-balanced.

Assume now that $\Lg_\bsigma$ is $C$-balanced and let $w$ be a limit sequence of $\bsigma$. Let $p$ be a prefix of $w$ and write $w=v_0v_1\dots$ where $v_k\in\A^{*}$ with $|v_k|=|p|$ for each $k\ge 0$. By $C$-balance, $\Vert\pi_{\bu,\bone}\mathbf{l}(v_k)-\pi_{\bu,\bone}\mathbf{l}(p)  \Vert_\infty\le C$ for each $k\in\N$ and, hence, $\big\Vert\frac1n\sum_{k=0}^{n-1}\pi_{\bu,\bone}\mathbf{l}(v_k)-\pi_{\bu,\bone}\mathbf{l}(p)  \big\Vert_\infty\le C$ for each $n\in \N$. By Lemma~\ref{lem:SadicFreq} (see proof of Part 1), the letter frequencies of $w$ are given by the entries of the vector $\bu/\Vert\bu\Vert_1$ which implies that $\lim_{n\to\infty}\frac1n\sum_{k=0}^{n-1}\pi_{\bu,\bone}\mathbf{l}(v_k)={\bf 0}$ and thus
\[
\Vert\pi_{\bu,\bone}\mathbf{l}(p)  \Vert_\infty =
\bigg\Vert\lim_{n\to\infty}\frac1n\sum_{k=0}^{n-1}\pi_{\bu,\bone}\mathbf{l}(v_k)-\pi_{\bu,\bone}\mathbf{l}(p)  \bigg\Vert_\infty \le C. 
\]

This finishes the proof of \eqref{eq:boundbal} for the case $\bw=\bone$. The full statement \eqref{eq:boundbal} follows from this because $\Ra_\bw=\pi_{\bu,\bw}\Ra$, which implies that $\Ra$ is bounded if and only if $\Ra_{\bw}$ is bounded for each $\bw\in\R_{\ge0}^d\setminus\{\mathbf{0}\}$. 
\end{proof}

Our next aim is to make sure that $\Ra_\bw(i)$ has nonempty interior for each $\bw\in\R_{\ge0}^d\setminus\{\mathbf{0}\}$ and each $i\in\A$. This will require much more work. In a first step observe that we have no hope to get nonempty interior if $\bu$ has coordinates which are \emph{rationally dependent}, {\it i.e.}, if there is $\bx\in\Z^d$ such that $\langle \bx,\bu \rangle=0$. Indeed, in this case the set $\Ra_\bw$ is contained in a finite union of proper affine subspaces of $\bw^{\bot}$. We wish to exclude this case first. This is related to an irreducibility property (going back to \cite[Section~2.2]{Berthe-Steiner-Thuswaldner}) of the underlying set of incidence matrices which we define now.

\begin{definition}[Algebraic irreducibility]\label{ef:algirr}\index{algebraic irreducibility}
Let $\bM=(M_n)$ be a sequence of nonnegative integer matrices. We say that $\bM$ is \emph{algebraically irreducible} if for each $m\in\N$ there is $n>m$ such that the characteristic polynomial of $M_{[m,\ell)}$ is irreducible for each $\ell \ge n$. 

A sequence $\bsigma$ of substitutions is called \emph{algebraically irreducible} if it has a sequence of incidence matrices which is algebraically irreducible.
\end{definition}

\begin{remark}
For our purposes we can replace algebraic irreducibility by the weaker condition that for each $m\in\N$ the matrix $M_m$ is regular and there is $n>m$ such that $M_{[m,\ell)}$ does not have $1$ as eigenvalue for each $\ell \ge n$.
This condition is easier to check than algebraic irreducibility.

However, since we will always have to assume balance in our setting all but the dominant eigenvalue of large blocks $M_{[m,\ell)}$ should be inside the closed unit disk anyway ({\it cf.}~also the definition of the Pisot condition in \eqref{eq:lyapunov}). Thus this new condition is not essentially weaker than algebraic irreducibility. For this reason we work with algebraic irreducibility in the sequel.
\end{remark}

Together with other properties, algebraic irreducibility of $\bsigma$ implies rational independence of the right eigenvector. We announce this in the following lemma, whose elegant proof is taken from~\cite[Lemma~4.2]{Berthe-Steiner-Thuswaldner}.

\begin{lemma}\label{irrirr}
Let $\bsigma$ be an algebraically irreducible sequence of substitutions with finitely balanced language $\Lg_\bsigma$ that admits a generalized right eigenvector $\bu\in\R_{\ge 0}^{d}\setminus\{\mathbf{0}\}$. Then $\bu$ has rationally independent coordinates.
\end{lemma} 
 
\begin{proof}
The proof is done by contradiction. Assume that $\bu$ has rationally dependent coordinates. Then there is $\bx\in \Z^d\setminus\{\mathbf{0}\}$ such that $\langle \bx, \bu \rangle=0$. This implies that $\langle (M_{[0,n)})^t\bx, \be_i \rangle=\langle \bx, M_{[0,n)}\be_i \rangle=\langle \bx, \mathbf{l}\sigma_{[0,n)}(i) \rangle = \langle \bx, \pi_{\bu,\bone}\mathbf{l}\sigma_{[0,n)}(i) \rangle$ is uniformly bounded in $i\in\A$ and $n\in\N$ by balance of $\Lg_\bsigma$. Thus $(M_{[0,n)})^t\bx \in \Z^d$ is bounded and, hence, there exists an integer $k$ and infinitely many $\ell>k$ with $(M_{[0,k)})^t\bx=(M_{[0,\ell)})^t\bx$. Multiplying by $((M_{[0,k)})^t)^{-1}$ we see that $\bx$ is an eigenvector of $(M_{[k,\ell)})^t$ with eigenvalue $1$. Since $\ell$ can be chosen arbitrarily large this contradicts algebraic irreducibility.
\end{proof}

In Definition~\ref{def:weakconv} the concept of weak convergence of a sequence of matrices is introduced. In what follows, we will need a stronger form of convergence, {\it viz.}\ \emph{strong convergence}. If we look back to Lemma~\ref{lem:rotCF} we see that the cascade of inductions we perform on the interval leads to smaller and smaller intervals (that are blown up by renormalization) whose lengths tend to $0$. To get an analogous behavior on $S$-adic Rauzy fractals we need to introduce a certain subdivision on them whose pieces have a diameter that tends to zero. It will turn out that strong convergence is the right condition to guarantee this behavior. We thus recall the definition of strong convergence which is well known in the theory of generalized continued fraction algorithms (see {\it e.g.}~\cite[Definition~19]{Schweiger:00}) and then derive it from the conditions we introduced so far.

\begin{definition}[Strong convergence]\index{strong convergence}\index{convergence!strong}
We say that a sequence $\bM=(M_n)$ of nonnegative integer matrices is \emph{strongly convergent} to $\bu\in\R_{\ge 0}^{d}\setminus\{\mathbf{0}\}$ if
\[
\lim_{n\to\infty} \pi_{\bu,\bone}M_{[0,n)}\be_i=\mathbf{0} \quad \hbox{for all }i\in\A. 
\]
If $\bsigma$ has a strongly convergent sequence of incidence matrices we say that $\bsigma$ is \emph{strongly convergent}.
\end{definition}

The difference between weak and strong convergence is explained and illustrated in Figure~\ref{fig:WSC}: while weak convergence of vectors can be seen on the unit ball, strong convergence takes place at their end points.

\begin{figure}[hh]
\includegraphics[trim=0 0 0 0,clip,width=0.5\textwidth]{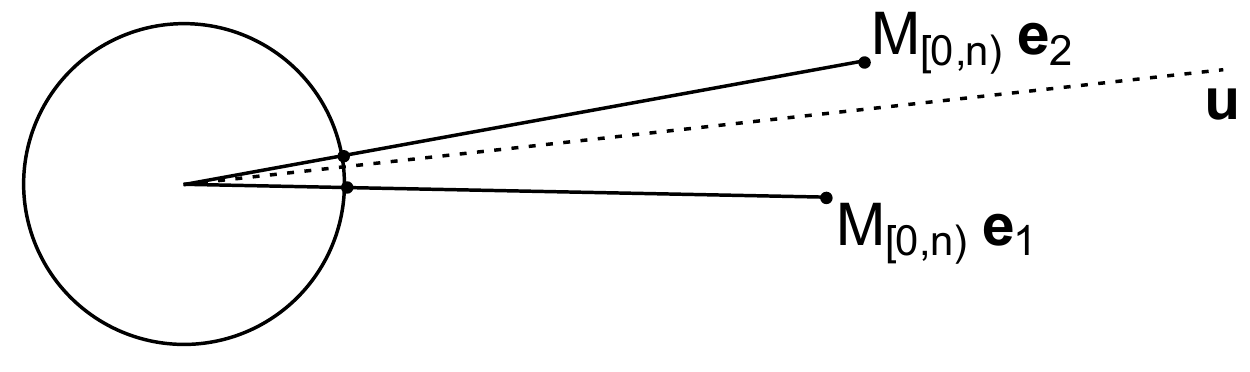}

\caption{The sequence $\bM=(M_{n})$ of matrices is weakly convergent, if the intersections of $M_{[0,n)}\be_{i}$ with the unit ball converge to the intersection of the generalized right eigenvector $\bu$ with the unit ball.
It is strongly convergent, if the minimal distance of the point $M_{[0,n)}\be_{i}$ to the ray $\R_+\bu$ converges to zero  for each $i\in\mathcal{A}$.
Summing up: weak convergence takes place on the unit ball while strong convergence concerns the end points of the vectors. \label{fig:WSC}}
\end{figure}

The following result on strong convergence will be needed in the sequel. It is the content of \cite[Proposition~4.3]{Berthe-Steiner-Thuswaldner}.

\begin{proposition}\label{prop:strongconv}
Let $\bsigma$ be a primitive, algebraically irreducible, and recurrent sequence of substitutions with finitely balanced language $\Lg_\bsigma$. Then
\[
\lim_{n\to\infty} \sup\{\Vert \pi_{\bu,\bone}M_{[0,n)}\mathbf{l}(v)\Vert_\infty\;:\; v\in\Lg_\bsigma^{(n)} \}=0.
\]
By primitivity this implies that $\bsigma$ is strongly convergent.
\end{proposition}

The proof of this result is quite tricky. We give a sketch to illustrate the ideas and refer to \cite[Proposition~4.3]{Berthe-Steiner-Thuswaldner} for details.

\begin{proof}[Sketch]
Let $w$ be a limit sequence of $\bsigma$. By primitivity we may apply Proposition~\ref{prop:sadicminimal}~(ii) to the shifted sequence $(\sigma_{n},\sigma_{n+1},\sigma_{n+2},\ldots)$. Thus the language $\Lg_\bsigma^{(n)}$ is equal to the language $L(w^{(n)})$ of factors of the $n$-th ``desubstitution'' $w^{(n)}$ of $w$ (see \eqref{eq:limitworddesubs}), {\it i.e.}, each $v\in\Lg_\bsigma^{(n)}$ satisfies $\mathbf{l}(v)=\mathbf{l}(p)-\mathbf{l}(q)$, where $p$ and $q$ are prefixes of $w^{(n)}$. Thus it suffices to prove
\begin{equation}\label{eq:43htp}
\lim_{n\to\infty} \sup\{\Vert \pi_{\bu,\bone}M_{[0,n)}\mathbf{l}(p)\Vert_\infty \;:\; p \hbox{ is a prefix of }w^{(n)}  \}=0.
\end{equation}

Let $(i_n)$ be the sequence of first letters of $w^{(n)}$ and choose $\varepsilon>0$ arbitrary. Set
\[
\mathcal{S}_n=  \{\pi_{\bu,\bone}\mathbf{l}(p) \;:\; p \hbox{ is a prefix of } \sigma_{[0,n)}(i_n) \} \quad\hbox{and}\quad
\tilde \Ra := \overline{ \{\pi_{\bu,\bone}\mathbf{l}(p) \;:\; p \hbox{ is a prefix of } w \} }. 
\] 
Then $\mathcal{S}_n \to \tilde\Ra$ for $n\to\infty$ in Hausdorff metric. Since, on the other hand, $\pi_{\bu,\bone}M_{[0,n)}\mathbf{l}(p)+\mathcal{S}_n\subset \tilde\Ra$ holds for each $p\in\A^*$ such that $pi_n$ is a prefix $w^{(n)}$ we obtain
\begin{equation}\label{eq:sadicprop43first}
\Vert \pi_{\bu,\bone}M_{[0,n)}\mathbf{l}(p)\Vert_\infty<\varepsilon
\end{equation}
for  each $p\in\A^*$ such that $pi_n$ is a prefix $w^{(n)}$ for a large enough $n$. We have to prove \eqref{eq:sadicprop43first} for arbitrary prefixes $p$ of $w^{(n)}$. If $N(p)=\{n\in\N\,:\, pi_n \hbox{ is a prefix of } w^{(n)}\}$ is infinite then \eqref{eq:sadicprop43first} yields
\begin{equation}\label{eq:npeq}
\lim_{n\in N(p),\, n\to\infty } \Vert\pi_{\bu,\bone}M_{[0,n)}\mathbf{l}(p)\Vert_\infty =0.
\end{equation}
Using algebraic irreducibility and balance by some tricky arguments it is now possible to find a set $P$ of prefixes of $w$ such that the abelianizations $\mathbf{l}(P)$ contain a basis of $\R^d$ and $N(P)=\bigcap_{p\in P} N(p)$ is an infinite set (moreover, the elements of $P$ can by ``synchronized'' in a certain way by using the recurrence of $\bsigma$). This implies that \eqref{eq:npeq} is true for each $p\in P$ when $N(p)$ is replaced by $N(P)$, {\it i.e.},
\begin{equation*}\label{eq:npeq2}
\lim_{n\in N(P),\, n\to\infty } \Vert\pi_{\bu,\bone}M_{[0,n)}\mathbf{l}(p)\Vert_\infty =0 \qquad(p\in P).
\end{equation*}
Since $\mathbf{l}(P)$ contains a basis of $\R^d$ we gain 
\begin{equation}\label{eq:npeq3}
\lim_{n\in N(P),\, n\to\infty } \Vert\pi_{\bu,\bone}M_{[0,n)}\bx\Vert_\infty =0 \qquad(\bx\in \mathbb{R}^d).
\end{equation}
Using primitivity and recurrence again, equation \eqref{eq:43htp} can be obtained using \eqref{eq:sadicprop43first} and \eqref{eq:npeq3}. This again requires some work and we omit the details.
\end{proof}

\section{Properties of $S$-adic Rauzy fractals}\label{sec:rauzyprop}

Based on the results of the previous section we will now study deeper properties of $S$-adic Rauzy fractals. In particular, the present section is devoted to the illustration of the proof of the following result from Berth\'e, Steiner, and Thuswaldner~\cite[Theorem~3.1~(ii)]{Berthe-Steiner-Thuswaldner}.
 \begin{theorem}\label{th:RauzyProperties}
Let $S$ be a finite set of unimodular substitutions over a finite alphabet $\A$ and let $\bsigma=(\sigma_n)$ be a primitive and algebraically irreducible sequence of substitutions taken from the set $S$. Assume that there is $C>0$ such that for every $\ell\in\N$ there exists $n\ge 1$ such that $(\sigma_n,\ldots,\sigma_{n+\ell-1})=(\sigma_0,\ldots,\sigma_{\ell-1})$ and the language $\Lg_\bsigma^{(n+\ell)}$ is $C$-balanced.

Then each subtile $\Ra(i)$, $i\in\A$, of the Rauzy fractal $\Ra$ is a nonempty compact set which is equal to the closure of its interior and has a boundary whose Lebesgue measure $\lambda_\bone$ is zero.
\end{theorem}

\begin{remark}\label{rem:commentThmRot}\mbox{}
\begin{enumerate}
\item[(i)] We can see that the assumptions of this theorem contain all the properties we discussed in the previous subsections. We could have used the stronger assumption that $\bsigma$ is primitive, recurrent, algebraically irreducible, and has $C$-balanced language $\Lg_\bsigma^{(n)}$ for each $n\in\N$. However, although this assumption is more handy and holds for many natural examples it would lead to a measure zero subset of the set of ``all'' sequences $\bsigma$. The conditions we give in Theorem~\ref{th:RauzyProperties} will turn out to be ``generic'' in the sense that they are true for ``almost all'' sequences $\bsigma$. All this will be made precise when we develop a metric counterpart of our theory in Section~\ref{sec:metric}.

\item[(ii)]  Let $\sigma$ be a substitution on the alphabet $\A$. It is easy to prove that for each $C>0$ there is $C'>0$ such that $\sigma(w)$ is $C'$-balanced for each $C$-balanced sequence $w\in\A^\N$. Applying this to the substitution $\sigma=\sigma_{[0,n+\ell)}$ for some $n,\ell$ with balanced language $\Lg_\bsigma^{(n+\ell)}$ we see that we can choose the constant $C$ in Theorem~\ref{th:RauzyProperties} in a way that also $\Lg_\bsigma$ is $C$-balanced. We will always assume that $C$ is chosen in this way in the sequel.

\item[(iii)] We confine ourselves to finite sets $S$ of substitutions to keep things as simple as possible. With a bit more effort it is possible to generalize Theorem~\ref{th:RauzyProperties} to infinite sets $S$. This is of interest because infinite sets $S$ correspond to multiplicative continued fraction algorithms like the important Jacobi-Perron algorithm or an acceleration of the Arnoux-Rauzy algorithm proposed recently by Avila, Hubert, and Skripchenko~\cite{AHS:15}. This more general setting is treated in \cite{Berthe-Steiner-Thuswaldner}.
\end{enumerate}
\end{remark}

Theorem~\ref{th:RauzyProperties} will enable us to study tiling properties of $\Ra_\bw$ and its subtiles which will finally lead to the measurable conjugacy of $(X_\bsigma,\Sigma)$ to a rotation. 

The proof of Theorem~\ref{th:RauzyProperties} is quite long and technical and we refer to \cite[Section~6]{Berthe-Steiner-Thuswaldner} for details. Our aim here is to illustrate the main ideas in a way that is hopefully more accessible to a broader readership than the original research paper. First we will establish a set equation for the subtiles $\Ra_\bw(i)$, $i\in\A$, of $\Ra_\bw$ that governs certain subdivisions of $\Ra_\bw(i)$. Using this set equation we will be able to establish the properties of $S$-adic Rauzy fractals stated in Theorem~\ref{th:RauzyProperties}. 

Theorem~\ref{th:RauzyProperties} has a number of predecessors. For instance, Lagarias and Wang~\cite{Lagarias-Wang:96a} proved that each \emph{self-affine tile} $\mathcal{T}$ is the closure of its interior and $\partial \mathcal{T}$ has Lebesgue measure zero. For substitutive Rauzy fractals the according result was proved by Sirvent and Wang~\cite{SirventWang02}. However, in all these cases the sets have strong self-affinity properties which are no longer present in our setting. We therefore need new ideas and more efforts to get the desired results (in particular, the proof of the fact that the boundary of an $S$-adic Rauzy fractal has measure zero will need quite some work).

\subsection{Set equations for $S$-adic Rauzy fractals and dual substitutions}\label{sec:seteq}

The first important tool in the proof of Theorem~\ref{th:RauzyProperties} will be a \emph{set equation} for the subtiles $\mathcal{R}_\bw(i)$, $\bw\in\R_{\ge 0}^d\setminus\{\mathbf{0}\}$ and $i\in\A$, of a sequence $\bsigma$ of unimodular substitutions as well as for related subtiles associated with ``shifts'' of $\bsigma$. This set equation equips the sets $\mathcal{R}_\bw(i)$ with a subdivision structure that is governed by $\bsigma$. We now give an idea on how this works.

Let $\bsigma=(\sigma_n)$ be a primitive and recurrent sequence of unimodular substitutions over the alphabet $\A$ with generalized right eigenvector $\bu\in\R_{> 0}^d$ and choose $\bw\in\R_{\ge 0}^d\setminus\{\mathbf{0}\}$. In all what follows, keep in mind the definition of the subtile $\Ra_\bw(i)$ from \eqref{eq:subtileRi}. We choose a limit sequence $w$ of $\bsigma$ and associate with it the sequence $(w^{(n)})$ of its ``desubstitutions'' according to \eqref{eq:limitworddesubs}. 

Consider the set $\{\pi_{\bu,\bw}\mathbf{l}(p)\;:\; pi \hbox{ is a prefix of }w \}$ and observe that by the definition of a limit sequence each $p\in\mathcal{A}^*$ for which $pi$ is a prefix of $w$ can be written as $p=\sigma_0(p')p_0$ with $p_0i$ a prefix of $\sigma_0(j)$ for some $j\in\A$ and $p'j$ some prefix of $w^{(1)}$. Using this decomposition of $p$ we obtain the decomposition
\begin{equation}\label{eq:set1}
\begin{split}
\{\pi_{\bu,\bw}\mathbf{l}(p)\;:\; &pi \hbox{ is a prefix of }w \} =\\
&\bigcup_{\begin{subarray}{c} 
 j\in \A,\, p_{0} \in \A^*\\
\sigma_0(j)=p_0i\A^*
\end{subarray}}
\{\pi_{\bu,\bw}\mathbf{l}(p_0) + \pi_{\bu,\bw}(\mathbf{l}\sigma_0(p')) \;:\; p'j \hbox{ is a prefix of }w^{(1)} \}.
\end{split}
\end{equation}
From \eqref{eq:MsigmaDiagram} we see that $\mathbf{l}\sigma_0(p')=M_0\mathbf{l}(p')$. Moreover, direct calculation (see~\cite[Lemma~5.2]{Berthe-Steiner-Thuswaldner}) yields that $\pi_{\bu,\bw} M_0 = M_0\pi_{M_0^{-1}\bu,M_0^{t}\bw}$. Inserting this in \eqref{eq:set1} we gain
\begin{equation*}\label{eq:set2}
\begin{split}
\{\pi_{\bu,\bw}\mathbf{l}(p)\;:\; &pi \hbox{ is a prefix of }w \} =\\
&\bigcup_{\begin{subarray}{c} 
j\in \A,\, p_{0} \in \A^*\\
\sigma_0(j)=p_0i\A^*
\end{subarray}}
\pi_{\bu,\bw}\mathbf{l}(p_0) + M_0\{ \pi_{M_0^{-1}\bu,M_0^{t}\bw}\mathbf{l}(p') \;:\; p'j \hbox{ is a prefix of }w^{(1)} \}.
\end{split}
\end{equation*}
Taking the union over all (finitely many, by Proposition~\ref{prop:sadicminimal}) limit sequences of $\bsigma$ and taking the closure we obtain by \eqref{eq:subtileRi} that
\begin{equation}\label{eq:set3}
\begin{split}
&\Ra_\bw(i)
=\\
&\bigcup_{\begin{subarray}{c} 
j\in \A,\, p_{0} \in \A^* \\
\sigma_0(j)=p_0i\A^*
\end{subarray}}
\pi_{\bu,\bw}\mathbf{l}(p_0) + M_0\overline{\{ \pi_{M_0^{-1}\bu,M_0^{t}\bw}\mathbf{l}(p') \;:\; p'j \hbox{ is a prefix of some limit sequence of } (\sigma_{n+1}) \}}.
\end{split}
\end{equation}
We now inspect the closures in the union in \eqref{eq:set3}. Looking at the definition of subtiles in \eqref{eq:subtileRi} we see that these are subtiles of the Rauzy fractal corresponding to the shifted sequence $(\sigma_{n+1})_{n\ge 0}$ of $\bsigma$. Indeed, it follows from Proposition~\ref{prop:furstmatrix} that $M_0^{-1}\bu$ is the right eigenvector of this shifted sequence. This motivates the following definitions.

For $k\in\N$ let 
\begin{equation}\label{eq:nproj}
\pi^{(k)}_{\bu,\bw} = \pi_{M_{[0,k)}^{-1}\bu,M_{[0,k)}^{t}\bw},
\end{equation}
denote the subtiles of the shifted sequence of substitutions $(\sigma_{n+k})_{n\in\N}$ which live in the hyperplane $(M_{[0,k)}^{t}\bw)^\bot$ by 
\begin{equation}\label{eq:RnRauzyfr}
\Ra_\bw^{(k)}(i) := \overline{\{ \pi^{(k)}_{\bu,\bw}\mathbf{l}(p') \;:\; p'j \hbox{ is a prefix of some limit sequence of } (\sigma_{n+k})_{n\in\N} \}},
\end{equation}
and set $\Ra_\bw^{(k)}=\bigcup_{i\in\A}\Ra_\bw^{(k)}(i)$.
Together with these notations \eqref{eq:set3} can be generalized by using similar arguments as we used in its proof. The generalized form of \eqref{eq:set3} reads as follows (for a detailed proof see~\cite[Proposition~5.6]{Berthe-Steiner-Thuswaldner}).

\begin{proposition}[The set equation]\label{prop:seteq}\index{Rauzy fractal!set equation}
Let $\bsigma$ be a primitive and recurrent sequence of unimodular substitutions with generalized right eigenvector $\bu$. Then for each $[\bx,i]\in\Z^d\times\A$ and every $k,\ell\in\N$ with $k <\ell$ we have
\begin{equation}\label{eq:seteq}
\pi^{(k)}_{\bu,\bw}\bx + \Ra_\bw^{(k)}(i)=
\bigcup_{[\by,j]\in E_1^*(\sigma_{[k,\ell)})[\bx,i]} M_{[k,\ell)}(\pi^{(\ell)}_{\bu,\bw}\by + \Ra_\bw^{(\ell)}(j)),
\end{equation}
where
\begin{equation}\label{eq:dualGeomSubs}
E_1^*(\sigma)[\bx,i] = 
\{
[M_\sigma^{-1}(\bx + \mathbf{l}(p)),j]\;:\; j\in\A,\, p\in\A^* \hbox{ such that }pi\hbox{ is a prefix of }\sigma(j)
\}.
\end{equation}
\end{proposition}

The elements in the union on the right hand side of \eqref{eq:seteq} are called the \emph{level $(\ell-k)$ subtiles} of  $\pi^{(k)}_{\bu,\bw}\bx + \Ra_\bw^{(k)}(i)$. The collection of all the elements in the union is called the \emph{$(\ell-k)$-th subdivision} of $\pi^{(k)}_{\bu,\bw}\bx + \Ra_\bw^{(k)}(i)$. This will often be applied for the case $k=0$. In Figure~\ref{fig:seteq} the set equation is illustrated for the situation discussed in Example~\ref{ex:ARseteq}.

The \emph{dual geometric realization}\index{dual geometric realization}  $E_1^*(\sigma)$ of a substitution $\sigma$ defined in \eqref{eq:dualGeomSubs} will turn out to be useful when we define so-called \emph{coincidence conditions} in Section~\ref{sec:coinc}. If we regard the pairs $[\bx,i]$ as hypercubes as we did in \eqref{eq:hypercube} this dual also has a geometric meaning. We explain this in the following example.

\begin{example}
Let $\sigma$ be the Tribonacci substitution defined in \eqref{eq:substribo}. Then by direct computation we see that $E_1^*(\sigma)$ is given by 
\begin{equation*}\label{eq:e1starTribo}
\begin{split}
E_1^*(\sigma)[\mathbf{0},1]&=\{[\mathbf{0},1],[\mathbf{0},2],[\mathbf{0},3]\},\\
E_1^*(\sigma)[\mathbf{0},2]&=\{[(0,0,1)^t,1]\},\\
E_1^*(\sigma)[\mathbf{0},3]&=\{[(0,0,1)^t,2]\}
\end{split}
\end{equation*}
together with the obvious fact that $E_1^*(\sigma)[\bx,i]=M_\sigma^{-1}\bx+E_1^*(\sigma)[\mathbf{0},i]$. One can extend the definition of $E_1^*(\sigma)$ to subsets of $Y\subset \Z^d\times \A$ in a natural way by setting 
\[
E_1^*(\sigma)Y=\bigcup_{[\bx,i]\in Y}E_1^*(\sigma)[\bx,i]. 
\]
Using this extension we can then iterate $E_1^*(\sigma)$. The geometric interpretation of $E_1^{*}(\sigma)^{12}[\mathbf{0},1]$ is depicted in Figure~\ref{fig:RauzyStep}. 
\begin{figure}[hh]
\includegraphics[trim=0 40 0 60,clip,width=0.7\textwidth]{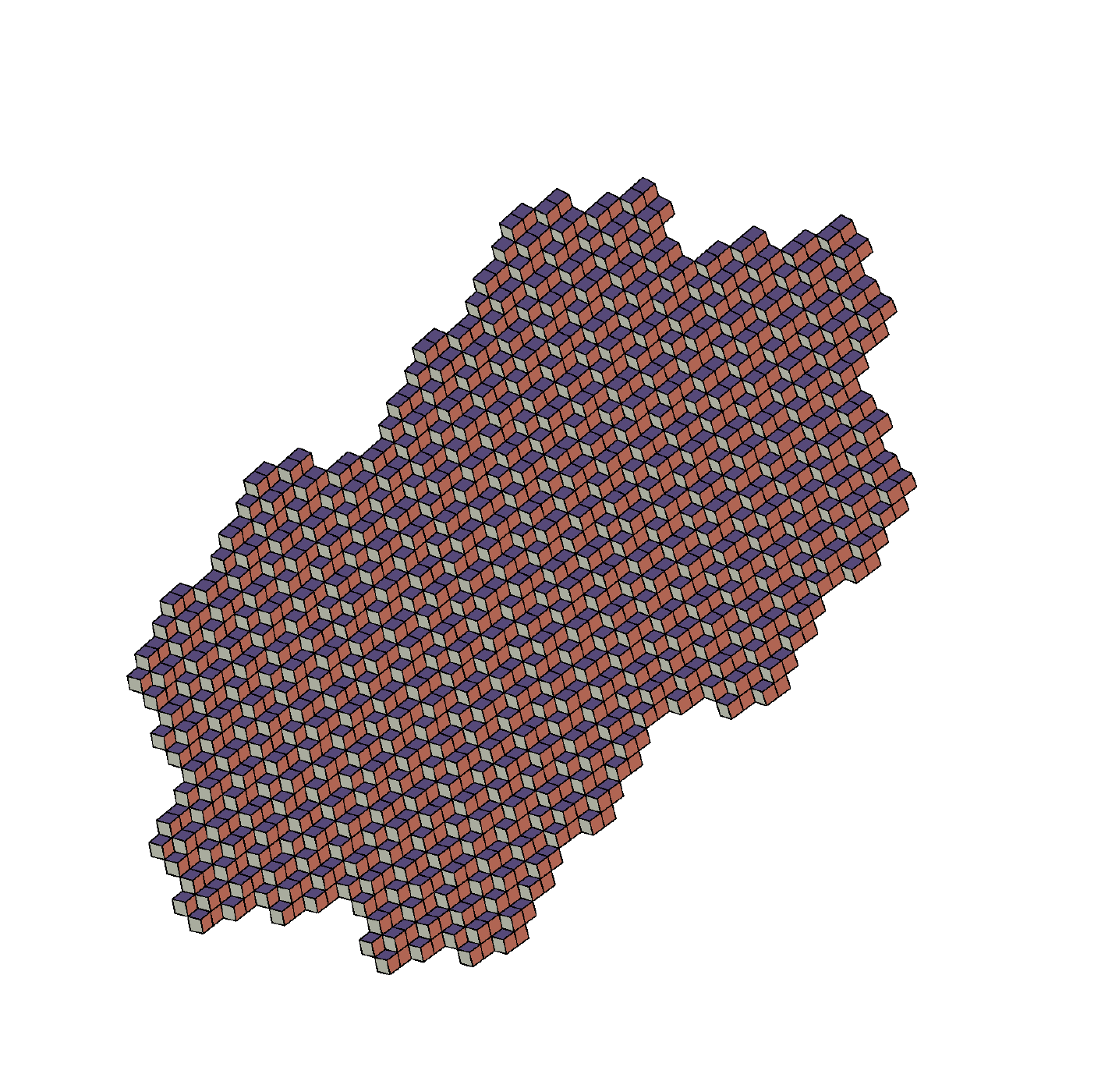}
\caption{An approximation of $\Ra$ using $E_1^*(\sigma)$.
\label{fig:RauzyStep}}
\end{figure}
It is not by accident that this image is a good approximation of (an affine image of) the classical Rauzy fractal corresponding to $\sigma$ depicted in Figure~\ref{fig:tribodomain}. In fact, $E_1^{*}(\sigma)$ can even be used to give an alternative definition of $\Ra$, see for example \cite{Arnoux-Ito:01,CANTBST}.
\end{example}

The dual $E_1^*(\sigma)$ and its higher dimensional generalizations have been investigated thoroughly in connection with the study of substitutive dynamical systems and their Rauzy fractals (see \cite{Arnoux-Ito:01,CANTBST,Ei:03,Ito-Rao:06,LM:17,SAI:01}). We need a result of Fernique~\cite{Fernique:06} that shows how $E_1^*(\sigma)$ behaves with respect to discrete hyperplanes. Before we state it we introduce some notation. Let $\bsigma$ be a sequence of substitutions with generalized right eigenvector $\bu\in\R^d_{>0}$ and let a fixed vector $\bw\in \R_{\ge 0}^d\setminus\{\mathbf{0}\}$ be given (such that the Rauzy fractal $\Ra_\bw$ can be defined). Then, motivated by the projections \eqref{eq:nproj} we needed in the formulation of the set equation we set
\begin{equation}\label{eq:ukwk}
\bu^{(k)} = (M_{[0,k)})^{-1} \bu, \qquad \bw^{(k)} = (M_{[0,k)})^t \bw\qquad\qquad{(k\in\N)}.
\end{equation}

\begin{lemma}\label{lem:fernique}
Let $\bsigma=(\sigma_n)$ be a sequence of unimodular substitutions. Then for all $k< \ell$ the following assertions hold.
\begin{enumerate}
\item[(i)] $M_{[k,\ell)}(\bw^{(\ell)})^\bot=(\bw^{(k)})^\bot$,
\item[(ii)] $E_1^*(\sigma_{[k,\ell)})\Gamma(\bw^{(k)})=\Gamma(\bw^{(\ell)})$,
\item[(iii)] for distinct pairs $[\bx,i],[\bx',i']\in \Gamma(\bw^{(k)})$ the images $E_1^*(\sigma_{[k,\ell)})[\bx,i]$ and $E_1^*(\sigma_{[k,\ell)})[\bx',i']$ are disjoint patches of $\Gamma(\bw^{(\ell)})$.
\end{enumerate} 
\end{lemma}

\begin{proof}
Assertion~(i) is an immediate consequence of the definition of $\bw^{(k)}$, assertions~(ii) and (iii) are the content of \cite[Theorem~1]{Fernique:06}. Their proof is a bit tedious, however, it just uses the definition of discrete hyperplane and checks the required conditions (assertion~(iii) is essentially already contained in \cite[Lemma~3]{Arnoux-Ito:01}).
\end{proof}

Combining Proposition~\ref{prop:seteq} and Lemma~\ref{lem:fernique} we get the following result in which we use the notation
\[
\Co_\bw^{(k)}=\{\pi_{\bu,\bw}\bx + \Ra^{(k)}_\bw(i)\;:\; [\bx,i]\in \Gamma(\bw^{(k)})  \} \qquad(k\in\N)
\]
for the collection of the subtiles associated with the shifted sequence $(\sigma_{n+k})_{n\in\N}$ of $\bsigma$. 

\begin{proposition}\label{prop:seteqII}
Let $\bsigma$ be a primitive and recurrent sequence of unimodular substitutions with generalized right eigenvector $\bu$. Then for each $[\bx,i]\in\Z^d\times\A$ and every $k,\ell\in\N$ with $k <\ell$ we have
\[
\bigcup_{[\bx,i]\in \Gamma(\bw^{(k)})}\pi_{\bu,\bw}\bx + \Ra_\bw^{(k)}(i)=
\bigcup_{[\by,j]\in \Gamma(\bw^{(\ell)})} M_{[k,\ell)}(\pi^{(\ell)}_{\bu,\bw}\by + \Ra_\bw^{(\ell)}(j)).
\]
The collection $M_{[k,\ell)}\Co_\bw^{(\ell)}$ is a refinement of $\Co_\bw^{(k)}$ in the sense that each element of the latter is a finite union of elements of the former.
\end{proposition}

The following lemma shows that the set equation subdivides Rauzy fractals into sets whose diameter eventually tends to zero (see~\cite[Lemma~5.5]{Berthe-Steiner-Thuswaldner}).

\begin{lemma} \label{l:smallsubtiles}
Let $\boldsymbol{\sigma} = (\sigma_n)\in S^{\mathbb{N}}$ be a primitive, algebraically irreducible, and recurrent sequence of unimodular substitutions with balanced language $\mathcal{L}_{\boldsymbol{\sigma}}$, and let $\mathbf{w} \in \mathbb{R}_{\ge 0}^d \setminus \{\mathbf{0}\}$. Then 
\[
\lim_{n\to\infty} M_{[0,n)} \mathcal{R}^{(n)}_\mathbf{w} = \{\mathbf{0}\}.
\]
\end{lemma}

\begin{proof}
As $M_{[0,n)} \pi_{\mathbf{u},\mathbf{w}}^{(n)} = \pi_{\mathbf{u},\mathbf{w}}\, M_{[0,n)}$  and $\pi_{\mathbf{u},\mathbf{w}} = \pi_{\mathbf{u},\mathbf{w}}\, \pi_{\mathbf{u},\mathbf{1}}$, we conclude that $M_{[0,n)} \pi_{\mathbf{u},\mathbf{w}}^{(n)}\, \mathbf{l}(v) = \pi_{\mathbf{u},\mathbf{w}}\, \pi_{\mathbf{u},\mathbf{1}}\, M_{[0,n)}\, \mathbf{l}({v})$ for all $v\in\Lg_\bsigma^{(n)}$. Now, the result follows from Proposition~\ref{prop:strongconv} and the definition of $\mathcal{R}^{(n)}_\mathbf{w}$ in \eqref{eq:RnRauzyfr}.

\end{proof}

We explain the concepts and results of this section in the following example.

\begin{example}\label{ex:ARseteq}
Recall the definition of the Arnoux-Rauzy substitutions $\sigma_1,\sigma_2,\sigma_3$ from \eqref{eq:ARsubs} and consider a sequence 
\[
\bsigma=
(\sigma_1,\sigma_2,\sigma_3, \sigma_1,\sigma_2,\sigma_3,\sigma_1,\sigma_2,\sigma_3,
\dots),
\]
where the dots ``$\dots$'' mean that the sequence is continued in a way that $\bsigma$ is primitive and recurrent (if we start with three blocks of the form $\sigma_1,\sigma_2,\sigma_3$ it turns out that $\Ra_\bw$ is close to the classical Rauzy fractal studied in Example~\ref{ex:rauzy} in Hausdorff metric, which, of course, doesn't say anything about its topological properties or tiling properties; we just did it this way to get nice pictures in Figure~\ref{fig:seteq}).

\begin{figure}[hh]
\includegraphics[trim=0 0 0 0,clip,width=0.30\textwidth]{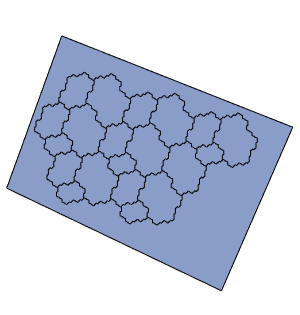}\hskip 3cm
\includegraphics[trim=0 0 0 0,clip,width=0.20\textwidth]{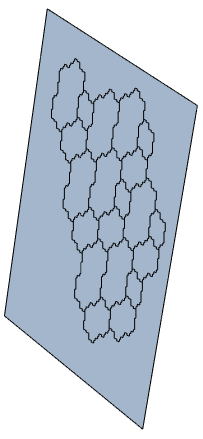}
\\
\mbox{} \hskip 2cm (a) \hskip 6cm (b)
\\
\vskip -0.9cm
\hskip 3cm\includegraphics[trim=0 0 0 0,clip,width=0.30\textwidth]{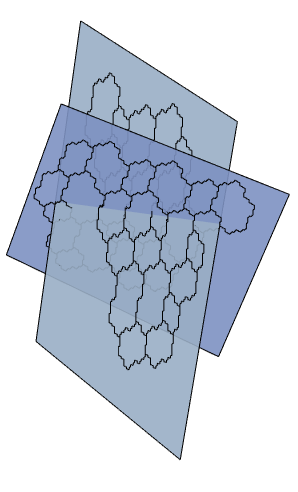}\\
\mbox{} \hskip 5cm (c)
\\
\vskip -1cm
\includegraphics[trim=0 0 0 0,clip,width=0.30\textwidth]{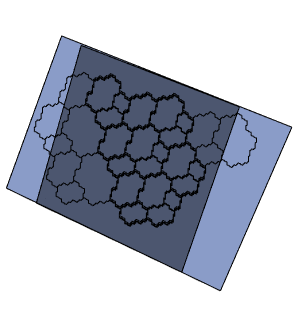}\hskip 2cm
\includegraphics[trim=0 0 0 0,clip,width=0.30\textwidth]{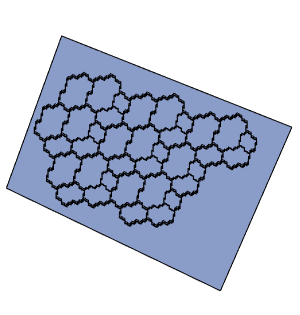}\\
\mbox{} \hskip 2cm (d) \hskip 6cm (e)
\caption{An illustration of the set equation. (a) shows a patch $P_0$ of the collection $\Co_\mathbb{\bw}=\Co^{(0)}_\mathbb{\bw}$, (b) contains a patch $P_1$ of $\Co^{(1)}_\bw$. In (c) $P_0$ and $P_1$ are drawn together to illustrate that they lie in different planes. In (d) the matrix $M_0$ is applied to $P_1$: the image $M_0P_1$ is located in the same plane as $P_0$ and, according to the set equation, forms a subdivision of some tiles of $P_0$. The subdivision of $P_0$ in patches of $M_0\Co^{(1)}_\mathbb{\bw}$ is shown in (e) for the whole patch $P_0$.
\label{fig:seteq}}
\end{figure}

In Figure~\ref{fig:seteq}~(a) we show a patch $P_0$ of the collection $\Co_\bw$ (for some convenient vector $\bw\in\R^3_{\ge 0}\setminus\{\mathbf{0}\}$) of subtiles associated with $\bsigma$, while Figure~\ref{fig:seteq}~(b) shows a patch $P_1$ of the collection $\Co_\bw^{(1)}$ associated with the shifted sequence
\[
\bsigma^{(1)}=
(\sigma_2,\sigma_3, \sigma_1,\sigma_2,\sigma_3,\sigma_1,\sigma_2,\sigma_3,
\dots).
\]
Note that, since $\bw$ and $\bw^{(1)}$ are not collinear, these patches live in two different planes which is illustrated in Figure~\ref{fig:seteq}~(c).

In this setting, the set equation in Proposition~\ref{prop:seteq} says that each element of the collection $\Co_\bw$ can be viewed as the union of elements from $M_0\Co_\bw^{(1)}$. In other words, if we take the patch $P_1$ depicted in Figure~\ref{fig:seteq}~(b) and apply the linear mapping $M_0$ to it, the resulting patch $M_0P_1$ lies in the same plane $\bw^\bot$ as the collection $\Co_\bw$ and some elements of $P_0$ are unions of elements from $M_0P_1$. In Figure~\ref{fig:seteq}~(d) this is illustrated: the image of the patch $P_1$ from Figure~\ref{fig:seteq}~(b) under the mapping $M_0$ is subdividing some parts of $P_0$. Figure~\ref{fig:seteq}~(e) illustrates that, according to Proposition~\ref{prop:seteqII}, each element of $\Co_\bw$ is a union of elements from $M_0\Co_\bw^{(1)}$.

Note that in Figure~\ref{fig:seteq} the collections $\Co_\bw$ and $\Co^{(1)}_\bw$ are depicted as tilings and the patches of $M_0\Co^{(1)}_\bw$ subdivide the elements of $\Co_\bw$ without overlap. This is the situation we ``dream'' of. So far, we only know that elements of  $\Co_\bw$ are unions of elements of $M_0\Co^{(1)}_\bw$. To realize this ideal situation we need to work more.
\end{example}

\subsection{An $S$-adic Rauzy fractal is the closure of its interior}

The present section is devoted to the interior of the subtiles. We start with a covering result taken from~\cite[Proposition~6.2]{Berthe-Steiner-Thuswaldner}. In its statement we use the following terminology. Let $\mathcal{K}$ be a collection of subsets of a set $D$. The \emph{covering degree} of $\mathcal{K}$ (in $D$) is the largest number $m$ having the property that each $x\in D$ is contained in at least $m$ elements of $\mathcal{K}$.

\begin{lemma}\label{lem:covDegInc}
Let $\bsigma$ be a sequence of unimodular substitutions and $\bw\in\R_{\ge 0}\setminus\{\mathbf{0}\}$. If $\bsigma$ is primitive, recurrent, algebraically irreducible, and has finitely balanced language $\Lg_\bsigma$ then $\Co_\bw^{(n)}$ covers $(\bw^{(n)})^\bot$ with finite covering degree for each $n\in\N$. The covering degree of $\Co_\bw^{(n)}$ increases monotonically with $n$.
\end{lemma}

\begin{proof}
We prove the covering property for $\Co_\bw$. The covering property for $\Co_\bw^{(n)}$ as well as the monotonicity of the covering degree follow from this by the set equation in Proposition~\ref{prop:seteq}. 

By Proposition~\ref{prop:seteqII} with the choices $k=0$ and $\ell=n\ge n_0$ we know that
\begin{equation}\label{eq:denseunion}
\bigcup_{T\in \Co_\bw}T= \bigcup_{n\ge n_0}\bigcup_{T\in \Co_\bw^{(n)}}M_{[0,n)}T= 
\bigcup_{n\ge n_0} \bigcup_{[\by,j]\in \Gamma(\bw^{(n)})} M_{[0,n)}(\pi^{(n)}_{\bu,\bw}\by + \Ra_\bw^{(n)}(j))
\end{equation}
holds for each $n_0\in \N$. Because $\Co_\bw$ is a locally finite collection of compact sets it suffices to show that $\bigcup_{T\in \Co_\bw}T$ is dense in $\bw^{\bot}$. To prove this we show that
the right hand side of \eqref{eq:denseunion} is dense in $\bw^{\bot}$ for each $n_0\in \N$. To see this note that by the definition of the discrete hyperplane $\Gamma(\bw^{(n)})$ the set of translates in this union satisfies (recall from \eqref{eq:ukwk} that $\bw^{(n)} = (M_{[0,n)})^t\bw$)
\[
\begin{split}
\{M_{[0,n)} \pi^{(n)}_{\bu,\bw}\by\;:\;[\by,j]\in & \Gamma(\bw^{(n)})\} \\
=&
\{
\pi_{\bu,\bw}M_{[0,n)}\by \;:\; \by \in\Z^d,\, 0\le \langle \by,(M_{[0,n)})^t\bw \rangle \le \max_{i\in\A}\langle \be_i,(M_{[0,n)})^t\bw\rangle\} \\
=&
\{
\pi_{\bu,\bw}\bz \;:\; \bz\in\Z^d,\, 0\le \langle\bz,\bw \rangle \le \max_{i\in\A}\langle M_{[0,n)}\be_i , \bw\rangle
\}.
\end{split}
\]
As $\bu$ has rationally independent coordinates by Lemma~\ref{irrirr}, the set 
\[
\{\pi_{\bu,\bw}\bz \;:\; \bz\in\Z^d,\, 0\le \langle\bz,\bw \rangle\}
\]
is dense in $\bw^\bot$. Since $\max_{i\in\A}\langle M_{[0,n)}\be_i, \bw \rangle\to\infty$ for $n\to\infty$ by primitivity, this yields that
\begin{equation}\label{eq:dd2}
\bigcup_{n\ge n_0} \bigcup_{[\by,j]\in \Gamma(\bw^{(n)})} M_{[0,n)} \pi^{(n)}_{\bu,\bw}\by
=\{\pi_{\bu,\bw}\bz \;:\; \bz\in\Z^d,\, 0\le \langle\bz,\bw \rangle\}.
\end{equation}
is dense in $\bw^{\bot}$ for each $n_0\in \N$. Because $n_0$ was arbitrary and $\lim_{n\to \infty} M_{[0,n)}\Ra_\bw^{(n)}(i)=\{\mathbf{0}\}$ by Lemma~\ref{l:smallsubtiles} this implies that the right hand side of \eqref{eq:denseunion} is dense in $\bw^{\bot}$ for each $n_0\in \N$ and we are done.
\end{proof}

From this result we get the assertion on the interiors of $S$-adic Rauzy fractals.

\begin{proposition}\label{prop:intclos}
Let $\bsigma$ be a sequence of unimodular substitutions over the alphabet $\A$ and $\bw\in\R_{\ge 0}\setminus\{\mathbf{0}\}$. If $\bsigma$ is primitive, recurrent, algebraically irreducible, and has finitely balanced language $\Lg_\bsigma$,  then $\Ra(i)$ is the closure of its interior for each $i\in\A$.
\end{proposition}

\begin{proof}
Choose some $\bw\in\R_{\ge 0}\setminus\{\mathbf{0}\}$. By Lemma~\ref{lem:covDegInc} the collection $\Co_\bw^{(n)}$ is a locally finite covering of $(\bw^{(n)})^\bot$ by compact sets for each $n\in\N$. Thus by Baire's theorem for each $n\in \N$ there is $i_n\in \A$ such that ${\rm int}(\Ra^{(n)}_\bw(i_n))\not=\emptyset$. By primitivity of $\bsigma$ the set equation in Proposition~\ref{prop:seteq} implies that each $\Ra^{(n)}_\bw(i)$ contains $\Ra^{(k)}_\bw(i_k)$ for some $k>n$. Thus for each $n\in\N$ and each $i\in\A$ we have ${\rm int}(\Ra^{(n)}_\bw(i))\not=\emptyset$.

For each $i\in \A$ and each $n\in \N$, Proposition~\ref{prop:seteq} yields a subdivision of $\Ra_\bw(i)$ in translates of sets of the form $M_{[0,n)}\Ra_\bw^{(n)}(j)$, $j\in\A$. The diameters of these sets tend to $0$ by Lemma~\ref{l:smallsubtiles}. Since they all contain inner points, the set of inner points of $\Ra_\bw(i)$ is dense in $\Ra_\bw(i)$. In other words,  $\Ra_\bw(i)$ is the closure of its interior. The result now follows by taking $\bw=\bone$.
\end{proof}

\subsection{The generalized left eigenvector}\label{sec:genEV}

Let $\bsigma$ be a primitive and recurrent sequence of unimodular substitutions. 
If we look at the set equation in Proposition~\ref{prop:seteq} for $k=0$ and $\ell=n$ we see that it subdivides the sets $\Ra_\bw(i)$, $i\in\A$, into translates of sets of the form $\Ra_\bw^{(n)}(j)$, $j\in \A$. In the well-studied substitutive case $\Ra_\bw^{(n)}(i)=\Ra_\bw(i)$ holds for each $n$, {\it i.e.}, the sets $\Ra_\bw(i)$ are subdivided into small copies of themselves. This fact is crucial in most of the proofs of properties of substitutive Rauzy fractals (see {\it e.g.}~\cite{SiegelThuswaldner10}). In our case, in general the sets $\Ra_\bw^{(n)}$ are not only different for each $n\in\N$, but also live in different hyperplanes $(\bw^{(n)})^\bot$ of $\mathbb{R}^n$.

In what follows we want to deal with this problem by choosing a strictly increasing sequence $(n_k)$ of integers such that $\Ra_\bw^{(n_k)}(i)$ is at least getting closer and closer to  $\Ra_\bw(i)$ in Hausdorff metric when $k\to\infty$. 

To this matter let $\bsigma$ be a sequence of substitutions that satisfies the assumptions of Theorem~\ref{th:RauzyProperties}. We now successively choose subsequences of the integers to get the desired properties.

\begin{enumerate}
\item[(a)] Consider the set equation in Proposition~\ref{prop:seteq} for the choices $k=0$, $\ell=m$ and  $k=n$, $\ell=n+m$. Look at the subdivision of $\Ra_\bw(i)$ and $\Ra_\bw^{(n)}(i)$. We can hope to get $\Ra_\bw(i)$ and $\Ra_\bw^{(n)}(i)$ close to each other in Hausdorff metric if they have the same subdivision structure. From Proposition~\ref{prop:seteq}  we see that these subdivision structures are the same if 
$(\sigma_0,\ldots,\sigma_{m-1}) = (\sigma_n,\ldots,\sigma_{n+m-1})$.
Since $\bsigma$ is recurrent, there exist strictly increasing sequences $(n_k)$ and $(\ell_k)$ such that 
\begin{equation}\label{eq:cond1}
(\sigma_0,\ldots,\sigma_{\ell_k-1}) = (\sigma_{n_k},\ldots,\sigma_{n_k+\ell_k-1}).
\end{equation}
By recurrence and primitivity it is possible to choose $(n_k)$ and $(\ell_k)$ in a way that there is some $h$ such that $M_{[\ell_k-h,\ell_k)}$ is the same primitive matrix for all $k\in\N$.

\item[(b)] We know that $M_{[0,\ell_k)}\Ra_\bw^{(\ell_k)}(j)$ tends to $\{\mathbf{0}\}$ in Hausdorff metric for $k\to\infty$ by Lemma~\ref{l:smallsubtiles} so that the subdivision corresponding to the choice $k=0$, $\ell=\ell_k$ in the set equation gives a subdivision of $\Ra_\bw(j)$ into sets whose diameter tends to $0$ for $k\to\infty$. However, if we consider $\Ra_\bw^{(n_k)}(i)$, there is no reason for $M_{[n_k,n_k+\ell_k)}\Ra_\bw^{(n_k+\ell_k)}(j)=M_{[0,\ell_k)}\Ra_\bw^{(n_k+\ell_k)}(j)$ to tend to $\{\mathbf{0}\}$ unless $\Ra_\bw^{(n_k+\ell_k)}(j)$ is bounded uniformly in $k$. To this end we need to assume that $\Lg_{\bsigma}^{(n_k+\ell_k)}$ is $C$-balanced as this implies that $\Ra_\bw^{(n_k+\ell_k)}(j)$ is indeed bounded by Proposition~\ref{prop:Rcompact}. In view of the conditions imposed on $\bsigma$ in Theorem~\ref{th:RauzyProperties} it is, however, possible to change the sequence $(n_k)$ and $(\ell_k)$ chosen in (a) in a way that also $\Lg_{\bsigma}^{(n_k+\ell_k)}$ is $C$-balanced for $C\in\N$ not depending on $k$.

\item[(c)] Still (a) and (b) give us no reason for $\Ra_\bw^{(n_k)}$ living in a hyperplane $\bw^{(n_k)}$ close to $\bw$ which is needed in order to get $\Ra_\bw^{(n_k)}$ close to $\Ra_\bw$ in Hausdorff metric.  By the compactness of the space of directions in $\R^d$, using the Hilbert metric from Proposition~\ref{prop:furstmatrix} it is possible to exhibit a vector $\bv\in\R_{\ge 0}\setminus\{\mathbf{0}\}$ for which there exists subsequences of $(n_k)$ and $(\ell_k)$ (called $(n_k)$ and $(\ell_k)$ again) such that $\lim_{k\to\infty}\bv^{(n_k)}/\Vert\bv^{(n_k)}\Vert_1=\bv/\Vert\bv\Vert_1$. Here we set $\bv^{(n)}=(M_{[0,n)})^t \bv$.
\end{enumerate}

Summing up, if the conditions of Theorem~\ref{th:RauzyProperties} are in force we can choose sequences $(n_k)$ and $(\ell_k)$ satisfying (a), (b), and (c). The vector $\bv$ defined in (c) deserves special attention.

\begin{definition}[Generalized left eigenvector]\index{generalized left eigenvector}
A vector $\bv$ as in (c) is called a \emph{generalized left eigenvector} of $\bsigma$.
\end{definition}

Sequences $(n_k)$ and $(\ell_k)$ associated with $\bsigma$ in the above way will just be called \emph{associated sequences} for $\bsigma$ in the sequel (they are related to the property \emph{PRICE} of \cite[Definition~5.8]{Berthe-Steiner-Thuswaldner}).
It turns out that associated sequences are suitable for our purposes. In particular, we get the following result (we refer to \cite[Proposition~5.12]{Berthe-Steiner-Thuswaldner} for details).

\begin{proposition}\label{prop:RauzyHausdorff}
Let $\bsigma$ be a sequence of substitutions that admits associated sequences $(n_k)$ and $(\ell_k)$ and has a generalized left eigenvector $\bv$. Then for each 
$i\in\A$
\[
\lim_{k\to\infty}\Ra_\bv^{(n_k)}(i)=\Ra_\bv(i)
\]
in Hausdorff metric.
\end{proposition}

\begin{proof}[Sketch]
By \eqref{eq:cond1} in (a) the sets $\Ra_\bv(i)$ and $\Ra^{(n_k)}_\bv(i)$ have the same subdivision structure governed by $E_1^{*}(\sigma_{[0,\ell_k)})$ for $k\in\N$. More precisely,
\begin{equation}\label{eq:comp0nk}
\begin{split}
\Ra_\bv(i) &= \bigcup_{[\by,j]\in E_1^{*}(\sigma_{[0,\ell_k)})[\mathbf{0},i]}
M_{[0,\ell_k)}(\pi^{(\ell_k)}_{\bu,\bv}\by + \Ra_\bv^{(\ell_k)}(j)),\\
\Ra^{(n_k)}_\bv(i) &= \bigcup_{[\by,j]\in E_1^{*}(\sigma_{[0,\ell_k)})[\mathbf{0},i]}
M_{[0,\ell_k)}(\pi^{(n_k+\ell_k)}_{\bu,\bv}\by + \Ra_\bv^{(n_k+\ell_k)}(j)).
\end{split}
\end{equation}
By Proposition~\ref{prop:strongconv} the sets $M_{[0,\ell_k)}\Ra_\bv^{(\ell_k)}(j)$ tend to $\{\mathbf{0}\}$ in Hausdorff metric for $k\to\infty$. With more effort, using the balance conditions of (b) and the convergence properties of (c), one can also show that the sets  $M_{[0,\ell_k)}\Ra_\bv^{(n_k+\ell_k)}(j)$ tend to $\{\mathbf{0}\}$ in Hausdorff metric for $k\to\infty$. So replacing all these sets by $\{\mathbf{0}\}$ on the right hand side of \eqref{eq:comp0nk} changes the sets on the left hand side of \eqref{eq:comp0nk} only very little in Hausdorff metric for large $k\in\N$.
Thus for large $k\in\N$ the Hausdorff distance between $\Ra_\bv^{(n_k)}(i)$ and $\Ra_\bv(i)$ is (up to an error tending to $0$ for $k\to\infty$) bounded by 
\[
\begin{split}
&\max\left\{
\Vert M_{[0,\ell_k)}(\pi^{(\ell_k)}_{\bu,\bv}\by-\pi^{(n_k+\ell_k)}_{\bu,\bv}
\by\Vert_\infty \;:\;[\by,j]\in
E_1^{*}(\sigma_{[0,\ell_k)})[\mathbf{0},i]
\right\}
\\
&\hskip 3cm =\max\left\{
\Vert\pi_{\bu,\bv}M_{[0,\ell_k)}\by-\pi^{(n_k)}_{\bu,\bv}M_{[0,\ell_k)}\by\Vert_\infty \;:\; [\by,j]\in E_1^{*}(\sigma_{[0,\ell_k)})[\mathbf{0},i] 
\right\}.
\end{split}
\]
One can now show that the latter maximum tends to $0$ for $k\to\infty$.  Here one uses that by the definition of the generalized left eigenvector in (c) the hyperplanes $(\bv^{(n_k)})^{\bot}$ converge to $\bv^{\bot}$.
\end{proof}

\subsection{An $S$-adic Rauzy fractal has a boundary of measure zero}

We now turn to the boundary of an $S$-adic Rauzy fractal. We start with a result on level $\ell$ subtiles contained in the interior of a given subtile whose detailed proof is contained in \cite[Lemma~6.6]{Berthe-Steiner-Thuswaldner}.

\begin{lemma}\label{lem:intsubtile}
Let $\bsigma$ be a sequence of unimodular substitutions that satisfies the properties of Theorem~\ref{th:RauzyProperties} and let associated sequences $(n_k)$, $(\ell_k)$, and a generalized  left eigenvector $\bv$ be given. 

Then there is $\ell\in\N$ such that for each $i,j\in\A$ there is $[\by,j]\in E_1^*(\sigma_{[0,\ell)})[\mathbf{0},i]$ such that 
\begin{enumerate}
\item[(i)] $M_{[0,\ell)}(\pi^{(\ell)}_{\bu,\bv}\by + \Ra_\bv^{(\ell)}(j)) \subset {\rm int}(\Ra_\bv(i) )$,
\item[(ii)] $M_{[0,\ell)}(\pi^{(n_k+\ell)}_{\bu,\bv}\by + \Ra_\bv^{(n_k+\ell)}(j)) \subset {\rm int}(\Ra^{(n_k)}_\bv(i) )$ for each sufficiently large $k\in\N$.
\end{enumerate}
Moreover, the covering degree of $\Co_\bv^{(n)}$ does not depend on $n$. 
\end{lemma}

\begin{proof}[Sketch]
Since the conditions in the lemma imply that ${\rm int}(\Ra_\bv(i) ) \not=\emptyset$ (see Proposition~\ref{prop:intclos}) and that the diameter of $M_{[0,\ell)}\Ra_\bv^{(\ell)}(j)$ becomes arbitrarily small for large $\ell$ (see Lemma~\ref{l:smallsubtiles}), assertion~(i) follows easily from primitivity.

The fact that $\ell$ and $\by$ can be chosen in a way that (i) and (ii) hold simultaneously is more difficult to prove. By Proposition~\ref{prop:RauzyHausdorff} we get that $\Ra^{(n_k)}_\bv(i) \to \Ra_\bv(i)$ in Hausdorff metric. Moreover, $\Ra_\bv(i)$ and $\Ra^{(n_k)}_\bv(i)$ have the same subdivision structure for $\ell_k$ steps. This implies that the ``inner structure'' of these tiles is similar for large $k$. However, as inner points are not respected by the Hausdorff metric, technical difficulties occur and also the ``outer structure'', {\it i.e.}, the structure of the collections $\Co_\bv$ and $\Co^{(n_k)}_\bv$ has to be exploited. One can show that if a patch $P$ occurs in a discrete hyperplane $\Gamma(\bw)$, then translates of $P$ occur relatively densely in each discrete hyperplane $\Gamma(\tilde\bw)$ provided that  $\Vert\bw-\tilde\bw\Vert_\infty$ is small enough. In particular, containment of translates of a patch $P$ is an open property of discrete hyperplanes. Thus, if $k$ is large then at level $n_k$ there is a translation $\by_k\in\Z^d$ such that the sets $\Gamma(\bv)$ and $\Gamma(\bv^{(n_k)})-\by_k$ have a large patch around the origin in common. 

Summing up, this means that the collections $\Co^{(n_k)}_\bv-\pi^{(n_k)}_{\bu,\by}\by_k$ converge\footnote{In \cite{Radin-Wolff:92} a space of tilings is equipped with a topology by saying that two tilings are close to each other if their tiles are close to each other in Hausdorff metric inside a large ball around the origin. Although $\Co_\bv$ and $\Co^{(n_k)}_\bv$ are no tilings, an analogous topology can be used here: $\Co_\bv$ and $\Co^{(n_k)}_\bv-\pi^{(n_k)}_{\bu,\by}\by_k$ are said to be close to each other if $\Gamma(\bv)$ and $\Gamma(\bv^{(n_k)})-\by_k$ coincide inside a large ball $B$ around the origin and the tiles associated to an element of $[\by,i]\in \Gamma(\bv) \cap B$ in each of these two collections are close to each other in Hausdorff metric.} 
to $\Co_\bv$ for $k\to\infty$. This implies that the covering degree of $\Co^{(n_k)}_\bv-\pi^{(n_k)}_{\bu,\by}\by_k$ is less than or equal to the covering degree of $\Co_\bv=\Co_\bv^{(0)}$ for $k$ large enough. Since the covering degree of $\Co^{(n)}_\bv$ is monotonically increasing in $n$ by Lemma~\ref{lem:covDegInc}, the last assertion of the lemma follows. 

The fact that the collections $\Co^{(n_k)}_\bv-\pi^{(n_k)}_{\bu,\by}\by_k$ converge to $\Co_\bv$ and have the same covering degree can now be used to show that inner points of elements of $\Co_\bv$ are close to inner points of elements of $\Co^{(n_k)}_\bv-\pi^{(n_k)}_{\bu,\by}\by_k$ for large $k$. Using this together with the fact that $\Ra_\bv(i)$ and $\Ra^{(n_k)}_\bv(i)$ have the same subdivision structure for $\ell_k$ steps, one can show that (i) and (ii) holds simultaneously as claimed. 
\end{proof}

After these preparations we can also prove the result on the measure of the boundary of $S$-adic Rauzy fractals announced in Theorem~\ref{th:RauzyProperties}.

\begin{proposition}\label{prop:measurezero}
Let $\bsigma$ be a sequence of unimodular substitutions over the alphabet $\A$ that satisfies the assertions of Theorem~\ref{th:RauzyProperties}. Then $\lambda_{\bone}(\partial\Ra(i))=0$ for each $i\in\A$.
\end{proposition}

\begin{proof}[Sketch]
Choose $\ell\in\N$ and the sequences $(n_k)$ and $(\ell_k)$ as in Lemma~\ref{lem:intsubtile} and consider $\Ra_\bv(i)$ for some $i\in\A$ (see Figure~\ref{fig:measure0}~(a)), where $\bv$ is a generalized left eigenvector of $\bsigma$. Then subdivide $\Ra_\bv(i)$ into its level $\ell$ subtiles as shown in Figure~\ref{fig:measure0}~(b). According to Lemma~\ref{lem:intsubtile}~(i) there is at least one level $\ell$ subtile $M_{[0,\ell)}(\pi^{(\ell)}_{\bu,\bv}\by+\Ra_\bv^{(\ell)}(j))$ which is a subset of ${\rm int}(\Ra_\bv(i))$; this is indicated with a black boundary in Figure~\ref{fig:measure0}~(b). Letting $m_{ij}=\lambda_{\bv}(M_{[0,\ell)}\Ra_\bv^{(\ell)}(j))/\lambda_\bv(\Ra_\bv(i))$ and $m=\min\{m_{ij}\,:\, i,j\in\A\}$  we therefore gain 
\[
\lambda_{\bv} (\partial \Ra_\bv(i)) = \lambda_{\bv} (\Ra_\bv(i)\setminus {\rm int}(\Ra_\bv(i))) \le (1-m)\lambda_\bv(\Ra_\bv(i)).
\]
Now we subdivide all level $\ell$ subtiles of $\Ra_\bv(i)$ apart from $M_{[0,\ell)}(\pi^{(\ell)}_{\bu,\bv}\by+\Ra_\bv^{(\ell)}(j))$ in level $n_{k}$ subtiles where $k$ is chosen in a way that $n_{k}\ge \ell$. This is illustrated in Figure~\ref{fig:measure0}~(c).

\begin{figure}
\includegraphics[angle=90, trim=0 0 0 0,clip,width=0.40\textwidth]{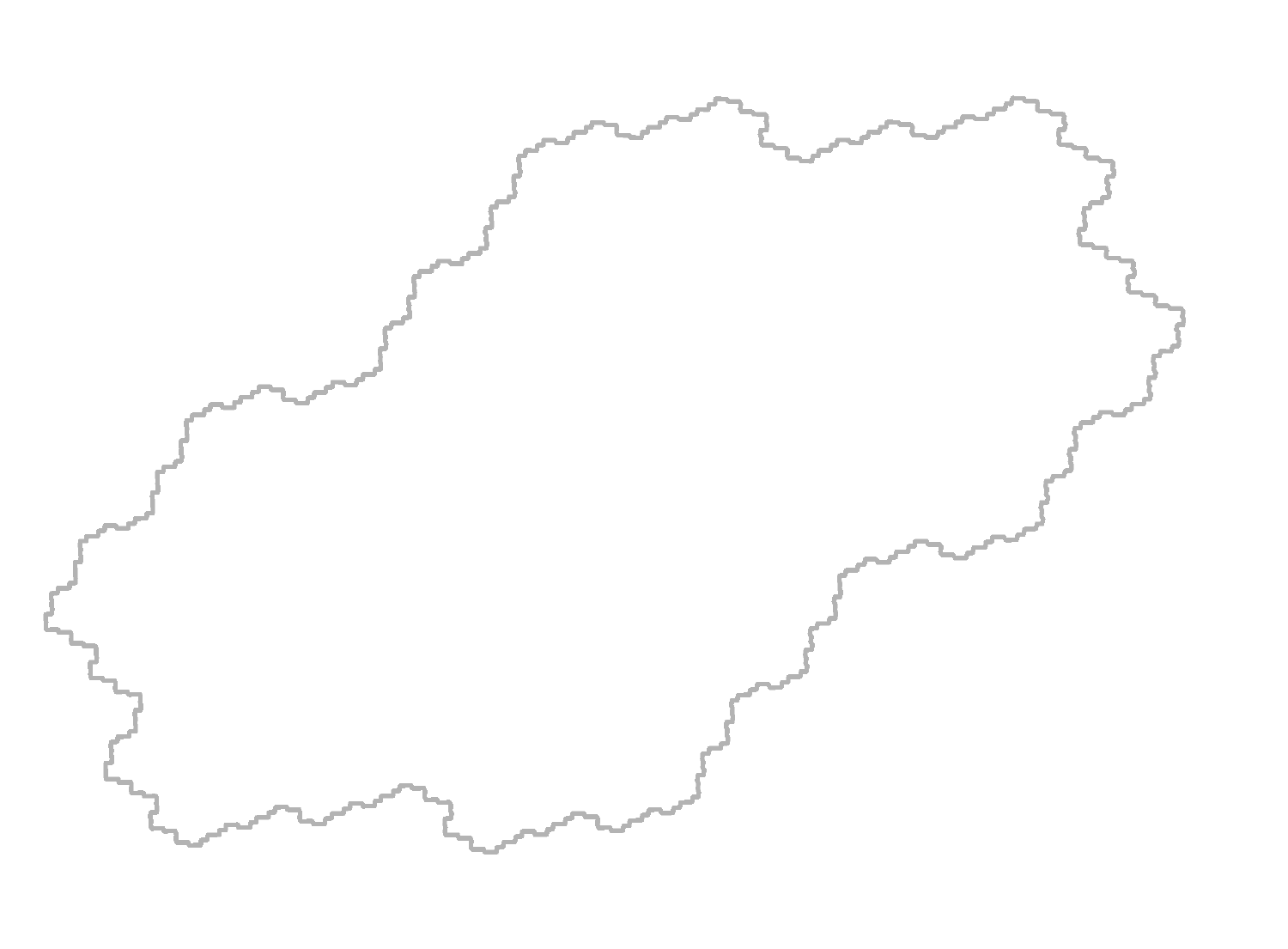}
\includegraphics[angle=90, trim=0 0 0 0,clip,width=0.40\textwidth]{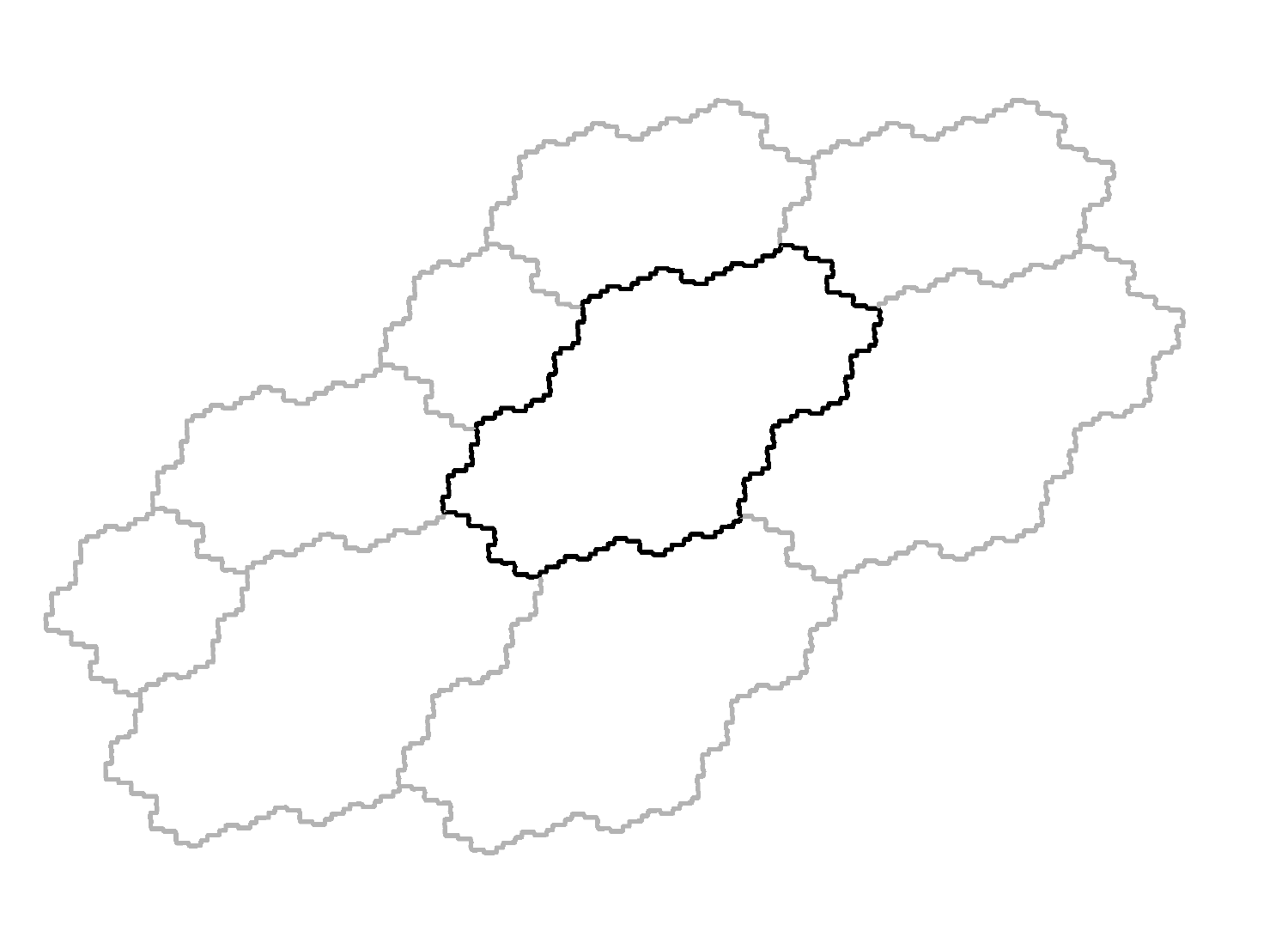}
\\
\mbox{} \hskip 2.5cm (a) \hskip 5.5cm (b)
\\
\includegraphics[angle=90, trim=0 0 0 0,clip,width=0.40\textwidth]{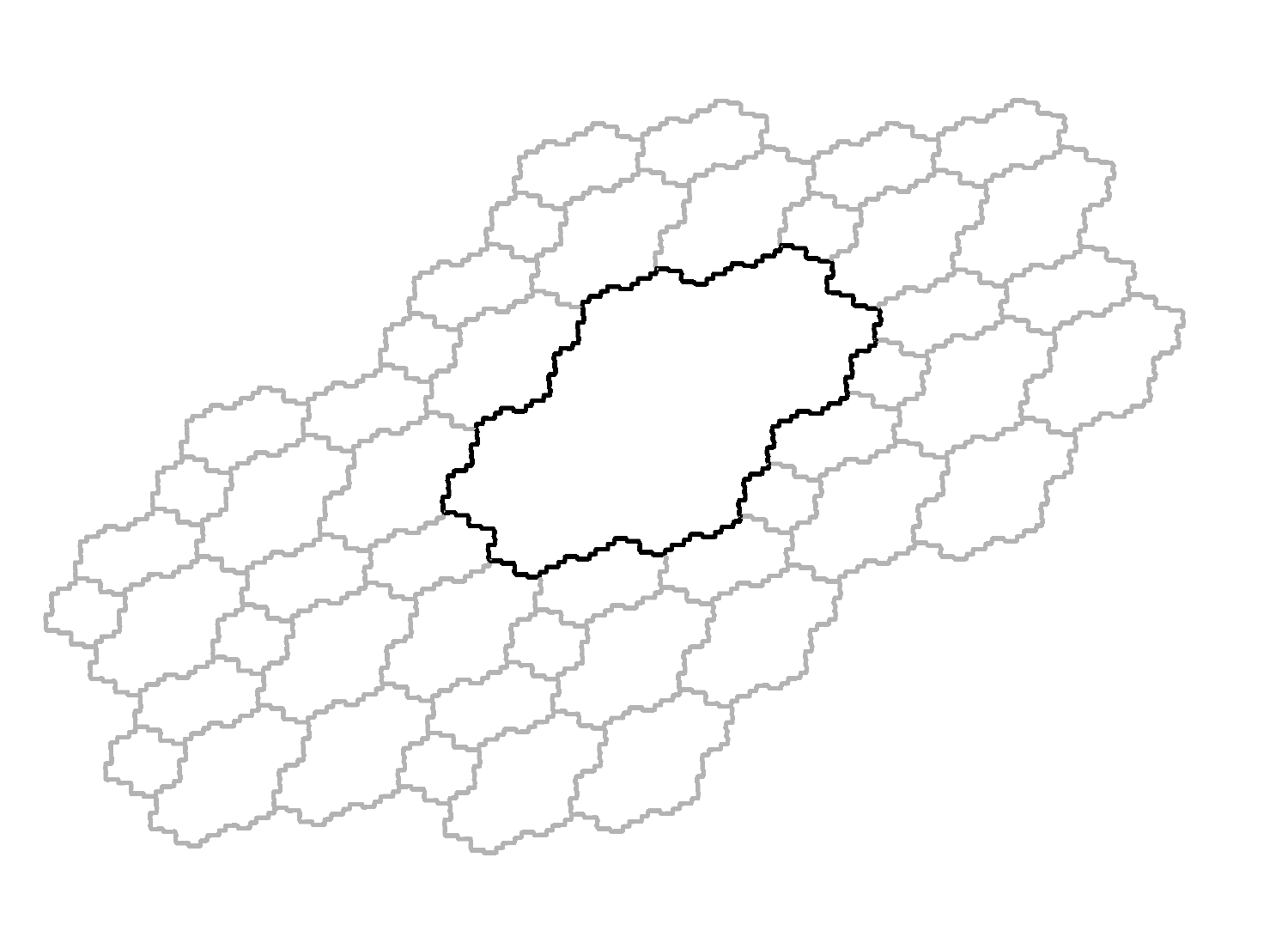}
\includegraphics[angle=90, trim=0 0 0 0,clip,width=0.40\textwidth]{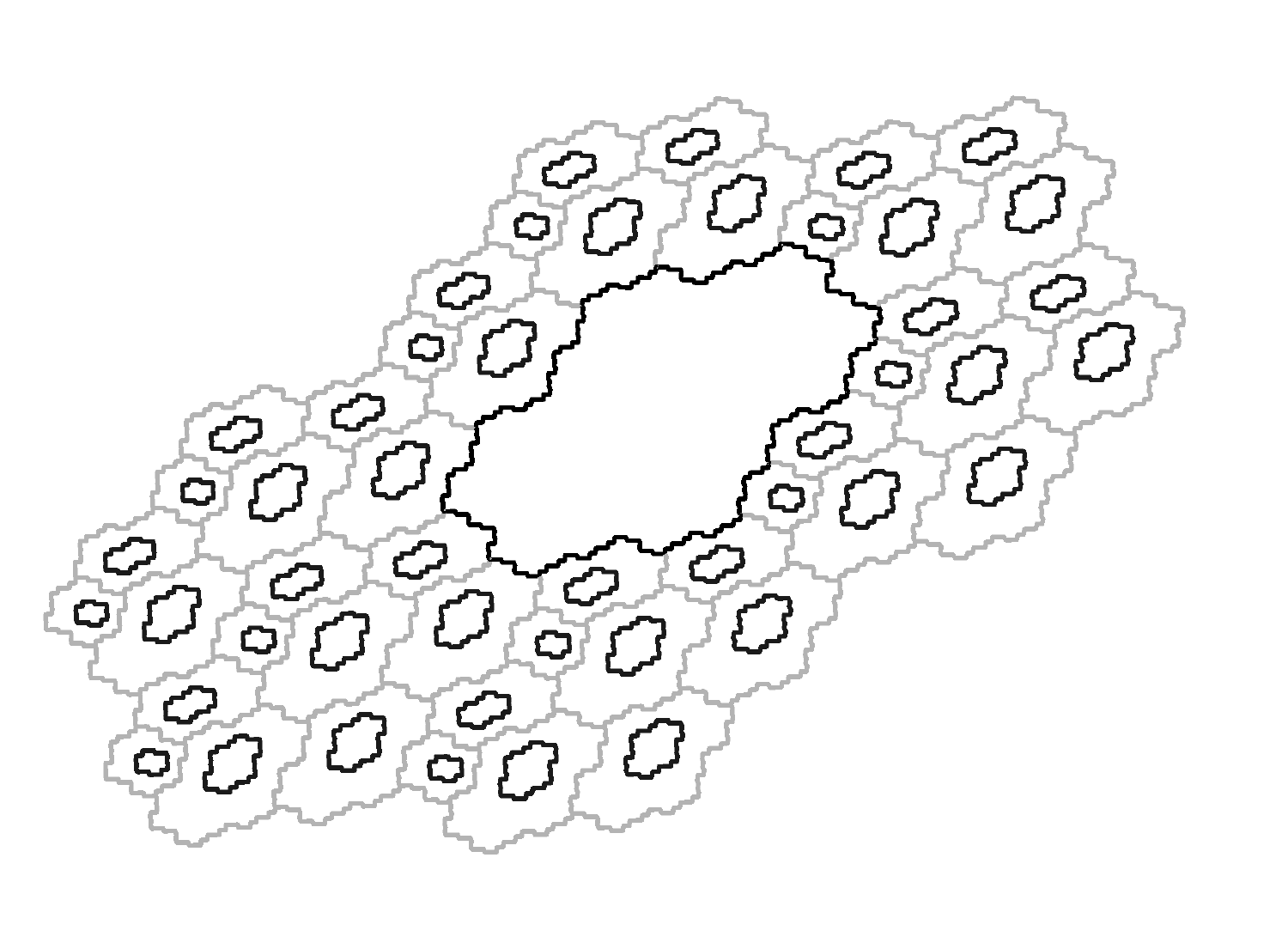}\\
\mbox{} \hskip 2.5cm (c) \hskip 5.5cm (d)
\caption{Illustration of the proof of Proposition~\ref{prop:measurezero}. In (a) a subtile $\Ra_\bv(i)$, $i\in \A$, is shown. In (b) we see the $\ell$-th subdivision of $\Ra_\bv(i)$. The level $\ell$ subtile contained in ${\rm int}(\Ra_\bv(i))$ has black boundary. In (c) all other level $\ell$ subtiles are further subdivided in level $n_k$ subtiles. Each of them contains a level $n_k+\ell$ subtile in its interior. These level $n_k+\ell$ subtiles, which \emph{a fortiori} are also contained in ${\rm int}(\Ra_\bv(i))$, are depicted in (d) also with black boundary. 
\label{fig:measure0}}
\end{figure}

We iterate this procedure: each level $n_{k}$ subtile $R_{n_k}$ we got in this way is subdivided in level $n_{k}+\ell$ subtiles. By Lemma~\ref{lem:intsubtile}~(ii) one of these level $n_{k}+\ell$ subtiles lies in the interior of $R_{n_k}$ (see Figure~\ref{fig:measure0} (d) for an illustration of this) and, {\it a fortiori}, in the interior of $\Ra_\bv(i).$ If we set 
\[
m_{ij}^{(n_k)}=\frac{\lambda_{\bv}(M_{[0,n_k+\ell)}\Ra_\bv^{(n_k+\ell)}(j))}{\lambda_\bv(M_{[0,n_k)}\Ra_\bv^{(n_k)}(i))} 
=\frac{\lambda_{\bv}(M_{[n_k,n_k+\ell)}\Ra_\bv^{(n_k+\ell)}(j))}{\lambda_\bv(\Ra_\bv^{(n_k)}(i))}
= \frac{\lambda_{\bv}(M_{[0,\ell)}\Ra_\bv^{(n_k+\ell)}(j))}{\lambda_\bv(\Ra_\bv^{(n_k)}(i))}
\]
(note that the last equation follows from recurrence of $\bsigma$ if $k$ is chosen large enough) and $m^{(n_k)}=\min\{m^{(n_k)}_{ij}\,:\, i,j\in\A\}$ we obtain 
\[
\lambda_{\bv} (\partial \Ra_\bv(i)) \le (1-m)(1-m^{(n_k)})\lambda_\bv(\Ra_\bv(i)).
\]
Iterating this further we get for some infinite set $K\subset \N$ that
\begin{equation}\label{eq:m0_1}
\lambda_{\bv} (\partial \Ra_\bv(i)) \le (1-m)\prod_{k\in K}(1-m^{(n_k)})\lambda_\bv(\Ra_\bv(i)).
\end{equation}
One can show that $m^{(n_k)}$ is uniformly bounded away from $0$. To this end one needs Proposition~\ref{prop:RauzyHausdorff}  and the fact that $\ell_k$ is chosen in a way that there is some $h$ such that $M_{[\ell_k-h,\ell_k)}$ is the same primitive matrix for all $k\in\N$ (see (a) in Section~\ref{sec:genEV}).
Now \eqref{eq:m0_1} yields $\lambda_{\bv} (\partial \Ra_\bv(i))=0$  and, hence, $\lambda_{\bone} (\partial \Ra(i))=0$.
\end{proof}

Proposition~\ref{prop:intclos} and Proposition~\ref{prop:measurezero} imply Theorem~\ref{th:RauzyProperties}.

\section{Tilings, coincidence conditions, and combinatorial issues}

We now turn to tiling conditions of Rauzy fractals. Already in the substitutive case combinatorial conditions like the \emph{strong coincidence condition} (see {\it e.g.}~\cite{Arnoux-Ito:01}) or the \emph{super coincidence condition} and its variants ({\it cf.}~\cite{Barge-Kwapisz:06,CANTBST,Ito-Rao:06}) have to be imposed in order to gain all the tiling results on Rauzy fractals required for our purposes. Here we discuss an $S$-adic version of these concepts and establish a variety of tiling results. For detailed proofs we refer again to Berth\'e, Steiner, and Thuswaldner~\cite{Berthe-Steiner-Thuswaldner}. As before, our aim is to discuss the main ideas and to make these ideas understandable without going into all the technical details.

\subsection{Multiple tiling and inner subdivision of the subtiles}

In this section we prove tiling properties of Rauzy fractals that hold without further combinatorial conditions. Our first result contains a multiple tiling property of the collections of Rauzy fractals $\Co_\bv$ defined in \eqref{eq:cVcoll}.

\begin{proposition}\label{prop:multiTiling}
Let $\bsigma=(\sigma_n)$ be a primitive and algebraically irreducible sequence of unimodular substitutions. Assume that there is $C>0$ such that for every $\ell\in\N$ there exists $n\ge 1$ such that $(\sigma_n,\ldots,\sigma_{n+\ell-1})=(\sigma_0,\ldots,\sigma_{\ell-1})$ and the language $\Lg_\bsigma^{(n+\ell)}$ is $C$-balanced.

If $\bv$ is a generalized left eigenvector of $\bsigma$ then the collection $\Co_\bv$ forms a multiple tiling of the hyperplane $\bv^{\bot}$.
\end{proposition}

\begin{proof}[Sketch]
Let $(\ell_k)$ and $(n_k)$ be associated sequences for $\bsigma$. We subdivide the proof in seven observations. In the sequel $B_X(\bx,\varepsilon)$ denotes an open ball in a metric space $X$ centered at $\bx$ with radius $\varepsilon$. 

\begin{enumerate}
\item[(i)] 
Let $\bw\in \R_{\ge0}^d\setminus\{\mathbf{0}\}$.
As mentioned in the proof of Lemma~\ref{lem:intsubtile} one can show that each patch $P\subset \Gamma(\bw)$ is \emph{repetitive} in the following sense: there exists $\delta_P>0$ and a radius $r_P>0$ such that for each $\tilde \bw\in \R_{\ge0}^d\setminus\{\mathbf{0}\}$ with $\Vert\tilde \bw- \bw\Vert_\infty<\delta_P$ and each  $\bz$ with $[\bz,i]\in\Gamma(\tilde\bw)$ a translate of $P$ occurs in $\Gamma(\tilde\bw)\cap B_{\R^n}(\bz,r_P)$. This means that each patch occurring in a discrete hyperplane $\mathcal{D}$ occurs uniformly repetitively in each hyperplane $\mathcal{D}'$ which is close enough to $\mathcal{D}$. This general property of discrete hyperplanes is proved in \cite[Lemma~6.5]{Berthe-Steiner-Thuswaldner}.

\item[(ii)] Let $m$ be the covering degree of $\Co_\bv$. Then each point $\bx \in \bv^\bot$ which is covered exactly $m$ times by elements of $\Co_\bv$ is not contained in the boundary of any element of $\Co_\bv$. Suppose this was wrong and let $R_1, \ldots, R_m \in \Co_\bv$ be the elements containing $\bx$. Since $\Co_\bv$ is a locally finite union of compact sets there is $\varepsilon > 0$ such that $B_{\bv^{\bot}}(\bx,\varepsilon)$ doesn't intersect any $R\in \Co_\bv \setminus \{R_1, \ldots, R_m \}$. By assumption $\bx \in \partial R_i$ for some $1\le i \le m$. Thus there is $\by\in B_{\bv^{\bot}}(\bx,\varepsilon)$ with $\by\not \in R_i$ and, hence, $\by$ is covered by at most $m-1$ elements of $\Co_\bv$, a contradiction.

\item[(iii)] 
Choose $\bx$ which is covered exactly $m$ times by elements of $\Co_\bv$. Since the elements of $\Co_\bv$ are uniformly bounded, the set of elements of $\Co_\bv$ which contain $\bx$ is contained in a set $\{\pi_{\bu,\bv}\bx + \Ra_\bv(i)\;:\; [\bx,i]\in P\}$, where $P$ is a patch of $\Gamma(\bv)$ which is chosen so large that, regardless of how the elements of $\Gamma(\bv)$ continue outside $P$, they will not contribute elements of $\Co_\bv$ containing $\bx$ because they are bounded and located ``too far away'' from $\bx$. Thus, whenever we encounter a translate $P+\bt$ of $P$ in $\Gamma(\bv)$, the point $\bx+\pi_{\bu,\bv}\bt$ will be covered $m$ times by elements of $\Co_\bv$ as well. Thus by (i) and (ii) there exist $r_m$ and $r'_m$ such that in each ball of radius $r_m'$ the hyperplane $\bv^\bot$ contains a ball of radius $r_m$ that is covered by exactly $m$ elements of $\Co_\bv$.

\item[(iv)] $\Co^{(n_k)}_\bv$ converges to $\Co_\bv$ in a sense described in the proof of Lemma~\ref{lem:intsubtile}. Thus by (i) the radii $r_m$ and $r'_m$ in (iii) can be chosen in a way that in each ball of radius $r_m'$ the hyperplane $(\bv^{(n_k)})^\bot$ contains a ball of radius $r_m$ that is covered by exactly $m$ elements of $\Co^{(n_k)}_\bv$ for $k$ large enough.

\item[(v)] Suppose that $\Co_\bv$ is not a multiple tiling. Then there is a set $X\subset \bv^\bot$ with $\lambda_\bv(X)>0$ which is covered at least $m+1$ times. Since the boundaries of the elements of $\Co_\bv$ have measure $0$ by Proposition~\ref{prop:measurezero}, there is $\bx$, which is covered a least $m+1$ times and which is not contained in the boundary of any element of $\Co_\bv$. Thus there is $\varepsilon >0$ such that $B_{\bv^{\bot}}(\bx,\varepsilon)$ is covered at least $m+1$ times.

\item[(vi)] Suppose that $\Co_\bv$ is not a multiple tiling. By analogous arguments as in (iii), by (v) there exist $r_{m+1}$ and $r'_{m+1}$ such that in each ball of radius $r_{m+1}'$ the hyperplane $\bv^\bot$ contains a ball of radius $r_{m+1}$ that is covered by at least $m+1$ elements of $\Co_\bv$.

\item[(vii)] By Proposition~\ref{prop:seteqII} each element of $\Co_\bv$ can be subdivided into elements of $M_{[0,n_k)}\Co^{(n_k)}_\bv$. The diameters of the elements of $M_{[0,n_k)}\Co^{(n_k)}_\bv$ tend to $0$ for $k\to\infty$ by Lemma~\ref{l:smallsubtiles} and the balls of radius $r_{m}'$ occurring in (iv) are shrunk by $M_{[0,n_k)}$ to ellipsoids contained in balls of radius less than $r_{m+1}$. Thus by (iv) we can chose $k$ so large that in each ball of radius $r_{m+1}$ in $\bv^\bot$ there are points which are covered exactly $m$ times by $M_{[0,n_k)}\Co^{(n_k)}_\bv$. Thus, by Proposition~\ref{prop:seteqII}, in each ball of radius $r_{m+1}$ there are points which are covered at most $m$ times by $\Co_\bv$. This contradicts (vi) and the result follows.
\end{enumerate}
\end{proof}

This result can be generalized to $\Co_\bw$ for arbitrary $\bw\in\R_{\ge 0}^d\setminus\{\mathbf{0}\}$. To establish this generalization one needs to show first that the measures of the subtiles of $\Ra_\bv$ are determined by
\[
(\lambda_\bv(\Ra_\bv(1)),\ldots,\lambda_\bv(\Ra_\bv(d)))=
m(\lambda_\bv(\pi_{\bu,\bv}[\mathbf{0},1]), \ldots,\lambda_\bv(\pi_{\bu,\bv}[\mathbf{0},d])),
\]
where $m$ is the covering degree of the multiple tiling $\Co_\bv$. This can be proved along similar lines as in the substitutive case, see \cite[Lemma~2.3]{Ito-Rao:06}. Using this, measure theoretical considerations lead to the following generalization of Proposition~\ref{prop:multiTiling}.

\begin{proposition}\label{prop:multiTiling2}
Let $\bsigma=(\sigma_n)$ be a primitive and algebraically irreducible sequence of unimodular substitutions. Assume that there is $C>0$ such that for every $\ell\in\N$ there exists $n\ge 1$ such that $(\sigma_n,\ldots,\sigma_{n+\ell-1})=(\sigma_0,\ldots,\sigma_{\ell-1})$ and the language $\Lg_\bsigma^{(n+\ell)}$ is $C$-balanced.

Then for each $\bw\in \R_{\ge0}^d\setminus\{\mathbf{0}\}$ the collection $\Co_\bw$ forms a multiple tiling of the hyperplane $\bw^{\bot}$.
\end{proposition}

It remains to show that this multiple tiling is actually a tiling. As we will see later, additional assumptions are needed to prove this. However, there is one tiling result which holds without additional assumptions. This result, which concerns the ``inner tiling'' of $\Ra_\bw(i)$ by the set equation \eqref{eq:seteq} will be proved next.

\begin{proposition}\label{prop:TilingSetEq}
Let $\bsigma=(\sigma_n)$ be a primitive and algebraically irreducible sequence of unimodular substitutions. Assume that there is $C>0$ such that for every $\ell\in\N$ there exists $n\ge 1$ such that $(\sigma_n,\ldots,\sigma_{n+\ell-1})=(\sigma_0,\ldots,\sigma_{\ell-1})$ and the language $\Lg_\bsigma^{(n+\ell)}$ is $C$-balanced.

Then the unions in the set equation \eqref{eq:seteq} of Proposition~\ref{prop:seteq} are disjoint in measure.
\end{proposition}

\begin{proof}
From Proposition~\ref{prop:multiTiling2} we know that $\Co_\bw$ is a multiple tiling for each $\bw\in \R_{\ge0}^d\setminus\{\mathbf{0}\}$ with multiplicity $m$ not depending on $\bw$. Together with Proposition~\ref{prop:seteqII} this implies that the $(\ell-k)$-th subdivisions of all the tiles in the multiple tiling $\Co_{\bw^{(k)}}$ form a multiple tiling $M_{[k,\ell)}\Co_{\bw^{(\ell)}}$ of the same covering degree (for all $k,\ell\in\N$ with $k<\ell$). This is possible only if each tile of $\Co_{\bw^{(k)}}$ is tiled without overlaps by elements of $M_{[k,\ell)}\Co_{\bw^{(\ell)}}$. This proves the result.
\end{proof}

\subsection{Coincidence conditions and tiling properties}\label{sec:coinc}

Let $\bsigma$ be a sequence of unimodular substitutions over an alphabet $\A$. In view of Example~\ref{ex:rauzy} in order to prove that $(X_\bsigma,\Sigma)$ is measurably conjugate to a rotation on a torus we need two properties of the associated Rauzy fractal $\Ra$. Firstly, the subtiles $\Ra(i)$, $i\in\A$, need to be disjoint in measure and secondly, the Rauzy fractal itself has to be a fundamental domain of a (well-chosen) torus. The latter property is equivalent to the fact that $\Ra$ admits a lattice tiling of $\bv^{\bot}$. Setting $\bw=\bone$ in Proposition~\ref{prop:multiTiling2} we obtain that $\Co_\bone$ is a multiple tiling of $\bone^\bot$. Since the discrete hyperplane $\Gamma(\bone)$ can be written as $\Gamma(\bone) = \{[\bx,i] \,:\, \langle\bx, \bone\rangle=0,\ i\in\A \}$ we see that 
\[
\bigcup_{R\in\Co_\bone}R = \bigcup_{[\bx,i] \in \Gamma(\bone)} \pi_{\bu,\bone}\bx + \Ra(i)
= \bigcup_{\bx \in \Z^d\,:\, \langle\bx, \bone\rangle=0} \pi_{\bu,\bone}\bx + \Ra.
\]
Thus $\Ra$ is a covering of $\bone^\bot$ w.r.t.\ the lattice $\{\bx \in \Z^d\,:\, \langle\bx, \bone\rangle=0\}$ and we have to prove that the elements of the union on the right hand side are measure disjoint to get tiling properties of $\Ra$.

We have therefore three types of unions which we want to be disjoint in measure:
\begin{enumerate}
\item[(i)] The unions of subtiles on the right hand side of the set equation \eqref{eq:seteq}.
\item[(ii)] The union $\Ra = \Ra(1)\cup\dots\cup \Ra(d)$.
\item[(iii)] The union $\bone^\bot= \bigcup_{\bx \in \Z^d\,:\, \langle\bx, \bone\rangle=0} \pi_{\bu,\bone}\bx + \Ra=\bigcup_{[\bx,i] \in \Gamma(\bone)} \pi_{\bu,\bone}\bx + \Ra(i)$.
\end{enumerate} 

The elements of the unions in (i) are disjoint in measure by Proposition~\ref{prop:TilingSetEq}. One can use this fact in order to prove that the unions in (ii) are disjoint in measure as well. However, to make this proof work we need an additional assumption on $\bsigma$.

\begin{definition}[Strong coincidence condition]\index{strong coincidence condition}
A sequence $\bsigma$ of substitutions over an alphabet $\A$ satisfies the \emph{strong coincidence condition} if there is $\ell\in\N$ such that for each pair $(j_1,j_2)\in\A^2$ there are $i\in\A$ and $p_1,p_2\in\A^*$ with $\mathbf{l}(p_1)=\mathbf{l}(p_2)$ such that $\sigma_{[0,\ell)}(j_1)\in p_1i\A^*$ and $\sigma_{[0,\ell)}(j_2)\in p_2i\A^*$.
\end{definition}

This definition has an easy geometric meaning: it says that the broken lines associated with $\sigma_{[0,\ell)}(j_1)$ and $\sigma_{[0,\ell)}(j_2)$ have at least one line segment in common for each pair $(j_1,j_2)\in\A^2$. 

\begin{example}
Figure~\ref{fig:coinc} shows that the strong coincidence condition is satisfied for the constant sequence $\bsigma=(\sigma)$ with $\sigma(1)=121$, $\sigma(2)=21$. Because we are in a case with a two letter alphabet we only have to deal with the instance $(j_1,j_2)=(1,2)$.
\begin{figure}[hh]
\includegraphics[width=0.05\textwidth]{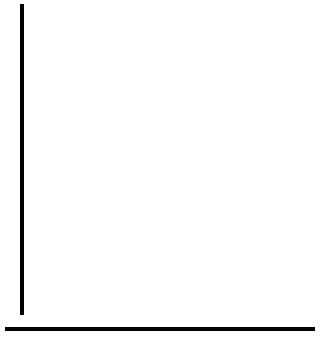}\qquad
\includegraphics[width=0.1\textwidth]{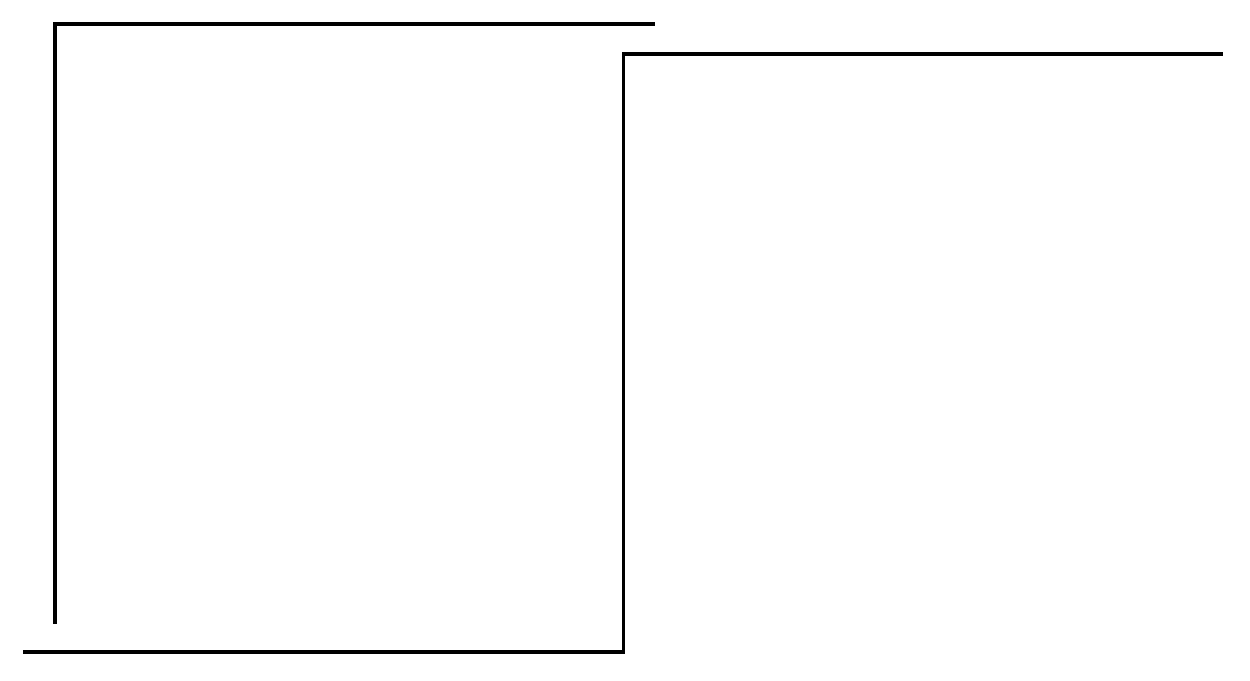}\qquad
\includegraphics[width=0.25\textwidth]{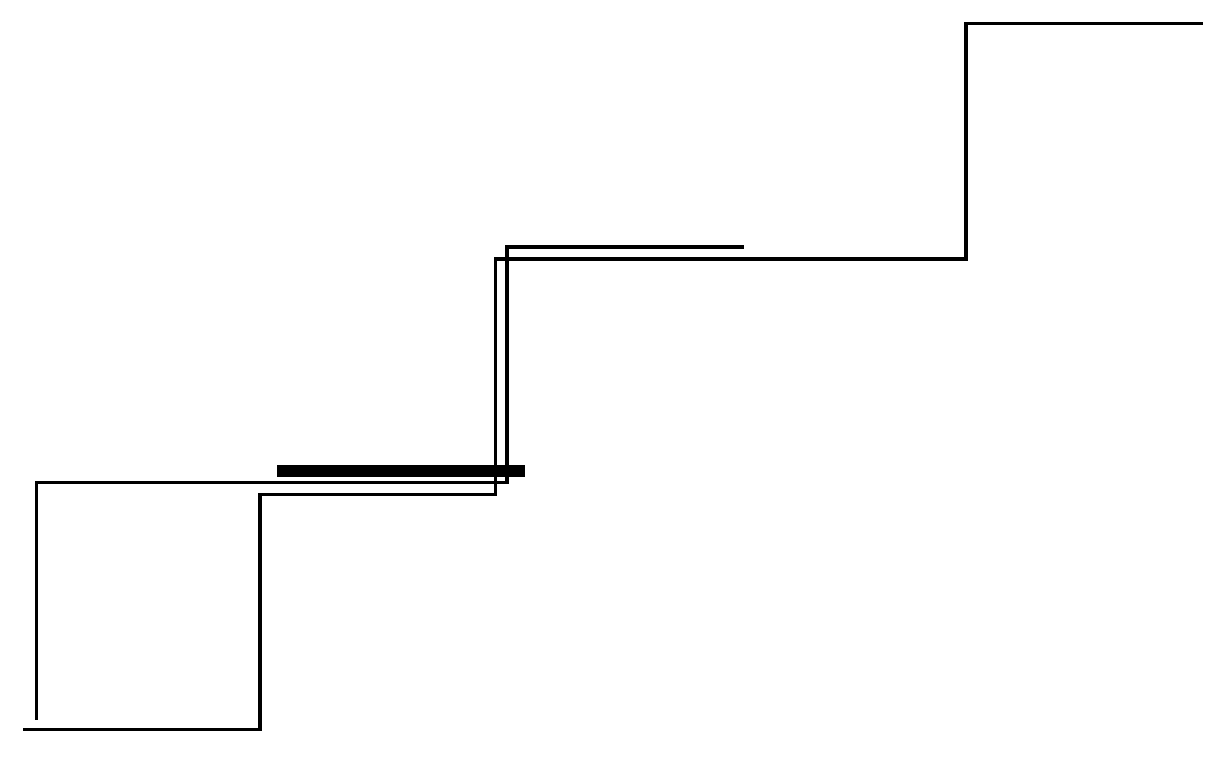} \\[3mm]
{\small \hskip -0.5cm $i$ \hskip 1.15cm $\sigma_{[0,1)}(i)$ \hskip 1cm $\sigma_{[0,2)}(i)$}
\caption{The broken lines associated with $i$, $\sigma_{[0,1)}(i)$, and $\sigma_{[0,2)}(i)$ for $i\in\{1,2\}$. Coincidence is indicated by the bold line. \label{fig:coinc}}
\end{figure}
\end{example}

Using the strong coincidence condition we get the following result.

\begin{proposition}\label{prop:TilingSubtiles}
Let $\bsigma=(\sigma_n)$ be a primitive and algebraically irreducible sequence of unimodular substitutions. Assume that there is $C>0$ such that for every $\ell\in\N$ there exists $n\ge 1$ such that $(\sigma_n,\ldots,\sigma_{n+\ell-1})=(\sigma_0,\ldots,\sigma_{\ell-1})$ and the language $\Lg_\bsigma^{(n+\ell)}$ is $C$-balanced.

If the strong coincidence condition holds then the subtiles $\Ra(i)$, $i\in\A$, are disjoint in measure.
\end{proposition}

\begin{proof}[Sketch]
Let $(n_k)$ and $(\ell_k)$ be the associated sequences of $\bsigma$. 
Let $\Ra(j_1)$ and $\Ra(j_2)$ be two subtiles with $j_1,j_2\in\A$ distinct and assume that the strong coincidence condition holds with $\ell\in\N$. By the definition of the dual $E_1^*$ in \eqref{eq:dualGeomSubs} this implies that for $k$ satisfying $n_k\ge \ell$ there is $\bz_k\in\Z^d$ and a letter $i\in\A$ such that $[\bz_k,j_1],[\bz_k,j_2]\in E_1^*(\sigma_{[0,n_k)})[\mathbf{0},i]$. Thus the set equation 
\begin{equation*}
\Ra(i)=
\bigcup_{[\by,j]\in E_1^*(\sigma_{[0,n_k)})[0,i]} M_{[0,n_k)}(\pi^{(n_k)}_{\bu,\bone}\by + \Ra^{(n_k)}(j)),
\end{equation*}
(see \eqref{eq:seteq})
contains $M_{[0,n_k)}(\bz_k + \Ra^{(n_k)}(j_1))$ and $M_{[0,n_k)}(\bz_k + \Ra^{(n_k)}(j_2))$ in the union on the right hand side. Proposition~\ref{prop:TilingSetEq} now implies that $\Ra^{(n_k)}(j_1)$ and $\Ra^{(n_k)}(j_2)$ are disjoint in measure. Since this is true for arbitrarily large $k$, using results along the line of Proposition~\ref{prop:RauzyHausdorff} (in particular, \cite[Lemma~6.8]{Berthe-Steiner-Thuswaldner}) this implies that $\Ra(j_1)$ and $\Ra(j_2)$ are disjoint in measure as well. 
\end{proof}

What we did in the proof of Proposition~\ref{prop:TilingSubtiles} can be explained in a simple way. If the strong coincidence condition holds, each intersection of the subtiles $\Ra(j_1) \cap \Ra(j_2)$ can be realized as an intersection of two elements in the union on the right hand side of the set equation \eqref{eq:seteq}. Since we know that the elements in the union of the set equation are measure disjoint, the same is true for $\Ra(j_1)$ and $\Ra(j_2)$. More briefly: in case of strong coincidence the elements in the union in (ii) are special cases of the elements in some union in (i).

The same strategy can be used in order to prove that the unions in (iii) are measure disjoint. To this end we need another type of coincidence condition.

\begin{definition}[Geometric coincidence condition]\label{def:geomcoinc}\index{geometric coincidence condition}
A sequence $\bsigma$ of unimodular substitutions over an alphabet $\A$ satisfies the \emph{geometric coincidence condition} if the following is true. For each $r>0$ there is $n_0\in\N$ such that for each $n\ge n_0$ the set $E_1^*(\sigma_{[0,n)})[\mathbf{0},i_n]$ contains a ball of radius $r$ of the discrete hyperplane $\Gamma((M_{[0,n)})^t\bone)$ for some $i_n\in\A$. 
\end{definition}

Along similar lines as Proposition~\ref{prop:TilingSubtiles} one can prove the following tiling criterion for Rauzy fractals (see \cite[Proposition~7.9]{Berthe-Steiner-Thuswaldner}).

\begin{proposition}\label{prop:tilingR}
Let $\bsigma=(\sigma_n)$ be a primitive and algebraically irreducible sequence of unimodular substitutions. Assume that there is $C>0$ such that for every $\ell\in\N$ there exists $n\ge 1$ such that $(\sigma_n,\ldots,\sigma_{n+\ell-1})=(\sigma_0,\ldots,\sigma_{\ell-1})$ and the language $\Lg_\bsigma^{(n+\ell)}$ is $C$-balanced. Then the following assertions are equivalent.

\begin{enumerate}
\item[(i)] The collection $\Co_\bone$ forms a tiling of $\bone^\bot$.
\item[(ii)] The sequence $\bsigma$ satisfies the geometric coincidence condition.
\item[(iii)] The sequence $\boldsymbol{\sigma}$ satisfies the strong coincidence condition and for each $r > 0$ there exists $n_0 \in \mathbb{N}$ such that $\bigcup_{i\in\mathcal{A}} E_1^*(\sigma_{[0,n)})[\mathbf{0},i]$ contains a ball of radius~$r$ of $\Gamma((M_{[0,n)})^t\, \mathbf{1})$ for all $n \ge n_0$.
\item[(iv)] The sequence $\bsigma$ satisfies the following effective condition: There are $n \in \mathbb{N}$, $i \in \mathcal{A}$, and $\mathbf{z} \in \mathbb{R}^d$, such that 
\[
\big\{[\mathbf{y},j] \in \Gamma((M_{[0,n)})^t\, \mathbf{1}):\, 
\|\pi_{(M_{[0,n)})^{-1}\mathbf{u},\mathbf{1}} (\mathbf{y} - \mathbf{z})\| 
\le C\big\} \subset E_1^*(\sigma_{[0,n)})[\mathbf{0},i],
\]
where $C \in \mathbb{N}$ is chosen in a way that $\mathcal{L}_{\boldsymbol{\sigma}}^{(n)}$ is $C$-balanced.
\end{enumerate}
\end{proposition}

An (essentially) more restrictive condition than the geometric coincidence condition and its variants in Proposition~\ref{prop:tilingR} is the following one. 

\begin{definition}[Geometric finiteness property]\label{def:geomfin}
A sequence $\bsigma$ of unimodular substitutions over an alphabet $\A$ satisfies the \emph{geometric finiteness property} if for each $r > 0$ there is $n_0\in \mathbb{N}$ such that $\bigcup_{i\in\mathcal{A}} E_1^*(\sigma_{[0,n)})[\mathbf{0},i]$ contains the ball $\{[\mathbf{x},i] \in \Gamma((M_{[0,n)})^t\, \mathbf{1}):\, \|\mathbf{x}\| \le r\}$ for all $n \ge n_0$. 
\end{definition}

The geometric finiteness property implies that $\bigcup_{i\in\mathcal{A}} E_1^*(\sigma_{[0,n)})[\mathbf{0},i]$ generates a whole discrete plane for $n\to\infty$, and that $\mathbf{0}$ is an inner point of  the Rauzy fractal $\Ra$ (as is proved in \cite[Proposition~7.10]{Berthe-Steiner-Thuswaldner}). 
 It is immediate that together with the strong coincidence condition the geometric finiteness property is more restrictive than the condition in Proposition~\ref{prop:tilingR}~(iii). The name \emph{geometric finiteness property} comes from the fact that it is related to certain finiteness properties in number representations w.r.t.\ positional number systems (see for instance~Barat {\it et al.}~\cite{BBLT:06} for a survey on these objects). By Proposition~\ref{prop:tilingR}~(iii) strong coincidence plus geometric finiteness imply that $\Co_\bone$ forms a tiling of $\bone^\bot$.

\subsection{How to check geometric coincidence and geometric finiteness?}\label{sec:howto}

In most cases it is easy to check strong coincidence of a sequence $\bsigma=(\sigma_n)$ of substitutions over an alphabet $\A$. For instance, this property trivially holds if $\sigma_0(i)$ starts with the same letter for each $i\in\A$.  However, it is {\it a priori} not so clear how to check geometric coincidence or geometric finiteness and although there is an effective criterion for geometric coincidence contained in Proposition~\ref{prop:tilingR}~(iv) this is only suitable for checking single instances. Geometric coincidence asserts that a large piece of a discrete hyperplane can be generated by the dual substitution $E_1^*(\sigma_{[0,n)})$ acting on $[\mathbf{0},i_n]$ if $n$ is large. If geometric finiteness holds, even a whole discrete hyperplane can be generated by the patches $E_1^*(\sigma_{[0,n)})\bigcup_{i\in\mathcal{A}}[\mathbf{0},i]$ for $n\to\infty$. The idea of generating discrete hyperplanes in this way using sequences of substitutions coming from generalized continued fraction algorithms goes back to Ito and Ohtsuki~\cite{ItoOtsuki94}. More recently, Berth\'e {\it et al.}~\cite{BBJS16,Berthe-Jolivet-Siegel:12} provide a systematic study on how to check geometric coincidence as well as geometric finiteness. While \cite{Berthe-Jolivet-Siegel:12} concentrates on Arnoux-Rauzy substitutions, the more general treatment in \cite{BBJS16} uses substitutions related to the Brun as well as the Jacobi-Perron algorithm as guiding examples.  In this section we give a brief discussion of their ideas which are centered around an ``annulus property'' of stepped hyperplanes generated by $E_1^*(\sigma_{[0,n)})$.

Let $\bsigma=(\sigma_n)$ be a sequence of unimodular substitutions  over an alphabet $\A$ and let $S=\{\sigma_n\,:\, n\in\N\}$. The fact that $\bsigma$ satisfies the geometric coincidence condition in Definition~\ref{def:geomcoinc} roughly says that the patch $E_1^*(\sigma_{[0,n)})[\mathbf{0},i_n]$ contains a larger and larger ball when $n$ is growing. In this section, for the sake of simplicity, we will deal with the geometric finiteness property. Indeed, we will assume that this ball is centered at the origin and instead of $[\mathbf{0},i_n]$ we will use $\mathcal{U}=\bigcup_{i\in\mathcal{A}}[\mathbf{0},i]$ as our ``seed''.  So we want to show that for each $R > 0$ there is $n_0 \in \mathbb{N}$ such that $ E_1^*(\sigma_{[0,n)})\mathcal{U}$ contains the ball $\{[\mathbf{x},i] \in \Gamma((M_{[0,n)})^t\, \mathbf{1}):\, \|\mathbf{x}\| \le R\}$ for all $n \ge n_0$. 

Following \cite{BBJS16} we shall reformulate the geometric finiteness property in a more combinatorial way. Let $P$ be a patch of a discrete hyperplane containing $\mathcal{U}$ and interpret its elements as faces as in \eqref{eq:hypercube}. Then the \emph{minimal combinatorial radius} 
${\rm rad}(P)$ of $P$ is equal to the length $\ell$ of the shortest sequence of faces $[\bx_1,j_1],\ldots,[\bx_\ell,j_\ell] \in P$ satisfying $[\bx_1,j_1]\in\mathcal{U}$, $[\bx_\ell,j_\ell]$ contains a part of the boundary of $P$ (regarded as a topological manifold), and $[\bx_k,j_k]\cap[\bx_{k+1},j_{k+1}]\not=\emptyset$ for $1\le k\le \ell-1$.
Intuitively, ${\rm rad}(P)$ is the minimal distance between
$\mathbf{0}$ and the boundary of $P$. For instance, one easily checks that the minimal combinatorial radius of the patch on the left hand side of Figure~\ref{fig:steppedsurface} is equal to six. Clearly a sequence $\bsigma$ enjoys the geometric finiteness property if and only if ${\rm rad}\left(E_1^*(\sigma_{[0,n)})\mathcal{U} \right)$ tends to $\infty$ for $n\to\infty$.

Let $P_{[m,n)}=E_1^*(\sigma_{[m,n)})\mathcal{U}$. We have to show that 
the minimal combinatorial radii of the patches $P_{[0,n)}$ tend to $\infty$ for $n\to\infty$.  Since the patches  $P_{[0,n)}$ can have complicated shapes there is no obvious way to do this. One approach to prove this property goes back to Ito and Ohtsuki~\cite{ItoOtsuki94} and makes use of ``annuli''. Let $\ell<m<n$ and suppose that $\mathcal{U}\subset E_1^*(\sigma)\mathcal{U}$ holds for each $\sigma\in S$ (this is not a crucial assumption and, if it is not true, can often be gained by blocking the substitutions of the sequence $\bsigma$). Then $P_{[m,n)}\subset P_{[\ell,n)}$ holds by the definition of $E_1^*(\sigma)$ (note in particular that $E_1^*(\tau)E_1^*(\sigma)=E_1^*(\sigma\tau)$ for $\sigma,\tau\in S$). The idea is to make sure that whenever $(\sigma_\ell,\ldots, \sigma_{m-1})$ is of a certain shape then $P_{[\ell,n)} \setminus P_{[m,n)}$ contains an annulus of positive width. One can then show that if $(\sigma_0,\ldots, \sigma_{n})$ contains the block $(\sigma_\ell,\ldots, \sigma_{m-1})$ for $k$ times, the patch $P_{[0,n)}$ contains $k$ ``concentric'' annuli and has a minimal combinatorial radius greater than or equal to $k$.

To achieve this we first search for a block $(\sigma_0,\ldots,\sigma_{m-1})$ such that $A=P_{[0,m)}\setminus \mathcal{U}$ contains an annulus of positive width, {\it i.e.}, $\partial P_{[0,m)} \cap \mathcal{U}=\emptyset$. If $\bsigma$ is recurrent, the block $(\sigma_0,\ldots,\sigma_{m-1})$ occurs infinitely often in $\bsigma$. Let $(n_j)$ with $n_0=0$ and $n_j\ge n_{j-1}+m$ be an increasing sequence such that $(\sigma_{n_j},\ldots,\sigma_{n_j+m-1})=(\sigma_0,\ldots,\sigma_{m-1})$. Fix $k\in \N$ and set $A_0=P_{[0,n_k+m)}\setminus P_{[m,n_k+m)}$ and $A_j:= P_{[n_{j-1}+m,n_k+m)} \setminus P_{[n_j+m,n_k+m)}$ for $j\ge1$. Then 
\begin{equation}\label{eq:annul}
\begin{split}
P_{[0,n_k+m)} &= (P_{[0,n_k+m)}\setminus P_{[m,n_k+m)}) \cup P_{[m,n_k+m)} \\
&= A_0  \cup P_{[m,n_k+m)} \\
&= A_0 \cup (P_{[m,n_k+m)}\setminus P_{[n_1+m,n_k+m)}) \cup P_{[n_1+m,n_k+m)}\\
&=  A_0 \cup A_1 \cup P_{[n_1+m,n_k+m)} \\
&=  A_0 \cup A_1 \cup (P_{[n_1+m,n_k+m)} \setminus P_{[n_2+m,n_k+m)}) \cup P_{[n_2+m,n_k+m)}\\
&=  A_0 \cup A_1 \cup A_2 \cup P_{[n_2+m,n_k+m)}\\&= \cdots = A_0 \cup \dots \cup A_k \cup \mathcal{U}.
\end{split}
\end{equation}
Because 
\[
\begin{split}
A_j&=P_{[n_{j-1}+m,n_k+m)} \setminus P_{[n_j+m,n_k+m)} \\
&\supset P_{[n_j,n_k+m)} \setminus P_{[n_j+m,n_k+m)}\\
&= E_1^*(\sigma_{[n_j+m,n_k+m)}) (P_{[n_j,n_j+m)}\setminus\mathcal{U})\\
&=E_1^*(\sigma_{[n_j+m,n_k+m)}) A
\end{split}
\]
for $j\ge 1$ (the last step comes from the recurrence property; the case $j=0$ follows along similar lines) each $A_j$ contains some image of $A$ under $E_1^*$. 
If the annulus $A$ has certain ``covering properties'' that are described in detail in \cite{BBJS16,Berthe-Jolivet-Siegel:12}, one can show that images of $A$ under $E_1^*$ are annuli of positive width as well. Thus such an annulus of positive width is contained in each of the pairwise disjoint subsets $A_0,\ldots,A_k$ of $P_{[0,n_k+m)}$ and therefore \eqref{eq:annul} implies that the patch $P_{[0,n_k+m)}$ contains a ``concentric'' annulus for each of the $k+1$ (non overlapping) occurrence of the block $(\sigma_0,\ldots,\sigma_{m-1})$ in $(\sigma_0,\ldots,\sigma_{n_k+m-1})$. 
Since an application of $E_1^*$ maps disjoint annuli to disjoint annuli also $P_{[0,n)}=E_1^*(\sigma_{[n_k+m,n)})P_{[0,n_k+m)}$ with $n_k+m\le n < n_{k+1}+m$ contains $k+1$ such ``concentric'' annuli. 
Thus if $n\to\infty$, the number of such annuli in $P_{[0,n)}$ tends to $\infty$. Since the above-mentioned covering properties of $A$ imply that $A_0 \cup \dots \cup A_k\cup \mathcal{U}=P_{[0,n_k+m)}$ is simply connected for each $k\in\N$ and that the same is true for all the patches $P_{[0,n)}$ (see~\cite{Berthe-Jolivet-Siegel:12}), we gain that the minimal combinatorial radii of the patches $P_{[0,n)}$ tend to $\infty$ for $n\to\infty$.

The following example shows that this method can be used in order to prove geometric finiteness for large classes of sequences of substitutions.

\begin{figure}[hh]
\includegraphics[trim=0 160 0 60,clip,width=1\textwidth]{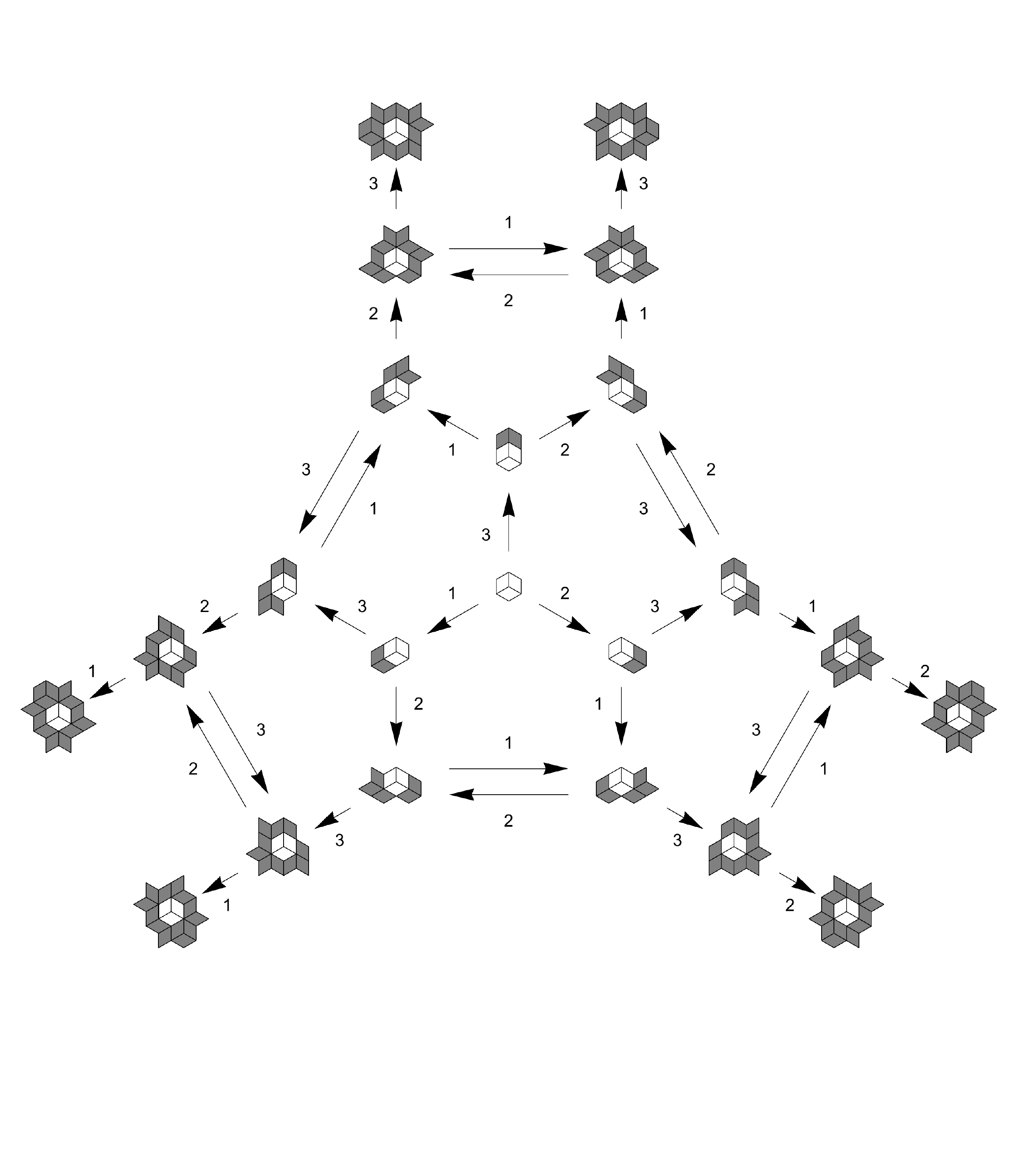}
\caption{An illustration of the annulus property for sequences of Arnoux-Rauzy substitutions.
\label{fig:TimoAnnulus}}
\end{figure}

\begin{example}\label{ex:timoAR}
We want to illustrate the construction of the annulus $A$ around $\mathcal{U}$ for the case of sequences of Arnoux-Rauzy substitutions $\bsigma=(\sigma_n)$ (all details for this case can be found in \cite{Berthe-Jolivet-Siegel:12}). Suppose that $\bsigma$ is a recurrent sequence of Arnoux-Rauzy substitutions which contains each of the three Arnoux-Rauzy substitutions \eqref{eq:ARsubs}. Then, by recurrence, $\bsigma$ contains a block $(\sigma_0,\ldots, \sigma_{m-1})$ in which each Arnoux Rauzy substitution occurs at least twice.
In the graph depicted in Figure~\ref{fig:TimoAnnulus} the action of $E_1^*$ on $\mathcal{U}$ is illustrated.\footnote{We note that in \cite{Berthe-Jolivet-Siegel:12} the dual $E_1^*(\sigma)$ is defined using suffixes of the images of $\sigma$ instead of prefixes. Nevertheless, this difference does not change the behavior of $E_1^*(\sigma)$ significantly and in Figure~\ref{fig:TimoAnnulus} we get the same image as the authors obtained in \cite[Figure~1]{Berthe-Jolivet-Siegel:12}.} The vertices of this graph are patches and there is an edge $P_1\xrightarrow{i} P_2$ if $P_2\subset E_1^*(\sigma_i)P_1$. Thus each vertex has an outgoing edge for each $i\in\{1,2,3\}$ (loops and outgoing edges of patches that contain an annulus of positive width around $\mathcal{U}$ are suppressed). Examining the graph we see that $E_1^*(\sigma_{[k,n)})\mathcal{U}$ contains an annulus around $\mathcal{U}$ of positive width whenever the block $(\sigma_{k},\ldots,\sigma_{n-1})$ contains at least two occurrences of each Arnoux-Rauzy substitution. Thus, $P_{[0,m)}$ is a patch which contains $\mathcal{U}$ together with an annulus $A$ of positive width around it. 

If one proves that the annulus $A$ has the above-mentioned covering properties (which was done in  \cite{Berthe-Jolivet-Siegel:12}) one can iterate this procedure as indicated above and prove that $P_{[0,n)}$ is simply connected and contains a growing number of ``concentric'' annuli for growing $n$. Thus
the minimal combinatorial radius of $P_{[0,n)}$ tends to $\infty$ for $n\to\infty$ and, hence, $\bsigma$ has the geometric finiteness property.  
\end{example}

Summing up, in Example~\ref{ex:timoAR} we have sketched a proof of the following result.

\begin{proposition}\label{prop:ARgeometric}
Let $\bsigma$ be a sequence of Arnoux-Rauzy substitutions. If $\bsigma$ is recurrent and contains each of the three Arnoux-Rauzy substitutions then $\bsigma$ satisfies the geometric finiteness property.
\end{proposition}

\section{$S$-adic systems and torus rotations}

Let $\bsigma$ be a sequence of unimodular substitutions over an alphabet $\A$ with $d$ letters.  In the past sections we proved a variety of properties of Rauzy fractals. Using all these results makes Rauzy fractals suitable to ``see'' a rotation on the torus $\mathbb{T}^{d-1}$ acting on them. This rotation turns out to be measurably conjugate to the underlying $S$-adic system $(X_{\bsigma},\Sigma)$. In this section we prove the according results which form special cases of the main results of \cite{Berthe-Steiner-Thuswaldner} and provide some examples.

In Section~\ref{sec:191} we state Theorem~\ref{th:TilingRotation}, a result that gives the measurable conjugacy between $(X_{\bsigma},\Sigma)$ and a torus rotation together with some of its consequences under a set of natural conditions. Section~\ref{sec:rotproof} is devoted to the proof of this result. In Section~\ref{sec:metric} we formulate a metric version of Theorem~\ref{th:TilingRotation}. In particular, for a finite set $S$ of substitutions we consider the shift\footnote{Note that there are two kinds of shifts: the one just defined acts on the sequence of substitutions $S^\N$, the other one (the $S$-adic shift) acts on the set of sequences $X_\bsigma$ which is defined in terms of a single sequence of substitutions $\bsigma\in S^\N$. It should cause no confusion that both of these shift mappings are denoted by $\Sigma$.}
 $(S^\N,\Sigma,\nu)$ acting on all infinite sequences of substitutions taken from $S$. 
 The measure $\nu$ is chosen in a way that this shift becomes ergodic. 
 We prove that the conditions of Theorem~\ref{th:TilingRotation} are ``generic'' w.r.t.\ the measure $\nu$ if the {\em Pisot condition}  \eqref{eq:lyapunov} on the Lyapunov exponents associated with a linear cocycle of $(S^\N,\Sigma,\nu)$ is in force. Thus under this Pisot condition we gain that $\nu$-almost all $\bsigma\in S^\N$ give rise to an $S$-adic system $(X_\bsigma,\Sigma)$ that is measurably conjugate to a torus rotation. This result is the content of Theorem~\ref{t:3}. Section~\ref{sec:194} is devoted to the proof of this result. Finally, Section~\ref{sec:corex} gives examples for $S$-adic systems associated with Arnoux-Rauzy and Brun substitutions. This shows that the Pisot condition is satisfied in many natural situations.

\subsection{Statement of the conjugacy result}\label{sec:191}

Before we state the first main result of this survey we give some terminology. We start with a spectral property of a measurable dynamical system that is ``the opposite'' of continuous spectrum (see Section~\ref{sec:WM}; we refer to this section also for the definition of an \emph{eigenfunction}).

\begin{definition}[Pure discrete spectrum, see {\cite[Defintion~3.2]{Walters:82}}] \label{def:PDS}\index{pure discrete spectrum}\index{spectrum!pure discrete}
An ergodic dynamical system $(X,T,\mu)$ on a probability space $X$ has \emph{pure discrete spectrum} if there exists an orthonormal basis of $L^2(\mu)$ which consists of eigenfunctions of $T$.
\end{definition}

It is well known that an ergodic dynamical system on a probability space that has pure discrete spectrum is measurably conjugate to a rotation on a compact abelian group. On the other hand, each ergodic rotation on a compact abelian group has pure discrete spectrum (see for instance \cite[Theorems~3.5 and~3.6]{Walters:82}; these results can be proved by using character theory and Pontryagin duality for compact abelian groups).

The notion of \emph{natural coding} came up already in Sections~\ref{sec:sturmrot} and~\ref{sec:ARImbalance} in the framework of Sturmian sequences and Arnoux-Rauzy sequences. Sloppily speaking a natural coding is a coding of a torus rotation that induces translations on the atoms of the partition that was used to define the coding. We give a precise general definition of this concept. 

\begin{definition}[Coding and natural coding]\label{def:NC}
\index{natural coding}
Let $\Lambda$ be a full-rank lattice in~$\mathbb{R}^d$ and $T_\mathbf{t}: \mathbb{R}^d/\Lambda \to \mathbb{R}^d/\Lambda$, $\mathbf{x} \mapsto \mathbf{x} + \mathbf{t}$ a rotation on the torus $\mathbb{R}^d/\Lambda$. Let $\Omega\subset\mathbb{R}^d$ be a fundamental domain for the lattice $\Lambda$ and $\tilde T_\mathbf{t}:\Omega\to \Omega$ the mapping induced by $T_\mathbf{t}$ on $\Omega$. Assume that $\Omega = \Omega_1 \cup \cdots \cup \Omega_k$ is a (measure theoretic w.r.t.\ the Lebesgue measure) partition of $\Omega$. 

A sequence $w=w_0w_1\ldots\in\{1,\ldots,k\}^\N$ is the {\it coding} of a point $\mathbf{x} \in \Omega$ with respect to this partition if $\tilde T_\mathbf{t}^j(\mathbf{x})\in \Omega_{w_j}$ holds for each $j\in \N$. If, in addition, for each $1\le i\le k$ the restriction~$\tilde T_\mathbf{t}|_{\Omega_i}$  is given by the translation ${\mathbf x}\mapsto{\mathbf x}+{\mathbf t}_i$ for some~${\mathbf t}_i\in\mathbb{R}^d$ we call~$w$ a \emph{natural coding} of $T_\mathbf{t}$.
\end{definition}

For the sake of completeness we give the definition of \emph{bounded remainder set}.

\begin{definition}[Bounded remainder set]\label{def:BRS}\index{bounded remainder set}
Let $\Lambda$ be a full-rank lattice in~$\mathbb{R}^d$. 
A~subset~$A$ of $\mathbb{R}^d/\Lambda$ 
is called a \emph{bounded remainder set} for the rotation $T_\mathbf{t}: \mathbb{R}^d/\Lambda \to \mathbb{R}^d/\Lambda$, $\mathbf{x} \mapsto \mathbf{x} + \mathbf{t}$ if there exist $\gamma,C > 0$ such that, for  a.e.\ $\bx \in  \mathbb{R}^d/\Lambda$,
\[
|\#\{n  < N:\, T _\mathbf{t}^n(\bx) \in A \} -  \gamma N  |  <  C 
\]
holds for all $N \in \mathbb{N}$.
\end{definition}

The following result gives sufficient conditions for an $S$-adic system $(X_\bsigma,\Sigma)$ to be measurably conjugate to an irrational rotation on a torus. The subtiles $\Ra(i)$ of the Rauzy fractal $\Ra$ turn out to be bounded remainder sets for this rotation and induce natural codings of the elements of $(X_\bsigma,\Sigma)$.
 
\begin{theorem}[{see \cite[Theorem~3.1]{Berthe-Steiner-Thuswaldner}}]
\label{th:TilingRotation}
Let $S$ be a finite set of unimodular substitutions over a finite alphabet $\A=\{1,2,\ldots,d\}$ and let $\bsigma=(\sigma_n)$ be a primitive and algebraically irreducible sequence of substitutions taken from the set $S$. Assume that there is $C>0$ such that for every $\ell\in\N$ there exists $n\ge 1$ such that $(\sigma_n,\ldots,\sigma_{n+\ell-1})=(\sigma_0,\ldots,\sigma_{\ell-1})$ and the language $\Lg_\bsigma^{(n+\ell)}$ is $C$-balanced.

If the collection $\mathcal{C}_\mathbf{1}$ forms a tiling of~$\mathbf{1}^\bot$ then the following results hold.
\begin{enumerate} \setcounter{enumi}{0}
\itemsep1ex
\item \label{i:16}
The $S$-adic shift $(X_{\boldsymbol{\sigma}},\Sigma,\mu)$, with $\mu$ being the unique $\Sigma$-invariant Borel probability measure on $X_\bsigma$, is measurably conjugate to a rotation $T$ on the torus~$\mathbb{T}^{d-1}$; in particular, its 
measure-theoretic spectrum is purely discrete. 
\item\label{i:17}
Each element of $X_{\boldsymbol{\sigma}}$ is a natural coding of the torus rotation~$T$ with respect to the partition $\{\mathcal{R}(i):\,i \in \mathcal{A}\}$ of the fundamental domain $\Ra$.
\item\label{i:18}
The subtile $\mathcal{R}(i)$ is a bounded remainder set for the torus rotation~$T$ for each $i\in\mathcal{A}$.
\end{enumerate}
\end{theorem}

For the special case of two letter alphabets the tiling condition does not have to be assumed. It can be derived from the remaining assumptions of Theorem~\ref{th:TilingRotation}. The corresponding result is proved in~\cite{BMST:16} and generalizes an analogous result for substitutive systems from~\cite{Barge-Diamond:02}.

\subsection{Proof of the conjugacy result}\label{sec:rotproof}

In this section we illustrate the proof of Theorem~\ref{th:TilingRotation} given in \cite{Berthe-Steiner-Thuswaldner}. We assume throughout this section that the sequence $\bsigma$ satisfies the conditions of Theorem~\ref{th:TilingRotation}. The main part is the proof of the measurable conjugacy between $(X_{\bsigma},\Sigma,\mu)$ and a rotation on the torus $\mathbb{T}^{d-1}$, where $d$ is the cardinality of the underlying alphabet. Here $\mu$ is the unique $\Sigma$-invariant Borel probability measure on $X_{\bsigma}$ (see Theorem~\ref{prop:SadicUE}).

Our first aim is to set up the \emph{representation map} from $X_{\bsigma}$ to the Rauzy fractal. We define this map using a nested sequence of the subsets 
\[
\mathcal{R}(u) := \overline{ \{\pi_{\bu,\bw}\mathbf{l}(p) \;:\; pu \hbox{ is a prefix of a limit sequence of } \bsigma \} } \qquad(u\in\A^*)
\]
of the Rauzy fractal $\mathcal{R}$. In particular, we set
\begin{equation}\label{eq:repmapphi}
\varphi : X_{\bsigma}\to\Ra; \quad
v_{0}v_{1}v_{2}\ldots \mapsto \bigcap_{n\in \N}\Ra(v_{0}v_{1}\ldots v_{n-1}).
\end{equation}
To show that $\varphi$ is a well-defined continuous surjection one has to prove that the intersection on the right-hand side of \eqref{eq:repmapphi} is a single point. Using the minimality of $(X_{\bsigma},\Sigma)$ and the strong convergence property from Proposition~\ref{prop:strongconv} this is done in \cite[Section~8]{Berthe-Steiner-Thuswaldner}.

In the next step one proves that $(X_{\bsigma},\Sigma,\mu)$ is measurably conjugate to the \emph{domain exchange} $(\Ra,E,\lambda_{\bone})$, where $E$ is given by 
\[
E: \Ra \to \Ra; \quad \bx \mapsto \bx  +\pi_{\bu,\bone}\mathbf{l}(i) \quad \hbox{for }\bx\in \Ra(i) \setminus \bigcup_{j\not=i}\Ra(j)
\]
which is illustrated in Figure~\ref{fig:tribodomain}. Since $\mathcal{C}_{\bone}$ is a tiling, the overlaps of the subtiles $\Ra(i)$ have measure $0$ and, hence, $E$ is well defined a.e.\ w.r.t.\ the measure $\lambda_{\bone}$ on $\Ra$. To prove the asserted conjugacy, we have to show that $\varphi$ is bijective $\mu$-a.e.\ and that the diagram
\begin{equation}\label{eq:Ediag}
\begin{CD}
X_{\boldsymbol{\sigma}} @> \Sigma >> X_{\boldsymbol{\sigma}} \\
@VV\varphi V @VV\varphi V\\
\mathcal{R} @> E >> \mathcal{R}
\end{CD}
\end{equation}
commutes. Since 
\begin{equation}\label{eq:comm}
E  \circ \varphi = \varphi \circ  \Sigma  
\end{equation}
follows easily by direct calculation it remains to prove the bijectivity assertion. This runs as follows (all statements are true up to measure zero). First observe that $E$ satisfies
\[
E(\mathcal{R}(i)) = \overline{\{\pi_{\mathbf{u},\mathbf{1}}\, \mathbf{l}(p\hspace{.1em}i):\, p \in \mathcal{A}^*,\ \mbox{$p\hspace{.1em}i$ is a prefix of a limit word of $\boldsymbol{\sigma}$}\}} \quad (i \in\mathcal{A}).
\]
Therefore, we have $\bigcup_{i\in\mathcal{A}} E(\mathcal{R}(i)) = \mathcal{R}$ and, hence, $E$ is a surjective piecewise isometry. Therefore, $E$ is bijective. Since the subtiles $\Ra(i)$, $i\in\A$, are disjoint in measure and 
\begin{equation} \label{e:R0n}
\mathcal{R}(w_0 w_1 \cdots w_{n-1}) = \bigcap_{\ell=0}^{n-1} E^{-\ell} \mathcal{R}(w_\ell),
\end{equation}
the injectivity of~$E$ implies that also the elements of the collection of ``length $n$ subtiles''\footnote{Not to be confused with the level $n$ subtiles introduced in Section~\ref{sec:seteq}.} $\mathcal{K}_{n}=\{\Ra(u)\,:\, u\in \Lg_{\bsigma} \hbox{ with } |u|=n\}$ are disjoint in measure.  By 
\eqref{eq:comm} the measure $\lambda_\mathbf{1} \circ \varphi$ is a shift invariant probability measure on~$X_{\boldsymbol{\sigma}}$. As by Theorem~\ref{prop:SadicUE} there is only one such measure, $\mu = \lambda_\mathbf{1} \circ \varphi$.
Now, essential disjointness of the elements of $\mathcal{K}_{n}$ implies that $\varphi(\mathbf{x}) \ne \varphi(\mathbf{y})$ for all distinct $\mathbf{x}, \mathbf{y} $ satisfying $\varphi(\mathbf{x}),\varphi(\mathbf{y} )  \in \mathcal{R} \setminus \bigcup_{n\in\mathbb{N}, K\in\mathcal{K}_n} \partial K$.
As, by \eqref{e:R0n} and Theorem~\ref{th:RauzyProperties}, $\lambda_\mathbf{1}(\partial K)=\mu(\varphi^{-1}(\partial K)) = 0$ for all $K\in\mathcal{K}_n$, $n\in\mathbb{N}$, the map~$\varphi$ is $\mu$-a.e.\ injective. Since surjectivity follows from the definition of $\varphi$ this proves $\mu$-a.e.\ bijectivity.
 Finally, using~\eqref{eq:comm}, the commutativity of the diagram \eqref{eq:Ediag}  follows from the bijectivity of $\varphi$.

Since $\mathcal{C}_\mathbf{1}$ forms a tiling of~$\mathbf{1}^\bot$ by assumption, the Rauzy fractal $\mathcal{R}$ is a fundamental domain of the lattice $\Lambda = \mathbf{1}^\bot \cap \mathbb{Z}^d$ spanned by $\mathbf{e}_1 - \mathbf{e}_i$, $i \in \mathcal{A} \setminus \{1\}$. 
But as $\pi_{\mathbf{u},\mathbf{1}}\, \mathbf{e}_i \equiv \pi_{\mathbf{u},\mathbf{1}}\, \mathbf{e}_1 \pmod \Lambda$ holds for each $i\in\mathcal{A}$, the canonical projection of~$E$ onto the torus $\mathbf{1}^\bot / \Lambda \simeq \mathbb{T}^{d-1}$ is equal to the rotation $T:  \mathbb{T}^{d-1}\to \mathbb{T}^{d-1},\,\mathbf{x} \mapsto\mathbf{x} + \pi_{\mathbf{u},\mathbf{1}}\, \mathbf{e}_{1}$. Thus, if we denote by $\overline{\varphi}$ the canonical projection of $\varphi$ to the torus $\mathbf{1}^\bot/\Lambda$,
the diagram
\[
\begin{CD}
X_{\boldsymbol{\sigma}} @> \Sigma >> X_{\boldsymbol{\sigma}} \\
@VV\overline{\varphi} V @VV\overline{\varphi} V\\
\mathbf{1}^\bot/\Lambda @> + \,\pi_{\mathbf{u},\mathbf{1}}\, \mathbf{e}_{1} >> \mathbf{1}^\bot/\Lambda
\end{CD}
\]
commutes. Note that $\overline{\varphi}$ is $m$ to~$1$ onto, where $m$ is the covering degree of~$\mathcal{C}_1$, and, hence, a bijection as $\mathcal{C}_1$ forms a tiling. This proves the first assertion of Theorem~\ref{th:TilingRotation}.

The second assertion of Theorem~\ref{th:TilingRotation} follows from the definition of a natural coding because the rotation~$T$ was defined in terms of an exchange of domains. Finally, due to \cite[Proposition~7]{Adamczewski:03}, the $C$-balance of~$\mathcal{L}_{\boldsymbol{\sigma}}$ implies that~$\mathcal{R}(i)$ is a bounded remainder set for each $i \in \mathcal{A}$, which also proves the last assertion.

\subsection{A metric result}\label{sec:metric}

As mentioned already in Remark~\ref{rem:commentThmRot}~(i), the assumptions of Theorems~\ref{th:RauzyProperties} and~\ref{th:TilingRotation} allow for a metric version of these results. To be more precise, let $S$ be a finite set of substitutions and consider the full shift $(S^\N, \Sigma, \nu)$, where $\nu$ is an ergodic $\Sigma$-invariant probability measure satisfying some mild conditions. Our aim is to state a version of Theorems~\ref{prop:SadicUE}, \ref{th:RauzyProperties}, and~\ref{th:TilingRotation} that is valid for $\nu$-a.e.\ $\bsigma\in S^\N$. This second main result of the present survey is also a special case of a result from Berth\'e, Steiner, and Thuswaldner~\cite{Berthe-Steiner-Thuswaldner}.

To state our result we need to introduce some new concepts. Let $S$ be a finite set of substitutions over the alphabet $\A=\{1,2,\ldots,d\}$ and consider the shift $(S^\N,\Sigma,\nu)$, where $\nu$ is some $\Sigma$-invariant probability measure on $S^\N$.  With each $\bsigma=(\sigma_n)_{n\ge 0}$ we associate the \emph{linear cocycle operator} $A(\bsigma)=(M_0)^t$ (recall that $M_0$ is the incidence matrix of $\sigma_0$) and define the \emph{Lyapunov exponents} $\vartheta_1,\ldots,\vartheta_d$ of $(S^\N,\Sigma,\nu)$ iteratively by
\begin{align}
\vartheta_1 + \vartheta_2 + \cdots + \vartheta_k & = \lim_{n\to\infty} \frac{1}{n} \int_{S^\N} \log \|\wedge^k \big(A(\Sigma^{n-1}(\bsigma)) \cdots A(\Sigma(\bsigma)) A(\bsigma)\big)\|_\infty\, d\nu(\bsigma) \nonumber \\
& = \lim_{n\to\infty} \frac{1}{n} \int_{S^\N} \log \|\wedge^k (M_{[0,n)})^t\|_\infty\, d\nu \label{eq:transposeequal}
\\
&= \lim_{n\to\infty} \frac{1}{n} \int_{S^\N} \log \|\wedge^k M_{[0,n)}\|_\infty\, d\nu \nonumber
\end{align}
for $1 \le k \le d$, where $\wedge^k$ denotes the $k$-fold wedge product. We say that $(S^\N,\Sigma,\nu)$ satisfies the \emph{Pisot condition} if
\begin{equation}\label{eq:lyapunov}
\vartheta_1 > 0 > \vartheta_2 \ge \cdots \ge \vartheta_d
\end{equation}
({\em cf.}~\cite[\S6.3]{Berthe-Delecroix}).
Using these definitions we get the following metric version of Theorems~\ref{prop:SadicUE}, \ref{th:RauzyProperties}, and~\ref{th:TilingRotation}.

\begin{theorem}[{see \cite[Theorem~3.3]{Berthe-Steiner-Thuswaldner}}] \label{t:3}
Let $S$ be a finite set of unimodular substitutions and assume that the shift $(S^\N,\Sigma,\nu)$ is ergodic and satisfies the Pisot condition. Assume further that $\nu$ assigns positive measure to every cylinder and that there exists a cylinder corresponding to a substitution with positive incidence matrix.
Then, for $\nu$-almost every $\boldsymbol{\sigma} \in S^\N$ the following assertions hold.
\begin{enumerate} \setcounter{enumi}{0}
\itemsep1ex
\item $(X_\bsigma,\Sigma)$ is minimal and uniquely ergodic (denote the unique $\Sigma$-invariant measure by $\mu$).
\item Each subtile $\Ra(i)$, $i\in\A$, is equal to the closure of its interior and satisfies $\lambda_\bone(\partial\Ra(i))=0$.
\item If the collection $\mathcal{C}_\mathbf{1}$ associated with~$\boldsymbol{\sigma}$ forms a tiling of $\mathbf{1}^\bot$ then $(X_\bsigma,\Sigma,\mu)$ is measurably conjugate to a rotation $T$ on $\mathbb{T}^{d-1}$, each element of $X_\bsigma$ is a natural coding of $T$ w.r.t.\ the partition $\{\Ra(i)\,:\, i\in \A\}$ of $\Ra$, and each $\Ra(i)$, $i\in\A$, is a bounded remainder set for~$T$.
\end{enumerate}
\end{theorem}

\subsection{Proof of the metric result}\label{sec:194}

In the present section we give a quite complete proof of Theorem~\ref{t:3}. The idea is to show that each of the conditions posed in Theorem~\ref{th:TilingRotation} is {\em generic}. A prominent tool in this proof is the {\em Multiplicative Ergodic Theorem} (also called {\em Oseledec Theorem}; see for instance~\cite[3.4.1~Theorem]{Arnold98}). Also the famous {\em Poincar\'e Recurrence Theorem} ({\em cf.~e.g.}~\cite[Theorem~1.4]{Walters:82}), which states that a.e. orbit in a measurable dynamical system $(X,T,\mu)$ starting in a set of positive measure $E$ hits $E$ infinitely often, will be used. In our setting, the Oseledec theorem has the following consequence.

\begin{proposition}\label{prop:oseledec}
Let $S$ be a finite set of unimodular substitutions over the alphabet $\A=\{1,2,\ldots,d\}$ and assume that the shift $(S^\N,\Sigma,\nu)$ is ergodic with Lyapunov exponents $\vartheta_1,\ldots,\vartheta_d$ satisfying the Pisot condition~\eqref{eq:lyapunov}. Assume further that $\nu$ assigns positive measure to every cylinder and that there exists a cylinder corresponding to a substitution with positive incidence matrix. Then for $\nu$-a.e. $\bsigma\in S^\N$ the following assertions hold.
\begin{enumerate}
\item[(i)]
The sequence $\bsigma$ is primitive and recurrent, thus the letter frequency vector $\bu=\bu(\bsigma)$ exists. 
\item[(ii)] 
For each $\varepsilon >0$ there exists $n_0=n_0(\varepsilon,\bsigma)$ such that the sequence of incidence matrices $\bM=(M_n)=(M_n(\bsigma))$ satisfies\footnote{Here $\Vert \cdot \Vert_2$ is the operator norm w.r.t.\ the Euclidean norm on $\R^d$.}
\[
\Vert (M_{[0,n)})^t\vert_{\bu^\bot} \Vert_2 < e^{(\vartheta_2+\varepsilon)n}
\]
for each $n\ge n_0$.
\end{enumerate}
\end{proposition}

\begin{proof}
Since $\nu$ puts positive mass on each cylinder, $\nu$-a.e.\ $\bsigma$ is recurrent by Poincar\'e recurrence. Together with the fact that there is a cylinder corresponding to a positive incidence matrix Poincar\'e recurrence also implies primitivity for $\nu$-a.e.\ $\bsigma$. Thus $\nu$-a.e. $\bsigma$ has a letter frequency vector $\bu$ by Proposition~\ref{prop:furstmatrix}. This proves (i).

In order to apply the Multiplicative Ergodic Theorem \cite[3.4.1~Theorem]{Arnold98} we need to assure {\em log-integrability} of the cocycle which, in our case, means that 
\begin{equation}\label{eq:logint}
\max\{0, \log\Vert M_0(\bsigma)\Vert_2\} \in L^1(S^\N,\nu).
\end{equation}
Since $S$ finite, the quantity $\max\{0, \log\Vert M_0(\bsigma)\Vert_2\}$ is bounded and therefore \eqref{eq:logint} always holds. Thus, because $\vartheta_1$ is a simple Lyapunov exponent, \cite[3.4.1~Theorem]{Arnold98} implies that for $\nu$-a.e.\ $\bsigma$  there is a hyperplane $\mathcal{H}=\mathcal{H}(\bsigma)\subset \R^d$ such that
$\lim_{n\to\infty} \frac1n \log \Vert  M_{[0,n)}(\bsigma)^t|_\mathcal{H}  \Vert_2 \le \vartheta_2$.
This implies that for each $\varepsilon >0$ there is $n_0=n_0(\varepsilon,\bsigma)$ such that
\begin{equation}\label{eq:Hubot2}
\Vert M_{[0,n)}(\bsigma)^t\vert_\mathcal{H} \Vert_2 < e^{(\vartheta_2+\varepsilon)n} 
\end{equation}
holds for $n\ge n_0$. 
It remains to show that $\mathcal{H}=\bu^\bot$. However, this follows because for $\bx\not\in\bu^\bot$ we have that  $\langle M_{[0,n)}(\bsigma)^t \bx, \bone \rangle = \langle \bx, M_{[0,n)}(\bsigma) \bone \rangle$ is unbounded because for large $n$ the vector $M_{[0,n)}(\bsigma) \bone$ is a large vector close to the line $\R_+\bu$. Thus the only hyperplane for which \eqref{eq:Hubot2} can possibly hold is $\mathcal{H}=\bu^\bot$ and (ii) follows.

\end{proof}

Proposition~\ref{prop:oseledec} is now used in order to show that balance is generic for elements of a shift $(S^\N,\Sigma,\nu)$ satisfying the Pisot condition.

\begin{lemma}\label{lem:genericBalance}
Let $S$ be a finite set of unimodular substitutions over the alphabet $\A=\{1,2,\ldots,d\}$ and assume that the shift $(S^\N,\Sigma,\nu)$ is ergodic and satisfies the Pisot condition~\eqref{eq:lyapunov}. Assume further that $\nu$ assigns positive measure to every cylinder and that there exists a cylinder corresponding to a substitution with positive incidence matrix. Then the sets
\[
S(C) =\{\bsigma\in S^\N \;:\; \mathcal{L}_\bsigma \hbox{ \rm is $C$-balanced}  \} \qquad(C \in \N)
\]
satisfy
\[
\lim_{C\to\infty} \nu(S(C))=1,
\]
{\em i.e.}, balance of $\mathcal{L}_\bsigma$ is a generic property of $\bsigma\in S^\N$.
\end{lemma}

\begin{proof}
By Proposition~\ref{prop:oseledec} we see that for $\nu$-a.e.\ $\bsigma\in S^\N$ 
the sequence is primitive and recurrent, and for the letter frequency vector $\bu=(u_1,\ldots,u_d)^t$ (with $\Vert\bu\Vert_1=1$) we have 
\begin{equation}\label{eq:balFinit}
\sum_{n\ge 0}\Vert (M_{[0,n)})^t\vert_{\bu^\bot} \Vert_2 < \infty.
\end{equation}
We assume that $\bsigma\in S^\N$ has all these properties and follow the proof of \cite[Theorem~5.8]{Berthe-Delecroix}. Let $w\in X_\bsigma$ be arbitrary. Since by the proof of Proposition~\ref{prop:sadicminimal}~(iii) each element of $X_\bsigma$ has the same language, each factor $v$ of $w$ is a factor of a limit sequence of $\bsigma$ and, hence, by \eqref{def:sadicsequence} can be written as
\begin{equation}\label{eq:DMuvw}
v=p_0\sigma_0(p_1\ldots \sigma_{N-2}(p_{N-1}\sigma_{N-1}(x)s_{N-1})\ldots s_1)s_0
\end{equation}
where $p_n$ and $s_n$ is a prefix and a suffix of $\sigma_n(i)$ for some $i\in \A$, respectively,  for each $0\le n\le N-1$ and $x$ is a factor of $\sigma_N(i)$ for some $i\in \A$. To make the notation easier we set $p_N=x$ and $s_N=\varepsilon$. We mention that \eqref{eq:DMuvw} is the Dumont-Thomas decomposition of $v$ which was first introduced in \cite{Dumont-Thomas:89}. Using \eqref{eq:DMuvw} and denoting by $\be_1,\ldots, \be_d$ the standard basis vectors of $\R^d$ we have 
\[
\begin{split}
|v|_i-|v|u_i &= \sum_{n=0}^N
( |\sigma_{[0,n)}(p_n)|_i -  |\sigma_{[0,n)}(p_n)| u_i + |\sigma_{[0,n)}(s_n)|_i -  |\sigma_{[0,n)}(s_n)| u_i)\\
&= \sum_{n=0}^N \langle
\be_i-u_i(\be_1+\cdots+\be_d), M_{[0,n)}\mathbf{l}(p_n + s_n)
\rangle
\end{split}
\]
for each $i\in \A$. Since $u_1+\cdots+u_d=1$ we see that $\be_i-u_i(\be_1+\cdots+\be_d) \in \bu^\bot$. This can be used to get
\begin{align*}
\big|
|v|_i-|v|u_i 
\big| 
&\le 
\sum_{n=0}^N \big|
\langle
\be_i-u_i(\be_1+\cdots+\be_d), M_{[0,n)}\mathbf{l}(p_n + s_n)
\rangle\big|\\
&=
\sum_{n=0}^N \big|
\langle
(M_{[0,n)})^t(\be_i-u_i(\be_1+\cdots+\be_d)), \mathbf{l}(p_n + s_n)
\rangle\big|\\
& \le 2\sqrt{d} \sum_{n=0}^N
\Vert (M_{[0,n)})^t\vert_{\bu^\bot} \Vert_2 \Vert M_{n}\Vert_2.
\end{align*}
Since $S$ is a finite set, the quantity $\Vert M_n\Vert_2$ is uniformly bounded in $n$. Thus, using \eqref{eq:balFinit}  this implies that $w$ is finitely balanced. Since $\bsigma$ was taken from a set of full measure $\nu$ of $S^\mathbb{N}$ this finishes the proof.

\end{proof}

Before we can put everything together we need to deal with the genericness of algebraic irreducibility. This has been done in \cite[Lemma~8.7]{Berthe-Steiner-Thuswaldner} in the following fashion.

\begin{lemma}\label{lem:genericIrred}
Let $S$ be a finite set of unimodular substitutions over the alphabet $\A=\{1,2,\ldots,d\}$ and assume that the shift $(S^\N,\Sigma,\nu)$ is ergodic and satisfies the Pisot condition~\eqref{eq:lyapunov}. If $\nu$-a.e.\ sequence $\bsigma\in S^\N$ is primitive then $\nu$-a.e.\ sequence $\bsigma\in S^\N$ is algebraically irreducible.
\end{lemma}

\begin{proof}[Sketch]
Let $\bsigma$ be a generic sequence with sequence of incidence matrices $\bM=(M_n)$ and fix $k\in \N$. Then for $\ell \to \infty$ the matrix $M_{[k,\ell)}$ maps the unit sphere into an ellipse whose largest semiaxis tends to infinity and all of whose other semiaxes tend to zero by the Pisot condition. We prove that for $\ell$ large enough there can be only one eigenvalue $\lambda$ with $|\lambda|\ge 1$.

Indeed, if $\ell$ is large enough then $M_{[k,\ell)}$ is strictly positive, thus there is a dominant Perron-Frobenius eigenvalue $\lambda_0 >1$. It corresponds to an eigenvector $\bw_0$ with strictly positive entries. Suppose that there is another real eigenvalue $\lambda$ with $|\lambda|\ge 1$ and corresponding eigenvector $\bw$. Since the image of the unit sphere under $M_{[k,\ell)}$ is an ellipse with the above mentioned properties, the corresponding eigenvector has to have a direction close to $\bw_0$ for $\ell$ large, because otherwise its length would be shrunk by the application of $M_{[k,\ell)}$ as can be seen in Figure~\ref{fig:ExtEll}. 
\begin{figure}[hh]
\includegraphics[width=0.7\textwidth]{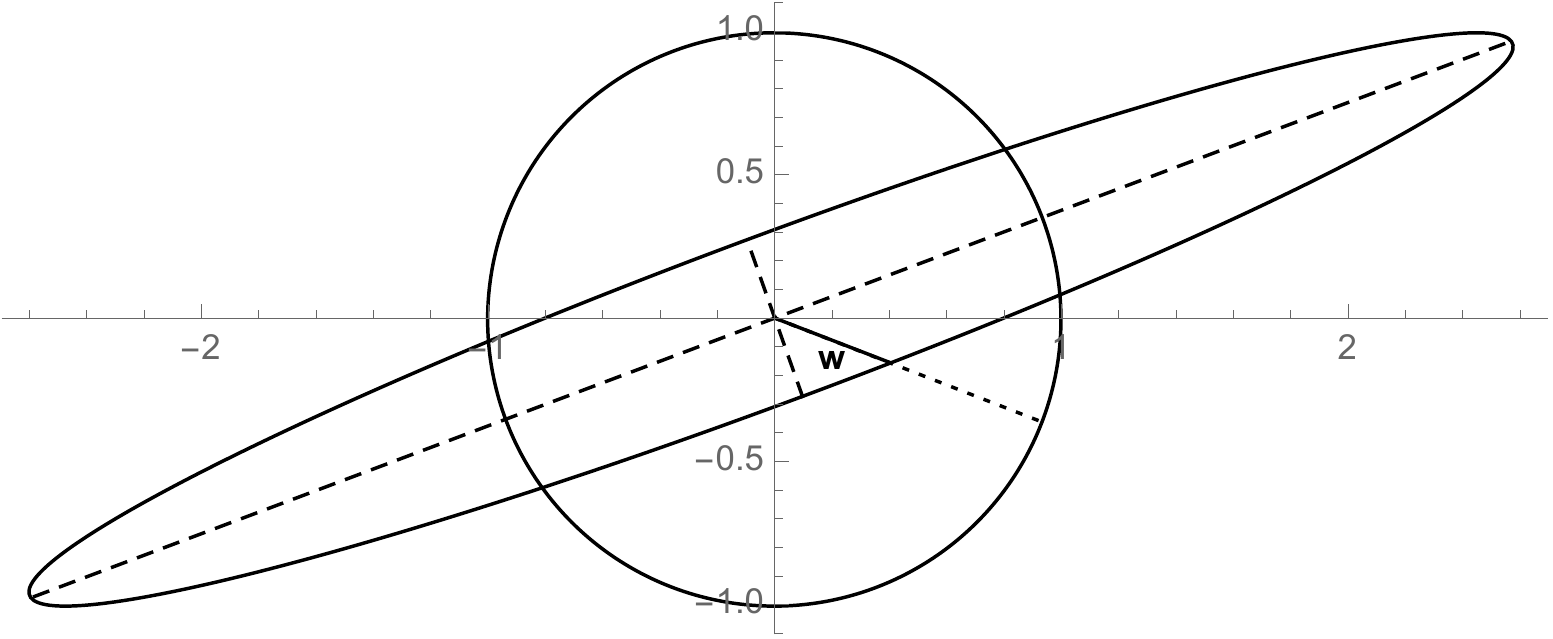}
\caption{An illustration of the elliptic image of the unit circle under $M_{[k,\ell)}$. The dashed lines are the axes of the ellipse, the largest axis being the direction of the Perron-Frobenius eigenvector $\bw_0$. If the indicated vector $\bw$ is an eigenvector of $M_{[k,\ell)}$ for another eigenvalue, its direction has to be far from the direction of $\bw_0$ (because not all of its entries can be positive). This entails that its length is less than $1$ and so it can only correspond to an eigenvalue less than $1$ in modulus. 
\label{fig:ExtEll}}
\end{figure}
Thus, if $\ell$ is large enough then $\bw$ must have strictly positive entries. However, such an eigenvector has to belong to the Perron-Frobenius eigenvalue, a contradiction. The case of nonreal eigenvalues can be treated similarly. Thus $M_{[k,\ell)}$ has only one eigenvalue of modulus greater than or equal to $1$. Since $M_{[k,\ell)}$ is an unimodular integer matrix, it cannot have $0$ as an eigenvalue. This implies that the characteristic polynomial of $M_{[k,\ell)}$ is irreducible and, hence, $\bsigma$ is algebraically irreducible. Indeed, we even proved that the characteristic polynomial of $M_{[k,\ell)}$ is the minimal polynomial of the Pisot number $\lambda_0$.

\end{proof}

We now have all the necessary ingredients to finish the proof of Theorem~\ref{t:3}. 

\begin{proof}[Conclusion of the proof of Theorem~\ref{t:3}]
We show that the conditions of Theorem~\ref{th:TilingRotation} are satisfied for $\nu$-a.e. $\bsigma\in S^\N$. To keep things simple we give the proof only for $\nu$ being a Bernoulli measure. Primitivity and algebraic irreducibility hold $\nu$-a.e.\ by Proposition~\ref{prop:oseledec}~(i) and Lemma~\ref{lem:genericIrred}, respectively.

It remains to deal with the condition involving recurrence and balance. We claim that there is $C\in \N$ such that 
\begin{equation}\label{eq:th:3pf}
\nu([\sigma_0,\ldots,\sigma_{\ell-1}] \cap \Sigma^{-\ell} S(C) ) > 0 \hbox{ for each } (\sigma_n)\in S^\N \hbox{ and each } \ell \ge 0.
\end{equation}
Indeed, since $\nu$ is a Bernoulli measure, $[\sigma_0,\ldots,\sigma_{\ell-1}]$ is independent from $\Sigma^{-\ell} S(C)$. Thus we have
\[
\nu([\sigma_0,\ldots,\sigma_{\ell-1}] \cap \Sigma^{-\ell} S(C) ) = \nu([\sigma_0,\ldots,\sigma_{\ell-1}])
\nu(S(C))
\]
and the claim \eqref{eq:th:3pf} follows because $\nu([\sigma_0,\ldots,\sigma_{\ell-1}])>0$ by assumption and $\nu(S(C))>0$ for $C$ large enough by Lemma~\ref{lem:genericBalance}. By another application of Poincar\'e recurrence \eqref{eq:th:3pf} yields that for $\nu$-a.e.\ $\bsigma \in S^\N$ and for every $\ell \in \N$ there is $n>0$ such that $\Sigma^n\bsigma \in [\sigma_0,\ldots,\sigma_{\ell-1}]$ and $\Sigma^{n+\ell}\bsigma \in S(C)$. 

Summing up we see that the assumptions of Theorem~\ref{th:TilingRotation} are satisfied for $\nu$-a.e. $\bsigma\in S^\N$. Thus Theorem~\ref{t:3} (1) follows from Theorem~\ref{prop:SadicUE}, Theorem~\ref{t:3} (2) follows from Theorem~\ref{th:RauzyProperties}, and Theorem~\ref{t:3} (3) follows from Theorem~\ref{th:TilingRotation}.

\end{proof}

\begin{remark}
With small amendments in the conclusion of the proof of Theorem~\ref{t:3} it is possible to prove Theorem~\ref{t:3} for sofic subshifts $(X,\Sigma,\nu)$ of $(S^\N,\Sigma,\nu)$. Even the case of infinitely many substitutions ({\it i.e.}, $|S|=\infty$) can be treated provided that the log-integrability condition \eqref{eq:logint} is satisfied. In this case one has to deal with the {\em $S$-adic graph} introduced in \cite{Berthe-Delecroix}. As mentioned above, the general result is contained in \cite{Berthe-Steiner-Thuswaldner}.
\end{remark}

\subsection{Corollaries for Arnoux-Rauzy and Brun systems}\label{sec:corex}

We now want to apply the two main theorems to Arnoux-Rauzy as well as Brun $S$-adic systems. Since these systems and their related generalized continued fraction algorithms have been studied quite well in the literature this will yield unconditional results on measurable conjugacy to a torus rotation, natural codings, and bounded remainder sets.

We start with the case of Arnoux-Rauzy systems. Let $S=\{\sigma_1,\sigma_2,\sigma_3\}$ be the set of Arnoux-Rauzy substitutions defined in \eqref{eq:ARsubs}. First we give a version of Theorem~\ref{t:3} for the $S$-adic sequences taken from $S^\N$.

\begin{corollary}[{see \cite[Theorem~3.8]{Berthe-Steiner-Thuswaldner}}]\label{cor:AR1}
Let $S$ be the set of Arnoux-Rauzy substitutions defined in \eqref{eq:ARsubs} and consider the full shift $(S^\N,\Sigma, \nu)$ equipped with an ergodic invariant measure $\nu$ that assigns positive mass to each cylinder. Then $\nu$-a.e. $\bsigma\in S^\N$ defines an $S$-adic system $(X_\bsigma,\Sigma)$ that is measurably conjugate to a rotation $T$ on the $2$-torus $\mathbb{T}^2$. Moreover, each element of $X_\bsigma$ forms a natural coding of $T$ w.r.t.\ the partition $\{\Ra_i\,:\, i\in \A\}$ defined by the subtiles of the Rauzy fractal $\Ra$. Each of these subtiles is a bounded remainder set of $T$.
\end{corollary} 

\begin{proof}[Sketch]
It is easy to see that each cylinder containing each of the three substitutions has positive incidence matrix.
Thus the result follows from Theorem~\ref{t:3} if we can establish that $(S^\N,\Sigma,\nu)$ satisfies the Pisot condition and that for $\nu$-a.e.\ $\bsigma\in S^\N$ the associated collection $\Co_\bone$ of Rauzy fractals forms a tiling. The fact that the Pisot condition holds was proved by Avila and Delecroix~\cite{AD15}. The tiling property is a consequence  of Proposition~\ref{prop:tilingR}. Indeed, assertion~(iii) of this proposition holds by the following results. Firstly, strong coincidence follows from \cite[Proposition~4]{Barge-Stimac-Williams:13} (or \cite[Section~9]{Berthe-Steiner-Thuswaldner} where ``negative coincidence'' was used). The other assertion from Proposition~\ref{prop:tilingR}~(iii) is a weaker form of the geometric finiteness property which holds by Proposition~\ref{prop:ARgeometric}
(see also \cite[Theorem~4.7]{Berthe-Jolivet-Siegel:12}).

\end{proof}

With help of the balance properties of Arnoux-Rauzy sequences proved in \cite{Berthe-Cassaigne-Steiner} it is possible to use Theorem~\ref{th:TilingRotation} in order to show results for concrete Arnoux-Rauzy systems. For instance it is proved in \cite[Corollary~3.9]{Berthe-Steiner-Thuswaldner} that any linearly recurrent Arnoux-Rauzy sequence with recurrent directive sequence generates an $S$-adic system $(X_\bsigma,\Sigma)$ that is measurably conjugate to a rotation on a $2$-torus.

\medskip

For the second class of examples let $S=\{\sigma_1,\sigma_2,\sigma_3\}$ be the set of Brun substitutions defined in \eqref{eq:brun}. In this case a version of Theorem~\ref{t:3} completely analogous to Corollary~\ref{cor:AR1} holds. 

\begin{corollary}[{see \cite[Theorem~3.10]{Berthe-Steiner-Thuswaldner}}]\label{cor:BRUN11}
Let $S$ be the set of Brun substitutions defined in \eqref{eq:brun} and consider the full shift $(S^\N,\Sigma, \nu)$ equipped with an ergodic invariant measure $\nu$ that assigns positive mass to each cylinder. Then $\nu$-a.e. $\bsigma\in S^\N$ defines an $S$-adic system $(X_\bsigma,\Sigma)$ that is measurably conjugate to a rotation $T$ on the $2$-torus $\mathbb{T}^2$. Moreover, each element of $X_\bsigma$ forms a natural coding of $T$ w.r.t.\ the partition $\{\Ra_i\,:\, i\in \A\}$ defined by the subtiles of the Rauzy fractal $\Ra$. Each of these subtiles is a bounded remainder set of $T$.
\end{corollary} 

\begin{proof}[Sketch]
First observe that $\sigma_1\sigma_2\sigma_1\sigma_2$ has positive incidence matrix.
One uses again \cite{AD15} to ensure that the Pisot condition holds (see also~\cite{FUKE96,MEESTER,Schratzberger:98} for similar results). The tiling property follows from geometric coincidence which is established in \cite{BBJS16} for the Brun class. 

\end{proof}

Contrary to the Arnoux-Rauzy continued fraction algorithm, the Brun algorithm 
can be performed for all elements $(x_1,x_2)\in\Delta$ with $\Delta$ as in \eqref{eq:deltaset}. Thus, using Brun systems we get natural codings for a.a.\ torus rotations $\bt\in\mathbb{T}^2$. 

\begin{corollary}[{see \cite[Corollary~3.12]{Berthe-Steiner-Thuswaldner}}]\label{cor:BRUN2}
Let $S$ be the set of Brun substitutions defined in \eqref{eq:brun}. Then for almost every $\bt\in\mathbb{T}^2$ (w.r.t.\ the Haar measure on $\mathbb{T}^2$) there is $\bsigma\in S^\N$ such that the shift $(X_\bsigma,\Sigma)$ is measurably conjugate to the rotation $T_\bt$ by $\bt$ on $\mathbb{T}^2$. Moreover, the sequences in $X_\bsigma$ form natural codings of the rotation $T_\bt$.
\end{corollary}

To create concrete examples of Brun $S$-adic shifts being measurably conjugate to a rotation, one can use Theorem~\ref{th:TilingRotation} together with the balance results established in \cite{Delecroix-Hejda-Steiner}.

\section{Concluding remarks: Natural extensions, flows, and their Poincar\'e sections}

It remains to extend the ideas and results presented in Section~\ref{sec:NatlGauss} to generalized continued fraction algorithms and $S$-adic systems on $d$ letters. This is the subject of the ongoing paper by Arnoux {\it et al.}~\cite{ABMST:18}. 

It is possible to study natural extensions of generalized continued fraction algorithms (see for instance~\cite{AL:15,Arnoux-Nogueira}). In the way we do it in \cite{ABMST:18}, the analogs of the $L$-shaped regions of Section~\ref{sec:NatlGauss}  are ``Rauzy-Boxes'' which are defined as suspensions of $S$-adic Rauzy fractals. They were introduced in the $S$-adic setting in \cite[Section~2.9]{Berthe-Steiner-Thuswaldner} but have been studied earlier in the substitutive case, see for instance Ito and Rao~\cite{Ito-Rao:06}. These Rauzy boxes allow nonstationary Markov partitions for so-called ``mapping families'' in the sense studied by Arnoux and Fisher~\cite{AF:05} that can be visualized by restacking $S$-adic Rauzy fractals in a suitable way.

Also Artin's idea of viewing continued fraction algorithms as Poincar\'e sections of the geodesic flow on $\mathrm{SL}_2(\Z)\backslash \mathrm{SL}_2(\R)$ can be generalized. In this generalization the role of the geodesic flow is played by the \emph{Weyl Chamber Flow}, a diagonal $\R^{d-1}$-action on the space $\mathrm{SL}_d(\Z)\backslash \mathrm{SL}_d(\R)$ of $d$-dimensional lattices. It turns out that each coordinate direction of this $\R^{d-1}$-action has a Poincar\'e section which is arithmetically coded by a generalized continued fraction algorithm. Geometrically, this is visualized by deforming a given Rauzy box (one for each coordinate) by the action of the Weyl Chamber Flow and restacking it accordingly as soon as a Poincar\'e section is reached. 

\medskip

Details of all this will be contained in~\cite{ABMST:18}.

\section{Acknowledgements}

I warmly thank Shigeki Akiyama and Pierre Arnoux for inviting me to contribute to this volume. I very much appreciate their constant encouragement and support, and the many discussions I had with them. Moreover, I am indebted to Val\'erie Berth\'e, S\'ebastien Labb\'e, and Wolfgang Steiner for their suggestions.

\bibliographystyle{spmpsci}
\bibliography{MorletJT}

\end{document}